\documentclass[12pt]{article}

\usepackage{amssymb, amsthm, amsmath}	
\usepackage{enumitem, hyperref}
\usepackage{comment}
\usepackage{url}
\usepackage{color}
\usepackage{float}
\usepackage{todonotes}
\usepackage[margin=1.2in]{geometry}

\usepackage{longtable}

\usepackage{tikz}
\usetikzlibrary{arrows, chains, positioning, shapes.symbols,calc}
\usepackage{tikz-3dplot}
\tdplotsetmaincoords{70}{110}
\usepackage{tikz-cd}

\numberwithin{equation}{section}
\theoremstyle{plain}
\newtheorem{theorem}{Theorem}[section]
\newtheorem{lemma}[theorem]{Lemma}
\newtheorem{proposition}[theorem]{Proposition}
\newtheorem{corollary}[theorem]{Corollary}

\theoremstyle{definition}
\newtheorem{definition}[theorem]{Definition}
\newtheorem{notation}[theorem]{Notation}
\newtheorem{remark}[theorem]{Remark}
\newtheorem{example}[theorem]{Example}

\newtheorem{conjecture}[theorem]{Conjecture}

\newtheorem*{rmk}{Remark}

\newcommand{\C}{\mathbb{C}}
\newcommand{\Q}{\mathbb{Q}}
\newcommand{\R}{\mathbb{R}}
\newcommand{\Z}{\mathbb{Z}}
\renewcommand{\P}{\mathbb{P}}
\newcommand{\CP}{\mathbb{CP}}

\newcommand{\mB}{\mathcal{B}}

\newcommand{\mE}{\mathcal{E}}
\newcommand{\mF}{\mathcal{F}}
\newcommand{\mG}{\mathcal{G}}
\newcommand{\mK}{\mathcal{K}}
\newcommand{\mO}{\mathcal{O}}
\newcommand{\mH}{\mathcal{H}}
\newcommand{\mI}{\mathcal{I}}

\newcommand{\mL}{\mathcal{L}}
\newcommand{\mN}{\mathcal{N}}
\newcommand{\mP}{\mathcal{P}}
\newcommand{\mQ}{\mathcal{Q}}
\newcommand{\mS}{\mathcal{S}}
\newcommand{\mU}{\mathcal{U}}
\newcommand{\mV}{\mathcal{V}}

\newcommand{\ba}{\mathbf{a}}
\newcommand{\bb}{\mathbf{b}}
\newcommand{\be}{\mathbf{e}}

\newcommand{\bx}{\mathbf{x}}

\newcommand{\mLi}{\mathcal{L}_{\textup{irr}}}
\newcommand{\mLr}{\mathcal{L}_{\textup{red}}}
\newcommand{\p}{\partial}
\newcommand{\iprod}{\mathbin{\lrcorner} \,}

\newcommand{\tL}{\widetilde{L}}

\newcommand{\tH}{\widetilde{H}}

\newcommand{\tp}{\widetilde{p}}
\newcommand{\tX}{\widetilde{X}}
\newcommand{\tD}{\widetilde{D}}

\newcommand{\tP}{\widetilde{P}}
\newcommand{\tQ}{\widetilde{Q}}
\newcommand{\tV}{\widetilde{V}}
\newcommand{\newpar}{\vspace{0.15in}}

\newcommand{\tv}{\tilde{v}}

\newcommand{\bc}{\mathbf{c}}
\newcommand{\bv}{\mathbf{v}}

\DeclareMathOperator{\rk}{rank}
\DeclareMathOperator{\spn}{span}
\DeclareMathOperator{\codim}{codim}
\DeclareMathOperator{\parch}{par-ch}
\DeclareMathOperator{\parc}{par-c}
\DeclareMathOperator{\pardeg}{par-deg}
\DeclareMathOperator{\ch}{ch}
\DeclareMathOperator{\Hom}{Hom}
\DeclareMathOperator{\Gr}{Gr}

\DeclareMathOperator{\Pic}{Pic}
\DeclareMathOperator{\Bl}{Bl}

\DeclareMathOperator{\pr}{pr}
\DeclareMathOperator{\vol}{vol}

\DeclareMathOperator{\PGL}{PGL}
\DeclareMathOperator{\Tan}{Tan}
\DeclareMathOperator{\Irr}{Irr}
\DeclareMathOperator{\wt}{wt}

\DeclareMathOperator{\Ric}{Ric}
\DeclareMathOperator{\Cl}{Cl}
\DeclareMathOperator{\loc}{loc}
\DeclareMathOperator{\FS}{FS}
\DeclareMathOperator{\RF}{RF}
\DeclareMathOperator{\KE}{KE}

\DeclareMathOperator{\PG}{PG}

\newcommand{\Address}{{% additional braces for segregating \footnotesize
		\bigskip
		\begin{flushright}
			\textsc{Loughborough University, Department of Mathematics\\
				Schofield Building, LE11 3TU, United Kingdom} \\
				\small{\href{mailto:m.de-borbon@lboro.ac.uk}{\nolinkurl{m.de-borbon@lboro.ac.uk}}} \\
			\vspace{3mm}
			\textsc{King's College London, Department of Mathematics\\
				Strand, WC2R 2LS, United Kingdom} \\
				\small{\href{mailto:dmitri.panov@kcl.ac.uk}{\nolinkurl{dmitri.panov@kcl.ac.uk}}}
		\end{flushright}
}}

\makeatletter
\def\blfootnote{\xdef\@thefnmark{}\@footnotetext}
\makeatother

\title{A Miyaoka-Yau inequality for hyperplane arrangements in $\mathbb {CP}^n$ }

\date{\today}
\author{Martin de Borbon and Dmitri Panov}

\begin{document}

\maketitle
\blfootnote{\noindent 2020 \textit{Mathematics Subject Classification}. Primary 14N20, 14J60; Secondary 32S22, 52C35.}

\begin{abstract}
	Let \(\mH\) be a hyperplane arrangement in \(\CP^n\).
	We define a quadratic form \(Q\) on \(\R^{\mH}\) that is entirely determined by the intersection poset of \(\mH\).
	Using the Bogomolov-Gieseker inequality for parabolic bundles, we show that if \(\ba \in \R^{\mH}\) is such that the weighted arrangement \((\mH, \ba)\) is \emph{stable}, then \(Q(\ba) \leq 0\). 
	
	As an application, we consider the symmetric case where all the weights are equal. The inequality \(Q(a, \ldots, a) \leq 0\) gives a lower bound for the total sum of multiplicities of codimension \(2\) intersection subspaces of \(\mH\). The lower bound is attained when every \(H \in \mH\) intersects all the other members of \(\mH \setminus \{H\}\) along \((1-2/(n+1))|\mH| + 1\) codimension \(2\) subspaces; extending from \(n=2\) to higher dimensions a condition found by Hirzebruch for line arrangements in the complex projective plane.   
\end{abstract}

\tableofcontents

\section{Introduction}
Let \(\mH\) be an arrangement of complex hyperplanes \(H \subset \CP^n\).
Let \(N = |\mH|\) be the number of hyperplanes, which we assume to be finite, and fix a labelling, say \(\mH = \{H_1, \ldots, H_N\}\). We assume that
\(H_i \neq H_j\) for \(i \neq j\).
In this paper,
we introduce a quadratic form \(Q: \R^N \to \R\) associated to the arrangement \(\mH\) and show that \(Q \leq 0\) on a certain convex polyhedral cone \(C \subset \R_{\geq 0}^N\), that we call the \emph{semistable cone} of \(\mH\). \newpar

\noindent {\bf The quadratic form \(Q\) of \(\mH\).}
The multiplicity \(m_L\) of a linear subspace \(L \subset \CP^n\) is the number of hyperplanes \(H_i \in \mH\) that contain \(L\).
A codimension \(2\) subspace \(L \subset \CP^n\) is \emph{reducible} if \(m_L = 2\), and \emph{irreducible} if \(m_L \geq 3\).
Let \(\sigma_i\) be the number of codimension \(2\) irreducible subspaces contained in the hyperplane \(H_i\). 
The quadratic form \(Q: \R^N \to \R\) of \(\mH\) is defined by the symmetric matrix with entries given by
\[
Q_{ij} = \begin{cases}
- (n+1) \sigma_i + 2n &\text{ if } i = j \,, \\
-2 &\text{ if }  i \neq j \text{ and } L = H_i \cap H_j \text{ is reducible}, \\
\,\, n-1 &\text{ if } i \neq j \text{ and } L = H_i \cap H_j \text{ is irreducible} .
\end{cases}
\]

\begin{rmk}
If \(n=2\), then \((-Q_{ij})\) is equal to the matrix given by \cite[Equation (3)]{hirzebruch2} in the context of H\"ofer's formula for `the proportionality' \(3c_2 - c_1^2\) of algebraic surfaces obtained as branched covers of the projective plane branching along a line arrangement. 
\end{rmk}

\noindent{\bf The semistable and stable cones.}
A \emph{basis} \(\mB\) of \(\mH\) is a subset \(\mB \subset \mH\) of \(n+1\) linearly independent hyperplanes. The indicator function \(\be_{\mB}\) of a basis \(\mB\) is the vector in \(\R^N\) whose \(i\)-th component is \(1\) if \(H_i \in \mB\) and \(0\) otherwise. The \emph{semistable cone} \(C\) is the {\it conical hull} of the vectors \(\be_{\mB}\) with \(\mB\) ranging over all bases of \(\mH\).
Put differently, if \(P\) is the convex hull of the vectors \(\be_{\mB}\)\,, then \(C = \R_{\geq 0} \cdot P\) is the cone over \(P\). 
The convex set \(P\) is called the \emph{matroid polytope} of \(\mH\).
%\todo[inline]{Maybe  The convex set \(P\) is called the \emph{matroid polytope}? Indeed, this is not our definition, and in a way this is a result?} 

Dually, the cone \(C\) can be described in terms of defining linear inequalities.
Let \(\mL\) be the set of non-empty and proper linear subspaces \(L \subset \CP^n\)
obtained by intersecting hyperplanes in \(\mH\). The semistable cone \(C\)
is the set of points \(\ba = (a_1, \ldots, a_N) \in \R^N\) with non-negative components \(a_i \geq 0\) such that, for every \(L \in \mL\) , the following holds:
\[
\sum_{i \,|\, L \subset H_i} a_i \leq \frac{\codim L}{n+1} \cdot \sum_{i=1}^{N} a_i \,.
\]

The \emph{stable cone} \(C^{\circ}\) is the interior of the semistable cone \(C\). Equivalently, \(C^{\circ}\) is the subset of \(\R_{>0}^N\) where the above inequalities hold strictly. 

\begin{rmk}
The hyperplanes \(H_i \subset \CP^n\) correspond to points \(p_i \in (\CP^n)^*\). 
If the numbers \(a_i\) are positive integers, then \(\ba = (a_1, \ldots, a_N)\) belongs to \(C\) (resp. \(C^{\circ}\)) if and only if the weighted configuration of points \(\{(p_i, a_i)\}_{i=1}^N\) is semistable (resp. stable) in the sense of Geometric Invariant Theory, as follows from \cite[Theorem 11.2]{dolgachev}.
\end{rmk}

\noindent {\bf Miyaoka-Yau inequality.}
Our main result, in its most general form, is Theorem \ref{thm:maingeneral}. It 
asserts that the quadratic form \(Q\) is non-positive on the semistable cone:
\[
C \subset \{ Q \leq 0 \} \,.
\]
We think of this as a version of the Miyaoka-Yau inequality because we have the following expectation.\newpar

{\bf Conjecture:} suppose that \(\ba \in C^{\circ}\) is such that \(Q(\ba) = 0\) and \(0 < a_i < 1\) for all \(i\). Then there is a K\"ahler metric on \(\CP^n\) of constant holomorphic sectional curvature with cone angles \(2\pi\alpha_i\) along the hyperplanes \(H_i\)\,, where \(\alpha_i = 1-a_i\). (For a more precise version, in the case of zero curvature, see Conjecture \ref{conjecture}.)\newpar

This conjecture holds for \(n=1\), in which case the quadratic form \(Q\) is identically zero, and the existence of the metric is proven in \cite{mcowen} and \cite{troyanov} (see also \cite{dimasphere} for an approach similar to this paper).
For \(n=2\) and zero curvature the conjecture holds  by \cite[Theorem 1.12]{panov}. 
An important class of examples that fit into the above conjecture are complex reflection arrangements, for which the existence of the metric is proved in \cite{chl}.
In light of our expectation, we propose the following.
\newpar

\emph{Problem:} classify the arrangements for which there exists \(\ba \in C^{\circ}\)  with \(Q(\ba) = 0\). \newpar

Note that if \(\mH\) is as above then its quadratic form  \(Q\) is degenerate and negative semidefinite.
We believe that such arrangements should be very special and rigid for $n\ge 2$.   For example, in dimension $2$, their complements are $K(\pi,1)$, see \cite[\S 11]{dimapetrunin}.
\newpar

\noindent {\bf Hirzebruch arrangements.} 
We analyse the above problem in the particular case that \(\ba\) is the vector \(\mathbf{1} = (1, \ldots, 1) \in \R^N\). It is easy to see that  \(\mathbf{1}\) belongs to the kernel of \(Q\) precisely when every \(H_i \in \mH\) intersects the other members of \(\mH \setminus \{H_i\}\) along 
\[
\left( 1 - \frac{2}{n+1} \right) \cdot N + 1
\] 
codimension \(2\) subspaces.
If \(\mH\) satisfies this property, and the additional requirement that \(\mathbf{1} \in C\), we say that \(\mH\) is a \emph{Hirzebruch arrangement}.
We show that complex reflection arrangements, defined by irreducible unitary reflection groups \(G \subset U(n+1)\), are Hirzebruch. 
It remains an open problem to understand if
all Hirzebruch arrangements come from reflection groups.
If \(n=2\), this is an old question posed by Hirzebruch in \cite[\S 3]{hirzebruch2}, and it has a positive answer for \emph{real} line arrangements by \cite{dima}.\newpar

\noindent {\bf The Hirzebruch quadratic form of a matroid.} In Section \ref{sec:matroid}\,, we state our results in the wider context of matroid theory. Given a simple matroid \(M\) on the ground set with \(N\) elements, we define a quadratic form \(Q: \R^N \to \R\), which we call \emph{the Hirzebruch quadratic form} of \(M\). Our results show that, if \(M\) is representable over \(\C\), then the quadratic form \(Q\) is \(\leq 0\) on the cone over the matroid polytope of \(M\).

\subsection{Main result: klt and CY arrangements}\label{sec:mainresults}

Let \(\mH\) be a finite set of pairwise distinct complex hyperplanes \(H \subset \CP^n\), and let \(\ba\) be a vector in \(\R^{\mH}\) whose components are positive real numbers \(a_H > 0\) indexed by the elements \(H \in \mH\). We say that \((\mH, \ba)\) is a weighted hyperplane arrangement.

Let \(\mL\) be the set of non-empty and proper linear subspaces \(L \subset \CP^n\) obtained as intersections of hyperplanes in \(\mH\). 
\newpar

\noindent {\bf Klt and CY arrangements.}
The weighted arrangement \((\mH, \ba)\) is:
\begin{itemize}
	\item klt (which stands for Kawamata log terminal) if
	\begin{equation}\label{eq:klt}
			\forall \, L \in \mL: \,\, \sum_{H \,| \, L \subset H} a_H < \codim L \,,
	\end{equation}
	where the sum runs over all the hyperplanes \(H \in \mH\) that contain \(L\);
	
	\item Calabi-Yau, or CY for short, if 
	\begin{equation}\label{eq:cy}
		\sum_{H \in \mH} a_H = n+1 \,.
	\end{equation}
\end{itemize}

\begin{rmk}
	 Equation \eqref{eq:klt} applied to \(\mH \subset \mL\), together with the assumption that the weights \(a_H\) are positive, imply that
	\[
	\forall \, H \in \mH: \,\, 0 < a_H < 1 \,.
	\]
\end{rmk}

Our main result, in its primitive technical form is the following.

\begin{theorem}\label{thm:main}
	Suppose that \((\mH, \ba)\) is a weighted arrangement that is both klt and CY.  
	Then the following inequality holds:
	\begin{equation}\label{eq:mythm}
		\sum_{L \in \mLi^{n-2}} a_L^2 - \frac{1}{2} \sum_{H \in \mH}  B_H \cdot a_H^2  - \frac{n+1}{2} \leq 0 \,.
	\end{equation}
\end{theorem}

The terms in Equation \eqref{eq:mythm} are defined as follows: 
\begin{itemize}
	\item \(\mLi^{n-2}\) is the set of codimension \(2\) irreducible subspaces \(L \subset \CP^n\)\,;
	\item \(a_L\) for \(L \in \mLi^{n-2}\) is given by \(2 \cdot a_L = \sum\limits_{H \,|\, L \subset H} a_H\)\,;
	\item \(B_H + 1\) is equal to the number of elements \(L \in \mLi^{n-2}\) that are contained in \(H\).
\end{itemize}

\begin{remark}
    The set of weights \(\ba \in \R^{\mH}\) such that \((\mH, \ba)\) is klt and CY is equal to the relative interior \(P^{\circ}\) of the matroid polytope \(P\) of \(\mH\) as a subset of the hyperplane \(\sum a_H = n+1\).
    It is well-known that, if the arrangement \(\mH\) is \emph{essential} and \emph{irreducible},\footnote{I.e. \(\bigcap_{H \in \mH}H = \emptyset\) and \(\mH\) is not isomorphic to a \emph{product} -as in Definition \ref{def:prod}- of two non-empty arrangements, see Corollary \ref{cor:redarr}.}
    then \(\dim P = |\mH|-1\) and therefore \(P^{\circ}\) is non-empty. See \cite[Theorem 1.12.9]{borovikgelfand} and Appendix \ref{app:essirr} for a self contained proof.
\end{remark}

\noindent\textbf{Calabi-Yau metrics on log pairs.}
We explain here why Theorem \ref{thm:main} can be viewed as a Miyaoka-Yau type inequality in the setting of hyperplane arrangements, by relating it to the existence of singular Calabi-Yau metrics.

We begin by recalling the smooth case. Let \(X\) be a compact \(n\)-dimensional Calabi-Yau manifold and let \(\omega\) be a Ricci-flat K\"ahler metric on \(X\). Then the energy of \(\omega\), namely the squared \(L^2\)-norm of its curvature \( \int_X |\mathrm{Riem}|^2\)
is, up to a positive constant depending on the dimension, given by the topological quantity \([\omega]^{n-2} \cdot c_2(X)\).
In particular, one obtains the inequality \([\omega]^{n-2} \cdot c_2(X) \geq 0\), with equality if and only if \(\omega\) is flat.

We now turn to log pairs \((X,\Delta)\). To match the context of this paper, we assume that \(X = \CP^n\) and that \(\Delta = \sum a_i H_i\), where \(H_i \subset \CP^n\) are complex hyperplanes and the coefficients \(a_i\) lie in the interval \((0,1)\).
It follows from Kolodziej's work \cite{kolodziej} that, if the weighted arrangement \((\mH,\ba)\) is klt and CY, then there exists a (unique up to scale) \emph{weak} Ricci-flat K\"ahler metric \(\omega_{\RF}\) on \(\CP^n\) whose volume form has conical-type singularities along \(\mH\) (see Proposition \ref{prop:weakrf} in Appendix \ref{app:ke} for a precise statement). 

The metric \(\omega_{\RF}\) is smooth on the complement \(\CP^n \setminus \mH\).
By analogy with the smooth case, one might expect that the energy \(\int_{\CP^n \setminus \mH} |\mathrm{Riem}|^2\) of \(\omega_{\RF}\) should be determined by topological data together with the cone angles, or more precisely by a suitable logarithmic second Chern class \(c_2(X,\Delta)\) associated to the pair.
At present, such a formula is far from being established. The main difficulty lies in the limited regularity of \(\omega_{\RF}\) near \(\mH\); in particular, it is not even known in general whether the above energy integral is finite. When \(n=2\) some partial progress can be found in \cite{dbs}.

For log smooth pairs, that is, when \(\Delta\) has simple normal crossing support, notions of logarithmic second Chern classes \(c_2(X,\Delta)\) are well established, and Miyaoka-Yau type inequalities involving these classes have been proved; see for instance \cite[Theorem~1.13]{li}. Moreover, one can easily check that if the arrangement \(\mH\) is simple normal crossing, then, up to a negative constant depending on the dimension, our quantity \(Q(\ba)\) coincides with the logarithmic second Chern class defined in \cite[Equation~(4)]{li}. However, for such configurations one always has \(Q(\ba) < 0\) if $\dim X>1$; see Example \ref{ex:Qgeneric}. Thus, in order to capture flat metrics, it is essential to go beyond the normal crossing setting.

In this paper, following \cite{panov}, we associate a natural parabolic bundle to the pullback of the tangent bundle on a log resolution of the arrangement. This allows us to invoke Mochizuki's theory \cite{mochizuki}, where a refined regularity theory for singular connections is available. In this way, we bypass the analytic difficulties arising from the weak Ricci-flat K\"ahler metric \(\omega_{\RF}\).

Finally, we comment on the sign of our inequality \(Q(\ba) \leq 0\), which is opposite to the classical Miyaoka--Yau inequality \([\omega]^{n-2} \cdot c_2(X) \geq 0\). This discrepancy is explained by the fact that \(Q(\ba)\) should be viewed, roughly speaking, as the second parabolic Chern character \(\parch_2(\mE_*)\) of the parabolic bundle we consider. On the other hand, if \(E\) is a vector bundle with \(c_1(E)=0\), then \(\ch_2(E) = -c_2(E)\), which accounts for the difference in sign.

\subsection{Sketch proof of Theorem \ref{thm:main}}

If \(n=2\), then Theorem \ref{thm:main} is part of Theorem 1.12 in \cite{panov}. Our proof follows the same approach: we work on a logarithmic resolution of the arrangement, define a natural parabolic structure on the pullback tangent bundle, show that it is slope stable, and deduce Theorem \ref{thm:main} from Mochizuki's Bogomolov-Gieseker inequality for parabolic bundles \cite{mochizuki}.
Next, we present a more detailed account of these steps.
\newpar

\noindent {\bf The resolution.} 
We take the logarithmic resolution of the arrangement given by the \emph{minimal De Concini-Procesi model} of \(\mH\). This is a smooth projective variety \(X\) together with a birational morphism
\[
X \xrightarrow{\pi} \CP^n
\]
such that \(D = \pi^{-1}(\mH)\) is a simple normal crossing divisor.
We refer to \(X\) as \emph{the resolution} of \(\mH\).
The variety \(X\) is constructed as an iterated blowup of \(\CP^n\) along the \emph{irreducible subspaces} of the arrangement, and it is a particular instance of the  \emph{wonderful models} of subspace arrangements introduced in  \cite{conciniprocesi}.

Concretely, a linear subspace \(L \subset \CP^n\) obtained as intersection of hyperplanes in \(\mH\) is  \emph{irreducible}, if the localized arrangement \(\mH_L = \{H \in \mH \, |\, H \supset L\}\) is irreducible (Definition \ref{def:irrarr}). 
%\todo{Maybe add a reference to the def of irreducible?}
We denote by \(\mLi\) the set of all non-empty and proper irreducible subspaces. The irreducible components of the divisor \(D = \pi^{-1}(\mH)\) are in one to one correspondence with elements \(L \in \mLi\).
Specifically, the irreducible decomposition of \(D\) is
\[
D = \bigcup_{L \in \mLi} D_L \,,
\]
where \(D_L\) is the unique irreducible component of \(D\) such that \(\pi(D_L) = L\).\newpar

\noindent {\bf The parabolic bundle.}
Let \(X \xrightarrow{\pi} \CP^n\) be the resolution of \(\mH\) and let \(\mE = \pi^*(T\CP^n)\) be the pullback tangent bundle. We define a natural \emph{parabolic bundle} (Definition \ref{def:parbun}) \(\mE_{*}\) on \((X, D)\) where \(D = \pi^{-1}(\mH)\). 

The parabolic bundle \(\mE_{*}\) is defined by increasing filtrations of \(\mE|_{D_L}\) (for \(L \in \mLi\)) by vector subbundles \(F^L_a\) with \(0 < a \leq 1\). The vector subbundles \(F^L_a\) are defined as
\begin{equation*}
    F^L_a = 
    \begin{cases}
	\pi^*(TL) &\text{ if } a < a_L, \\
	\mE|_{D_L} &\text{ if } a \geq a_L ,
  \end{cases}
\end{equation*}
where the weights \(a_L\) are given by
\[
a_L = (\codim L)^{-1} \sum_{H | L \subset H} a_H \,.
\]
Note that the klt condition \eqref{eq:klt} implies that \(0 < a_L < 1\). 

The \emph{parabolic first Chern class} \(\parc_1(\mE_{*})\)
(Definition \ref{def:pch1}) is an element  of the real cohomology group \(H^2(X, \R)\).
Using the CY condition \eqref{eq:cy}, we show that
\[
\parc_1(\mE_{*}) = 0 \,.
\]

\noindent {\bf The locally abelian property and the parabolic second Chern character.}
A parabolic bundle \(\mE_{*}\) on \((X, D)\) is \emph{locally abelian} (Definition \ref{def:locab}) if the filtrations that define it satisfy certain compatibility conditions at the intersections \(\cap D_i\)\,, where \(D_i\) are the irreducible components of the divisor \(D\). In a few words, the locally abelian condition requires that \(\mE_{*}\) is locally isomorphic to a direct sum of parabolic line bundles. If the complex dimension is \(n=2\), then every parabolic bundle is locally abelian. However, if \(n > 2\), the locally abelian condition imposes strong restrictions on the parabolic structure, and the compatibility conditions are not satisfied for generic choices of filtrations. 

Theorem \ref{thm:locab} asserts that the parabolic bundle \(\mE_{*}\) on the resolution \(X\) is locally abelian. The proof of this theorem is non-trivial and relies on the fact that 
a collection of divisors \(D_L\) with \(L \in \mS \subset \mLi\) has non-empty common intersection \(\bigcap_{L \in \mS} D_L\) if and only if the set \(\mS\) is \emph{nested} relative to \(\mLi\) (Definition \ref{def:nested}). 
The locally abelian property of \(\mE_{*}\) is used to define the \emph{parabolic second Chern character} \(\parch_2(\mE_{*}) \in H^4(X, \R) \) (Definition \ref{def:pch2}).

We calculate \(\parch_2(\mE_{*})\) in terms of the intersection poset of the arrangement \(\mH\).
Let \(h = c_1(\mO_{\P^n}(1))\) be the generator of \(H^2(\CP^n, \Z)\),
and let \(\gamma_L \in H^2(X, \Z)\) -with \(L \in \mLi\)- be the Poincar\'e dual of the divisor \(D_L\).
We derive the formula
\[
\parch_2(\mE_{*}) = - \frac{n+1}{2} \cdot (\pi^*h)^2 
	+  \frac{1}{2} \sum_{L \in \mLi} a_L^2 \cdot \codim L \cdot \gamma_L^2 + \sum_{L \subsetneq M} a_L  a_{M} \cdot \codim M \cdot \gamma_L \cdot \gamma_M \,,
\]
where the last sum is over all pairs \((L, M) \in \mLi \times \mLi\) with \(L \subsetneq M\). \newpar

\noindent {\bf The Picard group of \(X\) and the polarization.}
The Picard group of \(X\) is the free abelian group generated by \([\pi^*(\mO_{\P^n}(1)]\) together with
the classes \([D_L]\) of the exceptional divisors \(D_L\) of the birational morphism \(\pi: X \to \CP^n\). The exceptional divisors \(D_L\) correspond to non-empty irreducible subspaces \(L \subset \CP^n\) of codimension \(\geq 2\), we denote the set of all such subspaces by \(\mLi^{\circ} := \mLi \setminus \mH\).

We fix positive integers \(b_L\) for each \(L \in \mLi^{\circ}\)\,, that we call \emph{polarization coefficients}, such  that
\begin{equation*}
	P_k = k \cdot \pi^*\big(\mO_{\P^n}(1)\big) - \sum_{L \in \mLi^{\circ}} b_L \cdot D_L
\end{equation*}
is an ample divisor on \(X\) for all \(k \gg 1\). \newpar 

\noindent {\bf Stability of the parabolic bundle.}
Theorem \ref{thm:stability} asserts that the parabolic bundle \(\mE_{*}\) is \(P_k\)-\emph{stable} (Definition \ref{def:parstable}) for all \(k \gg 1\). 
Theorem \ref{thm:stability} is a central result of this paper,  we refer to it as the \emph{stability theorem}.
To prove it, we must show that for every non-zero and proper \emph{saturated subsheaf} \(\mV \subset \mE\) (Definition \ref{def:saturated}), we have
\[
\pardeg_{P_k}(\mV_{*}) < 0 \,,
\]
where \(\mV_{*}\) is the naturally induced parabolic structure on \(\mV\) (Definition \ref{def:indpar}). 

The saturated subsheaf \(\mV \subset \mE\) defines a \emph{distribution} on \(\CP^n\) (Section \ref{sec:distributions}) which, by a slight abuse of notation, we denote by \(\mV \subset T\CP^n\). To prove stability, it is easy to reduce to the case where the \emph{index} \(\imath\) (Definition \ref{def:index}) of the distribution \(\mV\) is positive (this is essentially a consequence of the fact that, with our conventions,
the parabolic degree is less or equal than the degree).
We show that, if \(\imath \geq 0\), then
\[
\pardeg_{P_k}(\mV_{*}) \leq \left( \imath \, - \,\sum_{H \,|\, H \pitchfork \mV} a_H \right) \cdot k^{n-1} \,+\, O(k^{n-2}) \,.
\]
The above sum is over all hyperplanes \(H \in \mH\) that are \emph{transverse} to \(\mV\). The term \(O(k^{n-2})\) denotes a polynomial in \(k\) of degree at most \(n-2\) whose coefficients are uniformly bounded in absolute value by a number that depends only on the arrangement \(\mH\) and the polarization coefficients \(b_L\).

Given the above inequality for \(\pardeg_{P_k}(\mV_{*})\), the key estimate necessary to prove the stability theorem is provided by Proposition \ref{prop:key}. This proposition  asserts that, if the weighted arrangement \((\mH, \ba)\) is klt and CY, then there is \(\delta>0\) such that, for every distribution \(\mV \subset T\CP^n\) of index \(\imath \geq 0\) the following holds:
\[
\sum_{H \,|\, H \pitchfork \mV} a_H \geq \imath + \delta \,.
\]
The stability Theorem \ref{thm:stability} follows from this. \newpar

\noindent {\bf The Bogomolov-Gieseker inequality.}
Since the parabolic bundle \(\mE_{*}\) is stable with respect to the polarization \(P_k\) for all \(k \gg 1\) and \(\parc_1(\mE_{*}) = 0\)\,, the Bogomolov-Gieseker inequality
\cite[Theorem 6.5]{mochizuki} implies that
\[
\forall k \gg 1: \,\, c_1(P_k)^{n-2} \cdot \parch_2(\mE_{*})  \leq 0 \,.
\]

The cup product \(p(k) = c_1(P_k)^{n-2} \cdot \parch_2(\mE_{*})\) defines a polynomial of degree \(n-2\) in \(k\) that we write as
\[
p(k) = C_{n-2} \cdot k^{n-2} + O(k^{n-3}) \,.
\]
The leading order coefficient of \(p(k)\) is  given by \(C_{n-2} = (\pi^*h)^{n-2} \cdot \parch_2(\mE_{*})\)\,.
More geometrically, the coefficient \(C_{n-2}\) is the parabolic second Chern character of the restriction of \(\mE_{*}\) to a generic \(2\)-plane. 

Note that, since \(p(k)<0\) for \(k \gg 1\)\,, we must have \(C_{n-2} \leq 0\).
Using results on the cohomology ring of \(X\) from Section \ref{sec:intnum} together with our previous formula for \(\parch_2(\mE_{*})\), we show that
\[
C_{n-2} = \sum_{L \in \mLi^{n-2}} a_L^2 - \frac{1}{2} \sum_{H \in \mH}  B_H \cdot a_H^2  - \frac{n+1}{2} \,.
\]
Theorem \ref{thm:main} follows from this together with the inequality \(C_{n-2} \leq 0\).

\subsection{Context: PK manifolds}\label{sec:pk}

\begin{comment}
The study of constant curvature metrics on the Riemann sphere having a finite number of conical singularities with cone angles less than \(2\pi\) is a classical and well-established subject. In higher complex dimensions, one possible generalization is to replace a configuration of points on the Riemann sphere with an arrangement of complex hyperplanes in \(\CP^n\), and seek for K\"ahler metrics of constant holomorphic sectional curvature that have conical singularities in transverse directions to the prescribed hyperplanes. Our ultimate goal is to construct and classify such metrics, in the regime when the cone angles are less than \(2\pi\). For \(n = 2\) and zero curvature, a solution in terms of combinatorial properties of line arrangements is presented in \cite{panov}. In all dimensions, complex reflection arrangements provide an interesting class of examples \cite{chl}.
In this paper, we develop some algebro-geometric aspects of this  problem by establishing a version of the Miyaoka-Yau inequality within this context. 
\end{comment}

Our motivation for proving Theorem \ref{thm:main} comes from the study of \emph{polyhedral K\"ahler} (PK) manifolds. Let \(X\) be a complex manifold and let \(D \subset X\) be a complex hypersurface with irreducible decomposition \(D = \cup_i D_i\). We say that \(g\) is a PK metric on \(X\) with cone angles \(2\pi \alpha_i\) along \(D_i\) if the following conditions hold:
\begin{enumerate}[label=\textup{(\roman*)}]
	\item \(g\) is a flat K\"ahler metric on \(X \setminus D\)\,;
	\item  the metric completion of \(g\) is a polyhedral manifold \(M\) and the inclusion map \(X \setminus D \subset X\) extends continuously to \(M\) as a homeomorphism \(M \cong X\);
	\item if \(p \in D_i\) is a smooth point of \(D\) then \(g\) is holomorphically isometric near \(p\) to the product of a \(2\)-cone of total angle \(2\pi\alpha_i\) with a flat Euclidean factor \(\C^{n-1}\).
\end{enumerate}

The Levi-Civita connection \(\nabla\) of a PK metric is a meromorphic flat torsion-free unitary connection on \(TX\) with simple poles along the hypersurfaces \(D_i\). The residues of \(\nabla\) give a number of topological relations on the pair \((X, D)\). Conjecturally, these topological relations characterize the existence of PK metrics. 
In the particular case when \(X = \CP^n\) and \(D\) is a union of hyperplanes, we believe that these topological relations can be derived from quadratic expression \eqref{eq:mythm} and we propose the following.

\begin{conjecture}\label{conjecture}
	If equality holds in \eqref{eq:mythm} then there is a PK metric on \(\CP^n\) with cone angles \(2\pi\alpha_H\) along the hyperplanes \(H \in \mH\), with \(\alpha_H = 1-a_H\). 
\end{conjecture}

As we already mentioned, this conjecture holds for \(n=2\)  by \cite[Theorem 1.12]{panov}. The proof for \(n=3\), under some additional stability assumptions and applying the foundational results established in \cite{pkc}, will be the subject of a forthcoming paper.	

%Our plan to show existence of PK metrics
%has two steps.
%\begin{itemize}
%	\item Step 1: use the Kobayashi-Hitchin correspondence for parabolic bundles to construct a flat unitary meromorphic connection \(\nabla\) on \(TX\) with appropriate behaviour at \(D\).
%\end{itemize}

%Once Step 1 is done we can fix \(x \in X \setminus D\) and define \(g\) by parallel transport of an Hermitian inner product on \(T_xX\) preserved by \(\nabla\).

%\begin{itemize}
%	\item Step 2: show that \(g\) is a PK metric as defined above.
%\end{itemize}

%This paper achieves part of Step 1, showing that the relevant parabolic bundle is stable.

\subsection{Outline}

Section \ref{sec:ha} contains background material on hyperplane arrangements. \newpar

\noindent In Section \ref{sec:res}\,, we review the minimal De Concini-Procesi model \(X\) of \(\mH\), calculate certain cup products
in the integer cohomology \(H^*(X, \Z)\), and introduce the polarizations \(P_k\) used in the stability theorem.\newpar

\noindent In Section \ref{sec:parbun}\,, we define the parabolic bundle \(\mE_{*}\), show that it is locally abelian, and calculate its first and second Chern characters.\newpar

\noindent In Section \ref{sec:stability}\,, we show that the parabolic bundle \(\mE_{*}\) is \(P_k\)-stable for \(k \gg 1\) and deduce our main Theorem \ref{thm:main} from this.\newpar

\noindent In Section \ref{sec:qfsc}\,, we introduce the quadratic form \(Q\) of the arrangement \(\mH\), we define the semistable cone \(C\), and show that \(C \subset \{Q \leq 0\}\) (Theorem \ref{thm:maingeneral}).\newpar

\noindent In Section \ref{sec:hirzebruch}\,, we introduce the class of Hirzebruch arrangements, show that complex reflection arrangements defined by
irreducible unitary reflection groups belong to this class, and define the quadratic form \(Q\) in the broader context of matroids. We propose an extension (Conjecture \ref{conj:provocation}) of Theorem \ref{thm:maingeneral} to pseudoline arrangements, which relates to Bogomolov-Miyaoka-Yau inequality for symplectic $4$-maniolds.\newpar

\noindent Appendix \ref{app:auxres} contains supplementary results needed in our proofs.
We collect some basic results on filtrations of vector spaces, adapted basis, and nested sets; which are used in the proof of the locally abelian property of our parabolic bundle.
We also provide some background results on exterior algebra, saturated subsheaves, and distributions on \(\CP^n\); that are used in the proof of stability of our parabolic bundle.
\newpar

\noindent In Appendix \ref{app:essirr}\,, we prove that if \(\mH\) is essential and irreducible then the stable cone \(C^{\circ}\) is non-empty.\newpar

\noindent In Appendix \ref{app:ke}\,, we show that if the weighted arrangement \((\mH, \ba)\) is klt and CY then there is a (unique up to scale) weak Ricci-flat K\"ahler metric on \(\CP^n\) whose volume form has prescribed singularities of conical type along the hyperplanes of the arrangement.

\subsubsection*{Acknowledgments}
We thank Nikolai Mn\"ev, Artie Prendergast-Smith, and Calum Spicer for useful discussions. We also thank Jorge Vit\'orio Pereira for giving us Example \ref{ex:pereira}. We would like to thank the referee for a very careful reading of the manuscript and for providing many helpful comments and suggestions that improved the clarity of the paper.

\subsection*{Recent developments}

Conjecture \ref{conjecture} has recently been resolved in \cite{cpn}. As an offshoot, a more elementary proof of Theorem \ref{thm:main} was obtained in \cite{dunklnew}. However, it is worth noting that the methods developed in the present paper (most notably the stability of the parabolic bundle established in Theorem \ref{thm:stability}) remain fundamental; indeed, they play a central role in the proof of Conjecture \ref{conjecture} in \cite{cpn}.

%Conjecture \ref{conjecture} has recently been proved in \cite{cpn}. As an off shot, a more elementary proof of Theorem \ref{thm:main} has been found in \cite{dunklnew}. However, the approach taken in this paper, and in particular the stability of the parabolic bundle proved in Theorem \ref{thm:stability}, play a central role in \cite{cpn} and in the proof of Conjecture \ref{conjecture}. 

\section{Hyperplane arrangements}\label{sec:ha}

In Section \ref{sec:habasicdef}\,, we review standard terminology related to hyperplane arrangements.\newpar 

\noindent In Section \ref{sec:irrarr}\,, we introduce irreducible arrangements and subspaces, and discuss their basic properties.\newpar

\noindent In Section \ref{sec:weightarr}\,, we consider weighted arrangements. We show that if the weighted arrangement \((\mH, \ba)\) is klt and CY, then \(\mH\) is essential and irreducible.
This implies that there are no non-zero holomorphic vector fields on \(\CP^n\) that are tangent to all the hyperplanes in \(\mH\). We will make use of this fact later, in Section \ref{sec:stability}, to show that our parabolic bundle is stable.

%The central notion is that of an \emph{irreducible subspace}. An element \(L \in \mL\) of the intersection poset of \(\mH\) is irreducible if the localized arrangement \(\mH_L\), consisting of the hyperplanes in \(\mH\) that contain \(L\), can not be decomposed as a product of non-empty arrangements.
%disjoint union of two non-empty subsets \(\mH_L = \mH_1 \cup \mH_2\) such that, after a linear change of coordinates, the defining equations for \(\mH_1\) and \(\mH_2\) share no common variables.
%The set of all non-empty and proper irreducible subspaces \(L \in \mL\) is denoted by \(\mLi\).

\subsection{Basic definitions}\label{sec:habasicdef}

Let \(\CP^n = \P(\C^{n+1})\) be the complex projective space of dimension \(n\),
\begin{equation*}
\CP^n = \left(\C^{n+1} \setminus \{0\}\right) \big/ \, \C^*	\,.
\end{equation*}

A hyperplane arrangement \(\mH\), or an arrangement for short, is a finite collection of complex hyperplanes \(H \subset \CP^n\).
We don't allow multiplicities, meaning that the elements of \(\mH\) are distinct hyperplanes. By abuse of notation, we write \(\mH \subset \CP^n\) if we want to emphasize the ambient projective space.  Two arrangements \(\mH\) and \(\mathcal{K}\) are isomorphic if they are linearly equivalent, i.e., if there is an element of \(\PGL(n+1, \C)\) that maps the hyperplanes in \(\mH\) to the hyperplanes in \(\mK\).\newpar

\noindent {\bf Linear subspaces.}
A (projective or linear) subspace \(L \subset \CP^n\) is the image of a linear subspace of  \(\C^{n+1}\) under the quotient projection
\[
\C^{n+1} \setminus \{0\} \to \CP^n .
\]
We write \(\P(L^{\bc}) = L\), where \(L^{\bc} \subset \C^{n+1}\) is the unique linear subspace that projects to \(L\). 
Similarly, \(\mH^{\bc}\) denotes the linear arrangement of hyperplanes \(H^{\bc} \subset \C^{n+1}\) with \(H \in \mH\), and we write \(\P(\mH^{\bc}) = \mH\).
The codimension of \(L \subset \CP^n\)
is equal to the codimension of the linear subspace \(L^{\bc} \subset \C^{n+1}\). 
We also consider the empty subset of \(\CP^n\) to be a linear subspace with \(\emptyset^{\bc}=\{0\}\) and \(\codim \emptyset = n+1\).\newpar

\noindent {\bf Centre and essential arrangements.}
The \emph{centre} \(T\) of \(\mH\) is the common intersection of all its members, i.e.,
\begin{equation*}
	T = \bigcap_{H \in \mH} H .
\end{equation*}	
We say that \(\mH\) is \emph{essential} if \(T = \emptyset\), or alternatively, if the common intersection of the elements of \(\mH^{\bc}\) is the origin \(\{0\} \subset \C^{n+1}\). Clearly, \(\mH\) is essential if and only if it contains \(n+1\) linearly independent hyperplanes.\newpar

%Our definition of hyperplane arrangement allows for \(\mH\) to be empty. In this case,
%the centre \(T\) of the empty arrangement \(\mH= \emptyset\) is equal to the whole ambient space \(T=\CP^n\).

\noindent {\bf Sum and complementary subspaces.}
If \(U\) and \(V\) are two projective subspaces of \(\CP^n\) then their \emph{sum} (or join) \(U + V\) is the smallest projective subspace that contains both \(U\) and \(V\). Equivalently,
\[
U + V = \P (U^{\bc} + V^{\bc}) \,,
\]
where \(U^{\bc} + V^{\bc}\) denotes the usual sum of vector subspaces of \(\C^{n+1}\). The subspaces \(U\) and \(V\) are said to be \emph{complimentary} if \(U+V = \CP^n\) and \(U \cap V = \emptyset\). In this case, we also say that \(V\) is a \emph{complement} of \(U\), and vice-versa.\newpar

\noindent {\bf Essentialization.}
Let \(\mH \subset \CP^n\) be an arrangement with centre \(T\). Choose a subspace \(S\) that is a complement of \(T\).
The \emph{essentialization} of \(\mH\) is the arrangement given by
\[\mH/T = \{ H \cap S \text{ with } H \in \mH \} .\]
Since \(T \cap S = \emptyset\), the arrangement \(\mH/T\) is essential.
Suppose that \(S'\) is another complement of \(T\) and let \(f: S \to S'\) be the projective transformation which maps \(p \in S\) to \(p'=f(p)\) with \(p' = (p + T) \cap S'\). Since every hyperplane \(H \in \mH\) contains \(T\), it follows that \(f\) maps the intersection \(H \cap S\) to \(H \cap S'\) for all \(H \in \mH\); showing that \(\mH/T\) is independent of the choice of complement.\newpar

\noindent {\bf Intersection poset.}
The \emph{intersection poset} \(\overline{\mL}\) of \(\mH\) is the set of all subspaces \(L \subset \CP^n\) obtained as intersection of hyperplanes in \(\mH\).
We equip \(\overline{\mL}\) with the partial order given by reverse inclusion. The intersection poset \(\overline{\mL}\) has a unique minimal element \(L=\CP^n\) corresponding to the intersection over the empty subset of \(\mH\) and a unique maximal element \(L=T\) given by the centre of the arrangement.
We will mainly work with the subposet \(\mL \subset \overline{\mL}\) consisting of all non-empty and proper intersection subspaces. In other words, we exclude \(\CP^n\) and \(\emptyset\) (if \(\mH\) is essential) from \(\overline{\mL}\) and let \(	\mL = \overline{\mL} \, \setminus \, \{\emptyset, \, \CP^n\}\).\newpar

\noindent {\bf Localization and multiplicity.}
The \emph{localization} of \(\mH\) at a subspace \(L \in \overline{\mL}\) is the set of all hyperplanes of the arrangement that contain \(L\),
	\begin{equation*}
	\mH_L = \{H \in \mH \, \text{ such that } \, L \subset H \} .
	\end{equation*}
%If \(L = \emptyset\) then \(\mH_L = \mH\) and if \(L = \CP^n\) then \(\mH_L\) is the empty arrangement.
The centre of \(\mH_L\) is equal to \(L\). The \emph{link} of \(\mH\) at \(L\) is the essential arrangement \(\mH_L / L \subset \CP^{m}\) where \(m = \codim L -1\).
The \emph{multiplicity} \(m_L\) of \(L\) is the number of hyperplanes that contain \(L\), i.e., \(m_L = |\mH_L|\).\newpar

\noindent {\bf Induced arrangement and complements.}
Let \(L \in \mL\). The \emph{induced arrangement} \(\mH^L\) is the hyperplane arrangement obtained by intersecting \(L\) with the elements \(H \in \mH\) such that \(L \not\subset H\),
\[
\mH^L = \{ H \cap L \textup{ with } H \in \mH \setminus \mH_L \} .
\]
The complement of \(\mH^L\) in \(L\) is denoted by \(L^{\circ}\), i.e.,
\[
L^{\circ} = L \setminus \bigcup_{H | L \not\subset H} (L \cap H) .
\]
Similarly, we write \((\CP^n)^{\circ}\) for the \emph{arrangement complement},
\[
(\CP^n)^{\circ} = \CP^n \setminus \bigcup_{H \in \mH} H \,.
\]

\subsection{Irreducible arrangements and subspaces}\label{sec:irrarr}

A splitting of a hyperplane arrangement \(\mH \subset \CP^n\) consists of two subsets \(\mH_1, \mH_2 \subset \mH\) satisfying the following two properties.
	\begin{enumerate}[label=\textup{(\roman*)}]
		\item Every hyperplane of \(\mH\) belongs to either \(\mH_1\) or \(\mH_2\), i.e.,
		\begin{equation}\label{eq:uniondec}
		\mH = \mH_{1} \bigcup \mH_{2} \,.
		\end{equation}
		\item If \(T_1\) and \(T_2\) are the centres of \(\mH_1\) and \(\mH_2\), then
		\begin{equation}\label{eq:sumdec}
		T_1 + T_2 = \CP^n \,.
		\end{equation}
	\end{enumerate} 
	The splitting is non-trivial if furthermore:
	\begin{enumerate}[label=\textup{(\roman*)}, resume]
		\item Both \(\mH_1\) and \(\mH_2\) are non-empty.
	\end{enumerate} 

We write \(\mH = \mH_1 \uplus \mH_2\) to indicate that \(\mH_1\) and \(\mH_2\) form a splitting of \(\mH\). It follows from Equation \eqref{eq:sumdec} that the union \eqref{eq:uniondec} is disjoint. 
If \(\mH = \mH_1 \uplus \mH_2\) then \(\mH_1 = \mH_{T_1}\) and \(\mH_2 = \mH_{T_2}\). If the splitting is non-trivial then \(T_1\) and \(T_2\) belong to the poset \(\mL\) of non-empty and proper hyperplane intersections. 

\begin{definition}\label{def:irrarr}
	The arrangement \(\mH\) is \emph{reducible} if it admits a  non-trivial splitting. The arrangement \(\mH\) is  
	\emph{irreducible} if it is not reducible. 	
\end{definition}

\begin{example}
	The empty arrangement is irreducible.
	An arrangement consisting of only one hyperplane is irreducible. An arrangement made of two hyperplanes is always reducible. 
\end{example}

\begin{example}\label{ex:irrP1}
	If \(n=1\) then \(\mH \subset \CP^1\) is reducible if and only if \(|\mH| = 2\). Indeed, if \(|\mH| \geq 3\) and
	\(\mH = \mH_1 \uplus \mH_2\) is a splitting of \(\mH\), then at least one of the subarrangements, say \(\mH_1\), must contain \(2\) or more points, so \(T_1 = \emptyset\) and \(T_2 = \CP^1\), i.e., \(\mH_2 = \emptyset\); showing that \(\mH\) is irreducible.
\end{example}

\begin{lemma}\label{lem:redarr}
	An arrangement \(\mH\) is reducible if and only if, up to a linear change of coordinates, we can write \(\mH = \mH_1 \cup \mH_2\) where \(\mH_1\) and \(\mH_2\) are non-empty and the defining linear equations for the hyperplanes in \(\mH_1\) and \(\mH_2\) share no common variables.
\end{lemma}

\begin{proof}
	Suppose that \(\mH\) is reducible and let \(\mH = \mH_1 \uplus \mH_2\) be a non-trivial splitting of \(\mH\). Take linear coordinates \(x_1, \ldots, x_{n+1}\) in \(\C^{n+1}\) such that 
	\[
	T_1 = \{x_i = 0 \textup{ for } i \in I_1\} \,\, \textup{ and } \,\, T_2 = \{x_i = 0 \textup{ for } i \in I_2\}
	\]
	where \(I_1\) and \(I_2\) are subsets of \([n+1] = \{1, \ldots, n+1\}\). The defining linear equation for the hyperplanes in \(\mH_1\) depend only  on the variables \(x_i\) for \(i \in I_1\) and similarly for \(\mH_2\). We will show that \(I_1 \cap I_2 = \emptyset\).
	The subspace \(T_1\) is the span in projective space of the vectors \(\p_{x_{i}}\) for \(i\) in the complement index set \(I^c_1 = [n+1] \setminus I_1\) and similarly for \(T_2\). The fact that \(T_1 + T_2 = \CP^n\) implies that \(I_1^c \cup I_2^c = [n+1]\). Taking complements we get that \(I_1 \cap I_2 = \emptyset\). The converse is similar.
\end{proof}

\begin{definition}\label{def:prod}
Let \(\mH_1 \subset \CP^{n_1}\) and \(\mH_2 \subset \CP^{n_2}\) be two arrangements. Let \(\CP^n = \P(\C^{n_1+1} \times \C^{n_2+1})\)
with \(n = n_1 + n_2 + 1\). Embed the projective spaces \(\CP^{n_i}\) as disjoint subspaces \(P_i \subset \CP^n\) with \(P_1 + P_2 = \CP^n\) given by
\[
P_1 = \P\left(\C^{n_1+1} \times \{0\}\right) \,\,\textup{ and }\,\, P_2 = \P\left(\{0\} \times \C^{n_2+1}\right) \,.
\]
The \emph{product} arrangement \(\mH_1 \times \mH_2 \subset \CP^n\) is obtained by taking the joins \(H_1 + P_2\) for \(H_1 \in \mH_1\)
together with \(P_1 + H_2\)  for \(H_2 \in \mH_2\)\,.	
\end{definition}

%The \emph{product} of two arrangements \(\mH_1 \subset \CP^{n_1}\) and \(\mH_2 \subset \CP^{n_2}\) is the arrangement \(\mH_1 \times \mH_2 \subset \CP^{n}\) with \(n=n_1+n_2+1\) given by the hyperplanes \(\P(H_1^{\bc} \times \C^{n_2+1})\) with \(H_1 \in \mH_1\) together with \(\P(\C^{n_1+1} \times H_2^{\bc})\) with \(H_2 \in \mH_2\).

\begin{remark}
	The arrangement  \(\mH_1 \times \mH_2\) is the projectivization of  \(\mH_1^{\bc} \times \mH_2^{\bc}\)\,, where \(\mH_1^{\bc} \times \mH_2^{\bc}\) is the usual product of linear arrangements in a vector space as defined in \cite[Definition 2.13]{orlikterao}.
\end{remark}
%In other words, If \(\CP^{n_1}\) and \(\CP^{n_2}\) are linearly embedded in \(\CP^n\) as two disjoint projective subspaces \(P_1\) and \(P_2\) then the product arrangement \(\mH_1 \times \mH_2\) is obtained by taking the joins \(H_1 + P_2\) together with \(P_1 + H_2\) with \(H_1 \in \mH_1\) and \(H_2 \in \mH_2\). 

From the Lemma \ref{lem:redarr}\,, we have the following.

\begin{corollary}\label{cor:redarr}
	The arrangement \(\mH\) is reducible if and only if it is linearly isomorphic to the product of two non-empty arrangements.
\end{corollary}  

\begin{remark}
Every arrangement is linearly isomorphic to the product of irreducible arrangements. Moreover, this decomposition is unique up to re-labelling of the factors. See \cite[Lemma 2.32]{pkc}.	
\end{remark}

\begin{remark}
	An arrangement is reducible if and only if the Euler characteristic of its complement is equal to zero.
	One direction is easy, namely if \(\mH\) is reducible then there is a free \(\C^*\)-action on \((\CP^n)^{\circ}\), hence its Euler characteristic is zero. The converse is more subtle and involves the theory of M\"obius functions, see \cite[Theorem 5 (2)]{stv}.
\end{remark}

%The properties of an arrangement being essential/irreducible are completely independent of each other. For example, the arrangement given by the \(n+1\) coordinates hyperplanes in \(\CP^n\) is essential but not irreducible; while the arrangement made of \(3\) lines in \(\CP^2\) meeting at common a point is irreducible but not essential.
%More generally, the product of essential arrangements remains essential; whereas the product of an irreducible arrangement with an empty arrangement is irreducible but not essential.

%The study of arrangements can be basically focused on the case where \(\mH\) is essential and irreducible. This is because if \(\mH\) is not essential, we can take its essentialization \(\mH/T\); while if \(\mH\) is not irreducible, we can take its decomposition into irreducible factors (see \cite[Lemma 2.32]{pkc}). Moreover, these two procedures commute with each other, since an arrangement \(\mH\) is irreducible if and only if its essentialization \(\mH/T\) is irreducible. 

The next result will be crucial in the proof of Theorem \ref{thm:stability}\,.

\begin{lemma}\label{lem:noauto}
	Suppose that \(\mH \subset \CP^n\) is essential and irreducible. If \(Y\) is a holomorphic vector field on \(\CP^n\) that is tangent to all the hyperplanes in \(\mH\) then \(Y = 0\). 
\end{lemma}

\begin{proof}
	The vector field \(Y\) is given by a linear transformation \(f\) of \(\C^{n+1}\). The tangency condition means that \(f(H^{\bc}) \subset H^{\bc}\) for all \(H \in \mH\).
	Let \(g = f^*\) be the dual action on \((\C^{n+1})^*\). Choose defining linear equations \(\ell_H\) for the hyperplanes of the arrangement. The tangency condition implies that \(g(\ell_H) = \lambda_H \cdot \ell_H\) for some \(\lambda_H \in \C\). Since \(\mH\) is essential, we can take \(n+1\) linearly independent hyperplanes \(H_1, \ldots H_{n+1}\) whose corresponding defining equations \(\ell_i\) make a basis of \((\C^{n+1})^*\) of eigenvectors of \(g\) with eigenvalues \(\lambda_i\). We claim that \(\lambda_1 = \ldots = \lambda_{n+1}\). If not, let \(W_1\) be the \(\lambda_1\)-eigenspace of \(g\) and let \(W_2\) be the direct sum of all the other eigenspaces. Then
	\[
	\mH = \{H \,\, | \,\, \ell_H \in W_1\} \uplus \{H \,\, | \,\, \ell_H \in W_2 \}
	\]
	is a non-trivial splitting of \(\mH\), contradicting irreducibility. We conclude that \(f\) is a scalar multiple of the identity, hence \(Y = 0\).
\end{proof}

\begin{remark}
	If \(\mH\) is essential and irreducible, then the proof above shows that the stabilizer of \(\mH\), consisting of all elements in \(PGL(n, \C)\) that preserve each of the hyperplanes in \(\mH\), is trivial. In particular, the automorphism group of \(\mH\), consisting of linear isomorphisms that permute the members of \(\mH\), is finite.
\end{remark}

\begin{definition}
	A subspace \(L \in \overline{\mL}\) is irreducible if the localization \(\mH_L\) is an irreducible arrangement. Similarly, a subspace \(L\) is reducible if \(\mH_L\) is reducible.	 
\end{definition}

\begin{example}
	If \(L= \emptyset\) then \(\mH_L = \mH\), so \(L\) is irreducible if and only if \(\mH\) is.
	In the other extreme case, if \(L = \CP^n\), then \(\mH_L = \emptyset\), which is always irreducible.
\end{example}

\begin{example}
	The hyperplanes \(H \in \mH\) are irreducible subspaces. 
\end{example}

\begin{example}
	A subspace \(L \in \mL\) of codimension \(2\) is irreducible if and only if its multiplicity \(m_L\) is \(\geq 3\), c.f. Example \ref{ex:irrP1}.
\end{example}

\begin{example}
	If \(L \in \mL\) is an irreducible subspace of codimension \(\geq 2\) then \(m_L > \codim L\). However, if \(\codim L \geq 3\) then the converse is not true. For example, consider the arrangement in \(\CP^3\) made of \(3\) planes intersecting along a line \(L\) together with an extra plane \(H\) transverse to \(L\). The intersection point \(p = L \cap H\) has multiplicity \(4\) but the point \(p\) is a reducible subspace.
\end{example}

\begin{notation}\label{not:Lirr}
	Write
	\begin{equation*}
		\mLi = \bigcup_{i=0}^{n-1} \mLi^i
	\end{equation*}
	for the set of non-empty and proper irreducible subspaces, where \(\mLi^i\) is the subset of all \(L \in \mLi\) with \(\dim L = i\). In particular, \(\mLi^{n-1}\) is equal to \(\mH\). 
\end{notation}

\begin{definition}
	Let \(L \in \mL\). The \emph{irreducible components} of \(L\) are the maximal (with respect to the order in \(\mL\) by reverse inclusion) elements of \(\mLi\) that contain \(L\). 
\end{definition}

The term irreducible component is taken from \cite[\S 2.1]{chl}. 

\begin{notation}\label{not:IrrL}
Write \(\Irr(L)\) for the irreducible components of \(L\). In particular,
\(L \in \mL\) is reducible if and only if \(|\Irr(L)| > 1\).	
\end{notation}

\begin{lemma}\label{lem:irrcomp}
	Let \(L \in \mL\) and let \(\Irr(L) = \{L_1, \ldots, L_k\}\). Then the following holds.
	\begin{enumerate}[label=\textup{(\roman*)}]
		\item If \(M \in \mLi\) contains \(L\) then \(M \supset L_i\) for some \(1 \leq i \leq k\).
		\item \(L\) is the \emph{transversal intersection}\footnote{See Definition \ref{def:transint}\,.} of the subspaces \(L_i\), i.e.,
		\begin{equation}\label{eq:transint}
		L = \bigcap_{i=1}^k L_i \hspace{2mm} \textup{ and } \hspace{2mm}	\codim L = \sum_{i=1}^{k} \codim L_i \,.
		\end{equation}
	\end{enumerate}
\end{lemma}

\begin{proof}
	Item (i) follows immediately from the definition. Item (ii) follows from the fact that the arrangement \(\mH_L/L\) is linearly isomorphic to the product 
	\[
	\mH_L / L \cong
	(\mH_{L_1} / L_1) \times \ldots \times (\mH_{L_k} / L_k) \,,
	\]
	see \cite[Lemma 2.44]{pkc}.
\end{proof}

\begin{lemma}\label{lem:l1l2}
	Let \(L_1\) and \(L_2\) be two irreducible subspaces such that their intersection \(L = L_1 \cap L_2\) is non-empty and reducible. Then \(\Irr(L) = \{L_1, L_2\}\). In particular, \(L_1 + L_2 = \CP^n\).
\end{lemma}

\begin{proof}
	Let \(\Irr(L) = \{L_1', \ldots, L_p'\}\). We want to show that \(p=2\) and that, up to a re-label, \(L_1=L_1'\) and \(L_2 = L_2'\). Since \(L_1\) is irreducible, it must contain one of the irreducible components of \(L\), say \(L_1 \supset L_1'\). Similarly, \(L_2 \supset L_i'\) for some \(i\). We claim that \(i \neq 1\). Indeed, if \(L_2 \supset L_1'\) then \(L = L_1 \cap L_2\) would contain \(L_1'\) but \(L \subsetneq L_1'\). We can assume that \(L_2 \supset L_2'\). Since \(L = L_1 \cap L_2\) and \(\codim L_i \leq \codim L_i'\) for \(i=1, 2\); we get that
	\[
	\codim L \leq \codim L_1 + \codim L_2 \leq \sum_{i=1}^{p} \codim L_i' = \codim L ,
	\]
	where the last equality follows from Equation \eqref{eq:transint}. We conclude that all inequalities must be equalities. In particular, 
	\(p=2\) and \(\codim L_i = \codim L_i'\) for \(i= 1, 2\); therefore \(L_1 = L_1'\) and \(L_2=L_2'\).
\end{proof}

\begin{notation}\label{not:transint}
	Let \(L_1, L_2 \in \mLi\)\,, we write \(L_1 \pitchfork L_2\) if \(L_1 \cap L_2\) is non-empty and reducible. 
\end{notation}

\begin{example}
	Let \(H_1\) and \(H_2\) be two hyperplanes of the arrangement \(\mH\) and let \(L\) be their intersection. Then \(H_1 \pitchfork H_2\) if and only if \(\mH_L = \{H_1, H_2\}\).
\end{example}

%\begin{lemma}
%	Let \(L_1, \ldots, L_p \in \mLi\). Assume that the following holds:
%	\begin{enumerate}[label=\textup{(\roman*)}]
%		\item The subspaces are pairwise non-comparable, i.e., if \(i \neq j\) then neither \(L_i \subset L_j\) nor \(L_j \subset L_i\).
%		\item Their common intersection \(L = \bigcap_{i=1}^p L_i\) is reducible.
%	\end{enumerate}
%	Then the \(L_i\)'s are the irreducible components of \(L\).
%\end{lemma}

\subsection{Weighted arrangements}\label{sec:weightarr}

\begin{definition}\label{def:weightedarr}
	A weighted arrangement \((\mH, \ba)\) is a hyperplane arrangement \(\mH\) in \(\CP^n\) together with a weight vector
	\(\ba \in \R^{\mH}\) whose components are positive real numbers \(a_H>0\) indexed by the elements \(H \in \mH\).
\end{definition}

Let \((\mH, \ba)\) be a weighted arrangement.
For an arbitrary non-empty and proper linear subspace \(L \subsetneq \CP^n\), consider the equation
\begin{equation}\label{eq:kltcondition}
\sum_{H | L \subset H} a _H < \codim L \,.
\end{equation}

\begin{lemma}\label{lem:kltcondition}
	The following conditions are equivalent: 
	\begin{enumerate}[label=\textup{(\roman*)}]
		\item Equation \eqref{eq:kltcondition} holds for every  \(L \in \mLi\)\,;
		\item Equation \eqref{eq:kltcondition} holds for every  \(L \in \mL\)\,;
		\item Equation \eqref{eq:kltcondition} holds for every  non-empty and proper linear subspace \(L \subsetneq \CP^n\).
	\end{enumerate}
\end{lemma}

\begin{proof}
	Let us show first that (i) implies (ii). Let \(L \in \mL\) and consider its irreducible decomposition 
	\begin{equation}\label{eq:irrdecpf}
	\mH_L = \mH_{L_1} \uplus \ldots \uplus \mH_{L_k} .
	\end{equation}
	Each \(L_i \in \mLi\) and \(L = \cap_{i=1}^k L_i\) is their common transverse intersection. In particular,
	\begin{equation}\label{eq:sumcod}
	\codim L = \codim L_1 + \ldots + \codim L_k .
	\end{equation}
	Equation \eqref{eq:irrdecpf} implies that every \(H \in \mH\) that contains \(L\) must contain exactly one of the irreducible components \(L_i\). It follows that
	\begin{equation}\label{eq:sumirrlem}
	\sum_{H | L \subset H} a_H = \sum_{H | L_1 \subset H} a_H + \ldots + \sum_{H | L_k \subset H} a_H 
	\end{equation}
	Applying Equation \eqref{eq:kltcondition} to each \(L_i\) together with Equations \eqref{eq:sumirrlem} and \eqref{eq:sumcod}, implies that Equation \eqref{eq:kltcondition} holds for \(L\). This finishes the proof that (i) implies (ii).
	
	Let us now show that (ii) implies (iii). Let \(L \subset \CP^n\) be a linear subspace and let \(L'\) be the common intersection of all \(H \in \mH\) that contain \(L\). 
	If there are no hyperplanes of the arrangement that contain \(L\) then Equation \eqref{eq:kltcondition} is trivially satisfied because the left hand side is equal to zero and \(\codim L \geq 1\). Since \(L \subset L'\), we get that \(L'\) is non-empty and therefore \(L' \in \mL\).
	Applying Equation \eqref{eq:kltcondition} to \(L'\) we obtain
	\[
	\sum_{H | L \subset H} a_H  = \sum_{H | L' \subset H} a_H < \codim L' \leq \codim L . 
	\]
	This finishes the proof that (ii) implies (iii). Since (iii) clearly implies (i), the lemma follows.
\end{proof}

Recall from Section \ref{sec:mainresults} that the weighted arrangement \((\mH, \ba)\) is klt if Equation \eqref{eq:kltcondition} holds for any \(L \in \mL\). As a direct consequence of Lemma \ref{lem:kltcondition}\,,
we obtain \(3\) slightly different but equivalent formulations of the klt condition.

\begin{corollary}\label{cor:kltcondition}
	Let \((\mH, \ba)\) be a weighted arrangement with weights \(a_H>0\). Then \((\mH, \ba)\) is klt if and only if any of the following equivalent conditions is satisfied:
	\begin{enumerate}[label=\textup{(\roman*)}]
		\item Equation \eqref{eq:kltcondition} holds for every  \(L \in \mLi\)\,;
		\item Equation \eqref{eq:kltcondition} holds for every  \(L \in \mL\)\,;
		\item Equation \eqref{eq:kltcondition} holds for every  non-empty and proper linear subspace \(L \subsetneq \CP^n\).
	\end{enumerate}
\end{corollary}

 The next result will be crucial in Section \ref{sec:stability}.

\begin{lemma}\label{lem:klycy}
	If \((\mH, \ba)\) is klt and CY then \(\mH\) is essential and irreducible.
\end{lemma}

\begin{proof}
	If the centre \(T\) of the arrangement is non-empty then the klt condition applied to \(L = T\)  implies that \(\sum_H a_H < n\) but this contradicts the CY condition \(\sum_H a_H = n+1\). Therefore, \(\mH\) is essential.
	
	If \(\mH\) is reducible then there are two linear subspaces \(L_1, L_2 \in \mL\) such that \(L_1 + L_2 = \CP^n\) and every \(H \in \mH\) contains either \(L_1\) or \(L_2\). Since \(\mH\) is essential, the subspaces \(L_1\) and \(L_2\) are disjoint, thus \(\codim L_1+ \codim L_2 = n+1\). It follows that 
	\[
	\begin{aligned}
	\sum_{H \in \mH} a_H &= \sum_{H | L_1 \subset H} a_H + \sum_{H | L_2 \subset H} a_H \\
	&< \codim L_1 + \codim L_2 = n + 1 \,,
	\end{aligned}
	\]
	which contradicts the CY condition.
\end{proof}

Lemmas \ref{lem:klycy} and \ref{lem:noauto} together yield the following:

\begin{corollary}\label{cor:noauto}
	If \((\mH, \ba)\) is klt and CY then there are no non-zero holomorphic vector fields tangent to all the members of \(\mH\).
\end{corollary}

Recall that the multiplicity \(m_L\) of a subspace \(L \in \mL\) is the number of hyperplanes \(H \in \mH\) that contain \(L\).
The next results asserts that arrangements with no subspaces of relatively high multiplicity are irreducible.

\begin{corollary}\label{cor:hir2irred}
	Let \(\mH \subset \CP^n\) be a non-empty arrangement such that 
	\begin{equation}\label{eq:mul}
		\forall \, L \in \mL: \,\,  m_L < \codim L \cdot \frac{|\mH|}{n+1} .
	\end{equation}
	Then \(\mH\) is essential and irreducible.
\end{corollary}

\begin{proof}
	Since for \(L \in \mL\) the multiplicity \(m_L\) is always greater or equal than \(\codim L\), we must have \(|\mH|> n+1\). Consider the CY weighted arrangement \((\mH, \ba)\) where all weights \(a_H\) are equal to \((n+1)/|\mH|\). As \(|\mH|>n+1\), we have \(0 < a_H < 1\). Equation \eqref{eq:mul} guarantees that \((\mH, \ba)\) is klt. The result follows from Lemma \ref{lem:klycy}\,.
\end{proof}

\section{The resolution}\label{sec:res}

In Section \ref{sec:blowup}\,, we recall a canonical compactification \(X\) of an arrangement complement that replaces \(\mH\) with a simple normal crossing divisor. This compactification is a particular instance of the \emph{wonderful models} of subspace arrangements, introduced by De Concini and Procesi \cite{conciniprocesi}.

We review the construction of \(X\) as an iterated blowup of \(\CP^n\) along linear subspaces and recall the notion of \emph{nested set}. The upshot is that the common intersection \(\bigcap_{L \in \mS} D_L\) is non-empty, where \(\{D_L \,|\, L \in \mS\}\) is a collection of irreducible components of \(D = \pi^{-1}(\mH)\), if and only if the set \(\mS\) is nested relative to \(\mLi\).\newpar

\noindent In Section \ref{sec:intnum}\,, we discuss the Picard group of \(X\). We also calculate some intersection numbers in the cohomology ring \(H^*(X, \Z)\). We will make use of the results on the Picard group and cohomology ring of \(X\) later, in Sections \ref{sec:parbun} and \ref{sec:stability}, in calculations of parabolic Chern classes of our parabolic bundle \(\mE_{*}\) and their products with the polarizations \(P_k\).\newpar

\noindent In Section \ref{sec:pol}\,, we introduce the polarizations \(P_k\) on \(X\) which are used in the \emph{stability} Theorem \ref{thm:stability}\,. 

\begin{comment}
In Section \ref{sec:res}\,, we recall a canonical resolution of the arrangement \(\mH\) known as the minimal 
De Concini-Procesi model of \(\mH\).
In \cite{conciniprocesi}\,, the authors define a resolution \(X_{\mG}\) of \(\mH\) for any \emph{building set} (\cite[\S 2.3]{conciniprocesi}) \(\mG\) of the intersection poset \(\mL\) of the arrangement. The irreducible subspaces \(\mLi\) form a building set of \(\mL\)
and any building set of \(\mL\) contains \(\mLi\). We take  \(X = X_{\mG}\) with \(\mG = \mLi\). If \(\mG_1\) and \(\mG_2\) are building sets with \(\mG_1 \subset \mG_2\), then \(X_{\mG_2}\) is obtained from \(X_{\mG_1}\) by performing a finite sequence of blowups along smooth subvarieties.
For this reason, the variety \(X\) is known as the \emph{minimal} De Concini-Procesi model of \(\mH\).

\(X \xrightarrow{\pi} \CP^n\) with the property that \(\pi^{-1}(\mH)\) is a simple normal crossing divisor. This resolution, known as the minimal De Concini-Procesi model of \(\mH\) \cite{conciniprocesi}, is constructed as an iterated blowup of \(\CP^n\) along the irreducible subspaces \(L \in \mLi\). The irreducible components of \(\pi^{-1}(\mH)\) are in one to one correspondence with the elements of \(\mLi\)\,, in a way that for each \(L \in \mLi\) there is a unique irreducible component \(D_L\) of \(\pi^{-1}(\mH)\) with \(\pi(D_L) = L\).
\end{comment}

\subsection{The minimal De Concini-Procesi model}\label{sec:blowup}

Before introducing De Concini and Procesi's wonderful models, we recall, as a warmup, one of the standard constructions of the blowup of \(\CP^n\) along a linear subspace.

Let \(L \subset \CP^n\) be a linear subspace and let \(P_L\) be a complementary subspace, that is \(L + P_L = \CP^n\) and \(L \cap P_L = \emptyset\).
The linear projection 
\[
\pr_L : \CP^n \setminus L \to P_L \,,
\]
sends a point \(p \in \CP^n \setminus L\) to the intersection of \(L + p\) with \(P_L\).
The blowup of \(\CP^n\) along \(L\) is the map \(B \xrightarrow{\sigma} \CP^n\), where \(B \subset \CP^n \times P_L\) is the closure of the graph of \(\pr_L\) and \(\sigma\) is the restriction to \(B\) of the projection to the first factor.

Consider now a hyperplane arrangement \(\mH \subset \CP^n\). Let
\((\CP^n)^{\circ}\) be the arrangement complement and let \(\mLi\) be the set of non-empty and proper irreducible subspaces. 
For each \(L \in \mLi\) choose a complementary subspace \(P_L\) and define
\begin{equation}\label{eq:pr}
\pr: (\CP^n)^{\circ} \to \prod_{L \in \mLi} P_L	\,,
\end{equation}
with components the linear projections \(\pr_L\)\,. 

\begin{definition}\label{def:resolution}
	Let \(X \subset \CP^n \times \prod P_L\) be the closure of the graph of \eqref{eq:pr} and let 
	\begin{equation}\label{eq:resol}
		\pi: X \to \CP^n 
	\end{equation}  	
	be the restriction to \(X\) of the projection to the first factor.
	
	The resolution of \(\mH\) is the variety \(X\) together with the map \(\pi\).
	For brevity, we will also refer to it simply as the resolution \(X\), omitting the map \(\pi\), or as the resolution \(\pi\), omitting the variety \(X\), depending on which aspect we wish to emphasize. 
\end{definition}

\begin{remark}
	De Concini and Procesi define a resolution
	\(X_{\mG}\) for any \emph{building set} \(\mG\) of the poset \(\mL\) \cite[\S 2.3]{conciniprocesi}. The irreducible subspaces \(\mLi\) form a building set of \(\mL\)
	and any building set of \(\mL\) contains \(\mLi\). Definition \ref{def:resolution}
	corresponds to  \(X = X_{\mG}\) with \(\mG = \mLi\). 
	If \(\mG_1\) and \(\mG_2\) are building sets with \(\mG_1 \subset \mG_2\), then \(X_{\mG_2}\) is obtained from \(X_{\mG_1}\) by performing a finite sequence of blowups along smooth subvarieties.
	For this reason, \eqref{eq:resol} is known as the \emph{minimal} De Concini-Procesi model of \(\mH\).
\end{remark}

The map \(\pi\) is a bijection on the preimage of \((\CP^n)^{\circ}\) with inverse \(\imath(x) = (x, \pr(x))\). 
By slight abuse of notation, suppressing \(\imath\), we can regard \((\CP^n)^{\circ}\) as an open subset of \(X\) and write
\[
X \setminus (\CP^n)^{\circ} = \pi^{-1}(\mH) .
\]	
The main result that we are after is the next.

\begin{theorem}[{\cite[\S 3.2 and \S 4.2]{conciniprocesi}}]\label{thm:conproc}
	The variety \(X\) is smooth and the preimage of the arrangement \(D = \pi^{-1}(\mH)\) is a simple normal crossing divisor whose irreducible components are in one to one correspondence with the elements of \(\mLi\).
	
	More precisely, 
	\begin{equation}
	D = \bigcup_{L \in \mLi} D_L ,
	\end{equation}	
	where \(D_L\) is the unique irreducible component of \(D\) such that \(\pi(D_L) = L\).
\end{theorem}

%The resolution \(\pi\) is a birational transformation of projective varieties, meaning that it has a rational inverse defined outside a subvariety of codimension \(\geq 2\). More precisely, if \(U = \CP^n \setminus Z\), where \(Z\) is the union of all irreducible subspaces of codimension \(\geq 2\), then \(\pi\) is a biholomorphism between \(\pi^{-1}(U)\) and \(U\). In order to see this, note that if \(L\) is a hyperplane then the complement \(P_L\) is a point and \(\pr_L\) is constant, so \(x \mapsto (x, \pr(x))\) extends holomorphically over \(U\) and provides a two-sided inverse for \(\pi\).

\begin{remark}
	It is shown in \cite[\(\S 1.6\)]{conciniprocesi} that \(D_L\) is the closure of \(\pi^{-1}(L^{\circ})\) where \(L^{\circ}\) is the complement of the induced arrangement \(\mH^L\). 
	Moreover, \cite[\(\S 4.3\)]{conciniprocesi} shows that \(D_L\) is biholomorphic to the product
	\[
	D_L \cong X({\mH^L}) \times X({\mH_L/L})
	\]
	where \(X(\mK)\) denotes the minimal De Concini-Procesi model of the arrangement \(\mK\).
\end{remark}

\subsubsection{Construction of \(X\) as an iterated blowup}

We present a more hands-on description of \(X\) as an iterated blowup. Before stating the result, we recall first the notion of blowup along a submanifold.

If \(N\) is a complex manifold and \(S \subset N\) is a complex submanifold, then the blowup of \(N\) along \(S\) is a complex manifold \(M\) together with a proper holomorphic map \(M \xrightarrow{\sigma} N\) such that \(\sigma: M \setminus E \cong N \setminus S\) is a biholomorphism
and the restriction of \(\sigma\) to \(E\) is equivalent to the bundle projection \(\P(N_S) \to S\), where \(\P(N_S)\) is the projectivization of the normal bundle \(N_S\) of \(S \subset N\).
%where \(E = \sigma^{-1}(S)\) is the exceptional divisor, 

Next, we recall the notions of proper transform and exceptional divisor in a slightly more general context.
Let \(M\) and \(N\) be complex manifolds of the same dimension and let 
\[
f: M \to N
\] 
be a proper holomorphic map of degree \(1\).
Let \(E \subset M\) be the set of critical points of \(f\) and let \(S=f(E)\) be the critical values, so that \(f: M \setminus E \to N \setminus S\) is a biholomorphism. We recall the following notions.
\begin{itemize}
	\item If \(V\) is a subvariety of \(N\) with \(V \not\subset S\) then the
	proper transform of \(V\) by \(f\) is the subvariety of \(M\) obtained as the closure of \(f^{-1}(V \setminus S)\).
	\item A divisor \(D \subset M\) is \(f\)-exceptional if \(f(D)\) is an analytic subset of \(N\) of codimension \(\geq 2\).
\end{itemize}

In our case of interest, \(f\) is a composition of blowups along complex submanifolds.  Furthermore, the set of critical points is a (possibly reducible) \(f\)-exceptional divisor. 

\begin{example}
The irreducible \(\pi\)-exceptional divisors of the resolution \eqref{eq:resol} are precisely of the form \(D_L\) with \(L \in \mLi\) of codimension \(\geq 2\).	
\end{example}

\begin{notation}\label{not:Lirrcirc}
	Let \(\mLi^{\circ} = \mLi \setminus \mH\). In other words, \(\mLi^{\circ}\) is the set of non-empty irreducible subspaces of codimension \(\geq 2\).
\end{notation}

\begin{proposition}[{\cite[Proposition 2.13]{lili}}]\label{prop:itblowup}
	Let \(L_1, \ldots, L_k\) be a labelling of all the elements in \(\mLi^{\circ}\) compatible with the inclusion relations, i.e., if \(L_i \subset L_j\) then \(i \leq j\). Then the resolution \eqref{eq:resol} is equal to \(X_k \xrightarrow{\pi_k} \CP^n\), where the maps \(X_i \xrightarrow{\pi_i} \CP^n\)
	for \(1 \leq i \leq k\) are defined inductively as follows:
	\begin{enumerate}
		\item \(X_1 \xrightarrow{\sigma_1} \CP^n\) is the blowup of \(\CP^n\) along \(L_1\) and \(\pi_1 = \sigma_1\); 
		
		\item Let \(i \geq 2\) and suppose that \(\pi_{i-1}: X_{i-1} \to \CP^n\) is defined. Let \(\tL_i\) be the proper transform of \(L_i\) by \(\pi_{i-1}\). Then \(\tL_i\) is smooth and
		 \(\pi_i = \pi_{i-1} \circ \sigma_i\) where
		\(\sigma_i: X_i \to X_{i-1}\) is the blowup of \(X_{i-1}\) along \(\tL_i\). 
	\end{enumerate}
\end{proposition}

Proposition \ref{prop:itblowup} guarantees that the end result \(X \xrightarrow{\pi} \CP^n\) is independent of the labelling, as long as it is compatible with inclusion relations.
For example, we can order the elements of \(\mLi^{\circ}\) by increasing dimension \(\mLi^0, \mLi^1, \ldots, \mLi^{n-2}\), where the members of \(\mLi^i\) are taken in any order,  c.f. \cite[Theorem 4.2.4]{othypergeometric}. 

The following three examples illustrate the proposition.

\begin{example}
	If \(\mH \subset \CP^2\) then \(X\) is obtained by blowing up the points \(p \in \CP^2\) of \(\mH\) of multiplicity \(m_p \geq 3\).
\end{example}

\begin{example}\label{ex:dim3}
	If \(\mH \subset \CP^3\) then \(X\) is obtained in two steps.
	\begin{itemize}
		\item Step 1: blowup the points in \(\mLi^0\).
		\item Step 2: blowup of the proper transforms \(\tL\) of the lines \(L \in \mLi^1\). 
	\end{itemize}
	Here, if \(L_1\) and \(L_2\) are two irreducible lines meeting at a point \(p\), then,  by Lemma \ref{lem:l1l2}, the intersection point \(p\) must be irreducible. Therefore, the proper transforms \(\tL\) in Step 2 are mutually disjoint. 	
\end{example}

\begin{example}
	Let \(P\) and \(Q\) be two projective planes in \(\CP^4\) meeting at a single point. Consider the arrangement \(\mH \subset \CP^4\) made of \(6\) hyperplanes, \(3\) of which intersect along \(P\) and the other \(3\) intersect along \(Q\). This arrangement contains no irreducible points and no irreducible lines. The resolution \(X\) is obtained by blowing up one of the planes in \(\mLi^2 = \{P, Q\}\) and then blowing up the proper transform of the other plane. The result is independent of the order in which the two planes are blown up. In an affine chart where \(P = \{0\} \times \C^2\) and \(Q = \C^2 \times \{0\}\) the resolution is the product
	\[
	\Bl_0 \C^2 \times \Bl_0 \C^2 \xrightarrow{(\sigma, \sigma)} \C^2 \times \C^2
	\] 
	where \(\Bl_0 \C^2 \xrightarrow{\sigma} \C^2\) is the blowup of \(\C^2\) at the origin.
\end{example}

\subsubsection{Nested subsets of \(\mLi\) and intersections of divisors \(D_L\)}

Next, we analyse when a collection of divisors \(D_L\) with \(L\) ranging over a subset \(\mS \subset \mLi\) has non-empty intersection. To do this, we recall the following.

\begin{definition}[{\cite[\S 2.4]{conciniprocesi}}]\label{def:nested}
	Let \(\mS \subset \mLi\) be a subset of irreducible subspaces. The set \(\mS\) is \emph{nested relative to} \(\mLi\) if, for any \(L_1, \ldots, L_k \in \mS\) with \(k \geq 2\) pairwise non-comparable (i.e. \(L_i \not\subset L_j\) for \(i\neq j\)), their common intersection \(\bigcap_{i=1}^k L_i\) is non-empty and reducible.
\end{definition}

\begin{remark}
	If \(\mS\) is nested relative to \(\mLi\) and \(L_1, \ldots, L_k \in \mS\) are pairwise non-comparable, then the intersections \(\bigcap_{i \in I} L_i\) are non-empty and reducible for any subset \(I \subset [k]\) with \(|I| \geq 2\). Indeed, the elements of \(\{L_i \,|\, i \in I\} \subset \{L_1, \ldots, L_k\}\) are also pairwise non-comparable, therefore by Definition \ref{def:nested} their common intersection  \(\bigcap_{i \in I} L_i\) is non-empy and reducible.
\end{remark}

\begin{example}\label{ex:l1l2}
If \(L, M \in \mLi\) then \(\mS=\{L, M\}\) is nested if and only if either one subspace is contained in the other or the intersection \(L \cap M\) is non-empty and reducible.	
\end{example}

\begin{example}
	Consider the arrangement \(\mH\) of \(4\) hyperplanes with defining equations \(x_1 = 0\,,\, x_2=0\,,\, x_1 + x_2 = 0\), and \(x_3=0\).
	
	Then the set \(\mS\) of hyperplanes \(x_1 = 0\,,\, x_2 = 0\,,\, x_3 = 0\) doesn't form a nested set. Indeed, the intersection of \(x_1 = 0\) with \(x_2 = 0\) is not reducible (while all other intersections are reducible).
\end{example}

Next, we provide an extension of Lemma \ref{lem:l1l2} that will be useful later on.
Recall that we write \(\Irr(L)\) for the irreducible components of \(L \in \mL\).

\begin{lemma}\label{lem:nestedint}
	Let \(\mS\) be nested relative to \(\mLi\) and let \(L_1, \ldots, L_k \in \mS\) be pairwise non-comparable. If \(L = \bigcap_{i=1}^k L_i\) is their common intersection, then \(\Irr(L) = \{L_1, \ldots, L_k\}\). In particular, the subspaces \(L_1, \ldots, L_k\) intersect transversely.
\end{lemma}

\begin{proof}
	Let \(\Irr(L) = \{L_1', \ldots, L_p'\}\). We want to show that \(p=k\) and that, up to a relabel, \(L_i=L_i'\) for all \(1 \leq i \leq k\). Since \(L_i\) is irreducible, it must contain one (and only one) of the irreducible components of \(L\), say \(L_i \supset L_{\sigma(i)}'\). 
	Clearly,
	\[
	L'_j \subset \bigcap_{i \,|\, \sigma(i) = j} L_i \,:= M_j
	\]
	and
	\[
	\bigcap_{j=1}^p M_j = \bigcap_{i=1}^k L_i = L \,.
	\]
	Therefore,
	\[
	\codim L \leq \sum_{j=1}^{p} \codim M_j
	\leq \sum_{j=1}^{p} \codim L_j' = \codim L \,.
	\]
	Thus, we must have \(\codim M_j = \codim L_j'\) for all \(j\), hence
	\[
	\forall j \in [p]: \,\, L'_j = M_j\,.
	\] 
	On the other hand, since the elements of \(\{L_i \,|\, \sigma(i) = j\} \subset \mS\) are pairwise non-comparable, their common intersection \(M_j\) must be reducible if \(|\sigma^{-1}(j)| \geq 2\). Since \(M_j = L'_j\) is irreducible, we must have  \(|\sigma^{-1}(j)| = 1\) for all \(j\). We conclude that \(\sigma\) is a bijection, so \(p=k\) and we can relabel so that \(L_i = L'_i\) for all \(i\).
\end{proof}

The main point of introducing nested sets is that we have the following.

\begin{proposition}[{\cite[\S 3.2]{conciniprocesi}}]\label{prop:nested}
	Let \(\mS \subset \mLi\). The intersection
	\[
	\bigcap_{L \in \mS} D_L
	\] 
	is non-empty if and only if \(\mS\) is nested relative to \(\mLi\). 
\end{proposition}

In particular, from Example \ref{ex:l1l2}, we have.

\begin{corollary}\label{cor:2intersect}
	Two divisors \(D_{L}\) and \(D_{M}\) intersect if and only if one of the following two cases happens:
	\begin{enumerate}[label=\textup{(\roman*)}]
		\item \(L \subset M\) or \(M \subset L\);
		\item \(L \cap M\) is non-empty and reducible.
	\end{enumerate}
\end{corollary}

\subsubsection{Example: tetrahedron in \(\CP^3\)}\label{sec:tetra}

Let \(p_1, \ldots, p_5 \in \CP^3\) be five points in general linear position, i.e., no \(3\) points lie on a line and no \(4\) points lie on a plane. We can represent these five points as the four vertices of a tetrahedron in \(\CP^3\) together with its barycentre. Each triplet of points determines a plane, giving rise to a collection of \(10\) planes \(\mH \subset \CP^3\) -see Figure \ref{fig:tetra}\,.

\begin{figure}[H]
	\centering
	\begin{tikzpicture}
		
		\fill[yellow!30] (0,0) -- (4,0) -- (1.5,-1);
		
		\node[circle,fill=black,inner sep=0pt,minimum size=5pt,label=left:{$p_1$}] (x1) at (0,0) {};
		\node[circle,fill=black,inner sep=0pt,minimum size=5pt,label=right:{$p_3$}] (x3) at (4,0) {};
		\node[circle,fill=black,inner sep=0pt,minimum size=5pt,label=right:{$p_4$}] (x4) at (2,3) {};
		\node[circle,fill=black,inner sep=0pt,minimum size=5pt,label=below:{$p_2$}] (x2) at (1.5,-1) {};
		\node[circle,fill=black,inner sep=0pt,minimum size=5pt,label=right:{$p_5$}] (x5) at (2,1) {};
		
		\draw (x1) -- (x2) -- (x3) -- (x4) -- (x1);
		\draw (x4) -- (x2);
		\draw[dashed] (x1) -- (x3);
		\draw[dashed, blue] (x5) -- (x1);
		\draw[dashed, blue] (x5) -- (x2);
		\draw[dashed, blue] (x5) -- (x3);
		\draw[dashed, blue] (x5) -- (x4);
		
		\node (h) at (4,-1) {\(H_{123}\)};
		\draw[<->, rounded corners, black!60!green] (2.5,-.3) -- (2.5, -1) -- (h);
		\node[black!30!blue] (l) at (.5,-.7) {\(L_{12}\)};
		\draw[thick, black!30!blue] (x1) -- (x2);
	\end{tikzpicture}
	\caption{The arrangement \(\mH\) of \(10\) planes in \(\CP^3\) spanned by triplets of \(5\) points in general linear position.}
	\label{fig:tetra}
\end{figure}
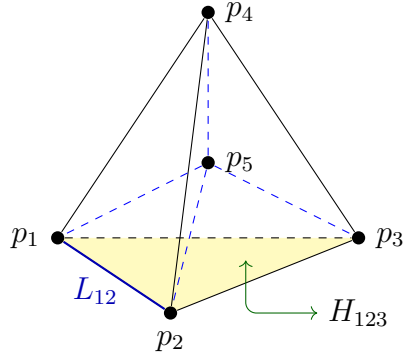

The irreducible subspaces of the arrangement are
\[
\mLi^2 = \{H_{ijk} = \overline{p_ip_jp_k}\}, \hspace{2mm} \mLi^1 = \{L_{ij} = \overline{p_ip_j}\}, \hspace{2mm} \mLi^0 = \{p_1, \ldots, p_5\} .
\] 
The resolution \(\pi = \sigma_1 \circ \sigma_2\) is constructed in two steps.
\begin{itemize}
	\item Step 1: \(X_1 \xrightarrow{\sigma_1} \CP^3\) is the blowup at the five points of \(\mLi^0\). 
	
	\item Step 2: \(X \xrightarrow{\sigma_2} X_1\) is the blowup at the \(10\) disjoint projective lines \(\tL\), where \(\tL\) are the proper transforms of the elements \(L \in \mLi^1\).
\end{itemize}

The preimage of the arrangement is a normal crossing divisor made of \(25\) irreducible components
\begin{equation*}
	\pi^{-1}(\mH) = \underbrace{\left( \bigcup_{H \in \mLi^2} D_H \right)}_{10 \text{ divisors } \cong \Bl_4 \P^2} \bigcup \underbrace{\left( \bigcup_{L \in \mLi^1} D_L \right)}_{10 \text{ divisors } \cong \P^1 \times \P^1} \bigcup \underbrace{\left( \bigcup_{p \in \mLi^0} D_p \right)}_{5 \text{ divisors } \cong \Bl_4 \P^2} .
\end{equation*}

\begin{remark}
	More generally, the braid or \(A_{n+1}\)-arrangement in \(\CP^n\) is the collection of \(\binom{n+2}{2}\) hyperplanes spanned by \(n\) out of \(n+2\) points in general linear position. The arrangement complement \((\CP^n)^{\circ}\) is naturally identified with the configuration space \(\mathcal{M}_{0, n+3}\) of \(n+3\) marked points in the Riemann sphere modulo M\"obius transformations. 
	It is well known that the minimal De Concini-Procesi model \eqref{eq:resol} agrees with the Deligne-Mumford-Knudsen compactification \(\overline{\mathcal{M}}_{0, n+3}\)\,, see  \cite[Theorem 4.3.3]{kapranov}. 
\end{remark}

\begin{remark}
	The action of the symmetric group \(S_6\) on \(X = \overline{\mathcal{M}}_{0, 6}\) is explained in \cite[Chapter 3]{hunt}. Continuing with the above notation,  
	\[X \xrightarrow{\sigma_2} X_1 \xrightarrow{\sigma_1} \CP^n .\]
	Contracting the ten projective lines \(\tL \subset X_1\), which have normal bundle \(\mO_{\P^1}(-1)^{\oplus 2}\), produces a \(3\)-fold with ten ordinary double point singularities which embeds in \(\CP^5\) as the Segre cubic
	\[
	\Bigl\{ \sum_{i=0}^{5} z_i = 0, \hspace{2mm} \sum_{i=0}^{5} z_i^3 = 0 \Bigr\} 
	\] 
	and \(S_6\) acts by permuting the coordinates.
\end{remark}

\subsection{Picard group and cohomology of \texorpdfstring{\(X\)}{X}}\label{sec:intnum}

In this section we calculate the Picard group of the resolution \(X\) of Section \ref{sec:blowup}. We also calculate some cup products in the integer cohomology of \(X\) that will be needed in the next sections. A set of generators and relations for the ring \(H^*(X, \Z)\) is given in \cite[\S 5]{conciniprocesi}. We provide quick and direct proofs of the formulas that we need rather than algebraic manipulation of the  relations there.

If \(X\) is a projective manifold, we write \(\Pic(X)\) for the group of isomorphisms classes of holomorphic -or equivalently algebraic- line bundles on \(X\). 
Given a divisor \(D \subset X\), we denote by \(\mO_X(D)\) the line bundle on \(X\) defined by the rank \(1\) locally free sheaf of rational functions on \(X\) with at most simple poles along \(D\).
We write \([D]\) for the class of  \(\mO_X(D)\) in \(\Pic(X)\).
We will use the following standard result.

%If \(X\) is a complex variety, we write \(\Pic(X)\) for the group of   isomorphisms classes of algebraic line bundles on \(X\) and \(\Cl(X)\) for the divisor class group of \(X\). We work with non-singular varieties, so there is a natural isomorphism \(\Cl(X) \cong \Pic(X)\) defined by \(D \mapsto \mO_X(D)\) where \(D \subset X\) is a divisor and \(\mO_X(D)\) is the rank \(1\) locally free sheaf of rational functions with at most simple poles along \(D\).
%Given a divisor \(D \subset X\), we write \([D]\) for its class in \(\Pic(X)\).

\begin{lemma}\label{lem:blowup}
	Let \(X\) be a nonsingular variety and let \(Y\) be a nonsingular subvariety of codimension \(r \geq 2\). Let \(\pi: \tX \to X\) be the blowup of \(X\) along \(Y\).
	\begin{enumerate}[label=\textup{(\roman*)}]
		\item If \(Y' = \pi^{-1}(Y)\) is the exceptional divisor then
			\begin{equation}
			\Pic(\tX) = \pi^* \left(\Pic(X)\right) \oplus \Z \cdot [Y'] .
		\end{equation}
		\item If \(D \subset X\) is a divisor with \(Y \not\subset D\) and \(\tD \subset \tX\) is the proper transform of \(D\) then
		\begin{equation}\label{eq:proptran}
			\pi^*([D]) = [\tD] .
		\end{equation}
	\end{enumerate}
\end{lemma}

\begin{proof}
	(i) This is Exercise 8.5 of Chapter II in \cite{hartshorne}.

	(ii) Let \(s\) be a section of \(\mO_X(D)\) with \(D=s^{-1}(0)\). Since \(Y \not\subset D\), the zero set of the pullback section \(\pi^*s\) is equal to \(\tD\), providing a trivialization of \(\pi^*(\mO_X(D)) \otimes \mO_{\tX}(-\tD)\). Taking isomorphism classes gives \eqref{eq:proptran}.
\end{proof}

\begin{comment}
(i) This is part of \cite[Ex. 8.5, Ch. II]{hartshorne}.
Let \(U = \tX \setminus Y'\).
By  \cite[Ch. II Proposition 6.5(c)]{hartshorne} we have an exact sequence
\begin{equation}\label{eq:clseq}
0 \to \Z \to \Cl(\tX) \to \Cl(U) \to 0 \,,
\end{equation}
where \(\Z \mapsto \Cl(\tX)\) is given by \(m \mapsto m \cdot [Y']\) and \(\Cl(X) \to \Cl(U)\) is given by restriction of divisors \(D \mapsto D \cap U\). 
Note that \(\Z \mapsto \Cl(\tX)\) is injective
because the restriction of the line bundle \(m \cdot [Y']\) to a fibre of \(\pi\) over a point \(y \in Y\) is isomorphic to \(\mO_{\P^{r-1}}(-m)\) which is non-trivial for \(m \neq 0\).

Since \(\pi\) is an isomorphism between \(U\) and \(X \setminus Y\), we have
\[
\Cl(U) \cong \Cl(X \setminus Y) \cong \Cl(X)
\]
where the second equality holds by \cite[Ch. II Proposition 6.5(b)]{hartshorne}.
Using the natural isomorphisms \(\Pic(X) \cong \Cl(X)\) and \(\Pic(\tX) \cong \Cl(\tX)\) we have an exact sequence
\begin{equation*}
0 \to \Z \to \Pic(\tX) \to \Pic(X) \to 0
\end{equation*}
which is split by \(\pi^*\). 
\end{comment}

In the rest of the section we let \(\mH \subset \CP^n\) be an arrangement and let \(X \xrightarrow{\pi} \CP^n\) be the resolution of \(\mH\) as in Definition \ref{def:resolution}\,.
Recall that we write \(\mLi^{\circ} = \mLi \setminus \mH\).

\begin{proposition}\label{prop:picard}
	The Picard group of the resolution \(X\) is the free abelian group generated by the classes \([D_L]\) for \(L \in \mLi^{\circ}\) together with the class of \(\pi^*\mO_{\P^n}(1)\).
\end{proposition}

\begin{proof}
	We begin by recalling the construction of \(X\) by a sequence of blowups as in Proposition \ref{prop:itblowup}\,. 
	The upshot is that \(X \xrightarrow{\pi} \CP^n\) is equal to \(X_k \xrightarrow{\pi_k} \CP^n\) where \(X_i \xrightarrow{\pi_i} \CP^n\) are constructed inductively as follows.
	
	Let \(L_1, \ldots, L_k\) be a labelling of the elements of \(\mLi^{\circ}\) such that if \(L_i \subset L_j\) then \(i \leq j\). Start with \(X_0 = \CP^n\) and \(\pi_0 =\) the identity map.
	Then \(\pi_i = \pi_{i-1} \circ \sigma_i\) for \(1 \leq i \leq k\), where \(X_i \xrightarrow{\sigma_i} X_{i-1}\) is the blowup of
	\(X_{i-1}\) along \(\tL_i\) and \(\tL_i\) is the proper transform of \(L_i\) by \(\pi_{i-1}\).
	
	By construction, each \(X_i\) contains exactly \(i\) irreducible \(\pi_i\)-exceptional divisors \(D^{i}_j\) with \(\pi_{i}(D^{i}_j) = L_j\) for \(1 \leq j \leq i\). 	Moreover, if \(j \leq i-1\) then \(D^i_j = \tD^{i-1}_{j}\) is the proper transform of \(D^{i-1}_{j}\) by \(\sigma_i\) while \(D^i_i = \sigma_i^{-1}(\tL_i)\) is the \(\sigma_i\)-exceptional divisor.
	
	\textbf{Claim.} The Picard group of \(X_i\) is a free abelian group of rank \(i+1\) generated by the \(\pi_i\)-exceptional divisors \([D^i_j]\) for \(1 \leq j \leq i\) together with the class of \(\pi_i^* \left(\mO_{\P^n}(1)\right)\).
	
	We prove the claim by induction on \(i\). The case \(i=0\) follows from the fact that \(\Pic(\CP^n)\) is generated by \(\mO_{\P^n}(1)\). Suppose that \(1 \leq i \leq k\) and assume that the Picard group of \(X_{i-1}\) is the free abelian group generated by the \(\pi_{i-1}\)-exceptional divisors \([D^{i-1}_j]\) for \(1 \leq j \leq i-1\) together with the class of \(\pi_{i-1}^* \left(\mO_{\P^n}(1)\right)\). 
	By Lemma \ref{lem:blowup} (i)  the Picard group of \(X_i\) is the free abelian group generated by
	\(\sigma_i^*(\pi_{i-1}^*(\mO_{\P^n}(1))) = \pi_i^*(\mO_{\P^n}(1))\) together with \(\sigma_i^*([D^{i-1}_j])\) for \(1 \leq j\leq i-1\) and \([\sigma_i^{-1}(\tL_i)] = [D^i_i]\). Furthermore, if \(\tL_i \subset D^{i-1}_j\) for some \(1 \leq j \leq i-1\) then
	\[
	L_i = \pi_{i-1}(\tL_i)  \subset \pi_{i-1}(D^{i-1}_j) = L_j
	\]
	which contradicts the inclusion preserving property of the labels. Thus, if \(1 \leq j \leq i-1\) then \(\tL_i \not\subset D^{i-1}_j\) and by Lemma \ref{lem:blowup} (ii) \(\sigma_i^*([D^{i-1}_j]) = [\tD^{i-1}_j] = [D^i_j]\). This finishes the proof of the claim.
	
	Having proved the claim, the statement follows since the divisors \(D^k_j\) for \(1 \leq j \leq k\) are precisely the divisors of the form \(D_L\) for \(L \in \mLi^{\circ}\).
\end{proof}

For the classes \(D_H\) with \(H \in \mH\) we have the following.

\begin{lemma}\label{lem:DH}
	If \(H \in \mH\) then
	\begin{equation}\label{eq:dh}
		[D_H] = [\pi^*(\mO_{\P^n}(1))] - \sum_{L | L \subsetneq H} [D_L]
	\end{equation}
	where the sum runs over all \(L \in \mLi\) properly contained in \(H\).
\end{lemma}

\begin{proof}
	Let \(s\) be a section of \(\mO_{\P^n}(1)\) with \(s^{-1}(0) = H\). The zero set of \(\pi^*s\) is equal to 
	\[
	\pi^{-1}(H) = \bigcup_{L | L \subset H} D_L
	\]
	so \(\pi^*(s)\) gives a trivialization of the tensor product of \(	\pi^*(\mO_{\P^n}(1))\) with \(\mO_X \left(-\sum D_L  \right)\) where the sum runs over all irreducible subspaces \(L\) contained in \(H\) (including \(L=H\)). Splitting the sum as \(H\) plus the sum pf all irreducible subspaces which are properly contained in \(H\) and taking classes in \(\Pic(X)\) gives \eqref{eq:dh}.
\end{proof}

Now we switch gears and discuss the integer cohomology of \(X\).

\begin{lemma}\label{lem:picx}
	The first Chern class is an isomorphism between \(\Pic(X)\) and \(H^2(X, \Z)\). 
\end{lemma}

\begin{proof}
	Since \(X\) is K\"ahler and simply connected, it follows that 
	\[b^{0,1}(X) = \dim H^1(X, \mO_X) = 0 .\]
	On the other hand, if \(\alpha\) is a holomorphic \(2\)-form on \(X\) then by Hartogs it defines a holomorphic \(2\)-form on \(\CP^n\), so \(\alpha = 0\). It follows that \(b^{2,0}(X) = 0\) and, since \(X\) is K\"ahler, we have
	\[
	b^{0,2}(X) = \dim H^2(X, \mO_X) = 0.
	\]
	Consider the exponential sequence \(0 \to \Z \to \mO_X \xrightarrow{\exp} \mO_X^* \to 0\). The long exact sequence in cohomology gives us
	\[
	0 \to H^1(X, \mO_X^*) \xrightarrow{c_1} H^2(X, \Z) \to 0 ,
	\]
	showing that \(c_1\) is an isomorphism between \(\Pic(X)\) and \(H^2(X, \Z)\). 
\end{proof}

\begin{notation}\label{not:gammaL}
	For \(L \in \mLi\) we write
	\begin{equation}
		\gamma_L = c_1(D_L).
	\end{equation}
	Equivalently, \(\gamma_L \in H^2(X, \Z)\) is the Poincar\'e dual of the divisor \(D_L \subset X\).
\end{notation}

\begin{definition}\label{not:h}
	We write \(h\) for the generator of \(H^2(\CP^n, \Z)\) given by the hyperplane class
	\begin{equation}
		h = c_1(\mO_{\P^n}(1)) .
	\end{equation}
	So \(\pi^*h\) is the Poincar\'e dual of the proper transform \(\tQ \subset X\) of a generic hyperplane \(Q \subset \CP^n\) that intersects transversely all the members of \(\mLi\).
\end{definition}

\begin{corollary}\label{cor:h2}
	\(H^2(X, \Z)\) is the free abelian group generated by the classes \(\gamma_L\) with \(L \in \mLi^{\circ}\) together with \(\pi^*h\).
\end{corollary}

\begin{proof}
	This follows from Proposition \ref{prop:picard} together with Lemma \ref{lem:picx}\,.
\end{proof}

\begin{corollary}
	For \(H \in \mH\) we have
	\begin{equation}\label{eq:tH}
	\gamma_H = \pi^*h - \sum_{L | L \subsetneq H} \gamma_L 
	\end{equation}
	where the sum runs over all \(L \in \mLi\) properly contained in \(H\).
\end{corollary}

\begin{proof}
	This follows from Equation \eqref{eq:dh} by taking the first Chern class.
\end{proof}

If \(\alpha \in H^i(X, \Z)\) and \(\beta \in H^j(X, \Z)\) then 
\begin{equation}
	\alpha \cdot \beta \in H^{i+j}(X, \Z)
\end{equation}
is their cup product. In the rest of the section we calculate diverse
cup products between the classes \(\gamma_L\) and \(\pi^*h\).

\begin{remark}\label{rmk:emptyint}
	If \(\alpha\) and \(\beta\) are the Poincar\'e duals of submanifolds \(A\) and \(B\) with transverse intersection, then \(\alpha \cdot \beta\) is the Poincar\'e dual of \(A \cap B\). In particular, if \(A \cap B = \emptyset\) then \(\alpha \cdot \beta = 0\).
\end{remark}

We need the following relation between the pullback of the Poincar\'e dual of a submanifold and the Poicar\'e dual of its proper transform.

\begin{lemma}\label{lem:pdpt}
	Let \(X\) be a nonsingular variety and let \(Y\) be a nonsingular subvariety. Let \(\pi: \tX \to X\) be the blowup of \(X\) along \(Y\). Suppose that \(V \subset X\) is a nonsingular subvariety of codimension \(r\) that is transversal to \(Y\).
	Let \(\tV \subset \tX\) be the proper transform of \(V\). Then
	\begin{equation}
		\gamma_{\tV} = \pi^*\gamma_V \, ,
	\end{equation}
	where \(\gamma_V \in H^{2r}(X, \Z)\) and \(\gamma_{\tV} \in  H^{2r}(\tX, \Z)\) are the Poincar\'e duals of \(V\) and \(\tV\)
\end{lemma}

The proof of Lemma \ref{lem:pdpt} is standard and we omit it, see \cite[Corollary 6.7.2]{fulton}.  Note that, under the hypothesis of Lemma \ref{lem:pdpt}, the proper transform \(\tV\) is nonsingular and is equal to the blowup of \(V\) along \(Y \cap V\).

Let \(P\) be a linear subspace \(P \subset \CP^n\) of dimension \(k\). We say that \(P\) is \emph{generic} if it intersects transversely the elements of \(\mL\). Concretely, if \(L \in \mL\) is such that  \(\codim L > k\) then \(P \cap L = \emptyset\), while if \(\codim L \leq k\) then \(P + L = \CP^n\).	
Clearly, the generic linear subspaces \(P \subset \CP^n\) of dimension \(k\) make an open dense subset of the Grassmannian of \(k\)-planes in \(\CP^n\).

\begin{lemma}\label{lem:ptgenplane}
	The class \((\pi^*h)^{n-k}\) is the Poincar\'e dual of the proper transform \(\tP\) of a generic \(k\)-plane \(P \subset \CP^n\).
\end{lemma}

\begin{proof}
	 We use the description of \(X\) by a sequence of blowups given in Proposition \ref{prop:itblowup}\,. The generic assumption implies that the proper transform \(\tP_i \subset X_i\) under \(X_i \xrightarrow{\pi_i} \CP^n\) is transversal to \(\tL_{i+1}\) for all \(i\). Then the statement follows by repeated application of Lemma \ref{lem:pdpt}\,.
\end{proof}

\begin{lemma}\label{lem:vanishing}
	Let \(L_1, \ldots, L_k\) be a set of not  necessarily distinct irreducible subspaces and assume that the common intersection
	\[
	L = \bigcap_{i=1}^k L_i
	\]
	has codimension \(r\). If \(k < r\) then
	\begin{equation}\label{eq:vanishingproductmixed}
	(\pi^*h)^{n-k} \cdot \prod_{i=1}^k \gamma_{L_i} = 0 .
	\end{equation}
\end{lemma}

\begin{proof}
	By Lemma \ref{lem:ptgenplane}, the class \((\pi^* h)^{n-k}\) is Poincar\'e dual of the proper transform \(\widetilde{P}\) of a generic \(k\)-plane \(P \subset \CP^n\). If \(k < r\) then \(P\)  does not intersect \(L\). Therefore,  \(\widetilde{P}\) does not intersect \(\bigcap_{i=1}^k D_{L_i}\) and the result follows from Remark \ref{rmk:emptyint}\,.
\end{proof}

\begin{corollary}\label{cor:vanishing}
	Let \(L \in \mLi\) with \(\codim L = r\).
	\begin{enumerate}[label=\textup{(\roman*)}]
		\item If \(r > 1\) then \((\pi^*h)^{n-1} \cdot \gamma_L = 0\).
		\item If \(r > 2\) then \((\pi^*h)^{n-2} \cdot \gamma_L^2 = 0\).
	\end{enumerate}
\end{corollary}

\begin{lemma}\label{lem:topth}
	For every \(H \in \mH\) we have
	\begin{equation}\label{eq:topth}
	(\pi^*h)^{n-1}  \cdot \gamma_H  = 1 .
	\end{equation}
\end{lemma}

\begin{proof}
	It follows from Equation \eqref{eq:tH} together with Corollary \ref{cor:vanishing} (i) that
	\begin{align*}
	(\pi^* h)^{n-1}  \cdot \gamma_H &= (\pi^* h)^{n-1}  \cdot \left( \pi^*h - \sum_{L | L \subsetneq H} \gamma_L  \right) \\
	&= (\pi^* h)^{n} = 1  
	\end{align*}
	where the last equality holds because \(h^n=1\) and \(\pi\) has degree \(1\).
\end{proof}

\begin{lemma}
	For every \(L \in \mLi\) with \(\codim L = 2\) we have
	\begin{equation}\label{eq:selftl}
	(\pi^*h)^{n-2}  \cdot \gamma_L^2  = -1 .
	\end{equation}
\end{lemma}

\begin{proof}
	Let \(P\) be a generic \(2\)-plane that intersects \(L\) transversely at a point \(p \in L^{\circ}\). The class \((\pi^* h)^{n-2}\) is the Poincar\'e dual of the proper transform \(\tP \subset X\) of \(P\).
	
	In a neighbourhood of \(p\), we can identify the resolution \(X \xrightarrow{\pi} \CP^n\) with the blowup of \(\CP^n\) along \(L\) so that:
	\begin{itemize}
		\item \(D_L = \P(N_L)\) is the exceptional divisor, where \(\P(N_L)\) is the projectivization of the normal bundle \(N_L\) of \(L \subset \CP^n\);
		\item the class \((\pi^*h)^{n-2} \cdot \gamma_L\) is the Poincar\'e dual of a curve \(C=\pi^{-1}(p)\) which is a fibre of the bundle projection \(\P(N_L) \xrightarrow{\pi} L\).
	\end{itemize}
	Since \(N_{D_L} = \mO_{\P(N_L)}(-1)\) is the tautological bundle on \(\P(N_L)\) that restricts to \(\mO_{\P^1}(-1)\) on each fibre, we have
	\[
	C \cdot \gamma_L = \deg (N_{D_L}|_C) = -1
	\]
	and Equation \eqref{eq:selftl} follows.
\end{proof}

\begin{lemma}\label{lem:lh}
	If \(L \in \mLi^{n-2}\) and \(H \in \mH\) contains \(L\), then
	\begin{equation}\label{eq:tlth}
	(\pi^*h)^{n-2}  \cdot \gamma_L \cdot \gamma_H  = 1 \,.
	\end{equation}
\end{lemma}

\begin{proof}
	By Equation \eqref{eq:tH}, the product \((\pi^*h)^{n-2}  \cdot \gamma_L \cdot \gamma_H\) is equal to
	\begin{equation}\label{eq:prodhl}
		(\pi^*h)^{n-2}  \cdot \gamma_L \cdot 
		\left(\pi^*h - \sum_{L' | L' \subsetneq H} \gamma_{L'} \right) .
	\end{equation}
	By Corollary \ref{cor:vanishing} (i) the product \((\pi^*h)^{n-1} \cdot \gamma_L\) is zero. By Lemma \ref{lem:vanishing}, if \(L' \neq L\) then \((\pi^*h)^{n-2}  \cdot \gamma_L \cdot \gamma_{L'} = 0\).
	It follows that \eqref{eq:prodhl} is equal to
	\[
	- (\pi^* h)^{n-2}  \cdot \gamma_L^2 
	\]
	and the result follows from Equation \eqref{eq:selftl}.
\end{proof}

\begin{definition}\label{def:bh}
	For \(H \in \mH\) let \(B_H\) be the number of irreducible subspaces of codimension \(2\) contained in \(H\) minus \(1\), i.e.,
	\begin{equation}\label{eq:bh}
	B_H = \left| \{L \in \mLi^{n-2} \, | \, L \subset H  \} \right| - 1 \,.
	\end{equation}
\end{definition}

\begin{lemma}
	For every \(H \in \mH\) we have
	\begin{equation}\label{eq:selfth}
	(\pi^*h)^{n-2}  \cdot \gamma_H^2  = -B_H \,.
	\end{equation}
\end{lemma}

\begin{proof}
	By the above results, we have
	\begin{equation*}
	\begin{aligned}
	(\pi^*h)^{n-2}  \cdot \gamma_H^2 &=
	(\pi^*h)^{n-2} \cdot \left( \pi^*h - \sum_{L | L \subsetneq H} \gamma_L \right)^2  \\
	&= (\pi^*h)^{n-2} \cdot \left( (\pi^*h)^2 + \sum_{\substack{L \subset H \\ \codim L = 2}} \gamma_L^2 \right) \\
	&= 1 - \left| \{L \in \mLi^{n-2} \, | \, L \subset H  \} \right| \,.
	\end{aligned}
	\end{equation*}
	The first equality uses Equation \eqref{eq:tH}\,. 
	The second equality gets rid of the parenthesis terms  \(\gamma_L \cdot \gamma_{L'}\)\,\,,\,\, \(\pi^* h \cdot \gamma_L\)\,\,, and \(\gamma_L^2\) if \(\codim L \geq 3\)\,\,;\,\, by using
	Equation \eqref{eq:vanishingproductmixed} together with items (i) and (ii) of Corollary \ref{cor:vanishing} respectively. Finally, the third equality follows from Equation \eqref{eq:selftl}.
\end{proof}

\subsection{Polarization}\label{sec:pol}

Let \(X \xrightarrow{\pi} \CP^n\) be the resolution of \(\mH\) as in Definition \ref{def:resolution}\,. Recall that we write 
\[\mLi^{\circ} = \mLi \setminus \mH\] 
for the subset of irreducible subspaces of codimension \(\geq 2\).

\begin{lemma}\label{lem:polarization}
	We can choose integers \(b_L>0\) for each \(L \in \mLi^{\circ}\) such  that
	\begin{equation}\label{eq:pol}
	P_k = k \cdot \pi^*\big(\mO_{\P^n}(1)\big) - \sum_{L \in \mLi^{\circ}} b_L \cdot D_L
	\end{equation}
	is an ample divisor on \(X\) for all \(k \gg 1\). 
\end{lemma}

\begin{proof}
	This can be proved inductively by considering a single blowup along a smooth subvariety, see item (ii) of Corollary 4.1.4 in \cite{lazarsfeld}. 
	There is a standard way to make a choice of coefficients \(b_L\) so that \(P_k := k \cdot \pi^*(\mathcal{O}(1)) - \sum b_L D_L\) is ample for all \(k \gg 1\), for completeness we recall the argument.
	
	Fix an ample divisor $H$ on $X_0 := X$ (in our case \(H= \mathcal{O}(1)\) and \(X_0 = \mathbb{CP}^n\)). Let
	\[
	X_r \xrightarrow{\pi_r} X_{r-1} \xrightarrow{\pi_{r-1}} \cdots 
	\xrightarrow{\pi_1} X_0
	\]
	be a sequence of blow-ups, where $\pi_i$ is the blow-up of $X_{i-1}$ along a smooth centre $Z_i$, and let $E_i \subset X_r$ denote the total transform of the exceptional divisor of $\pi_i$. Writing $\pi := \pi_1 \circ \cdots \circ \pi_r$, we let
	\[
	H_r := \pi^* H - \sum_{i=1}^r \varepsilon_i E_i \,,
	\]
	for some positive rational numbers
	\[
	0 < \varepsilon_r \ll \cdots \ll \varepsilon_1 \ll 1 .
	\]
	The \(\Q\)-divisor $H_r$ is ample on $X_r$ . Let \(b_i\) and \(k_0\) be positive integers such that \(\varepsilon_i = b_i / k_0\) for all \(i\), so
	\[
	k_0 \cdot H_r = k_0 \cdot \pi^* H - \sum_{i=1}^r b_i E_i 
	\]
	is ample on \(X_r\). Thus
	\[
	P_k := (k-k_0)\cdot \pi^* H + k_0 \cdot H_r  = k \cdot \pi^* H - \sum_{i=1}^r b_i E_i 
	\]
	is is also ample on \(X_r\) for any \(k \geq k_0\).
\end{proof}

\begin{notation}\label{not:bL}
	Fix positive integers \(b_L\) such that \(P_k\) as in Lemma \ref{lem:polarization} is an ample line bundle on \(X\) for all \(k \gg 1\). We refer to \(b_L\) as the \emph{polarization coefficients}.
\end{notation}

The next bound on the volumes of exceptional divisors and proper transforms of hyperplanes will be useful later in the proof of the main result.

\begin{lemma}\label{lem:boundvol}
	Let \(L \in \mLi\) and let 
	\[\vol_{P_k} (D_L) = c_1(P_k)^{n-1} \cdot \gamma_L\] 
	be the volume of \(D_L\) with respect to \(P_k\). Then
	\begin{equation}\label{eq:vol}
		\vol_{P_k} (D_L) = \begin{cases}
		f(k) &\textup{ if } \codim L \geq 2 ,\\
		k^{n-1} + f(k) &\textup{ if } \codim L = 1 ,
	\end{cases}	
	\end{equation}
	where \(f(k) = \sum_{j=0}^{n-2} C_j k^j\) is a polynomial in \(k\) of degree at most \(n-2\). 
	
	Moreover, there is \(\Lambda >0\) that depends only on \(\mH\) and the polarization coefficients such that \(|C_j| \leq \Lambda\) for all \(j\). 
\end{lemma}

\begin{proof}
	Write \(c_1(P_k) = k \cdot \pi^*h - e\) with \(e = \sum_{L' \in \mLi^{\circ}} b_{L'} \cdot \gamma_{L'}\) and use the binomial expansion to obtain
	\[
	\vol_{P_k}(D_L) = \sum_{j=0}^{n-1}  C_j \cdot k^j
	\]
	where
	\begin{equation}\label{eq:volcj}
		C_j = \binom{n-1}{j} \cdot (-e)^{n-1-j} \cdot (\pi^*h)^j \cdot \gamma_L .
	\end{equation}
	By Corollary \ref{cor:vanishing} (i) and Lemma \ref{lem:topth}, the leading coefficient \(C_{n-1} = (\pi^*h)^{n-1} \cdot \gamma_L\) is  
	\[
	C_{n-1} = 
	\begin{cases}
	0 \,\, &\text{ if } \,\, \codim L \geq 2 ,\\
	1 \,\, &\text{ if } \,\, L = H \in \mH 
	\end{cases}
	\]
	and \eqref{eq:vol} follows from this. It follows from \eqref{eq:volcj} that we can bound \(|C_j|\) in terms of \(\max b_L\) and the maximum absolute value of the top products
	\[
	(\pi^*h)^j \cdot \prod_{L' \in \mS} \gamma_{L'} 
	\]
	with \(\mS\) a set of not necessarily distinct \(n-j\) irreducible subspaces \(L' \in \mLi\), which only depends on \(\mH\).
\end{proof}

\begin{lemma}\label{lem:boundo1}
	The following holds 
	\[
	c_1(P_k)^{n-1} \cdot \pi^*h = k^{n-1} + f(k)
	\]
	where \(f(k)\) is as in Lemma \ref{lem:boundvol}\,.
\end{lemma}

\begin{proof}
	The same proof of Lemma \ref{lem:boundvol} applies by replacing \(\gamma_L\) with \(\pi^*h\).
\end{proof}

\begin{remark}
	An alternative differential-geometric proof of Lemma \ref{lem:boundvol} can be done as follows. Let \(\epsilon = 1/k\) and take a K\"ahler metric \(\omega_{\epsilon}\) in the cohomology class of \((1/k)P_k\) given by 
	\[
	\omega_{\epsilon} = \pi^*\omega_{\FS} + \epsilon \cdot \eta
	\] 
	where \(\omega_{\FS}\) is the Fubini-Study metric normalized so that  \([\omega_{\FS}] \in c_1(\mO_{\P^n}(1))\) and \(\eta\) is a fixed real and closed \((1,1)\)-form. The form \(\eta\) depends only on the polarization coefficients and it is positive when restricted to the kernel of \(\pi_{*}\), it can be constructed as in the proof of Proposition 3.24 in \cite{voisin}.
	
	If \(\codim L \geq 2\), then \(\pi^*\omega_{\FS}\) is degenerate along \(TD_L\) and we get that
	\[\int_{D_L} \omega_{\epsilon}^{n-1} = O(\epsilon) .\]
	Similarly, \(\int_{D_H} \omega_{\epsilon}^{n-1} = 1 + O(\epsilon)\) for every \(H \in \mH\).
\end{remark}

\section{The parabolic bundle}\label{sec:parbun}

In Section \ref{sec:genthry}\,, we present background material from the theory of parabolic bundles. \newpar

\noindent In Section \ref{sec:pardef}\,, 
we define a natural parabolic
bundle \(\mE_{*}\) on \((X, D)\), where \(X\) is the minimal De Concini-Procesi model  of \(\mH\)
%\todo[inline]{Could you please call it something more specific rather than "the above resolution" - above is usually a few lines above rather than a few pages} 
and \(D = \pi^{-1}(\mH)\). The underlying vector bundle is the pullback tangent bundle \(\mE = \pi^*(T\CP^n)\). The parabolic structure is defined by the filtrations of \(\mE|_{D_L}\) for \(L \in \mLi\) given by the subbundles \(\pi^*(TL) \subset \mE|_{D_L}\) together with weights \(a_L\).\newpar

\noindent In Section \ref{sec:locab}\,, we show that the parabolic bundle \(\mE_{*}\)
is locally abelian.\newpar

\noindent In Sections \ref{sec:parc1} and \ref{sec:parch2}\,, we calculate the first and second parabolic Chern characters of \(\mE_{*}\).
Specifically, we show that \(\parch_1(\mE_{*}) = 0\) and give a formula for \(\parch_2(\mE_{*})\) in terms of the  poset \(\mLi\).

\subsection{General theory} \label{sec:genthry}

Let \(X\) be a complex manifold and let \(D\) be a simple normal crossing divisor with irreducible decomposition \(D = \bigcup_{i \in I} D_i\)\,. Let \(\mE\) be a holomorphic vector bundle on \(X\).

\begin{definition}\label{def:parbun}
A parabolic structure \(\mE_{*}\) on \(\mE\) is given by a collection of locally free subsheaves \(\mE^i_a \subset \mE\) indexed by \(i \in I\) and \(a \in (0, 1]\) such that the following holds.
\begin{enumerate}[label=\textup{(\roman*)}]
	\item\label{it:support} Support/Increasing: for \(0 < a' < a \leq 1\) we have inclusions of \(\mO_X\)-modules
	\[
	\mE(-D_i) \subset \mE^i_{a'} \subset \mE^i_{a} \subset \mE \,,
	\]
	where \(\mE(-D_i)\) is the subsheaf of holomorphic sections of \(\mE\) that vanish along \(D_i\).
	\item\label{it:semicont} Semi-continuity: for any given \(a \in (0,1)\) there is \(\epsilon>0\) such that \(\mE^i_{a+\epsilon} = \mE^i_a\).
	\item\label{it:vb} Filtration in the category of vector bundles on \(D\): the quotient sheaves
	\begin{equation}\label{eq:fia}
		F^i_a = \left( \mE^i_a  \,\, \big/ \,\, \mE(-D_i) \right)|_{D_i}
	\end{equation}
	define an increasing filtration of \(\mE|_{D_i}\) by vector subbundles.
\end{enumerate}	
In this case, we say that \(\mE_{*}\) is a \emph{parabolic bundle} on \((X, D)\). 
\end{definition}

For an interesting source of examples of parabolic bundles on rank \(2\) vector bundles over \(\CP^1\) which relate to the existence of spherical metrics with conical singularities, see \cite{dimasphere}.

\begin{remark}\label{rmk:vb}
    Item \ref{it:vb} requires that the sheaf \(F^i_a\) obtained by restriction of the skyscraper sheaf \(\mE^i_a  \,\, \big/ \,\, \mE(-D_i)\) to \(D_i\) is a locally free \(\mO_{D_i}\)-module. Moreover, the inclusion of \(\mO_{D_i}\)-modules
	\[
	F^i_a \subset \mE|_{D_i} =   \left( \mE \,\, \big/ \,\, \mE(-D_i) \right)|_{D_i} 
	\]
	is an inclusion of vector bundles.
\end{remark}

\begin{remark}
	If we extend the index \(a \in (0,1]\) to \(a \in \R\) by requiring that \(\mE^i_{a-1} = \mE^i_a(-D_i)\) and define
	\(\mE_{\ba} = \bigcap_{i \in I} \mE^i_{a_i}\) for \(\ba = (a_i)_{i \in I} \in \R^I\). Then the collection of sheaves \(\mE_{\ba}\) form a parabolic bundle as defined in \cite[\S 2.1]{simpson}. 
	
	If \(\mE_{*}\) satisfies items \ref{it:support} and \ref{it:semicont} of Definition \ref{def:parbun} then \(\mE_{*}\) is a \emph{\(\textbf{c}\)-parabolic sheaf} as defined in \cite[\S 3.1.1]{mochizuki} where \(\mathbf{c} = \mathbf{1}\) is the vector in \(\R^I\) with all entries equal to \(1\). Item \ref{it:vb} is taken from \cite[Definition 3.12]{mochizuki}.
\end{remark}

\begin{notation}
Write
\[
F^i_{< a} = \bigcup_{a' < a} F^i_{a'} \,.
\]
By item (iii) of Definition \ref{def:parbun}\,, \(F^i_{<a}\) is a vector subbundle of \(F^i_a\).	
\end{notation}

\begin{definition}\label{def:gr}
	For \(i \in I\) and \(a \in (0,1]\), we let 
	\begin{equation}\label{eq:gr}
	\Gr^i_a = F^i_a \, \big/ \, F^i_{<a} \,.
	\end{equation}
	Thus, we have an exact sequence of vector bundles on \(D_i\)
	\[
	0 \to F^i_{< a} \to F^i_a \to \Gr^i_a \to 0 \,.
	\]
	The \(\Gr^i_a\) are the \emph{graded components} of the filtration \(F^i_a \subset \mE|_{D_i}\).
\end{definition}

\begin{definition}
	For \(i \in I\), let \(\wt(i) \subset (0,1]\) be the finite set of weights of the filtration \(F^i_a \subset \mE|_{D_i}\) given by
	\begin{equation}
		\wt(i) = \{a \,\, | \,\, \Gr^i_a \neq 0 \} \,, 
	\end{equation}
	and let
	\begin{equation}
		\lambda_i = \sum_{a \in \wt(i)} a \cdot \rk (\Gr^i_a) \,,
	\end{equation}
	where \(\rk (\Gr^i_a)\) is the rank of the vector bundle \(\Gr^i_a\) on \(D_i\)\,.
\end{definition}

\begin{definition}[{\cite[\S 3.1.2]{mochizuki}}]\label{def:pch1}
	The parabolic first Chern class of \(\mE_{*}\) is the element of \(H^2(X, \R)\) given by
	\begin{equation}
	\parc_1(\mE_*) = c_1(\mE) - \sum_{i \in I} \lambda_i \cdot  c_1(D_i) \,.
	\end{equation}	
\end{definition}

To define parabolic second Chern character, we must introduce an extra compatibility condition. 

\begin{definition}[{\cite[Definition 3.12]{mochizuki}}]\label{def:locab}
	Let \(\mE_{*}\) be a parabolic bundle on \((X, D)\) as in Definition \ref{def:parbun}\,.
	We say that \(\mE_{*}\) is \emph{locally abelian}, if for every subset \(J = \{i_1, \ldots, i_k\} \subset I\) such that the intersection \(D_J = \bigcap_{j \in J} D_j\) is non-empty, the following holds.
	\begin{itemize}
		\item (Compatibility condition.)
		There is a decomposition of \(\mE|_{D_{J}}\) locally on \(D_J\) as a direct sum of subbundles \(U_{\ba}\) indexed by \(\ba \in \R^k\)
		\[
		\mE|_{D_{J}} = \bigoplus_{\ba   \in \R^k} U_{\ba}
		\]
		such that for all \(\ba = (a_1, \ldots, a_k) \in (0,1]^k\) we have
		\begin{equation}\label{eq:compat}
		\bigcap_{j=1}^k \left.F^{i_j}_{a_j}\right|_{D_J}  = 
		\bigoplus_{\ba' \leq \ba} U_{\ba'}
		\end{equation}
		where the sum is over all \(\ba' = (a'_1, \ldots, a'_k)\) such that \(a'_j \leq a_j\) for all \(1 \leq j \leq k\).
	\end{itemize}
\end{definition}

\begin{remark}
If \(k = 1\) and \(D_J = D_i\)\,, then Equation \eqref{eq:compat} is satisfied by taking holomorphic complements to the subbundle \(F^i_{< a} \subset F^i_a\) locally along \(D_i\) in the analytic topology; so that \(F^i_a = F^i_{< a} \oplus U_a\) with \(U_a \cong \Gr^i_a\). 	
\end{remark}

\begin{remark}
	The term locally abelian is taken from \cite[\S 2.1]{simpson}. 
\end{remark}

Next, we reformulate the locally abelian condition in terms of local frames spanning the filtrations \(F^i_a\).
To do this, we introduce some notation.

Let \(\mE_{*}\) be a parabolic bundle on \((X, D)\) and
let \(J = \{i_1, \ldots, i_k\} \subset I\) be such that the intersection \(D_J = \bigcap_{j \in J} D_j\) is non-empty. 
For each \(j \in J\)\,, we have a filtration of \(\left.\mE\right|_{D_J}\) given by the subbundles \(\left.F^j_a\right|_{D_J}\). Altogether, we have \(k = |J|\) different filtrations
on \(D_J\).

\begin{notation}\label{not:Fa2}
	For \(\ba = (a_1, \ldots, a_k) \in \R^k\), we write \(F_{\ba}\) for the fibrewise intersection (see Remark \ref{rmk:compat} below)
	\[
	F_{\ba} = \bigcap_{j=1}^k \left.F^{i_j}_{a_j}\right|_{D_J} \,.
	\]
	If \(\ba = (a_1,\ldots, a_k)\) and \(\ba' = (a'_1, \ldots, a'_k)\) are vectors in \(\R^k\), we write \(\ba' \lneq \ba\) if \(a'_i \leq a_i\) for all \(i\) and \(\ba' \neq \ba\).
	Similarly, we write \(\sum_{\ba' \lneq \ba} F_{\ba'}\) for the fibrewise sum 
	\[
	\sum_{\ba' \lneq \ba} F_{\ba'} := F^{i_1}_{<a_1} \cap F^{i_2}_{a_2} \cap \ldots \cap F^{i_k}_{a_k} \,\, + \,\, \ldots \,\, + \,\, F^{i_1}_{a_1} \cap F^{i_2}_{a_2} \cap \ldots \cap F^{i_k}_{<a_k}  \,,
	\]
	where we have omitted all the restrictions \(|_{D_J}\) on the right hand side. 
\end{notation}

\begin{remark}\label{rmk:compat}
	Without compatibility conditions on \(\mE_{*}\), the above \(F_{\ba}\) and \(\sum_{\ba' \lneq \ba} F_{\ba'}\) are merely subsets of the total space of \(\mE|_{D_J}\) which intersect the fibres along linear subspaces; however the dimension of these subspaces might vary. By definition, we have an inclusion of sets \(\sum_{\ba' \lneq \ba} F_{\ba'} \subset F_{\ba}\).
\end{remark}

\begin{lemma}\label{lem:locab}
	Let \(\mE_{*}\) be a parabolic bundle on \((X, D)\). Then
	\(\mE_{*}\) is locally abelian if and only if for every non-empty intersection \(D_J = \bigcap_{j \in J} D_j\) the following conditions hold.
	\begin{enumerate}[label=\textup{(\roman*)}]
		\item For every point \(p \in D_J\) the \(k\)-tuple of filtrations \(\{\mF^j_p \,|\, j \in J\}\) of the fibre \(\mE|_p\) given by
		\(\mF^j_p = \{F^j_a|_p \,,\, a \in \R\}\)\,,   are compatible as in Definition \ref{def:compfilt}\,.
		
		\item For any \(\ba = (a_1, \ldots, a_k) \in \R^k\), the intersection
		\(F_{\ba}\) is a vector bundle on \(D_J\)\,; and the sum \(\sum_{\ba' \lneq \ba} F_{\ba'}\) is a vector subbundle of \(F_{\ba}\)\,.\footnote{For generic values of \(\ba\), e.g. if \(\ba\) does not belong to the finite set \(\prod_{j=1}^k \wt(i_j)\), the subbundle \(\sum_{\ba' \lneq \ba} F_{\ba'}\) is actually equal to \(F_{\ba}\).}
	\end{enumerate}
\end{lemma}

\begin{proof}
	Suppose that \(\mE_{*}\) is locally abelian. We want to prove items (i) and (ii). Item (i) is immediate, since  Definition \ref{def:locab} implies Definition \ref{def:compfilt} by restricting to fibres. To show (ii), note that Equation \eqref{eq:compat} implies that \(F_{\ba} = \oplus_{\ba' \leq \ba} U_{\ba'}\) is a vector bundle. Similarly,
	\[
	\sum_{\ba' \lneq \ba} F_{\ba'} = \bigoplus_{\ba' \lneq \ba} U_{\ba'}
	\]
	is a vector subbundle of \(F_{\ba}\)\,, proving (ii). 
	
	Conversely, suppose that items (i) and (ii) hold. Fix \(p \in D_J\),
	by item (i) we can find a direct sum decomposition \(\mE|_p = \oplus_{\ba} U_{\ba}|_p\) of \(\mE|_p\)  such that
	\begin{equation}\label{eq:ua}
	\left.F_{\ba}\right|_p = \left( \sum_{\ba' \lneq \ba} \left.F_{\ba'}\right|_p \right) \oplus \left.U_{\ba}\right|_p \,.		
	\end{equation}
	By item (ii), we can locally extend \(U_{\ba}|_p\) to a vector subbundle \(U_{\ba} \subset F_{\ba}\) such that Equation \eqref{eq:ua} holds for all points in a neighbourhood of \(p\). By construction, \(F_{\ba} = \sum_{\ba' \leq \ba} U_{\ba'}\) for any \(\ba\). On the other hand, since the subspaces \(U_{\ba}|_p\) form a direct sum at \(p\), they also form a direct sum near \(p\). This shows that
	\(F_{\ba} = \oplus_{\ba' \leq \ba} U_{\ba'}\) and therefore \(\mE_{*}\) is locally abelian.
\end{proof}

\begin{comment}
\begin{remark}
It follows from Equation \eqref{eq:ua} that, if \(\mE_{*}\) is locally abelian, then
\[
U_{\ba} \cong \dfrac{F_{\ba}}{\sum_{\ba' \lneq \ba} F_{\ba'}} \,.
\]
\end{remark}	
\end{comment}

\begin{definition}\label{def:adaptedframe}
	Let \(\mE_{*}\) be a parabolic bundle on \((X, D)\). 
	Let \(S = \{s_1, \ldots, s_r\}\) be a local frame of sections of \(\mE\) defined on an open set of \(D_J = \cap_{j \in J}D_j\). 
	The frame \(S\) is \emph{adapted} to \(\mE_{*}\),
	if for every \(j \in J\) and \(a \in \R\), the subset \(S^j_a \subset S\) given by
	\[
	S^j_a = \{s \in S \,\, | \,\, s \textup{ is a section of } F^j_a \}
	\] 
	is a frame of \(F^j_a\).
\end{definition}

The main characterization of the local abelian condition that we are after is given by the next.

\begin{lemma}\label{lem:locabadaptframe}
	Let \(\mE_{*}\) be a parabolic bundle on \((X, D)\). Then
	\(\mE_{*}\) is locally abelian if and only if for every non-empty intersection \(D_J\) and every \(p \in D_J\)\,, there is a frame of sections of \(\mE\) defined on a neighbourhood of \(p\) in \(D_J\) that is adapted to \(\mE_{*}\).
\end{lemma}

\begin{proof}
	If \(\mE_{*}\) is locally abelian, then we can decompose \(\mE = \oplus_{\ba} U_{\ba}\) locally on \(D_J\) as in Definition \ref{def:locab}\,.
	Choose a frame \(S_{\ba}\) for each component \(U_{\ba}\) and let
	\(S = \cup_{\ba} S_{\ba}\). It is easy to see that \(S\) is adapted to \(\mE_{*}\)\,. 
	%fix \(i_j \in J\) and \(a \in \R\). Let \(\ba\) be the vector with components equal to \(1\), except for the \(j\)-th component which is equal to \(a\). Equation \eqref{eq:compat} applied to \(\ba\) implies that \(F_{\ba} = F^j_a\) is a direct sum of \(U_{\bb}\)'s; thus the sections of \(S\) that are sections of \(F^j_a\) form a frame of \(F^j_a\). This shows that \(S\) is adapted.
	
	Conversely, suppose that such adapted frames exist. To show that \(\mE_{*}\) is locally abelian, we verify items (i) and (ii) of Lemma \ref{lem:locab}\,.
	Item (i) follows immediately from Lemma \ref{lem:comp}\,. To show item (ii),
	take a local trivialization of \(\mE|_{D_J}\) given by an adapted frame. In such a trivialization, the intersections \(F_{\ba}\) and sums \(\sum_{\ba' \lneq \ba} F_{\ba'}\) are constant, hence they are vector bundles. This proves (ii) and finishes the proof of the lemma.
\end{proof}

\begin{definition}\label{def:grij}
Suppose that \(\mE_{*}\) is a locally abelian parabolic bundle on \((X, D)\).	
For each pair \(i, j \in I\)\,, let \(D_{ij} = D_i \cap D_j\).
For \(a, b \in (0,1]\)\,, let \(\Gr^{i, j}_{a, b}\) be given by
\begin{equation}\label{eq:grij}
\Gr^{i, j}_{a, b} = \left(\left. F^i_{a} \right|_{D_{ij}} \cap \left. F^j_{b} \right|_{D_{ij}}\right) \, \big/ \, \left(\left.F^i_{<a}\right|_{D_{ij}} \cap \left.F^j_{b}\right|_{D_{ij}} +  \left.F^i_{a}\right|_{D_{ij}} \cap \left.F^j_{<b}\right|_{D_{ij}}\right) \,.
\end{equation}
By Lemma \ref{lem:locab} (ii), \(\Gr^{i, j}_{a, b}\) is a vector bundle on \(D_{ij}\).
\end{definition}

The quotient \(\Gr^{i, j}_{a, b}\) is zero unless \(a \in \wt(i)\) and \(b \in \wt(j)\). 
It is helpful to picture the \(\Gr^{i,j}_{a,b}\) inside the unit square \((0,1]^2\). The \(F^i_a\) are represented by intersections of the square with  left half planes \(x \leq a\) for \(a \in \wt(i)\); similarly the \(F^j_b\) are intersections of the square with lower half planes \(y \leq b\) for \(b \in \wt(j)\). The \(\Gr^{i,j}_{a,b}\) correspond to the smallest sub-rectangles of the square obtained as intersections of half spaces, indexed by the coordinates \((a, b)\) of their upper right corners.

\begin{definition}[{\cite[\S 3.1.5]{mochizuki}}\label{def:pch2}]
	Suppose that \(\mE_{*}\) is a locally abelian parabolic bundle on \((X, D)\).
	The parabolic second Chern character of \(\mE_{*}\) is
	the element of \(H^4(X, \R)\) given by
	\begin{equation}\label{eq:pch2}
	\begin{aligned}
	\parch_2(\mE_{*}) &= \ch_2(\mE) - \sum_{i} \sum_a a \cdot \imath_{*} \big(c_1(\Gr^i_a)\big)  \\ 
	&+ \,\, \frac{1}{2} \sum_{i} \sum_{a} a^2 \cdot \rk (\Gr^{i}_{a}) \cdot c_1(D_i)^2 \\	
	&+ \,\, \sum_{i < j} \sum_{a, b} a b \cdot \rk (\Gr^{i, j}_{a, b}) \cdot c_1(D_i) \cdot c_1(D_j) \,.
	\end{aligned}
	\end{equation}
\end{definition}

The terms in Equation \eqref{eq:pch2} have the following meaning:
\(\ch_2(\mE)\) denotes the second Chern character of the vector bundle \(\mE\) given by
\begin{equation*}
\ch_2(\mE) = \frac{1}{2} \big( c_1(\mE)^2 - 2c_2(\mE) \big) \,;
\end{equation*}
\(\imath\) is the inclusion \(D_i \subset X\) and \(\imath_{*}\) is the associated Gysin map in cohomology 
\[\imath_{*}: H^2(D_i, \Z) \to H^4(X, \Z) \,.  \]  
Finally, \(\rk (\Gr^{i,j}_{a, b})\) denotes the rank of the vector bundle
\(\Gr^{i,j}_{a, b}\)\,.

\subsubsection*{Stable bundles and Bogomolov-Gieseker inequality}

Next, we define (slope) stability for parabolic bundles.
To do this, we must consider \emph{saturated subsheaves} (see Definition \ref{def:saturated}).

We recall the way a parabolic bundle induces a parabolic structure on subsheaves.
Let \(\mE_{*}\) be a parabolic bundle on \((X, D)\) as in Definition \ref{def:parbun} and suppose that \(\mV \subset \mE\) is a saturated subsheaf.

\begin{definition}\label{def:indpar}
	The induced parabolic structure \(\mV_{*}\) is defined by the collection of subsheaves \(\mV^i_a \subset \mV\) indexed by \(i \in I\) and \(a \in (0,1]\) given by
	\[
	\mV^i_a = \mV \bigcap \mE^i_a .
	\]
\end{definition}

The filtration of \(\mV\) by the subsheaves \(\mV^i_a\) endow the sheaf \(\mV\) with a parabolic structure as defined in \cite[\S 3.1.1]{mochizuki}.

\begin{remark}\label{rmk:indpar}
	Since \(\mV\) is saturated, by Corollary \ref{cor:subbund}
	there is an analytic subset \(Z \subset X\) with \(\codim Z \geq 2\) such that \(\mV\) is a vector subbundle of \(\mE\) on \(X \setminus Z\). In particular, on the complements \(D_i \setminus Z\) we have an increasing filtration of the vector bundle \(\mV|_{D_i}\) by vector subbundles \(V^i_a =  F^i_a \bigcap \mV|_{D_i}\). 
\end{remark}

Let
\begin{equation}\label{eq:indlambda}
	\lambda^{\mV}_i = \sum_{a \in (0,1]} a \cdot \rk \left(V^i_a \, / \, V^i_{<a} \right) , 
\end{equation}
where \( \rk \left(V^i_a \, / \, V^i_{<a} \right)\) is the rank of the quotient vector bundle on \(D_i \setminus Z\).

%Let \(\mV \subset \mE\) be endowed with the induced parabolic structure \(\mV_{*}\) as in Definition \ref{def:indpar}\,. 

\begin{definition}
	The parabolic first Chern class of \(\mV_{*}\) is the element \(\parc_1(\mV_{*})\) of \(H^2(X, \R)\) defined as
	\[
	\parc_1(\mV_*) = c_1(\mV) - \sum_{i \in I} \lambda^{\mV}_i \cdot  c_1(D_i) \,,
	\] 
	where \(c_1(\mV)\) is as in Definition \ref{def:c1saturated}\,. 
\end{definition}

To define slope stability, we fix a polarization, given by an ample line bundle \(P\) on \(X\).
The parabolic degree of \(\mE_{*}\) is defined as
\[
\pardeg_P (\mE_{*}) = c_1(P)^{n-1} \cdot \parc_1(\mE_{*}) \,,
\]
where \(n\) is the dimension of \(X\). In the above equation we identify the top cohomology group \(H^{2n}(X, \R)\) with \(\R\) via integration on \(X\), thus we view the cup product as a non-degenerate bilinear pairing \(H^2(X, \R) \times H^{n-2}(X, \R) \to \R\).
Similarly, the parabolic degree of the induced parabolic structure \(\mV_{*}\) is defined as
\[
\pardeg_P (\mV_{*}) = c_1(P)^{n-1} \cdot \parc_1(\mV_{*}) \,.
\]

\begin{remark}\label{rmk:pc1lessc1}
	If \(\deg_P(\mV) =  c_1(P)^{n-1} \cdot c_1(\mV)\) is the degree of \(\mV \subset \mE\)\,, then
	\[
	\pardeg_P (\mV_*) = \deg_P(\mV) - \sum_i \lambda_i^{\mV} \cdot \vol_P(D_i)
	\]
	where \(\vol_P (D_i) = c_1(P)^{n-1} \cdot c_1(D_i) > 0\) is the volume of \(D_i \subset X\) with respect to \(P\). 
	Since \(\lambda^{\mV}_i \geq 0\),
	\begin{equation}\label{eq:parcilessc1}
	\pardeg_P (\mV_*) \leq \deg_P(\mV) \,.
	\end{equation}
	If the induced parabolic structure is non-trivial, in the sense that \(\lambda_i^{\mV}\) are not all zero, then
	strict inequality holds in \eqref{eq:parcilessc1}.
\end{remark}

A main notion we need is the next.

\begin{definition}[{\cite[\S 3.1.3]{mochizuki}}]\label{def:parstable}
	Let \(\mE_{*}\) be a parabolic bundle on \((X, D)\) with \(\parc_1(\mE_{*}) = 0\) and let \(P\) be an ample line bundle on \(X\).
	We say that \(\mE_{*}\) is stable with respect to \(P\), or \(P\)-stable for short, if for every non-zero and proper saturated subsheaf \(\mV \subset \mE\) we have
	\begin{equation}\label{eq:stabdef}
	\pardeg_P (\mV_{*}) < 0 \,.
	\end{equation}
\end{definition}

\begin{remark}
Our \(P\)-stable parabolic bundles are referred as \(\mu_L\)-stable in \cite{mochizuki}.	
\end{remark}

The main result we need is the next version of the Bogomolov-Gieseker inequality.

\begin{theorem}[{\cite[Theorem 6.5]{mochizuki}}]\label{thm:boggies}
	Let \(\mE_{*}\) be locally abelian parabolic bundle on \((X, D)\). Suppose that \(\parc_1(\mE_{*}) = 0\) and that \(\mE_{*}\)  is  \(P\)-stable. Then the following inequality holds:
	\begin{equation}\label{eq:boggies}
	c_1(P)^{n-2} \cdot \parch_2(\mE_{*}) \leq 0 \,.
	\end{equation}
\end{theorem}

\begin{remark}
	Mochizuki proves Theorem \ref{thm:boggies} in the more general case where \(\mE_{*}\) is only locally abelian in codimension \(2\), meaning that the compatibility condition in Definition \ref{def:locab} holds outside a set of codimension \(\geq 3\).
\end{remark}

\subsection{Weights and filtration}\label{sec:pardef}

Let \((\mH, \ba)\) be a weighted arrangement as in Theorem \ref{thm:main}\,.
Let \(X \xrightarrow{\pi} \CP^n\) be the resolution of \(\mH\) and let \(\mE\) be the pullback tangent bundle
\begin{equation}
\mE = \pi^*(T\P^n) \,.
\end{equation} 	 
We use the weights \(a_H\) to define a natural parabolic structure on \(\mE\) whose filtered subsheaves are indexed by the irreducible components of the simple normal crossing divisor \(D =\pi^{-1}(\mH)\).

\begin{definition}\label{def:weightaL}
	The weight \(a_L\) at an irreducible subspace \(L \in \mLi\) is 
	\begin{equation}\label{eq:aL}
	a_L = (\codim L)^{-1} \sum_{H | L \subset H} a_H .
	\end{equation}
	Note that \(0 < a_L < 1\) because of the klt condition \eqref{eq:klt}. 
\end{definition}

Recall that the irreducible components of \(D =\pi^{-1}(\mH)\) are of the form \(D_L\), where \(\pi(D_L) = L\) and \(L\) is a non-empty proper irreducible subspace.
For \(L \in \mLi\) let \(\mE^L\) be the subsheaf of \(\mE\) generated by local sections tangent to \(\pi^*TL\) when restricted to \(D_L\).\footnote{Alternatively, \(\mE^L\) is the pullback of the subsheaf of \(T\CP^n\) generated by vector fields tangent to \(L\).} Set
\begin{equation}\label{eq:filtration}
\mE^L_a = \begin{cases}
\mE^L &\mbox{ for } 0<a<a_L \,,\\
\mE &\mbox{ for } a_L \leq a \leq 1 \,.
\end{cases}
\end{equation}

\begin{lemma}
	The collection of subsheaves \(\mE^L_a \subset \mE\) given by Equation \eqref{eq:filtration} define a parabolic bundle \(\mE_{*}\) as in Definition \ref{def:parbun}\,.
\end{lemma}

\begin{proof}
	First of all, the sheaves \(\mE^L\) are locally free, for if \(p \in D_L\) then we can pick a frame of sections \(s_1, \ldots, s_n\) of \(T\CP^n\) in a neighbourhood of \(\pi(p)\) of which the first \(d = \dim L\) generate \(TL\). Then \(\mE^L\) is freely generated near \(p\) by the sections 
	\[
	\pi^* s_1 , \ldots, \pi^*s_d, \,\, z \cdot \pi^* s_{d+1}, \ldots z \cdot \pi^* s_n \, 
	\] 
	where \(z\) is a local defining equation of \(D_L=\{z=0\}\). Alternatively, the sections \(s_1, \ldots, s_n\) give a splitting \(T\CP^n = V \oplus W\) in a neighbourhood of \(\pi(p)\) where \(V|_L = TL\) and \(W|_L\) is isomorphic to the normal bundle of \(L\). Then \(\mE^L\) is locally isomorphic to the direct sum of \(\pi^*V\) and \( \pi^*W(-D_L)\).
	
	Let us now verify that \(\mE^L_a\) satisfy properties \ref{it:support}, \ref{it:semicont}, and \ref{it:vb} in Definition \ref{def:parbun}\,.
	Items \ref{it:support} support and \ref{it:semicont} semi-continuity are immediate from Equation \eqref{eq:filtration}. Finally, item \ref{it:vb}  follows from 
	\[
	\mE^L|_{D_L} \, \big/ \, \mE(-D_L)|_{D_L} = \pi^*(TL) ,
	\]  
	which implies that the quotient \(\mO_{D_L}\)-modules \(F^L_a = \mE_a^L|_{D_L} \, \big/ \, \mE(-D_L)|_{D_L}\) are given by
	\begin{equation}\label{eq:FLa}
		F^L_a = \begin{cases}
		\pi^*(TL) &\text{ if } a < a_L, \\
		\mE|_{D_L} &\text{ if } a \geq a_L .
		\end{cases}
	\end{equation}
	This shows that \(F^L_a\) is an increasing filtration of \(\mE|_{D_L}\) by vector subbundles.
\end{proof}

\begin{definition}\label{def:parstr}
	The parabolic bundle \(\mE_{*}\) is defined by the increasing filtrations \(\mE^L_a\) with \(L \in \mLi\) and \(a \in (0,1]\) given by Equation \eqref{eq:filtration}.
\end{definition}

\subsection{Locally abelian property}\label{sec:locab}

\begin{theorem}\label{thm:locab}
	The parabolic bundle \(\mE_{*}\) on \((X, D)\) is locally abelian.
\end{theorem}

\begin{proof}
	We use the characterization of the locally abelian property in terms of adapted frames given by Lemma \ref{lem:locabadaptframe}\,.
	Outside \(D = \pi^{-1}(\mH)\) there is nothing to prove. Take \(\bar{p} \in D\) and let \(\mS\) be the set of all irreducible subspaces \(L \in \mLi\) such that \(\bar{p} \in D_L\). Let \(D_{\mS}\) be their common intersection \(D_{\mS} = \bigcap_{L \in \mS} D_L\), so that \(\bar{p} \in D_{\mS}\) and \(\bar{p} \notin D_{L'}\) for \(L' \notin \mS\). Since the intersection \(D_{\mS}\) is non-empty, by Proposition \ref{prop:nested}\,, the set \(\mS\) is nested relative to \(\mLi\). By Lemma \ref{lem:nestedint}\,, \(\mS\) is a nested set of projective subspaces as in Definition \ref{def:nestedProj}.
	Let \(M = \bigcap_{L \in \mS} L\)
	be the projective subspace of \(\CP^n\) obtained as the common intersection of the members of \(\mS\). Let \(p = \pi(\bar{p}) \in M\). By Lemma \ref{lem:nestedframe} (ii), there is a frame of vector fields \(e_1, \ldots, e_n\) of \(T\CP^n\) defined on a neighbourhood of \(p \in U \subset M\) such that for every \(q \in U\) and \(L \in \mS\), the vectors \(e_i(q)\) that belong to \(T_qL\) form a basis of \(T_qL\). The pullbacks \(\pi^*e_1, \ldots, \pi^*e_n\) form a frame of \(\mE = \pi^*(T\CP^n)\) on \(\pi^{-1}(U)\). Since \(\pi(D_\mS) \subset M\), the preimage \(\pi^{-1}(U)\)
	contains an open neighbourhood \(\bar{U} \subset D_{\mS}\) of \(\bar{p}\).
	By construction, the frame \(\pi^*e_1, \ldots, \pi^*e_n\) is adapted -as per Definition \ref{def:adaptedframe}- to the parabolic structure \(\mE_{*}\). This finishes the proof of the theorem.
\end{proof}

\subsection{The parabolic first Chern class}\label{sec:parc1}

In this section we show that \(\parch_1(\mE_{*}) = 0\). We begin with a preliminary formula.

\begin{lemma}
	The parabolic first Chern class of \(\mE_{*}\) is given by
	\begin{equation}\label{eq:parch1}
	\parc_1(\mE_*) = c_1(\mE) - \sum_{L \in \mLi} a_L \cdot \codim L \cdot \gamma_L \,. 
	\end{equation}
\end{lemma}

\begin{proof}
	By Equation \eqref{eq:FLa},
	the graded terms \(\Gr^L_a =  F^L_{a}\big/F^L_{<a}\) are
	\begin{equation*}
	\Gr^L_{a}  = \begin{cases}
	\pi^* \left( T\CP^n|_L \big/ TL \right)  &\text{ if } a = a_L, \\
	0	&\text{ otherwise. }    
	\end{cases}
	\end{equation*}
	Therefore \(\rk (\Gr^L_{a_L}) = \codim L\) and Equation \eqref{eq:parch1} follows from the formula for \(\parch_1(\mE_{*})\) given in Definition \ref{def:pch1}\,.
\end{proof}

\begin{lemma}
	In \(H^2(X, \R)\) the following identity holds:
	\begin{equation}\label{eq:identitysum}
	\sum_{L \in \mLi} a_L \cdot \codim L \cdot \gamma_L = s \cdot \pi^*h \,,
	\end{equation}
	where \(s = \sum_{H \in \mH} a_H \) is the sum of all hyperplane weights.
\end{lemma}

\begin{proof}
	We split the left hand side of Equation \eqref{eq:identitysum} as a sum over hyperplanes and irreducible subspaces of codimension \(\geq 2\), 
	\begin{equation}\label{eq:h2sum1}
	\sum_{L \in \mLi} a_L \cdot \codim L \cdot \gamma_L = 
	\sum_{H \in \mH} a_H \cdot \gamma_H + \sum_{L \in \mLi^{\circ}} a_L \cdot \codim L \cdot \gamma_L \,.	
	\end{equation}
	On the other hand, Equation \eqref{eq:tH} implies that
	\begin{equation}\label{eq:h2sum2}
	\begin{aligned}
	\sum_{H \in \mH} a_H \cdot \gamma_H &= s \cdot \pi^*h - \sum_{L \in \mLi^{\circ}} \left(\sum_{H | L \subset H} a_H \right) \cdot \gamma_L \\
	&= s \cdot \pi^*h - \sum_{L \in \mLi^{\circ}} a_L \cdot \codim L \cdot \gamma_L \,,
	\end{aligned}
	\end{equation}
	where the second equality follows from the definition of \(a_L\) given by Equation \eqref{eq:aL}. Equation \eqref{eq:identitysum} follows from Equations \eqref{eq:h2sum1} and \eqref{eq:h2sum2}.
\end{proof}

\begin{lemma}
	The parabolic first Chern class of \(\mE_{*}\) is zero
	\begin{equation}
	\parch_1 (\mE_*) = 0 .
	\end{equation}
\end{lemma}

\begin{proof}
	Since \(\mE=\pi^*T\P^n\), we have
	\begin{equation*}
	c_1(\mE) = (n+1) \cdot \pi^*h .
	\end{equation*}
	It follows from Equations \eqref{eq:parch1} and \eqref{eq:identitysum} that
	\begin{equation*}
	\parch_1(\mE_{*}) = \big(n+1 - s \big) \cdot \pi^*h .
	\end{equation*}
	The lemma follows by using the CY condition \eqref{eq:cy} which requires \(s=n+1\).
\end{proof}

\subsection{The parabolic second Chern character}\label{sec:parch2}

By Equation \eqref{eq:pch2}, the parabolic second Chern character of \(\mE_{*}\) is given by
\begin{equation}\label{eq:pch20}
	\begin{gathered}
	\parch_2(\mE_{*}) = \ch_2(\mE) - \sum_{L \in \mLi} a_L \cdot \imath_{*} \big(c_1(\Gr^L_{a_L})\big)  \\ 
	+ \frac{1}{2} \sum_{(L, M)} a_L  a_{M} \cdot \nu_{L, M} \cdot \gamma_L \cdot \gamma_M \,.
	\end{gathered}
\end{equation}
The last sum runs over all ordered pairs \((L,M)\) of (not necessarily distinct) elements in \(\mLi \times \mLi\) with \(D_L \cap D_M \neq \emptyset\)\,, and
\[
\nu_{L, M} = \begin{cases}
 \rk \big( \Gr^{L, M}_{a_L, a_M} \big) &\textup{ if } L \neq M \,, \\
 \rk  \big( \Gr^{L}_{a_L} \big) = \codim L &\textup{ if } L = M \,,
\end{cases} 
\]
where \(\Gr^{L, M}_{a_L, a_M}\) is given by formula \eqref{eq:grij}.

\begin{lemma}
	For every \(L \in \mLi\) the following identity holds:
	\begin{equation}\label{eq:projform}
		\imath_{*} \big( c_1(\Gr^L_{a_L}) \big)  = \codim L \cdot \pi^*h \cdot \gamma_L \,.
	\end{equation}
\end{lemma}

\begin{proof}
	On \(D_L\) we have 
	\[ 
	\Gr^L_{a_L} = \pi^*\big( T\CP^n|_L \big/ TL \big) = \pi^* \big( \mO_{\P^n}(1)|_{L}^{\oplus r} \big) \,,
	\] 
	where \(r = \codim L\). This implies that 
	\[
	c_1(\Gr^L_{a_L})  = \imath^* \alpha \,,
	\] 
	where \(\alpha = r \cdot \pi^*h \in H^2(X, \Z)\) and \(\imath^*\) is the pullback by the inclusion map of \(D_L \subset X\).
	The projection formula \(\imath_{*} \left(\imath^*\alpha\right) =  \alpha \cdot \gamma_L\) -see Equation 1.6 in \cite{voisin2}- implies that
	\begin{equation*}
		\imath_{*} \big( c_1(\Gr^L_{a_L}) \big) = \imath_{*} \big( \imath^* \alpha \big) = \alpha \cdot \gamma_L = r \cdot \pi^*h \cdot \gamma_L \,. 
	\end{equation*}
    This finishes the proof.
\end{proof}

\begin{lemma}
	Let \(L\) and \(M\) be irreducible subspaces with \(D_L \cap D_M \neq \emptyset\)\,. Then 
	\begin{equation}
		\nu_{L, M} = \codim (L + M) \,.
	\end{equation}
\end{lemma}

\begin{proof}
	Since \(D_L \cap D_M \neq \emptyset\), the subspace \(N = L \cap M\) is non-empty and 
	\[
	\pi(D_L \cap D_M) \subset N .
	\]
	Since \(F^L_{< a_L} = \pi^*TL\)\,, \(F^L_{a_L} = \pi^*T\CP^n\)\,, and similarly for \(M\); on \(D_L \cap D_M\) the vector bundle \(\Gr^{L, M}_{a_L, a_M}\) (given by formula
	\eqref{eq:grij}) is equal to 
	\begin{equation*}\label{eq:grpair}
	\Gr^{L, M}_{a_L, a_M} = \pi^* \left( T\P^n \big/ \big(TL|_N + TM|_N \big) \right)
	\end{equation*}
	and the result follows.
\end{proof}

Split the terms that define \(\parch_2(\mE_{*})\) in Equation \eqref{eq:pch20} as
\begin{equation}
\parch_2(\mE_{*}) = \textup{A} + \textup{B}
\end{equation}
with
\begin{equation*}
\textup{A} = \ch_2(\mE) - \sum_{L \in \mLi} a_L \cdot \imath_{*} \big( c_1(\Gr^L_{a_L}) \big)  
\end{equation*}
and
\begin{equation*}
\textup{B} = \frac{1}{2} \sum_{(L, M)} a_L a_{M} \cdot \nu_{L, M} \cdot \gamma_L \cdot \gamma_M \,.
\end{equation*}

\begin{lemma}\label{lem:Ac2}
	The term \textup{A} in \(\parch_2(\mE_{*})\) is equal to
	\begin{equation}\label{eq:Atermc2}
	\textup{A} = - \frac{n+1}{2} \cdot (\pi^*h)^2 .
	\end{equation} 
\end{lemma}

\begin{proof}
	By definition, the second Chern class of the vector bundle \(\mE\) is
	\begin{equation*}
	\ch_2(\mE) = \frac{1}{2} \big( c_1(\mE)^2 - 2c_2(\mE) \big) .
	\end{equation*}
	Euler's exact sequence implies that
	\begin{equation*}
	c_k (T\CP^n) = \binom{n+1}{k} \cdot h^k 
	\end{equation*}
	and we get that the second Chern character of \(\mE = \pi^*\big(T\CP^n\big)\) is
	\begin{equation}\label{eq:ch2pn}
	\ch_2(\mE) = \frac{n+1}{2} \cdot (\pi^*h)^2 \,.
	\end{equation}	
	On the other hand, it follows from Equations \eqref{eq:projform} and \eqref{eq:identitysum} that
	\begin{equation}\label{eq:algrl}
	\begin{aligned}
	\sum_{L \in \mLi} a_L \cdot \imath_{*} \big( c_1(\Gr^L_{a_L})\big)  &= \left(\sum_{L \in \mLi} a_L \cdot \codim L \cdot \gamma_L \right) \cdot \pi^*h \\
	&= s \cdot (\pi^*h)^2 .
	\end{aligned}	
	\end{equation}
	Taking the difference of Equations \eqref{eq:ch2pn} and \eqref{eq:algrl}  gives us
	\begin{equation*}
	\textup{A} = \left( \frac{n+1}{2} - s \right) \cdot (\pi^*h)^2 .
	\end{equation*}
	The result follows from the CY condition \eqref{eq:cy} which requires \(s=n+1\).	
\end{proof}

\begin{lemma}\label{lem:Bc2}
	The term \textup{B} in \(\parch_2(\mE_{*})\) is equal to
	\begin{equation}\label{eq:c2B2}
	\textup{B} = \frac{1}{2} \sum_{L \in \mLi} a_L^2 \cdot \codim L \cdot \gamma_L^2 \, + \, \sum_{L \subsetneq M} a_L  a_{M} \cdot \codim M \cdot \gamma_L \cdot \gamma_M \,,
	\end{equation}
	where the second sum runs over all pairs of irreducible subspaces with \(L \subsetneq M\).
\end{lemma}

\begin{proof}
	Clearly,
	\begin{equation}\label{eq:B1}
	\textup{B} = \frac{1}{2} \sum_{L \in \mLi} a_L^2 \cdot \nu_L \cdot \gamma_L^2 \,\, + \,\, \frac{1}{2} \sum_{L \neq M} a_L  a_{M} \cdot \nu_{L, M} \cdot \gamma_L \cdot \gamma_M \,,		
	\end{equation}
	where the second sum runs over all pairs \((L, M) \in \mLi \times \mLi\) with \(L \neq M\) and \(D_L \cap D_M \neq \emptyset\). 
	By Proposition \ref{prop:nested}\,,
	the set \(\mS = \{L, M\}\) is nested relative to \(\mLi\), so one of the following must happen:
	\begin{enumerate}[label=\textup{(\roman*)}]
	\item one subspace is contained in the other; 
	\item the intersection \(L\cap M\) is non-empty and reducible.
	\end{enumerate}
	If (ii) occurs then, by Lemma \ref{lem:l1l2}\,, \(L+ M = \CP^n\) and \(\nu_{L, M} = \codim (L+M) = 0\). Therefore, we can assume that for all pairs \((L, M)\) occurring in the second sum in Equation \eqref{eq:B1} 
	we have either \(L \subset M\) or \(M \subset L\). Since every pair of irreducible subspaces \(L, M\) with one contained in the other contributes twice to the sum as \((L, M)\) and \((M, L)\); and since \(\nu_{L, M} = \codim (L+M)\), we have
	\begin{equation}\label{eq:B2}
	\frac{1}{2} \sum_{L \neq M} a_L  a_{M} \cdot \nu_{L, M} \cdot \gamma_L \cdot \gamma_M = \sum_{L \subsetneq M} a_L  a_{M} \cdot \codim M \cdot \gamma_L \cdot \gamma_M \,.	
	\end{equation}
	Finally, Equation \eqref{eq:c2B2} follows from Equations \eqref{eq:B1} and \eqref{eq:B2}.
\end{proof}

Lemmas \ref{lem:Ac2} and \ref{lem:Bc2} together give us the next.

\begin{corollary}
	The parabolic second Chern character of \(\mE_{*}\) is given by
	\begin{equation}\label{eq:parch2E}
	\begin{gathered}
	\parch_2(\mE_{*}) = - \frac{n+1}{2} \cdot (\pi^*h)^2 \\
	+ \, \frac{1}{2} \sum_{L \in \mLi} a_L^2 \cdot \codim L \cdot \gamma_L^2 + \sum_{L \subsetneq M} a_L  a_{M} \cdot \codim M \cdot \gamma_L \cdot \gamma_M \,,
	\end{gathered}
	\end{equation}
	where the last sum is over all pairs \((L, M) \in \mLi \times \mLi\) with \(L \subsetneq M\).
\end{corollary}

\section{Stability of the parabolic bundle}\label{sec:stability}

In this section, we show that the parabolic bundle \(\mE_{*}\) that 
we defined in Section \ref{sec:pardef} is slope stable. 
To state a precise result, 
fix polarization coefficients \(b_L\), as in Lemma \ref{lem:polarization}\,,  so that \(P_k\) is an ample line bundle on \(X\) for all \(k \gg 1\). 
In this section, we prove the following \emph{stability theorem}.

\begin{theorem}\label{thm:stability}
	The parabolic bundle \(\mE_*\) is \(P_k\)-stable for all \(k \gg 1\). More precisely, there is \(k_0\), depending only on \((\mH, \ba)\) and the polarization coefficients \(b_L\), such that
	for every non-zero and proper saturated subsheaf \(\mV \subset \mE\), we have
	\begin{equation}
	\forall k > k_0: \,\, \pardeg_{P_k}(\mV_*) < 0 \,, 
	\end{equation}
	where \(\mV_*\) is the parabolic structure on \(\mV\) induced from \(\mE_*\)\,.
\end{theorem}

\noindent In Section \ref{sec:auxres}\,, we establish some auxiliary results. \newpar

\noindent In Section \ref{app:key}\,, we prove the key estimate needed to show stability. This boils down to analysing the tangencies between \(\mH\) and distributions \(\mV \subset T\CP^n\) with \(c_1(\mV) \geq 0\), as described in Proposition \ref{prop:key}\,. \newpar

\noindent In Section \ref{sec:pfst}\,, we prove the stability Theorem \ref{thm:stability}\,.\newpar

\noindent In Section \ref{sec:pfmain}\,, we prove our main Theorem \ref{thm:main}\,. 
We derive Theorem \ref{thm:main} as a consequence of Theorem \ref{thm:stability} together with Mochizuki's version of the Bogomolov-Gieseker inequality for parabolic bundles \cite[Theorem 6.5]{mochizuki}.

\subsection{Auxiliary results}\label{sec:auxres}

The next result is an analogue of \cite[Lemma 7.7]{panov}.

\begin{lemma}\label{lem:key}
	Let \(M \subsetneq \CP^n\) be a non-empty linear subspace, then
	\begin{equation}
	\sum_{H| M \not\subset H} a_H > \dim M + 1 .
	\end{equation}
\end{lemma}

\begin{proof}
	The CY and klt conditions together give us
	\begin{align*}
	\sum_{H| M \not\subset H} a_H &= n+1 - \sum_{H | M \subset H} a_H \\
	&> n+1 - \codim M = \dim M + 1 \,,
	\end{align*}
	where we used Corollary \ref{cor:kltcondition} (iii) in the inequality step. 
\end{proof}

Let \(\mV\) be a saturated subsheaf of \(\mE\). By Corollary \ref{cor:h2} we can write
\begin{equation}\label{eq:c1V}
	c_1(\mV) = \imath \cdot \pi^*h + \sum_{L \in \mLi^{\circ}} d_L \cdot \gamma_L 
\end{equation}
for unique integers \(\imath\) and \(d_L\).
The next result is an analogue of \cite[Lemma 7.6]{panov}.

\begin{lemma}\label{lem:bounddL}
	Let \(\imath\) and \(d_L\) be as in Equation \eqref{eq:c1V}.
	Then
	\begin{equation}\label{eq:boundi}
	\imath \leq r \,,	
	\end{equation}
	where \(r\) is the rank of \(\mV\), and 
	\begin{equation}\label{eq:bounddL}
		\forall L : \,\, d_L \leq n - \imath \,.
	\end{equation}
\end{lemma}

\begin{proof}
	To show that \(\imath \leq r\), to take a generic line \(Q\) in \(\CP^n\) which does not intersect any of the irreducible subspaces of the arrangement of codimension \(\geq 2\) and such that \(\mV|_Q\) is a vector subbundle of \(T\CP^n|_Q\) which does not contain \(TQ\). The same argument used to prove Lemma \ref{lem:indlessrank} in the Appendix shows that
	Equation \eqref{eq:boundi} holds.
	
	To establish the bound on the coefficients \(d_L\), we argue similarly. Given \(L\), take a projective line \(Q \subset \CP^n\)
	such that: (i) \(Q\) meets \(L\) at single point; (ii) \(Q \cap M = \emptyset\) for all \(M \in \mL \setminus \{L\}\) with \(\codim M \geq 2\); (iii)  \(\mV|_{\tQ}\) is a vector subbundle of
	\begin{equation*}\label{eq:EP}
		\mE|_{\tQ} \cong \mO_{\P^1}(2) \oplus \mO_{\P^1}(1)^{\oplus(n-1)} \,,
	\end{equation*}
	 where \(\tQ\) is the proper transform of \(Q\). It is clear that we can take a line \(Q\) with the above three properties.
	
	Item (iii) implies that
	\begin{equation}\label{eq:bdl1}
		\deg (\mV|_{\tQ_L}) \leq n.
	\end{equation}
	On the other hand, items (i) and (ii) together with Equation \eqref{eq:c1V} give us
	\begin{equation}\label{eq:bdl2}
		\deg (\mV|_{\tQ_L}) = \imath + d_L \,.
	\end{equation}
	Equations \eqref{eq:bdl1} and \eqref{eq:bdl2}
	\[
	 \imath + d_L \leq n
	\]
	which is equivalent to Equation \eqref{eq:bounddL}.
\end{proof}

Next, we discuss tangencies between hyperplanes and distributions.

\begin{definition}\label{def:disttanhyp}
Let \(\mV \subset T\CP^n\) be a distribution and let \(H \subset \CP^n\) be a hyperplane.
We say that \(\mV\) is tangent to \(H\) (or that \(H\) is tangent to \(\mV\)), if for every point \(x \in H \cap U\) -where \(U\) is the regular set of \(\mV\)- we have 
\[
\mV_x \subset T_x H \,,
\] 
where \(\mV_x\) denotes the fibre at \(x\) of the vector subbundle \(\mV \subset T\CP^n\) on \(U\).
We denote by \(\Tan(\mV)\) the collection of all hyperplanes tangent to \(\mV\), 
\[
\Tan(\mV) = \{ H \,\, | \,\, H \text{ is tangent to } \mV \}\,.
\]	
\end{definition}

We analyse the set \(\Tan(\mV)\) for distributions with non-negative \emph{index} (see Definition \ref{def:index}). 
Our next result is an analogue of \cite[Lemma 7.4]{panov}. The proof relies on results from Appendix \ref{sec:distributions}\,.

%If the index is positive, we provide an upper bound on the number of linearly independent hyperplanes in \(\Tan(\mV)\). The reason that makes this work is that if \(\imath >0\) then \(\mV\) is defined by polynomials of low degree, while the tangency condition requires these polynomials to be divisible by the corresponding defining linear equations of the hyperplanes. Thus, the number of linearly independent hyperplanes in \(\Tan(\mV)\) is bounded above by the degree of the polynomials that define \(\mV\). If \(\imath = 0\) we show that there is a non-zero vector field on \(\CP^n\) which is tangent to all hyperplanes in \(\Tan(\mH)\).

\begin{proposition}\label{prop:posindexdist}\label{prop:zeroindexdist}
	Let \(\mV \subset T\CP^n\) be a distribution of index \(\imath \geq 0\).
	\begin{enumerate}[label=\textup{(\roman*)}]
		\item If \(\imath > 0\) then there is a linear subspace \(M \subset \CP^n\) with 
		\begin{equation}\label{eq:posindexdist}
		\dim M \geq \imath -1
		\end{equation}
		such that any hyperplane \(H \in \Tan(\mV)\) contains \(M\). 
		\item Suppose that \(\imath = 0\) and that \(\mV\) is tangent to the \(n+1\) coordinate hyperplanes. Then there is a non-zero holomorphic vector field \(Y\) on \(\CP^n\) that  is tangent to all the hyperplanes in \(\Tan(\mV)\).
	\end{enumerate} 
\end{proposition}

\begin{proof}
	Let \(r\) be the rank of \(\mV\) and let \(d =r-\imath\) be its degree. By Corollary \ref{cor:distmult}\,, the distribution \(\mV\) defines an
	\((r+1)\)-vector field \(\bv\)  on \(\C^{n+1}\) with homogeneous polynomial coefficients of degree \(d+1\). By
	Lemma \ref{lem:tanhm}\,, a hyperplane \(H \subset \CP^n\) belongs to \(\Tan(\mV)\) if and only if the corresponding linear hyperplane \(H^{\bc} \subset \C^{n+1}\) 
	belongs to \(\Tan(\bv)\), the set of linear hyperplanes in \(\C^{n+1}\) that are tangent to \(\bv\) (Definitions \ref{def:multvftanhyp} and \ref{def:tanv}).
	
	(i) By Proposition \ref{prop:posindex}\,, there is a linear subspace \(M^{\bc} \subset \C^{n+1}\) with \(\dim M^{\bc} \geq \imath\) that is contained in all hyperplanes in \(\Tan(\bv)\). Item (i) follows by taking \(M = \P(M^{\bc})\).
	
	(ii) Since \(\mV\) is tangent to the coordinate hyperplanes, the multivector field \(\bv\) belongs to the subspace \(T_{r+1}^{\circ}\) from Definition \ref{def:tr}\,. By Proposition \ref{prop:zeroindex}\,, there is a linear vector field \(\bv'\) on \(\C^{n+1}\) that is tangent to all the hyperplanes in \(\Tan(\bv)\) and such that \(e \wedge \bv'\) is non-zero. We let \(Y\) be the projection of \(\bv'\) down to \(\CP^n\). Then \(Y\) is tangent to all elements in \(\Tan(\mV)\) and, since \(e \wedge \bv' \neq 0\), the vector field \(Y\) is non-zero.
\end{proof}

\begin{example}
	Let \(M \subset \CP^n\) be a linear subspace with \(\dim M = r-1\) for some \(1 \leq r \leq n-1\). The collection of all \(r\)-dimensional subspaces that contain \(M\) defines a distribution \(\mV \cong \mO_{\P^n}(1)^{\oplus r}\) of index \(\imath = r\). A hyperplane \(H\) belongs to \(\Tan(\mV)\) if and only if \(H \supset M\). In particular,
	\[
	\bigcap_{H \in \Tan(\mV)} H = M \,.
	\]
	This example shows that the inequality \eqref{eq:posindexdist} is sharp. See also Example \ref{ex:distansubspace} in the appendix for a presentation of this distribution in terms of a multivector field.
\end{example}

\begin{remark}
The proof of item (ii) of Proposition \ref{prop:zeroindex} shows that
the vector field \(Y\) is tangent to \(\mV\), meaning that \(Y(x) \in \mV_x\) for all points \(x\) in the regular set of \(\mV\). 
The next example shows that, in general, distributions of index zero do not necessarily admit non-zero tangent vector fields.
\end{remark}

\begin{example}\label{ex:pereira}
	Let \(f\) and \(h\) be homogeneous polynomials in \(\C[x_0, \ldots, x_3]\) with \( \deg h =1\) and \(f\) generic with \(\deg f = 3\). The \(1\)-form on \(\C^4\) given by
	\[
	\omega = 3f dh - hdf
	\]
	has homogeneous polynomial coefficients of degree \(3\)
	and the contraction of \(\omega\) with the Euler vector field on \(\C^4\) is identically zero. The \(1\)-form \(\omega\) defines a codimension \(1\) distribution \(\mV\) on \(\CP^3\) of rank \(r=2\) and degree \(d = 2\) (see Remark \ref{rmk:homog1form}). The index of \(\mV\) is equal to \(\imath = r - d = 0\).
	The distribution \(\mV\) is integrable. The leaves of the foliation defined by \(\mV\) make a pencil of cubic surfaces  in \(\CP^3\) given by \(\{\lambda_1 \cdot f + \lambda_2 \cdot h^3 = 0\}\) with \([\lambda_1, \lambda_2] \in \CP^1\).
	A holomorphic vector field \(Y\) tangent to \(\mV\) must also be tangent to the smooth cubic surface \(\{f=0\}\) and therefore it must vanish. 
\end{example}

\subsection{Key estimate}\label{app:key}

Let \(\mV \subset T\CP^n\) be a distribution and let \(\Tan(\mV)\) be the set of all hyperplanes tangent to \(\mV\) (see Definition \ref{def:disttanhyp}).

\begin{definition}\label{def:HtransV}
	We say that a hyperplane \(H \subset \CP^n\)  is transverse to \(\mV\) if
	\begin{equation}
		H \notin \Tan(\mV) \,.
	\end{equation}
	We write this condition as \(H \pitchfork \mV\).
\end{definition}

According to the above definition, if \(H \notin \Tan(\mV)\) then there exists \(x \in U\) (the regular set of \(\mV\)) such that \(\mV_x \not\subset T_xH\). In particular, the subspaces \(\mV_x, T_xH \subset T_x \CP^n\) are transversal. By the openness of transversality, it follows that \(\mV_y\) and \(T_yH\) are transverse for all points \(y\) in a Zariski open subset of \(H\).

The next result lies at the core of the proof of Theorem \ref{thm:stability}\,.

\begin{proposition}\label{prop:key}
	Suppose that \((\mH, \ba)\) is a weighted arrangement that is  klt and CY.
	Then there is \(\delta>0\) such that, for any distribution \(\mV \subset T\CP^n\) with non-negative index \(\imath \geq 0\), we have
	\begin{equation}\label{eq:keyest}
		\sum_{H  | H \pitchfork \mV} a_H \geq \imath + \delta  \,,
	\end{equation}
	where the sum is over all \(H \in \mH\) that are transverse to \(\mV\).
\end{proposition}

\begin{proof}
	Take 
	\begin{equation}\label{eq:delta}
	\delta = \min \{\delta_1, \delta_2\}	
	\end{equation}
	with 
	\[
	\delta_1 = \min\limits_{H \in \mH} a_H  \,\, \text{ and } \,\,
	\delta_2 = \min_{L \subset \CP^n} \left( \codim L - \sum_{H | L \subset H} a_H \right) \,,
	\]
	where the minimum in the definition of \(\delta_2\) is taken over all non-empty linear subspaces \(L \subset \CP^n\) with \(\codim L \geq 2\). The klt condition implies that \(\delta>0\).
	
	We prove that Equation \eqref{eq:keyest} holds with the above \(\delta\). 	
	We analyse the cases \(\imath = 0\) and \(\imath > 0\) separately.
	\begin{itemize}
		\item Case \(\imath = 0\). We want to show that there is at least one hyperplane \(H' \in \mH\) such that \(\mV\) is not tangent to \(H'\).
		We proceed by contradiction and suppose that \(\mV\) is tangent to every hyperplane in \(\mH\). Then, since \(\mH\) contains \(n+1\) linearly independent hyperplanes, by Proposition \ref{prop:zeroindexdist} there is a non-zero holomorphic vector field \(Y\) that is tangent to all the members of \(\mH\), but this contradicts Corollary \ref{cor:noauto}\,. We conclude that there is \(H' \in \mH\) such that \(\mV\) is not tangent to \(H'\); therefore
		\begin{equation}\label{eq:esti0}
			\sum_{H  | H \pitchfork \mV} a_H  \geq a_{H'}  \geq \delta_1 = \imath + \delta_1 .		
		\end{equation}
		
		\item Case \(\imath > 0\). By Proposition \ref{prop:posindexdist} there is a subspace \(M \subset \CP^n\) with \(\dim M \geq \imath -1\) such that if \(\mV\) is tangent to \(H\) then \(H \supset M\). In particular, if \(M \not\subset H\) then \(H \pitchfork \mV\) and hence
		\begin{equation*}\label{eq:posind1}
			\sum_{H  | H \pitchfork \mV} a_H \geq \sum_{H | M \not\subset H} a_H .
		\end{equation*}
		Using the CY condition, we have
		\begin{equation*}\label{eq:posind2}
			\begin{aligned}
				\sum_{H | M \not\subset H} a_H 
				&= n+1 - \sum_{H | M \subset H} a_H \\
				&=  n+1 - \codim M + \left( \codim M - \sum_{H | M \subset H} a_H \right)  \\
				&\geq \imath + \delta_2 \,,
			\end{aligned}
		\end{equation*}
		where the last inequality uses the bound \(\dim M \geq \imath -1\) and the definition of \(\delta_2\). We conclude that
		\begin{equation}\label{eq:estipost}
		\sum_{H  | H \pitchfork \mV} a_H \geq \imath + \delta_2 .
		\end{equation}
	\end{itemize}
	The statement follows from Equations \eqref{eq:delta}, \eqref{eq:esti0}, and \eqref{eq:estipost}.
\end{proof}

\begin{remark}
	The number \(\delta>0\) in Proposition \ref{prop:key} depends only on the weighted arrangement \((\mH, \ba)\) -as given by Equation \eqref{eq:delta}- and not on the distribution \(\mV\).
\end{remark}

\subsection{Proof of Theorem \ref{thm:stability}}\label{sec:pfst}

Let \(\mV\) be a saturated subsheaf of \(\mE\).
We want to show that there is \(k_0\), independent of \(\mV\), such that \(\pardeg_{P_k}(\mV_{*}) < 0\) for all \(k \geq k_0\).
Let \(\imath\) be `the index' of \(\mV\), as defined by Equation \eqref{eq:c1V}.

\begin{notation}\label{not:On}
	We denote by \(O(k^{n-2})\) a polynomial in \(k\) of degree at most \(n-2\), say \(\sum_{j=0}^{n-2} C_j k^j\), such that there is a positive number \(K = K (\mH, b_L)\), depending only on \(\mH\) and the polarization coefficients \(b_L\), such that \(|C_j| \leq K\) for all \(j\). 
\end{notation}

\begin{lemma}\label{lem:negativeindex}
	Suppose that \(\imath < 0\). Then there is \(k_0\), that depends only on the arrangement \(\mH\) and the polarization coefficients \(b_L\), such that \(\deg_{P_k} (\mV) < 0\) for all \(k > k_0\).
\end{lemma}

\begin{proof}
	By Equation \eqref{eq:c1V}, we have
	\begin{equation*}
	\deg_{P_k} (\mV) = \imath \cdot c_1(P_k)^{n-1} \cdot \pi^*h + \sum_{L \in \mLi^{\circ}} d_L \cdot c_1(P_k)^{n-1} \cdot \gamma_L \,.		
	\end{equation*}
	By Lemmas \ref{lem:boundvol} and \ref{lem:boundo1}, we have \(c_1(P_k)^{n-1} \cdot \gamma_L = f_L\) with \(f_L = O(k^{n-2})\) and 
	\[c_1(P_k)^{n-1} \cdot \pi^*h = k^{n-1} + f_0\] 
	with \(f_0 = O(k^{n-2})\). Together with the bound \(d_L \leq n - \imath\) from Lemma \ref{lem:bounddL}, we obtain
	\begin{equation*}
	\begin{aligned}
	\deg_{P_k} (\mV) &= \imath \cdot k^{n-1} + \imath \cdot f_0 + \sum_{L \in \mLi^{\circ}} d_L \cdot f_L \\
	&\leq \imath \cdot \left( k^{n-1} +  f_0 - \sum_{L \in \mLi^{\circ}} f_L   \right)  + n \sum_{L \in \mLi^{\circ}}  f_L
	\end{aligned}
	\end{equation*}
	Dividing by \(k^{n-1}\), we get
	\begin{equation*}
	\frac{1}{k^{n-1}} \deg_{P_k} (\mV) \leq \,\, \imath \cdot \textup{A}  \,\, + \,\, \textup{B}
	\end{equation*}
	with
	\[
	\textup{A} = 1 + O(\epsilon) \,\, \text{ and } \,\, \textup{B} = O(\epsilon)	,
	\]
	where \(O(\epsilon)\) denotes a polynomial in the variable \(\epsilon = 1/k\) with zero constant term  whose coefficients are uniformly bounded in absolute value in terms of the arrangement and the fixed integers \(b_L\). Thus, if we choose \(k_0\) sufficiently big so that, say \(\textup{A} > 1/2\) and \(\textup{B} < 1/3\) for all \(k \geq k_0\), then
	\[
	\frac{1}{k^{n-1}} \deg_{P_k} (\mV) \leq \imath \cdot \textup{A} + \textup{B} < -\frac{1}{2} + \frac{1}{3} < 0
	\]
	for all \(\imath \leq -1\).
\end{proof}

%We continue with the proof of Theorem \ref{thm:stability}\,. 
The saturated subsheaf \(\mV \subset \mE\) induces in a natural way a distribution on \(\CP^n\). Concretely, \(\mV\) is a vector subbundle of \(\mE = \pi^* T\CP^n\) outside a codimension \(2\) analytic subset \(Z \subset X\). The analytic subset \(W \subset \CP^n\) given by
\[
W = \pi(Z) \cup \left( \bigcup_{L \in \mLi^{\circ}} L \right)
\]
has codimension \(\geq 2\) and \(\pi\) restricts to a biholomorphism between  \(\pi^{-1}(U)\) and \(U\), where \(U = \CP^n \setminus W\). On the open set \(U\), the push-forward sheaf \(\pi_*\mV\) is a vector subbundle of \(T\CP^n\), thus defining a distribution on \(\CP^n\). By slight abuse of notation, we shall also write \(\mV\) for the distribution \(\mV \subset T\CP^n\).

\begin{lemma}\label{lem:nonnegindex}
	If \(\imath \geq 0\), then
	\begin{equation}
		\pardeg_{P_k}(\mV_{*}) \leq \left(\imath - \sum_{H  | H \pitchfork \mV} a_H\right) \cdot k^{n-1} + O(k^{n-2}) \,.
	\end{equation}
	%where \(O(k^{n-2})\) is a polynomial in \(k\) of degree at most \(n-2\) whose coefficients are uniformly bounded in absolute value in terms of the arrangement \(\mH\) and the polarization coefficients \(b_L\).
\end{lemma}

\begin{proof}
The induced parabolic bundle \(\mV_{*}\) is given by the filtration
\[
\mV^L_a = \mV \cap \mE^L_a = \begin{cases}
\mV^L \, &\text{ for } \,  0 < a < a_L, \\
\mV \, &\text{ for } a_L \leq a \leq 1 ;
\end{cases}
\]
where \(L\) ranges over all elements in \(\mLi\) and \(\mV^L\) is the sheaf of sections of \(\mV\) which are tangent to \(\pi^*(TL)\) when restricted to \(D_L\). In particular, if on the regular set of \(\mV\) (c.f. Remark \ref{rmk:indpar}), the restriction of \(\mV\) to \(D_L\) is contained in \(\pi^*(TL)\), then the filtration \(\mV^L_a\) is trivial, in the sense that \(\mV^L_a = \mV\) for all \(0< a \leq 1\).

If \(L = H\) is a hyperplane, then the quotient \(\mV / \mV^H\) is non-zero if and only if
\[
\mV|_{D_H} \not\subset \pi^*(TH) \,.
\]
This is equivalent to \(H\) being transverse to the distribution \(\mV \subset T\CP^n\) as in Definition \ref{def:HtransV}. It follows that
\begin{equation}\label{eq:pf1}
\begin{gathered}
\pardeg_{P_k}(\mV_{*}) = \deg_{P_k}(\mV) \,\, - \\
\left( \sum_{H  | H \pitchfork \mV} a_H \cdot c_1(P_{k})^{n-1} \cdot \gamma_H + \sum_{L \in \mLi^{\circ}} r_L a_L \cdot c_1(P_k)^{n-1} \cdot \gamma_L   \right) \,,
\end{gathered}	
\end{equation}
where \(r_L\) is the rank (possibly zero) of the quotient sheaf on \(D_L\) given by \(\mV / \mV^L\). Same as in the proof of Lemma \ref{lem:negativeindex}\,,
\begin{equation*}
\begin{aligned}
\deg_{P_k} (\mV) &= \imath \cdot c_1(P_k)^{n-1} \cdot \pi^*h + \sum_{L \in \mLi^{\circ}} d_L \cdot c_1(P_k)^{n-1} \cdot \gamma_L		\\
&\leq \imath \cdot k^{n-1} +  \imath \cdot f_0 + \sum_{L \in \mLi^{\circ}} (n-\imath) \cdot f_L
\end{aligned}
\end{equation*}
where \(f_0 = O(k^{n-2})\) and \(f_L = O(k^{n-2})\). Since \(0 \leq \imath \leq n\), we get
\begin{equation}\label{eq:pf2}
\deg_{P_k}(\mV) \leq \imath \cdot k^{n-1} + O(k^{n-2}).	
\end{equation}
On the other hand, since the weights \(a_H, a_L \in (0,1)\) and \(0 \leq r_L \leq n\),
\begin{equation}\label{eq:pf3}
\begin{gathered}
\sum_{H  | H \pitchfork \mV} a_H \cdot c_1(P_{k})^{n-1} \cdot \gamma_H + \sum_{L \in \mLi^{\circ}} r_L a_L \cdot c_1(P_k)^{n-1} \cdot \gamma_L   \\
= \left(\sum_{H  | H \pitchfork \mV} a_H \right) \cdot k^{n-1} + O(k^{n-2}) .
\end{gathered}	
\end{equation}
The lemma follows from Equations \eqref{eq:pf1}, \eqref{eq:pf2}, and \eqref{eq:pf3}.
\end{proof}

\begin{lemma}\label{lem:polybound}
	Let \(p(k)\) be a polynomial with real coefficients 
	\[
	p(k) = c_{n-1} k^{n-1} + \ldots + c_1 k + c_0 .
	\]
	Suppose that \(c_{n-1} = - \delta\) with \(\delta>0\) and that \(c_{n-2}, \ldots, c_0 \leq C\) for some \(C>0\). Then \(p(k) < 0\) for all positive integers \(k > 2(n-1)C / \delta\).
\end{lemma}

\begin{proof}
	For \(k \geq 1\) we have 
	\[
	p(k) \leq \left(-\delta + \frac{(n-1)C}{k} \right) k^{n-1} 
	\]
    and the statement follows.
\end{proof}

\begin{proof}[Proof of Theorem \ref{thm:stability}]
	We divide the proof into cases according to the sign of \(\imath\). 
	
	The case \(\imath < 0\) follows from the obvious inequality:  
	\[
	\pardeg_{P_k}(\mV_{*}) \leq \deg_{P_k}(\mV)
	\]
	(see Remark \ref{rmk:pc1lessc1}) together with Lemma \ref{lem:negativeindex}.
	
	Therefore, we can assume that \(\imath \geq 0\). 
	By Lemmas \ref{lem:nonnegindex} and \ref{lem:polybound}, it is enough to show that 
	\begin{equation}
		 c := \imath - \sum_{H  | H \pitchfork \mV} a_H
	\end{equation}
	is \(< 0\).
	Since index of the distribution \(\mV \subset T\CP^n\) is equal to \(\imath \geq 0\),
	by Proposition \ref{prop:key} we get that \(c = -\delta\) for some \(\delta>0\) that depends only on \((\mH, \ba)\). This concludes the proof of Theorem \ref{thm:stability}\,.
\end{proof}

\begin{remark}
Tracing back the arguments in the proof of Theorem \ref{thm:stability}, we showed that we can take \(k_0 = C / \delta\), where \(C\) only depends on \(\mH\) and the polarization coefficients \(b_L\), and \(\delta\) is as in Proposition \ref{prop:key}\,.	
\end{remark}

\subsection{Proof of Theorem \ref{thm:main}}\label{sec:pfmain}

Let \(\mE_{*}\) be the parabolic bundle on \((X, D)\) as defined in Section \ref{sec:pardef} and let \(P_k\) be the polarization on \(X\) as given by Lemma \ref{lem:polarization}\,. Consider the top product of the polarization \(P_k\) with \(\parch_2(\mE_{*})\) given by
\begin{equation}
	p(k) = c_1(P_k)^{n-2} \cdot \parch_2(\mE_{*}) \,.
\end{equation}
The expression \(p(k)\) defines a polynomial in \(k\) of degree \(n-2\) with real coefficients. The coefficients of \(p(k)\) depend only on the weighted arrangement \((\mH, \ba)\) and the fixed integers \(b_L\) involved in the choice of polarization \(P_k\). More precisely, 
if we write 
\[
c_1(P_k) = k \cdot \pi^*h - e
\] 
with \(e = \sum_{L \in \mLi^{\circ}} b_L \cdot \gamma_L\), then
\begin{equation}\label{eq:pN}
p(k) = C_{n-2} k^{n-2} + C_{n-1} k^{n-1} + \ldots + C_0 \,,
\end{equation}
where the coefficients \(C_j\) are given by
\begin{equation}\label{eq:coefpol}
C_j = \binom{n-2}{j} \cdot  (\pi^*h)^j \cdot (-e)^{n-2-j} \cdot \parch_2(\mE_{*}) \,.
\end{equation}

\begin{lemma}\label{lem:hoc}
	The coefficient \(C_{n-2}\) in Equation \eqref{eq:pN} is given by
	\begin{equation}\label{eq:hoc}
	C_{n-2} =  \sum_{L \in \mLi^{n-2}} a_L^2 - \frac{1}{2} \sum_{H \in \mH} B_H \cdot a_H^2 - \frac{n+1}{2} \,,
	\end{equation}	
	where \(B_H + 1\) is the number of codimension \(2\) irreducible subspaces contained in \(H\).
\end{lemma}

\begin{proof}
	By Equation \eqref{eq:coefpol} the coefficient \(C_{n-2}\) is equal to
	\begin{equation}
	C_{n-2} = (\pi^*h)^{n-2} \cdot \parch_2(\mE_{*}) .
	\end{equation}
	Equation \eqref{eq:parch2E} expresses \(\parch_2(\mE_{*})\) as a sum of \(3\) terms:
	\begin{equation*}
	\begin{gathered}
	\parch_2(\mE_{*}) = -\frac{n+1}{2} \cdot (\pi^*h)^2 \\
	+ \frac{1}{2} \sum_{L \in \mLi} a_L^2 \cdot \codim L \cdot \gamma_L^2 + \sum_{L \subsetneq M} a_L  a_{M} \cdot \codim M \cdot \gamma_L \cdot \gamma_M \,.
	\end{gathered}
	\end{equation*}
	Taking the corresponding products of the above \(3\) terms with \((\pi^*h)^{n-2}\) we have
	\begin{equation}
	C_{n-2} = \textup{A} + \textup{B} + \textup{C} 
	\end{equation}
	where A, B, and C are given as follows
	\begin{itemize}
		\item A is the easiest and is equal to
		\begin{equation}\label{eq:hocA}
		\textup{A} = - \frac{n+1}{2} \cdot (\pi^*h)^n = - \frac{n+1}{2} \,.
		\end{equation}
		\item B is the middle term involving self-intersections. 
		We use Corollary \ref{cor:vanishing} (ii) together with Equations \eqref{eq:selftl} and \eqref{eq:selfth} to obtain correspondingly:
		\[
		(\pi^*h)^{n-2} \cdot \gamma_L^2 = \begin{cases}
		0 & \text{ if } \codim L > 2 \,,\\
		-1 & \text{ if } \codim L = 2 \,,\\
		-B_H & \text{ if } L = H \in \mH  \,.
		\end{cases}
		\]
		We conclude that
		\begin{equation}\label{eq:hocB}
		\begin{aligned}
		\textup{B} &= \frac{1}{2} \sum_{L \in \mLi} a_L^2 \cdot \codim L \cdot \left( (\pi^*h)^{n-2} \cdot \gamma_L^2 \right) \\
		&= - \sum_{L \in \mLi^{n-2}} a_L^2 \,\, - \,\, \frac{1}{2} \sum_{H \in \mH} B_H \cdot a_H^2 \,.
		\end{aligned}
		\end{equation}
		\item C involves mixed intersections for pairs \(L \subsetneq M\). It follows from Lemma \ref{lem:vanishing}  that if \(\codim L > 2\) then
		\begin{equation*}
		(\pi^*h)^{n-2} \cdot \gamma_L \cdot \gamma_M = 0 .
		\end{equation*}
		On the other hand, if \(\codim L =2\) and \(M = H\) is a hyperplane that contains \(L\) then by Lemma \ref{lem:lh} we get
		\begin{equation*}
		(\pi^*h)^{n-2} \cdot \gamma_L \cdot \gamma_H = 1 .
		\end{equation*}
		We conclude that
		\begin{equation}\label{eq:hocC}
		\begin{aligned}
		\textup{C} &= \sum_{L \subsetneq M} a_L  a_{M} \cdot \codim M \cdot \left( (\pi^*h)^{n-2} \cdot \gamma_L \cdot \gamma_M \right)  \\
		&= \sum_{L \in \mLi^{n-2}} a_L \cdot \left(\sum_{H | L \subset H} a_H\right) \\
		&= 2  \sum_{L \in \mLi^{n-2}} a_L^2 \,.
		\end{aligned}
		\end{equation}
	\end{itemize}
	The result follows from Equations \eqref{eq:hocA}, \eqref{eq:hocB}, and \eqref{eq:hocC}.
\end{proof}

\begin{remark}
	Note that \(C_{n-2}\) is independent of the integers \(b_L\) involved in the choice of polarization, unlike the other coefficients \(C_j\).
\end{remark}

\begin{proof}[Proof of Theorem \ref{thm:main}]
	Let \((\mH, \textbf{a})\) be klt and CY. By Theorem \ref{thm:stability}, the parabolic bundle \(\mE_{*}\) on \((X, D)\) is slope stable with respect to the polarization \(P_k\) for all \(k \gg 1\). By Theorem \ref{thm:locab} the parabolic bundle \(\mE_{*}\) is locally abelian, therefore we can apply Theorem \ref{thm:boggies}\,. 
	By the Bogomolov-Gieseker inequality \eqref{eq:boggies}, for all \(k \gg 1\) we have
	\begin{equation}
	p(k) = c_1(P_k)^{n-2} \cdot \parch_2(\mE_{*}) \leq 0 \,.
	\end{equation}
	In particular, the coefficient \(C_{n-2}\) of the highest order term of \(p(k)\) in Equation \eqref{eq:pN} must be non-positive, i.e.,
	\begin{equation}\label{eq:cneg}
		C_{n-2} \leq 0 \,.
	\end{equation} 
	Equation \eqref{eq:mythm} follows from Equation \eqref{eq:cneg} together with Lemma \ref{lem:hoc}\,.
\end{proof}

\section{The quadratic form and the stable cone}\label{sec:qfsc}

In Section \ref{sec:quadform}\,, we define the \emph{quadratic form} \(Q: \R^{\mH} \to \R\) associated to \(\mH\) by extending the left hand side of Equation \eqref{eq:mythm} from the affine hyperplane \(\{\sum_H a_H = n+1\} \subset \R^{\mH}\) to the whole space \(\R^{\mH}\) as a homogeneous degree \(2\) polynomial.\newpar

\noindent In Section \ref{sec:stabcone}\,,
we introduce the concept of \emph{stable weighted arrangements} 
and the \emph{stable cone} \(C^{\circ} \subset \R^{\mH}\) of the arrangement \(\mH\).\newpar

\noindent In Section \ref{sec:matpol}\,, we introduce the \emph{semistable cone} \(C\), the \emph{matroid polytope}, and show that the quadratic form \(Q\) is \(\leq 0\) on \(C\) (Theorem \ref{thm:maingeneral}).\newpar

\noindent In Section \ref{sec:pointsGIT}\,, we provide links to GIT and stability of pairs.

\subsection{The quadratic form of an arrangement}\label{sec:quadform}

Let \(\mH \subset \CP^n\) be a hyperplane arrangement. 
Let \(s\) be the linear function on \(\R^{\mH}\) that takes a  vector \(\ba \in \R^{\mH}\) and sends it to the sum of its components,
\begin{equation}\label{eq:s}
s = \sum_{H \in \mH} a_H \,.
\end{equation}
Recall that if \(L \in \mLi^{n-2}\) is an irreducible subspace of codimension \(2\) then the weight \(a_L\) at \(L\) is given by
\begin{equation}\label{eq:aLcod2}
a_L = \frac{1}{2} \sum_{H | L \subset H} a_H \,.	
\end{equation}
Thus, we can think of \(a_L\) as a linear function on \(\R^{\mH}\) as well.

\begin{definition}\label{def:quadraticform}
	The quadratic form \(Q = Q(\ba)\) of \(\mH\) is the homogeneous polynomial of degree \(2\) on \(\R^{\mH}\)  given by
	\begin{equation}\label{eq:Q}
	Q(\ba) = 4(n+1) \cdot \sum_{L \in \mLi^{n-2}} a_L^2 \, - \,  2(n+1) \cdot \sum_{H \in \mH} B_H \cdot a_H^2 \, - \, 2 \cdot s^2 \,,
	\end{equation}
	where \(B_H + 1\) is the number of irreducible codimension \(2\) subspaces contained in \(H\).
\end{definition}

Taking common factor \(4(n+1)\), we have
\begin{equation}\label{eq:Q2}
	Q(\ba) = 4(n+1) \cdot \left( \,
	\sum_{L \in \mLi^{n-2}} a_L^2 \, - \,  \frac{1}{2} \cdot \sum_{H \in \mH} B_H \cdot a_H^2 \, - \, \frac{s^2}{2(n+1)} 
	\right) \,.
\end{equation}
Up to the constant dimensional factor \(4(n+1)\)\,, the quadratic form \(Q\) agrees with the left hand side of Equation \eqref{eq:mythm} on the affine hyperplane \(\{s=n+1\} \subset \R^{\mH}\). We record this fact as a lemma.

\begin{lemma}\label{lem:Qsn}
	If \(s(\ba) = n+1\),  then
	\[
	Q(\ba) = 4(n+1) \left( \,
	\sum_{L \in \mLi^{n-2}} a_L^2 \, - \,  \frac{1}{2} \cdot \sum_{H \in \mH} B_H \cdot a_H^2 \, - \, \frac{n+1}{2} \,
	  \right) \,.
	\]
\end{lemma}

We can write down the symmetric matrix associated with the quadratic form \(Q\). To do this, label the hyperplanes, say
\[
\mH = \{H_1, \ldots, H_N\} \,,
\]
and let \(a_i = a_{H_i}\) be the weight at the hyperplane \(H_i\).
Then, we can write
\[
Q(\ba) = \sum_i Q_{ii} \cdot a_i^2 + \sum_{i < j}  Q_{ij} \cdot 2a_i a_j \,.
\]
The coefficients \(Q_{ij}\) can be easily calculated from Equation \eqref{eq:Q}. Let \(\sigma_i\) be the number of irreducible codimension \(2\) subspaces contained in \(H_i\), i.e., \(\sigma_i = B_{H_i} + 1\). Then 
\begin{equation}\label{eq:Qmatrix}
	Q_{ij} = \begin{cases}
	- (n+1) \sigma_i + 2n &\text{ if } i = j \,, \\
	-2 &\text{ if }  i \neq j \text{ and } H_i \cap H_j \text{ is reducible }, \\
	\,\, n-1 &\text{ if } i \neq j \text{ and }  H_i \cap H_j \text{ is irreducible } \,.
	\end{cases}
\end{equation}
By slight abuse of notation, we shall also write \(Q\) for the symmetric matrix \((Q_{ij})\).

\begin{remark}
	If \(n=2\) then \(Q_{ij} = -A_{ij}\) where \(A_{ij}\) are the entries of symmetric matrix defined  by Hirzebruch \cite[Equation (3)]{hirzebruch2} regarding H\"ofer's formula for the `proportionality' \(3c_2 - c_1^2\) of  coverings of the projective plane branched along line arrangements. 
\end{remark}

For the record, we note the following property of the quadratic form:

\begin{lemma}\label{lem:sumqij}
	The sum of the \(i\)-th column (or \(i\)-th row) of the matrix \((Q_{ij})\) is given by
	\begin{equation}\label{eq:sumqij}
	\sum_{j=1}^{N} Q_{ij} = (n - 1) \cdot N \,-\, (n+1) \cdot (t_i - 1) \,,
	\end{equation}
	where \(t_i\) is the number of codimension \(2\) subspaces \(L \in \mL\) contained in \(H_i\).
\end{lemma}

\begin{proof}
	Let \(\mathcal{R}\) and \(\mI\) be the set of all reducible and irreducible codimension \(2\) subspaces contained in \(H_i\) and let \(T = \mathcal{R} \cup \mI\). We write \(\tau_i = |\mathcal{R}|\) and
	\(\sigma_i = |\mI|\), so that \(t_i = \tau_i + \sigma_i\).
	It follows from Equation \eqref{eq:Qmatrix} that
	\[
	\begin{aligned}
		\sum_{j=1}^{n} Q_{ij} &= 2n -(n+1) \cdot \sigma_i - 2 \tau_i + (n-1) \cdot \sum_{L \in \mI} (m_L - 1) \\
		&= 2n - (n+1) \cdot t_i \,+\, (n-1) \cdot \tau_i + (n-1) \cdot \sum_{L \in \mI} (m_L - 1) \\
		&= 2n -(n+1) \cdot t_i 
		\,+\, (n-1) \cdot \sum_{L \in T} (m_L - 1) 
	\end{aligned}
	\]
	Since for every \(j \neq i\) the intersection \(H_i \cap H_j\) is a codimension \(2\) subspace contained in \(H_i\)
	\[
	\sum_{L \in T} (m_L -1) = N-1 \,.
	\]
	Therefore,
	\[
	\sum_{j=1}^{n} Q_{ij} = 2n - (n + 1) \cdot t_i  + (n-1) \cdot (N-1) 
	\]
	which implies \eqref{eq:sumqij}.
\end{proof}

\begin{remark}
	It follows from Lemma \ref{lem:sumqij} that the vector \(\mathbf{1}\) with all components equal to \(1\) is an eigenvector of \(Q\) if and only if every hyperplane \(H \in \mH\) intersects the other hyperplanes of the arrangement along the same number of codimension \(2\) subspaces.
\end{remark}

The next example shows that, for generic arrangements, the quadratic form \(Q\) is non-degenerate and indefinite.

\begin{example}\label{ex:Qgeneric}
	Let \(\mH = \{H_1, \ldots, H_N\}\) be an arrangement in \(\CP^n\) with \(n>1\) which is \emph{generic}, in the sense that no \(3\) distinct hyperplanes are linearly dependent. Suppose that the number of hyperplanes \(N\) is \(> n+ 1\). 
	The generic assumption implies that \(\sigma_i = 0\) for all \(i\). Equation \eqref{eq:Qmatrix} gives us
	\[
	Q = 2 \cdot \begin{pmatrix}
	n & -1 & \cdots & -1 \\
	-1 & n  & \cdots & -1 \\
	\cdots & \cdots  & \cdots & \cdots \\
	-1 & -1 & \cdots & n \\
	\end{pmatrix} \,.
	\]
	
	Let \(\mathbf{1} \in \R^N\) be the vector with all entries equal to \(1\). Then \(\mathbf{1}\) is an eigenvector with
	\[
	Q \cdot \mathbf{1} = 2 (n+1 - N) \cdot \mathbf{1} \,.
	\]
	The assumption that \(N > n+1\) implies that the eigenvalue is \(<0\). On the other hand, if \(\ba = (a_1, \ldots, a_N)\) is a vector with \(\sum_i a _i = 0\), i.e., \(\ba\) is orthogonal to \(\mathbf{1}\), then
	\[
	Q \cdot \ba = 2(n+1) \cdot \ba \,. 
	\]
	This shows that the orthogonal complement \(\mathbf{1}^{\perp}\) is an eigenspace of \(Q\) with positive eigenvalue. Thus, the quadratic form \(Q\) is non-degenerate of signature \((N-1, 1)\).
\end{example}

Next, we provide an example for which the quadratic form \(Q\) is negative semidefinite and has non trivial kernel.

\begin{example}\label{ex:Qbraid}
	Let \(\mH \subset \CP^n\) be the braid (or \(A_{n+1}\)) arrangement consisting of the \(\binom{n+2}{2}\) hyperplanes \(H_{ij} = \{x_i = x_j\}\) for \(1 \leq i < j \leq n+2\) in 
	\[
	\CP^n = \P \left( \C^{n+2} \big/ \{x_1 = x_2 = \ldots = x_{n+2}\} \right) .
	\]
	The elements \(L \in \mLi^{n-2}\) are of the form \(L = \{x_i = x_j = x_k\}\) for a triplet of distinct indices \(i, j, k\). 
	The number of codimension \(2\) irreducible subspaces contained in \(H_{ij}\) is equal to \(n = |[n+2] \setminus \{i, j\}|\). Therefore, \(B_H = n -1\) for all \(H\).  Equation \eqref{eq:Q} gives us
	\begin{equation}\label{eq:Qbraid}
	Q = (n+1) \sum_{i < j < k} \left(a_{ij} + a_{ik} + a_{jk}\right)^2 - 2\cdot(n^2-1) \cdot \sum_{i < j} a_{ij}^2 - 2 \cdot \left(\sum_{i < j} a_{ij}\right)^2 \,,	
	\end{equation}
	where \(a_{ij}\) is the weight at \(H_{ij}\). This quadratic form is negative semidefinite and has non trivial kernel, as shown by the next.
\end{example}

\begin{lemma}\label{lem:qbraid}
	The quadratic form \eqref{eq:Qbraid} is negative semidefinite on \(\R^{\mH}\) with kernel the linear subspace of \(K \subset \R^{\mH}\) parametrized by \(a_{ij} = a_i + a_j\) with \((a_1, \ldots, a_{n+2}) \in \R^{n+2}\). 
\end{lemma}

\begin{proof}
	Identify \(\R^{\mH}\) with \(\R^N\) where \(N = \binom{n+2}{2}\) with coordinates \((a_{ij})_{i < j}\) and basis vectors \(\mathbf{e}_{ij}\). If \(Q\) is the symmetric matrix of the quadratic form, then we can write
	\begin{equation}\label{eq:qbrmat}
	Q \cdot \ba =  \sum_{i < j} \langle \ba, \mathbf{q}_{ij} \rangle \cdot \mathbf{e}_{ij} \,,
	\end{equation}
	where \(\langle \cdot, \cdot \rangle\) is the Euclidean product on \(\R^N\) and the vectors \(\mathbf{q}_{ij}\) are the rows of the matrix \(Q\) given by Equation \eqref{eq:Qmatrix}
	\begin{equation}\label{eq:qbraid}
	(\mathbf{q}_{ij})_{kl} = \begin{cases}
	n(1-n) &\text{ if }  \{k, l\} = \{i, j\}\,, \\
	-2 &\text{ if }  \{k, l\} \cap \{i, j\} = \emptyset\,, \\
	n -1 &\text{ if } | \{k, l\} \cap \{i, j\}| = 1 \,.
	\end{cases}	
	\end{equation}
	
	Let \(\bv_i\) be the vector with components 
	\[
	(\bv_i)_{kl} = \begin{cases}
	1 \,\, &\textup{ if } \,\, i \in \{k, l\} \\
	0 \,\, &\textup{ if } \,\, i \notin \{k, l\} 
	\end{cases}
	\]
	so that \(\bv_1, \ldots, \bv_{n+2}\) make a basis of the linear subspace \(K\). By Equation \eqref{eq:qbraid},
	\begin{equation}\label{eq:qbrvec}
	\mathbf{q}_{ij} = (n+1) \cdot(\bv_i + \bv_j) -n(n+1) \mathbf{e}_{ij} -2 \cdot \mathbf{1} \,,	
	\end{equation}
	where \(\mathbf{1}\) is the vector with all entries equal to \(1\).
	By Equations \eqref{eq:qbrmat} and \eqref{eq:qbrvec}, together with
	\[
	\langle \bv_i , \mathbf{1} \rangle = n + 1\,, \hspace{2mm} 
	\langle \bv_i , \bv_j \rangle = \begin{cases}
	n + 1 &\text{ if } i = j \,, \\
	1 &\text{ if } i \neq j \,,
	\end{cases}
	\hspace{2mm} \text{ and } \mathbf{1} \in K \,,
	\]
	we deduce that \(Q \cdot \bv_i = 0\) for all \(i\) and \(Q \cdot \ba = -n(n+1) \cdot \ba\) for all \(\ba \in K^{\perp}\). Thus, the quadratic form \(Q\) is negative semidefinite with kernel \(K\).		
\end{proof}

\begin{remark}
	For \(n=2\), \cite[\S 5]{hirzebruch2} asserts that for all reflection line arrangements listed in Section 3 of that paper, the quadratic form \(Q\) is negative semidefinite. 
	In Theorem \ref{thm:refarr}\,,
	we provide an extension to higher dimensions.
	The calculation of the kernel of \(Q\) for reflection arrangements is also related to the classification of Dunkl connections in \cite[\S 2.6]{chl}.
\end{remark}

\begin{example}\label{ex:qprod}
	Let \(\mH = \mH_1 \times \mH_2\) be a product arrangement in \(\CP^n\) (Definition \ref{def:prod}) with \(\mH_1 \subset \CP^{n_1}\) and \(\mH_2 \subset \CP^{n_2}\). Then, in an obvious notation,
	\[
	\frac{Q}{n+1} = \frac{Q_1}{n_1+1}  + \frac{Q_2}{n_2 + 1} + \frac{2 s_1^2}{n_1+1} + \frac{2s_2^2}{n_2+1} - \frac{2s^2}{n+1} \,.
	\]
	In particular, 
	\[
	\frac{Q}{n+1} = \frac{Q_1}{n_1+1}  + \frac{Q_2}{n_2 + 1}
	\] 
	on the intersection \(\{s_1= n_1 + 1\} \cap \{s_2 = n_2 + 1\}\).
\end{example}

\subsection{The stable cone}\label{sec:stabcone}

Let \(\mH \subset \CP^n\) be a hyperplane arrangement and let
\(\ba \in \R^{\mH}\) be a weight vector with positive weights \(a_H > 0\) for all \(H \in \mH\). Let \(\ba'\) be the rescaled vector
\begin{equation}\label{eq:stable}
\ba' = \lambda \cdot \ba \,\, \text{ with } \,\, \lambda = \frac{n+1}{s} \,,
\end{equation}
where \(s\) is the total sum of the weights as in Equation \eqref{eq:s}.

\begin{definition}\label{def:stable}
	The weighted arrangement \((\mH, \ba)\) is \emph{stable}
	if  the rescaled weighted arrangement \((\mH, \ba')\) is klt. 
\end{definition}

We have the following restatement of Theorem \ref{thm:main}.

\begin{theorem}\label{thm:klt}
	If \((\mH, \ba)\) is a stable weighted arrangement, then \(Q(\ba) \leq 0\).
\end{theorem}

\begin{proof}
	Since \(Q\) is homogeneous, it suffices to show that \(Q(\ba') \leq 0\). The choice of rescaling factor \(\lambda\) given by Equation \eqref{eq:stable} ensures that the weighted arrangement \((\mH, \ba')\) is CY. On the other hand, the stability assumption requires that \((\mH, \ba')\) is klt.
	By Lemma \ref{lem:Qsn} and Theorem \ref{thm:main}\,,  \(Q(\ba') \leq 0\). 
\end{proof}

\begin{remark}
	If \(n=1\) and \((\mH, \ba)\) is stable then \(\mH\) must have at least \(3\) points. In this case, we take the convention that \(\mLi^{-1} = \{\emptyset\}\). Since \(\emptyset\) is contained in all \(H \in \mH\), we have that \(a_{\emptyset} = s/2\) and \(B_H = 1 - 1 = 0\) for all \(H\). Therefore, \(Q = s^2/4 - s^2/4 = 0\) and the statement of Theorem \ref{thm:klt} is trivial.	
\end{remark}

We provide \(3\) slightly different but equivalent systems of linear inequalities that characterize the weights for which the arrangement \(\mH\) is stable.
For an arbitrary non-empty and proper linear subspace \(L \subset \CP^n\),  consider the equation
\begin{equation}\label{eq:starr}
\sum_{H | L \subset H} a_H < \frac{\codim L}{n+1} \cdot \sum_{H \in \mH} a_H \,.
\end{equation}

\begin{lemma}\label{lem:stablearr}
	A weighted arrangement \((\mH, \ba)\) with positive weights \(a_H >0\) is stable if and only if
	any of the following equivalent conditions is satisfied: 
	\begin{enumerate}[label=\textup{(\roman*)}]
		\item Equation \eqref{eq:starr} holds for every \(L \in \mLi\)\,;
		\item Equation \eqref{eq:starr} holds for every \(L \in \mL\);
		\item Equation \eqref{eq:starr}  holds for every non-empty and proper linear subspace \(L \subset \CP^n\).
	\end{enumerate}
\end{lemma}

\begin{proof}
	The weighted arrangement \((\mH, \ba)\) is stable if the rescaled CY weighted arrangement \((\mH, \ba')\) is klt. Corollary \ref{cor:kltcondition} implies that (i), (ii), and (iii) are equivalent.
\end{proof}

\begin{definition}\label{def:stabcone}
	The \emph{stable cone} \(C^{\circ}\) of \(\mH\) is the set of all weights \(\ba \in \R^{\mH}_{>0}\) such that the weighted arrangement \((\mH, \ba)\) is stable. Equivalently, \(C^{\circ}\) is the cone over the set of weights \(\ba \in \R_{>0}^{\mH}\) for which the weighted arrangement \((\mH, \ba)\) is klt and CY. 
\end{definition}

By Lemma \ref{lem:stablearr}\,, the stable cone \(C^{\circ}\) is an open convex polyhedral cone; it is the subset of the positive octant  defined by the linear inequalities \eqref{eq:starr}. Next, we give two explicit examples of stable cones.

\begin{example}
		If \(\mH\) is normal crossing then the only irreducible subspaces are the hyperplanes of the arrangement, i.e., \(\mLi = \mH\).
		By Lemma \ref{lem:stablearr} (i), the weighted arrangement \((\mH, \ba)\) is stable if and only if 
		\begin{equation}\label{eq:sthyp}
			\forall H \in \mH: \,\, 0 < a_H < \frac{s}{n+1} \,.
		\end{equation}
\end{example}

\begin{example}\label{ex:Qbraid2}
	Let \(\mH \subset \CP^n\) be the braid arrangement as in Example \ref{ex:Qbraid}. The irreducible subspaces \(L_I\) correspond to subsets \(I \subset [n+2]\) with \(2 \leq |I| \leq n+1\) by letting \(L_I = \{x_i = x_j \,\, \textup{ for } \,\, i,j \in I \}\). In particular, \(\codim L_I = |I|-1\). 
	By Lemma \ref{lem:stablearr} (i),
	the weighted arrangement \((\mH, \ba)\) with weights \(a_{ij}>0\) is stable if and only if for every \(I \subset [n+2]\) with \(2 \leq |I| \leq n + 1\) we have
	\[
	\sum_{i < j \,\, | \,\, i, j \in I} a_{ij} < \frac{|I|-1}{n+1} \cdot s \,.
	\]
\end{example}

Theorem \ref{thm:klt} can be restated as follows. 

\begin{theorem}\label{thm:qcone}
	The quadratic form \(Q\) is negative semidefinite on the stable cone, 
	\begin{equation}
	C^{\circ} \subset \{Q \leq 0\} \,.	
	\end{equation}
\end{theorem}

\begin{remark}
	Of course, by continuity, \(Q \leq 0\) on the closure \(\overline{C^{\circ}}\).
\end{remark}

For generic arrangements, we can check Theorem \ref{thm:qcone} by direct calculation, as shown in the next example.

\begin{example}\label{ex:Qsnc}
	Let \(\mH \subset \CP^n\) be generic with \(|\mH|>n+1\) as in Example \ref{ex:Qgeneric}\,.  Then
	\begin{equation}\label{eq:snc1}
	Q = 2(n+1) \cdot \sum_{H \in \mH} a_H^2 \,-\, 2 \cdot s^2 \,.	
	\end{equation}
	is non-degenerate and it has signature \((|\mH|-1, 1)\).
	By Lemma \ref{lem:stablearr}\,, if \(\ba \in C^{\circ}\) then \(a_H < s / (n+1)\).
	Therefore,
	\begin{equation}\label{eq:snc2}
	\sum_{H \in \mH} a_H^2 < \frac{s}{n+1} \cdot \sum_{H \in \mH} a_H = \frac{s^2}{n+1} \,.	
	\end{equation}
	It follows from Equations \eqref{eq:snc1} and \eqref{eq:snc2} that \(Q < 0\) on \(C^{\circ}\).
\end{example}

\subsection{The matroid polytope and the semistable cone}\label{sec:matpol}

We present a geometric description of the set of weights for which an arrangement is stable, involving standard constructions of polytopes associated to matroids (see Section \ref{sec:matroidbasics} for basics on matroids).
For a more in depth discussion and relations to toric geometry, see \cite{alexeev}\,.

Let \(\mH \subset \CP^n\) be a hyperplane arrangement.

\begin{definition}
	A \emph{basis} \(\mB\) of \(\mH\) is a subset \(\mB \subset \mH\) with \(|\mB| = n+1\) and
	\[
	\bigcap_{H \in \mB} H = \emptyset \,.
	\]
\end{definition}

The arrangement \(\mH\) has a basis if and only if \(\mH\) is essential. If \(\mB\) is a basis, then up to a linear change of coordinates, \(\mB\) can be identified with the set of \(n+1\) coordinate hyperplanes in \(\CP^n\).

\begin{definition}
	The indicator function of a basis \(\mB \subset \mH\) is the vector \(\be_{\mB} \in \R^{\mH}\) with components
	\[
	(\be_{\mB})_H = \begin{cases}
	1 &\text{ if } H \in \mB\,, \\
	0 &\text{ if } H \notin \mB \,.
	\end{cases}
	\]
\end{definition}

\begin{definition}\label{def:matpol}
	The \emph{matroid polytope} \(P\) of \(\mH\) is the convex hull of the vectors \(\be_{\mB}\) with \(\mB \subset \mH\) a basis.
\end{definition}

The matroid polytope \(P\) is non-empty precisely when \(\mH\) is essential. 	
Since for every indicator function of a basis \(\mB \subset \mH\) we have \(s(\be_{\mB}) = n+1\), the matroid polytope is contained in the affine hyperplane \(s = n+1\),
\[
P \subset \{s = n+1\} \,.
\]
Therefore, the dimension of \(P\) is at most \(|\mH|-1\). It follows from \cite[Theorem 1.12.9]{borovikgelfand} that, if \(\mH\) is an essential arrangement, then
\[
\dim P = |\mH| - k \,,
\]
where \(k\) is the number of factors in the decomposition of \(\mH\) as a product of irreducible arrangements \(\mH \cong \mH_1 \times \ldots \times \mH_k\)\,.

A dual description of the matroid polytope in terms of defining linear inequalities is given by the following result of Edmonds.

\begin{theorem}[{\cite[Corollary 40.2d]{schrijver}}]\label{thm:edmonds}
	The matroid polytope \(P\) is the subset of the affine hyperplane \(\{s = n+1\}\) in \(\R^{\mH}\) defined by the following inequalities:
	\begin{equation}\label{eq:ip}
	\begin{gathered}
	\forall H \in \mH : \,\, a_H \geq 0 \,,\\
	\forall L \in \mL : \,\, \sum_{H | L \subset H} a _H \leq \codim L \,.
	\end{gathered}
	\end{equation}
\end{theorem}

Let \(P^{\circ}\) be the relative interior of the matroid polytope \(P\) inside the hyperplane \(\{s=n+1\}\). Specifically, \(P^{\circ}\) is the subset of \(\{s= n+1\}\) of points for which the inequalities \eqref{eq:ip} are strict. In particular, 
\(P^{\circ}\) is non-empty if and only if \(P\) has dimension \(|\mH|-1\).	

\begin{corollary}\label{cor:bp}
	The weighted arrangement \((\mH, \ba)\) is klt and CY if and only if 
	\begin{equation}
	\ba \in P^{\circ} \,.
	\end{equation}
\end{corollary}

\begin{proof}
	By definition, \((\mH, \ba)\) is klt if and only if all the inequalities in \eqref{eq:ip} are strict. The result follows from Theorem \ref{thm:edmonds}\,.
\end{proof}

Corollary \ref{cor:bp} implies the next.

\begin{corollary}\label{cor:stcone}
	The stable cone  \(C^{\circ}\) is the open polyhedral cone given by
	\begin{equation}
	C^{\circ} = \R_{>0} \cdot P^{\circ} \,.
	\end{equation}
\end{corollary}

\begin{definition}\label{def:semistabcone}
	The \emph{semistable cone} \(C \subset \R^{\mH}\) is the cone over the matroid polytope,
	\begin{equation}
		C = \R_{\geq 0} \cdot P \,.
	\end{equation}
\end{definition}

\begin{remark}\label{rmk:cones}
The stable cone \(C^{\circ}\) is the interior of the semistable cone \(C\), as a subset of \(\R^{\mH}\). If \(C^{\circ}\) is non-empty then \(C\) is equal to its closure, \(C = \overline{C^{\circ}}\). By Corollary \ref{cor:essirrst}, this happens precisely when \(\mH\) is essential and irreducible. However, if the essential arrangement \(\mH\) is reducible, then \(C^{\circ}\) is empty but \(C\) is not.	
\end{remark}

We can now state our main result in its more general form.

\begin{theorem}\label{thm:maingeneral}
	Let \(\mH \subset \CP^n\) be an essential arrangement. Then the quadratic form \(Q\) of \(\mH\) is \(\leq 0\) on the semistable cone \(C \subset \R^{\mH}\).
\end{theorem}

\begin{proof} 
	It suffices to show that \(Q \leq 0\) on the matroid polytope \(P\) of \(\mH\).
	If \(\mH\) is reducible, say \(\mH = \mH_1 \times \mH_2\), then the matroid polytope of \(\mH\) is the product of the matroid polytopes of the respective factors, say \(P = P_1 \times P_2\). 
	By Example \ref{ex:qprod}\,, on the polytope \(P\), we have 
	\[
	\frac{Q}{n+1} = \frac{Q_1}{n_1 + 1} + \frac{Q_2}{n_2 + 1} \,.
	\]
	Therefore, to prove the theorem, we can assume that \(\mH\) is irreducible.
	
	Suppose that \(\mH\) is irreducible. As pointed out in
	Remark \ref{rmk:cones}\,, if the arrangement \(\mH\) is irreducible then the semistable cone \(C\) is the closure of the stable cone \(C^{\circ}\). By Theorem \ref{thm:qcone}\,, the quadratic form \(Q\) of \(\mH\) is  \(\leq 0\) on \(C^{\circ}\). By continuity, \(Q \leq 0\) on \(C\). 
\end{proof}

\subsection{Relation to GIT and K-stability of pairs}\label{sec:pointsGIT}

In this section we justify the name stable cone, by providing links to Geometric Invariant Theory and stability of pairs.

\subsubsection{Configurations of points}

It is useful to work dually with configurations of points rather than hyperplane arrangements. 
We set-up the basic definitions.

\begin{definition}
	A weighted configuration of points \((\mP, \ba)\) is a finite set 
	\[
	\mP = \{p_1, \ldots, p_m\}
	\] 
	of points  \(p_i \in \CP^n\) together with a weight vector \(\ba = (a_1, \ldots, a_m) \in \R^m_{> 0}\).
\end{definition}

\begin{remark}
	Note that we require the weights \(a_i\) to be positive. Informally, we can think of \(a_i\) as the mass of the point \(p_i\)\,.
\end{remark}

The main concept of interest is the following.

\begin{definition}\label{def:stablepoints}
	A weighted configuration of points \((\mP, \ba)\) is \emph{stable}, if for every non-empty and proper projective subspace \(W \subset \CP^n\) the following holds:
	\begin{equation}\label{eq:stpointW}
	\sum_{i \,|\, p_i \in W} a_i < \frac{\dim W + 1}{n+1} \cdot \sum_{i=1}^{m} a_i \,.
	\end{equation}
\end{definition}

The \emph{sum poset} of \(\mP\) is the finite set \(\mU\) of all proper non-empty subspaces \(U \subset \CP^n\) spanned by points in \(\mP\), equipped with the partial order given by inclusion.
The next result is dual to Lemma \ref{lem:stablearr}\,, we omit its proof.
%To check stability it is enough to consider elements of the sum poset. 

\begin{lemma}\label{lem:stU}
	The weighted configuration of points \((\mP, \ba)\) is stable if and only if 
	\begin{equation}\label{eq:stU}
		\forall U \in \mU: \,\, \sum_{i \, | \, p_i \in U} a_i < \frac{\dim U + 1}{n+1} \cdot \sum_{i=1}^m a_i \,.
	\end{equation}
\end{lemma}

\begin{comment}
	\begin{proof}
		Suppose that \eqref{eq:stU} holds.
		Let \(W \subset \CP^n\) be a proper non-empty subspace. We want to show that Equation \eqref{eq:stpointW} holds.
		If \(W \cap \mP = \emptyset\) then Equation \eqref{eq:stpointW} is trivially satisfied because its left hand side is equal to \(0\) and its right hand side is \(> 0\). Thus, we can assume that \(W \cap \mP\) is non-empty.
		Let \(U\) be the subspace spanned by the points in \(W \cap \mP\). Since  \(W \cap \mP\) is non-empty, \(U \neq \emptyset\). Clearly, \(U \subset W\) and, since \(W\) is a proper subspace of \(\CP^n\), \(U\) too. Therefore, \(U \in \mU\). By Equation \eqref{eq:stU}\,, we have
		\[
		\begin{aligned}
			\sum_{i \, | \, p_i \in W} a_i &= \sum_{i \, | \, p_i \in U} a_i \\
			&< \frac{\dim U + 1}{n+1} \cdot \sum_{i=1}^m a_i \\
			&\leq \frac{\dim W + 1}{n+1} \cdot \sum_{i=1}^m a_i \,.
		\end{aligned}
		\]
		This proves Equation \eqref{eq:stpointW} and finishes the proof of the lemma.
	\end{proof}
\end{comment}

Next, we relate stability of weighted configurations of points to our previous notion of stable weighted arrangements.

\begin{notation}\label{not:ann1}
	If \(L \subset \CP^n\) is a linear subspace, we write \(L^{\perp} \subset (\CP^n)^*\) for the (projectivization) of the annihilator of \(L\). In particular,
	\begin{equation}\label{eq:codimL}
		\codim L = \dim L^{\perp} + 1 \,.	
	\end{equation}
\end{notation}
 
Let \(\mH \subset \CP^n\) be a hyperplane arrangement.
For each \(H \in \mH\), there is a unique up to scalar multiplication linear function \(\ell_H \in (\C^{n+1})^*\) such that \(H = \{\ell_{H}=0\}\). The annihilator \(H^{\perp} \) is the uniquely defined point in \((\CP^n)^* = \P((\C^{n+1})^*)\) given by \(H^{\perp} = [\ell_H]\). This way, the  arrangement \(\mH \subset \CP^n\) corresponds to a configuration of points \(\mP \subset (\CP^n)^*\) with
\[
\mP = \mH^{\perp} = \{H^{\perp} \,|\, H \in \mH\} \,. 
\]

Recall that \(\mL\) is the poset of non-empty and proper subspaces \(L \subset \CP^n\) obtained by intersecting members \(H \in \mH\), equipped with the partial order given by reverse inclusion.
Similarly, \(\mU\) is the poset of non-empty and proper subspaces \(U \subset (\CP^n)^*\) obtained as sums of points \(p \in \mP\), equipped with the partial order given by inclusion. The correspondence
\[
L \mapsto U = L^{\perp}
\]
defines an isomorphism of posets between \(\mL\) and \(\mU\).

Let \(\ba \in \R^{\mH}\) be a weight vector with components \(a_H > 0\) for all \(H \in \mH\).
Let \(m = |\mH|\) and label the hyperplanes \(\mH = \{H_1, \ldots, H_m\}\). Write \(\mP = \{p_1, \ldots, p_n\}\) with
\(p_i = H_i^{\perp}\) and 
\(\ba = (a_1, \ldots, a_m)\) with \(a_i = a_{H_i}\). 

\begin{lemma}\label{lem:stabHP}
	The weighted arrangement \((\mH, \ba)\) is stable (as in Definition \ref{def:stable}) if and only if the weighted configuration of points \((\mP, \ba)\) is stable (as in Definition \ref{def:stablepoints}).
\end{lemma}

\begin{proof}
	By Lemma \ref{lem:stablearr}\,, the weighted arrangement \((\mH, \ba)\) is stable if and only if,
	for every \(L \in \mL\) we have \(g_L(\ba) > 0\), where
	\[
	g_L(\ba) = \frac{\codim L}{n+1} \cdot \sum_{i=1}^m a_i - \sum_{i \,|\, L \subset H_i} a_i  \,.
	\]
	
	By Lemma \ref{lem:stU}, the weighted configuration \((\mP, \ba)\) is stable if and only if,
	for every \(U \in \mU\) we have \(f_U(\ba) > 0\), where
	\begin{equation*}
		f_U(\ba) = \frac{\dim U + 1}{n+1} \cdot \sum_{i=1}^m a_i - \sum_{i \,|\, p_i \in U} a_i \,.
	\end{equation*}
	
	If \(U = L^{\perp}\) for some \(L \in \mL\), then \eqref{eq:codimL} asserts that \(\dim U + 1 = \codim L\).
	On the other hand, \(L \subset H_i\) if and only if \(p_i  \in U\). Therefore,
	\[
	g_L(\ba) = f_{U}(\ba) \,.
	\]
	Using this, the lemma follows from the fact that \(L \mapsto U = L^{\perp}\) is a bijection between the intersection poset \(\mL\) of \(\mH\) and the sum poset \(\mU\) of \(\mP\).
\end{proof}

%Using Lemma \ref{lem:stabHP}, we can recast the results of Sections \ref{sec:git} and \ref{app:essirr} in terms of hyperplane arrangements. In particular, we have the following. 

\subsubsection{Geometric Invariant Theory}\label{sec:git} 
Given a reductive group \(G\) acting on a projective variety \(X\) together with a \(G\)-linearised ample line bundle \(L\), there is a standard notion of
a point \(p \in X\) being GIT stable.
In our case of interest, \(X\) is embedded in \(\CP^k\) (for some \(k \gg 1\)) via  sections of \(L\) and \(G\) acts on \(X\) through linear transformations of \(\C^{k+1}\). In this case, a point \(p \in X\) is GIT stable if the isotropy subgroup \(G_p\) is finite and the orbit \(G \cdot \tilde{p} \subset \C^{k+1}\) is closed, where \(\tp \in \C^{k+1}\) is any point that projects down to \(p \in \CP^k\).

To make the connection, let \((\mP, \ba)\) be a weighted configuration of \(m\) distinct points in \(\CP^n\) and assume that all the weights \(a_i\) are positive integers.
Let \(X\) be the product of \(m\)-copies of \(\CP^n\) and let 
\(L_{\ba}\) be the polarization on \(X\) given by
\[
L_{\ba} = \bigotimes_{j=1}^m \pr_j^*\left(\mO_{\P^n}(a_j)\right) \,,
\]
where \(X \xrightarrow{\pr_j} \CP^n\) is the projection to the \(j\)-factor.

\begin{lemma}
	The weighted configuration of points \((\mP, \ba)\) is stable if and only if the point \(p = (p_1, \ldots, p_n) \in X\) is GIT stable for the diagonal action of \(G = SL(n+1, \C)\) on \(X\)  linearised by the polarization \(L_{\ba}\). 	
\end{lemma}

\begin{proof}
	This is follows from \cite[Theorem 11.2]{dolgachev} and Definition \ref{def:stablepoints}\,.
\end{proof}

The link to differential geometry is provided by the Kempf-Ness theorem. To avoid distinctions between stable and polystable points, which have continuous isotropy subgroups, we assume that
\(G_p = \{\textup{identity}\}\). Equivalently, we assume that \(\mP\) is essential and irreducible (see Section \ref{app:essirr}).
There is a standard embedding of \(\CP^n \subset \mathfrak{su}(n+1)^*\), as a coadjoint orbit. Then \((\mP, \ba)\) is stable if and only if there is  \(F \in SL(n+1, \C)\) such that the centre of mass of the points \(F(p_i)\) with weights \(a_i\) in the Euclidean space \(\mathfrak{su}(n+1)^*\) is equal to zero. This interpretation of the stable condition remains valid for arbitrary positive real weights, see \cite[Example 6.3]{kapovich}.

\begin{example}
	If \((\mH, \ba)\) is a collection of points \(p_i \in \CP^1\) with weights \(a_i>0\). Then \((\mH, \ba)\) is stable \(\iff a_i < (1/2) \cdot \left(\sum_j a_j\right)\) for all \(i\);  equivalently:
	\begin{equation}\label{eq:pointsP1}
	\forall i: \,\,  a_i < \sum_{j \neq i} a_j \,.	
	\end{equation}
	In this case, the Kemp-Ness theorem becomes the familiar statement that a collection of \(3\) or more points on the unit sphere \(S^2 \subset \R^3\) with masses \(a_i\)  can be moved by an element of \(SL(2, \C)\) to have zero centre of mass if and only if \eqref{eq:pointsP1} is satisfied.
\end{example}

\subsubsection{Stability of pairs}

A weighted arrangement \((\mH, \ba)\) corresponds to a log pair \((\CP^n, \Delta)\) with 
\[
\Delta = \sum_{H \in \mH} a_H \cdot H \,.
\] 
From MMP, we recall the following.

\begin{definition}\label{def:logcan}
	The weighted arrangement \((\mH, \ba)\) is \emph{log canonical} if the following inequalities hold:
	\begin{equation}\label{eq:lc}
	\begin{gathered}
	\forall \, H \in \mH: \,\, 0 < a_H \leq 1 \,, \\ 
	\forall \, L \in \mL: \,\, \sum_{H | L \subset H} a_H \leq \codim L \,.
	\end{gathered}
	\end{equation}
\end{definition}

\begin{remark}
	Definition \ref{def:logcan} is broader than klt condition, in the sense that the strict inequalities \(<\) in \eqref{eq:klt} are relaxed to weak inequalities \(\leq\) in \eqref{eq:lc}.	
\end{remark}

Recall that we write \(s\) for the linear function on \(\R^{\mH}\) given by
\[
s = \sum_{H \in \mH} a_H \,.
\]

\begin{lemma}\label{lem:stpair}
	Let \((\mH, \ba)\) be a log canonical weighted arrangement.
	\begin{enumerate}[label=\textup{(\roman*)}]
		\item If \(s > n+1\) then \((\mH, \ba)\) is stable.
		\item If \(s=n+1\) then \((\mH, \ba)\) is stable if and only if \((\CP^n, \Delta)\) is klt.
		\item If \(s < n+1\) then \((\mH, \ba)\) is stable if and only if \((\CP^n, \Delta)\) is \(K\)-stable.
	\end{enumerate}
\end{lemma}

\begin{proof}
	(i) If \(s > n+1\) then the rescaling factor \(\lambda = (n+1) / s\) in Equation \eqref{eq:stable} is \(< 1\). Therefore, the weights \(a'_H = \lambda \cdot a_H\) of the rescaled arrangement \((\mH, \ba')\) satisfy \(a'_H < a_H\) for all \(H\). This implies that the weak inequalities \(\leq\) in \eqref{eq:lc} become strict inequalities \(<\); showing that \((\mH, \ba')\) is klt.
	(ii) This is just by Definition \ref{def:stable}.
	(iii) This is an immediate consequence of item (2) in \cite[Theorem 1.5]{fujita}.
\end{proof}

A weighted arrangement \((\mH, \ba)\) is of \emph{general type} if \(s > n+1\). It is \emph{Fano} if \(s < n+1\).

\begin{remark}
	Log canonical weighted arrangements of general type are studied in \cite{alexeev} under the acronym `shas', for `stable hyperplane arrangements'.
\end{remark}

\begin{remark}
	A basic result of Odaka \cite{odaka} asserts that if \((X, L)\) is a polarized complex projective manifold with \(c_1(X) \leq 0\) (with polarization \(L=K_X\) if \(c_1(X) < 0\)) then \((X, L)\) is \(K\)-stable. Actually, Odaka allows \(X\) to have (semi)-log-canonical singularities if \(c_1(X)<0\) and log-terminal singularities if \(c_1(X)=0\). A logarithmic version of Odaka's result would suggest that if the weighted arrangement \((\mH, \ba)\) is either log canonical of general type or klt CY, then \((\CP^n, \Delta)\) is \(K\)-stable.
\end{remark}

\section{Hirzebruch arrangements and matroids}\label{sec:hirzebruch}

In Section \ref{sec:critQ} we calculate the partial derivatives \(\p Q / \p a_H\) of the quadratic form and express them in terms of total sums
of weights on induced arrangements \(\mH^H\).\newpar

\noindent In Section \ref{sec:hirz}\,, we analyse the symmetric case where all the weights of the arrangements are equal. This leads us to define a class of arrangements, which we call \emph{Hirzebruch arrangements}, for which the main diagonal of \(\R^{\mH}\) is contained in the kernel of the quadratic form \(Q\). \newpar

\noindent In Section \ref{sec:compref}\,, we show that if \(\mH \subset \CP^n\) is a complex reflection arrangement defined by an irreducible unitary complex reflection group \(G \subset U(n+1)\), then \(\mH\) is Hirzebruch (Proposition \ref{prop:reflhir}) and its quadratic form is negative semidefinite. \newpar
%\todo[inline]{I think it would be good to mention here Theorem 7.29 as well - that the from is negative semi-definite for reflection arrangmenets?}

\noindent In Section \ref{sec:matroid}\,, we reformulate our results in the more general context of \emph{matroids}.

\subsection{Critical points of \texorpdfstring{\(Q\)}{Q} and induced arrangements}\label{sec:critQ}

We calculate the partial derivatives of \(Q\). Lemma \ref{lem:pq} expresses \(\p Q / \p a_H\) in terms of the sum of the induced weights on the induced arrangement \(\mH^H\), as defined next.

We use the notation for deletion and restriction triples as in \cite[Definition 1.14]{orlikterao}. 
Fix \(H \in \mH\) and write \((\mH, \mH', \mH'')\) where \(\mH' = \mH \setminus \{H\}\) and \(\mH'' = \mH^H\) is the induced arrangement
obtained by intersecting \(H\) with members of \(\mH'\)
as defined in Section \ref{sec:habasicdef}\,.

An element \(H'' \in \mH''\) is a codimension \(2\) subspace of \(\CP^n\). In particular, we can distinguish the elements of \(\mH''\) into two different types depending on whether \(H'' \in \mLi^{n-2}\) or not. Recall that \(H'' \in \mLi^{n-2}\) if and only if its multiplicity \(m_{H''} = |\mH_{H''}|\) is \( \geq 3\), where \(\mH_{H''}\) is the localization of \(\mH\) at \(H''\).
If \(H'' \in \mLi^{n-2}\), then \(a_{H''}\) is defined as one-half of the sum of all the weights of the hyperplanes in \(\mH\) that contain \(H''\).

\begin{definition}\label{def:inducedweights}
	The weights of the induced arrangement \(\mH''\) are defined as follows:
	\begin{equation}
		a''_{H''} = \begin{cases}
			a_{H''} & \text{if } H'' \in \mLi\,, \\
			a_{H'} & \text{if } \mH_{H''} = \{H, H'\} \,.
		\end{cases}
	\end{equation}
	The sum of the induced weights on \(\mH''\) is denoted by \(s_H\). More precisely,
	\begin{equation}\label{eq:sh}
		s_H = \sum_{H' | H' \pitchfork H} a_{H'} \,\, + \sum_{L \in \mLi^{n-2} | L \subset H} a_L \,.
	\end{equation}
\end{definition}

\begin{lemma}
	For every \(H \in \mH\) the following identity holds
	\begin{equation}\label{eq:ssh}
	B_H \cdot a_H = 
	s_H \,-\, s  \,\,+  \sum_{L \in \mLi^{n-2} | L \subset H} a_L  \,.
	\end{equation}
\end{lemma}

\begin{proof}
	Since every hyperplane \(H' \in \mH \setminus \{H\}\) intersects \(H\) either at \(L \in \mLi^{n-2}\) or \(H' \pitchfork H\), we have
	\[
	\begin{aligned}
	s &= a_H \,+\, \sum_{L \in \mLi^{n-2} | L \subset H} (2a_L - a_H) \,+\, \sum_{H' | H' \pitchfork H} a_{H'}    \\
	 &= -  B_H \cdot a_H  \,+\, s_H \,\,+ \sum_{L \in \mLi^{n-2} | L \subset H} a_L \,.
	\end{aligned}
	\] 
	and \eqref{eq:sh} follows.
\end{proof}

\begin{lemma}\label{lem:pq}
	The partial derivatives of the quadratic form are given by
	\begin{equation}\label{eq:pq}
		\frac{\p Q}{\p a_H} = 4n \cdot s - 4(n+1) \cdot s_H  \,.
	\end{equation}
\end{lemma}

\begin{proof}
	By Equation \eqref{eq:aLcod2} we have
	\[
	\frac{\p a_L}{\p a_H} = \begin{cases}
		1/2 &\text{ if } H \supset L \,, \\
		0 & \text{ otherwise.}
	\end{cases}
	\]
	Taking the partial derivative of \(Q\) with respect to \(a_H\) in Equation \eqref{eq:Q2} gives us
	\[
	\frac{1}{4(n+1)}\frac{\p Q}{\p a_H} =  \sum_{L \in \mLi^{n-2} | L \subset H} a_L \,-\, B_H \cdot a_H \,-\, \frac{s}{n+1} .
	\]
	Using Equation \eqref{eq:ssh} we obtain
	\[
	\frac{1}{4(n+1)}\frac{\p Q}{\p a_H} = \left(1 - \frac{1}{n+1}\right) s - s_H
	\]
	and \eqref{eq:pq} follows.
\end{proof}

A critical point \(\ba \in \R^H\) of \(Q\) is a point where all partial derivatives \(\p Q / \p a_H\) vanish. The set of all critical points of \(Q\) is the kernel of the quadratic form. Equation \eqref{eq:pq} gives us the following set of defining linear equations for the kernel of \(Q\).

\begin{corollary}\label{cor:critical}
		A weight vector \(\ba \in \R^{\mH}\) is a critical point of \(Q\) if and only if
	\begin{equation}
		\forall \, H \in \mH: \,\, s_H = \frac{n}{n+1} s \,.
	\end{equation}
\end{corollary}

We use the above calculation of partial derivatives to give an alternative expression for the quadratic form \(Q\).

\begin{notation}
	We denote by \(\mLr^{n-2}\) the set of all reducible codimension \(2\) subspaces. Concretely, \(L \in \mLr^{n-2}\) if and only if its multiplicity  is \(m_L =2\).
	In this case, we write  \(\mH_L = \{H, H'\}\) with \(H, H' \in \mH\) and \(L = H \pitchfork H'\).
\end{notation}

\begin{corollary}\label{cor:altq}
	The quadratic form of \(\mH\) is equal to
	\begin{equation}\label{eq:Qalt}
		Q = 2n \cdot s^2 \,-\, 4(n+1) \cdot \sum_{L \in \mLr^{n-2}} a_H \cdot a_{H'}  \,-\, 4(n+1) \cdot \sum_{L \in \mLi^{n-2}} a_L^2  \,\,,
	\end{equation}
	where the second sum runs over all \(L \in \mLr^{n-2}\) with \(L = H \pitchfork H'\).
\end{corollary}

\begin{proof}
	Since \(Q\) is a homogeneous polynomial of degree \(2\), we have
	\begin{equation}\label{eq:pQ}
		2 Q = \sum_{H \in \mH} a_H \frac{\p Q}{\p a_H} \,.
	\end{equation}
	Replacing the values for \(\p Q /  \p a_H\) given by Lemma \ref{lem:pq} together with Equation \eqref{eq:sh} for \(s_H\) gives us
    \begin{equation}\label{eq:pfalt1}
    \begin{aligned}
    Q &= \sum_{H \in \mH} a_H \cdot \left( 2n \cdot s - 2(n+1) \cdot s_H \right) \\
    &= 2n \cdot s^2 \,-\, 2(n+1) \cdot \sum_{H \in \mH} \,\, \sum_{H' |\, H' \pitchfork H} a_H \cdot a_{H'} \\ 
    &-\, 2(n+1) \cdot \sum_{H \in \mH} \,\, \sum_{L \in \mLi^{n-2} \,|\, L \subset H} a_H \cdot a_{L} \,.
    \end{aligned}     
    \end{equation}
    Note that 
    \begin{equation}\label{eq:pfalt2}
       \sum_{H \in \mH} \,\, \sum_{H' |\, H' \pitchfork H} a_H \cdot a_{H'} = 2 \cdot  \sum_{L \in \mLr^{n-2}} a_H \cdot a_{H'}  
    \end{equation}
    and
  \begin{equation}\label{eq:pfalt3}
      \sum_{H \in \mH} \,\, \sum_{L \in \mLi^{n-2} \,|\, L \subset H} a_H \cdot a_{L} = 2 \cdot  \sum_{L \in \mLi^{n-2}} a_L^2 \,.
  \end{equation}
  Equation \eqref{eq:Qalt} follows from Equations \eqref{eq:pfalt1}, \eqref{eq:pfalt2}, and \eqref{eq:pfalt3}.
\end{proof}

\begin{comment}
	We provide an alternative formula for the quadratic form \(Q\) of an arrangement. We continue with the same notation of Section \ref{sec:quadform}.

	We use the next elementary identity that holds in great generality.
	
	\begin{lemma}
		For any arrangement \(\mH \subset \CP^n\),
		the following identity holds
		\begin{equation}\label{eq:squareidentity}
			s^2 = - \sum_{H \in \mH} B_H \cdot a_H^2 \,\, + \,\, 4 \cdot \sum_{L \in \mLi^{n-2}} a_L^2 \,\, + \,\,  2 \cdot \sum_{L \in \mLr^{n-2}} a_H \cdot a_{H'}
		\end{equation}
		where the last sum runs over all \(L \in \mLr^{n-2}\) with \(L = H \pitchfork H'\).
	\end{lemma}
	
	\begin{proof}
		Straightforward calculation.
	\end{proof}
	
	\begin{corollary}\label{cor:altq}
		The quadratic form of \(\mH\) is equal to
		\begin{equation}\label{eq:Qalt}
			Q = -\sum_{L \in \mLi^{n-2}} a_L^2 - \sum_{L \in \mLr^{n-2}} a_H \cdot a_{H'} + \frac{n}{2(n+1)} s^2 \,.
		\end{equation}
	\end{corollary}
	
	\begin{proof}
		Use Equation \eqref{eq:squareidentity} to calculate
		\[
		- \frac{1}{2} \sum_{H \in \mH} B_H \cdot a_H^2 
		=
		\frac{s^2}{2}
		- 2 \cdot \sum_{L \in \mLi^{n-2}} a_L^2 \,\, -  \sum_{L \in \mLr^{n-2}} a_H \cdot a_{H'} 
		\]
		and replace this result into \eqref{eq:Q} to obtain \eqref{eq:Qalt}.
	\end{proof}	
\end{comment}

\subsection{Hirzebruch arrangements}\label{sec:hirz}

Let \(\mH \subset \CP^n\) be an arrangement with \(N = |\mH|\) hyperplanes.
We consider the case for where all hyperplanes have the same weight.
Our main result is Theorem \ref{thm:hir}.

We begin with an elementary lemma that calculates the total sum of the induced weights on the induced arrangements \(\mH^H\).
Fix \(H \in \mH\) and let \(s_H\) be, as in Definition \ref{def:inducedweights}, given by
\begin{equation}\label{eq:sh2}
s_H = \sum_{H' | H' \pitchfork H} a_{H'} + \sum_{L \in \mLi^{n-2} | L \subset H} a_L \,.
\end{equation} 

\begin{lemma}\label{lem:sh}
	Suppose that all the weights of \((\mH, \ba)\) are equal, say \(a_H = a\) for all \(H\). Then, for every \(H \in \mH\), we have
	\begin{equation}\label{eq:sheq}
		2 s_H = \left(|\mH^H| + N - 1 \right) \cdot a \,.
	\end{equation}
\end{lemma}

\begin{proof}
	We want to replace the values \(a_H=a\) in Equation \eqref{eq:sh2} but there is an observation to do first. The weight \(a_{H}\) of our fixed hyperplane \(H\) occurs only in the irreducible terms \(a_L\) but not in the reducible terms \(a_{H'}\). We remedy this asymmetry by using the identity \(a = (a+a)/2\). Indeed, since all the weights \(a_H\) are equal, we can replace each term \(a_{H'}\) with \((a_{H'} + a_{H})/2\).
	Since every hyperplane in \(\mH \setminus \{H\}\) intersects \(H\) at a unique codimension \(2\) subspace (see Figure \ref{fig:redirred}) we obtain 
	\begin{equation*}
	2 s_H =  \underbrace{(|\mH|-1) \cdot a}_{\text{ contribution of } \mH \setminus \{H\}}  +  \underbrace{|\mH^H| \cdot a}_{\text{ contribution of } H} 
	\end{equation*}
	from which the statement follows.
\end{proof}

\begin{figure}[h]
	\centering
	\scalebox{.7}{
		\begin{tikzpicture}
		\draw[thick, blue] (-10,0) to (5,0);
		\draw (-5,3) to (-5,-3);
		\draw (0,3) to (3,-3);
		\draw (0,-3) to (3,3);
		
		\filldraw[red] (-5, 0) circle (3pt);
		\filldraw[green!80!black] (1.5, 0) circle (3pt);
		
		\node[scale=1.4, blue] at (-9,.4){\(H\)};
		\node[scale=1.4] at (-5.4,2.4){\(H'\)};
		\node[scale=1.4] at (1.5,-.5){\(L\)};
		\node[scale=1.4] at (2.2,2.4){\(\tH\)};
		
		\node[scale=1.3] at (-2.8,-.6){\(2a_{H'} = a_{H'} + a_H\)};
		\node[scale=1.3] at (3.9,-.9){\(2a_L = \sum\limits_{\tH | L \subset \tH} a_{\tH}\)};
		
		\end{tikzpicture}		
	}
	\caption{Reducible (red) and irreducible (green) intersections.}
	\label{fig:redirred}
\end{figure}
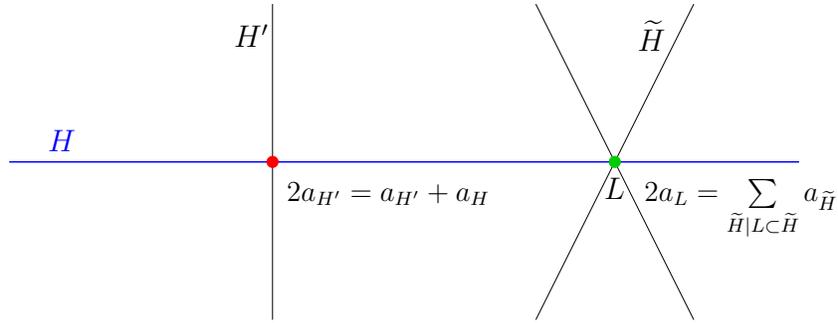

Recall that the multiplicity \(m_L\) of an element \(L \in \mL\) is the number of hyperplanes \(H \in \mH\) that contain \(L\). We denote by \(\mL^{n-2}\) the set of all codimension \(2\) subspaces of \(\CP^n\) obtained as intersection of elements \(H \in \mH\), so for any \(L \in \mL^{n-2}\) we have \(m_L \geq 2\) and \(m_L \geq 3\) if and only if \(L\) is irreducible.

\begin{lemma}
	The following identity holds:
	\begin{equation}\label{eq:summul2}
	\sum_{L \in \mL^{n-2}}m_L^2 = N^2 - N + \sum_{L \in \mL^{n-2}}m_L	\,.
	\end{equation}
\end{lemma}

\begin{proof}
	Since every pair of hyperplanes intersect at a codimension two subspace,
	\begin{equation}\label{eq:identityinter}
	\sum_{L \in \mL^{n-2}} m_L(m_L-1) = N(N-1) \,.
	\end{equation}
	Equation \eqref{eq:summul2} follows from Equation \eqref{eq:identityinter} by rearranging the terms.
\end{proof}

\begin{notation}
	Let \(\mathbf{1} \in \R^{\mH}\) be the vector with all components equal to \(1\).
\end{notation}

\begin{lemma}\label{lem:Qof1}
	The value of the quadratic form at \(\mathbf{1}\) is given by
	\begin{equation}\label{eq:Qof1}
		Q(\mathbf{1}) = (n-1) N^2 + (n+1) N - (n+1) \cdot  \sum_{L \in \mL^{n-2}}m_L \,.
	\end{equation}
\end{lemma}

\begin{proof}
	Since \(s(\mathbf{1}) = N\) and \(2a_L = m_L\) for all \(L \in \mLi^{n-2}\), Equation \eqref{eq:Qalt} gives us
	\[
	Q(\mathbf{1}) = 2n \cdot N^2 \, - \, (n+1) \cdot \sum_{L \in \mL^{n-2}} m_L^2 \,.
	\]
	Using Equation \eqref{eq:summul2} for the sum of squared multiplicities gives Equation \eqref{eq:Qof1}.
\end{proof}

\begin{lemma}\label{lem:1incone}
	The vector \(\mathbf{1}\) belongs to the stable cone \(C^{\circ}\) of \(\mH\) if and only if
	\begin{equation}
	\forall L \in \mL: \,\,	m_L < \codim L \cdot \frac{N}{n+1} \,.
	\end{equation}
\end{lemma}

\begin{proof}
	Immediate from item (ii) of Lemma \ref{lem:stablearr}\,.
\end{proof}

\begin{lemma}\label{lem:1ker}
	The vector \(\mathbf{1}\) belongs to the kernel of \(Q\) if and only if
	\begin{equation}\label{eq:1ker}
		\forall H \in \mH: \,\,	
		|\mH^H| =  \left( 1 - \frac{2}{n+1} \right) N + 1	\,.
	\end{equation}
\end{lemma}

\begin{proof}
	By Corollary \ref{cor:critical}\,, the vector \(\mathbf{1}\) belongs to the kernel of \(Q\) if and only if for all \(H \in \mH\) the following equation is satisfied:
	\[
	2 s_H(\mathbf{1}) = \frac{2n}{n+1} s(\mathbf{1}) \,.
	\]
	By Equation \eqref{eq:sheq} (and using that \(s(\mathbf{1}) = N\)) this condition is equivalent to
	\[
	|\mH^H| + N - 1 = \frac{2n}{n+1} \cdot N
	\]
	which rearranges to \eqref{eq:1ker}. Alternatively, the result also follows from Lemma \ref{lem:sumqij}\,.
\end{proof}

We shall need the following elementary result, whose proof we omit.

\begin{lemma}\label{lem:kernelquadratic}
	Let \(Q:\R^N \to \R\) be a real quadratic form. Suppose that \(Q \leq 0\) on an open set \(U \subset \R^N\). Then the following holds: 
	\begin{enumerate}[label=\textup{(\roman*)}]
		\item \(Q(\bx) = 0\) for some \(\bx \in U\) if and only if \(\bx\) is in the kernel of \(Q\)\,;
		\item if \(Q(\bx) = 0\) for some \(\bx \in U\), then \(Q \leq 0\) in all of \(\R^N\).
	\end{enumerate}
\end{lemma}

The main result of this section is the following extension of \cite[Corollary 7.8]{panov} to higher dimensions.

\begin{theorem}\label{thm:hir}
	Let \(\mH \subset \CP^n\) be an arrangement with \(N = |\mH| > n+1\) hyperplanes. Suppose that for all \(L \in \mL\) we have
	\begin{equation}\label{eq:sthir}
	m_L < \codim L \cdot \frac{N}{n+1} \,.
	\end{equation}
	Then
	\begin{equation}\label{eq:hirineq}
	\sum_{L \in \mL^{n-2}} m_L \geq \left( 1 - \frac{2}{n+1} \right) N^2 + N
	\end{equation}
	and equality holds if and only if 
	\begin{equation}\label{eq:hir}
	\forall H \in \mH: \,\,	
	|\mH^H| =  \left( 1 - \frac{2}{n+1} \right) N + 1	\,.
	\end{equation}
	Moreover, if \eqref{eq:hir} is satisfied, then the quadratic form \(Q\) of the arrangement \(\mH\) is negative semidefinite in all of \(\R^N\).
\end{theorem}

\begin{proof}
	Consider the vector \(\mathbf{1} \in \R^{\mH}\). The assumption \eqref{eq:sthir} together with Lemma \ref{lem:1incone} imply that \(\mathbf{1} \in C^{\circ}\). By Theorem \ref{thm:qcone}\,, \(Q(\mathbf{1}) \leq 0\). Using Equation \eqref{eq:Qof1} to evaluate \(Q\) at \(\mathbf{1}\), we obtain
	\[
	\frac{n-1}{n+1} \cdot N^2 \,+\,  N \,-\,  \sum_{L \in \mL^{n-2}}m_L  \leq 0 \,,
	\]
	which rearranges to \eqref{eq:hirineq}. Equality holds in \eqref{eq:hirineq} if and only if \(Q(\mathbf{1}) = 0\). By Lemma \ref{lem:kernelquadratic} (i) this is equivalent to \(\mathbf{1}\) being in the kernel of \(Q\). By Lemma \ref{lem:1ker} \(\mathbf{1}\) is in the kernel of \(Q\) if and only if Equation \eqref{eq:hir} is satisfied. By Lemma \ref{lem:kernelquadratic} (ii), if \eqref{eq:hir} is satisfied then \(Q\) is negative semidefinite.
\end{proof}

\begin{comment}
\begin{proof}
	Take weights \(a_H = a\) for all \(H\) with
	\begin{equation}
		a = \frac{n+1}{|\mH|} .
	\end{equation}
	Since \(|\mH| > n+1\), the weight \(a\) belongs to the open interval \((0,1)\).
	Equation \eqref{eq:sthir} guarantees that the CY arrangement \((\mH, \ba)\) is klt. 
	In particular, \(\ba\) belongs to the stable cone \(C^{\circ} \subset \R^{\mH}\).
	It follows from Theorem \ref{thm:klt} and Corollary \ref{cor:altq} that
	\[
	Q(\ba) = - (a^2/4) \cdot \sum_{L \in \mL^{n-2}} m_L^2 \,\, + \,\, \frac{n(n+1)}{2} \leq 0 .
	\]
	Replacing \(a = (n+1) / |\mH|\) we obtain
	\begin{equation}\label{eq:sum2}
		\sum_{L \in \mL^{n-2}} m_L^2 \geq  \frac{2n}{n+1} |\mH|^2 .	
	\end{equation}
	Equation \eqref{eq:hirineq} follows from Equations \eqref{eq:summul2} and \eqref{eq:sum2}.
	
	If equality holds in Equation \eqref{eq:hirineq} then
	\(Q(\ba) = 0\). Since \(Q \leq 0\) in a neighbourhood of \(\ba\) (because the stable cone \(C^{\circ}\) is an open set) we conclude that
	\(\ba\) is a critical point of \(Q\). By Lemma \ref{lem:pq}, for all \(H \in \mH\) we have
	\begin{equation}\label{eq:shn}
		s_H = n .
	\end{equation}
	Equation \eqref{eq:hir} follows Lemma \ref{lem:sh} and Equation \eqref{eq:shn}.
\end{proof}	
\end{comment}

We note that Theorem \ref{thm:hir} does not necessarily hold for arrangements defined over fields other than \(\C\), as the next example illustrates.

\begin{example}\label{ex:finiteps}
	Let \(p\) be a prime number and let \(\mathbb{F}_p\) be the finite field of order \(p\). For \(n \geq 2\), let
	\(\mH\) be the arrangement  consisting of all hyperplanes in \(\PG(n, p) = \P(\mathbb{F}_p^{n+1})\). The number \(N = |\mH|\) of hyperplanes is equal to
	\[
	\left| \big(\mathbb{F}_p^{n+1} \setminus \{0\} \big) \, \big/ \, \mathbb{F}_p^* \, \right| = \frac{p^{n+1}-1}{p-1} \,.
	\]
	
	The next two claims show that the arrangement \(\mH\) violates Theorem \ref{thm:hir}. Concretely, Claim 1 proves that \eqref{eq:sthir} holds and Claim 2 proves that \eqref{eq:hirineq} doesn't.
	
	\begin{itemize}
		\item \emph{Claim 1:} the vector \(\mathbf{1}\) belongs to the stable cone \(C^{\circ}\), i.e., if \(L\) is a non-empty and proper linear subspace, then
		\begin{equation}\label{eq:mulLp}
		\frac{m_L}{\codim L} < \frac{N}{n+1} \,.
		\end{equation}
		
		\item \emph{Claim 2:} 
		the sum of multiplicities of codimension two subspaces satisfies
		\begin{equation}\label{eq:sumLp}
		\sum_{L \in \mL^{n-2}} m_L < \left( 1 - \frac{2}{n+1} \right) N^2 + N \,.
		\end{equation}
	\end{itemize}

	\emph{Proof of Claim 1.} 
	The multiplicity \(m_L\) of a linear subspace \(L\) is the number of hyperplanes in \(\PG(n, p)\) that contain it and it is given by
	\(m_L = |\PG(c-1, p)|\) where \(c\) is the codimension of \(L\). Note that \(1 \leq c \leq n\), as \(L\) is proper and non-empty. Therefore, inequality \eqref{eq:mulLp} is equivalent to
	\begin{equation*}
		\frac{p^c - 1}{c} < \frac{p^{n+1} - 1}{n+1} \,,
	\end{equation*}
	which follows  by noticing that the ratio \(p^c / c\) is strictly increasing with \(c\).

	\emph{Proof of Claim 2.}
	By Lemma \ref{lem:Qof1}\,, Equation \eqref{eq:sumLp} is equivalent to \(Q(\mathbf{1}) > 0\). We prove this by showing that \(\mathbf{1}\) is an eigenvector of \(Q\) with positive eigenvalue.
	Since every \(H \in \mH\) intersects the other hyperplanes of \(\PG(n, p)\) along the same number \(t\) of codimension \(2\) subspaces, namely \(t = |\PG(n-1, p)|\),
	Lemma \ref{lem:sumqij} implies that \(Q \cdot \mathbf{1} = \lambda \cdot \mathbf{1}\) with
	\[
	\lambda = (n-1) \cdot N - (n+1) \cdot (t-1) \,.
	\]
	Replacing the values for \(N = |\PG(n, p)|\) and \(t = |\PG(n-1, p)|\) we have
	\[
	(p-1) \lambda = (n-1) \cdot (p^{n+1} - 1) - (n+1) \cdot (p^n - p)
	\]
	from which is straightforward to check that \(\lambda > 0\) for all \((n, p)\) with \(n \geq 2\) and \(p \geq 2\).
\end{example}

Next, we give a name for the complex arrangements that saturate the inequality \eqref{eq:hirineq}. 
Note that,	by Corollary \ref{cor:hir2irred}\,, if \(\mH\) is as in Theorem \ref{thm:hir} then \(\mH\) is essential and irreducible.
To include product arrangements, we relax \eqref{eq:sthir} to a non-strict inequality.

\begin{definition}
	Let \(\mH \subset \CP^n\) be an essential arrangement. We say that \(\mH\) is Hirzebruch if the following two conditions are satisfied:
	\begin{enumerate}[label=\textup{(H\arabic*)}]
			\item \label{item:h1} for every subspace \(L \in \mL\) we have
		\begin{equation}\label{eq:hir2}
		\frac{m_L}{\codim L} \leq \frac{|\mH|}{n+1} 
		\end{equation}
		where \(m_L\) is the multiplicity of \(L\);
		
		\item \label{item:h2} every hyperplane \(H \in \mH\) intersects the other hyperplanes in \(\mH \setminus \{H\}\) along
		\begin{equation}\label{eq:hir1}
			 \left( 1 - \frac{2}{n+1} \right) |\mH| + 1
		\end{equation}
		codimension \(2\) subspaces.
	\end{enumerate}
\end{definition}

%Condition \ref{item:h2} is equivalent to requiring that the quadratic form \(Q\) of \(\mH\) is degenerate and that the main diagonal \(\R \cdot (1, \ldots, 1) \subset \R^{\mH}\) belongs to its kernel. Condition \ref{item:h1} means that the CY weighted arrangement obtained by setting all weights equal to \((n+1)/|\mH|\) is log canonical; or equivalently the Fano arrangement \((\mH, \ba)\) with \(\ba = a \cdot (1, \ldots, 1)\) is K-semistable for all  \(0 < a < (n+1)/|\mH|\), see \cite{fujita}.

\begin{example}\label{ex:coordhyp}
	The arrangement of \(n+1\) coordinate hyperplanes in \(\CP^n\) is Hirzebruch. 
\end{example}

If \(n=1\) then any finite collection of points \(\mH \subset \CP^1\) with \(|\mH| \geq 2\) is Hirzebruch. However, if \(n \geq 2\) the Hirzebruch condition is much more rigid. 
A line arrangement \(\mH \subset \CP^2\) is Hirzebruch if and only if \(|\mH|=3k\) for some positive integer \(k\) and every line in \(\mH\) intersects the other lines of the arrangement at \(k+1\) points, all known examples are listed in \cite{hirzebruch2}\,. In dimension \(3\) we have the following characterization.

\begin{lemma}\label{lem:hirdim3}
	Let \(\mH \subset \CP^3\) be an essential arrangement of \(2k\) planes with \(k \geq 2\) such that every plane in \(\mH\) intersects the others along \(k +1\) lines. Then \(\mH\) is Hirzebruch.
\end{lemma}

\begin{proof}
	We have to prove that \(\mH\) satisfies \ref{item:h1}\,. This amounts to show
	two things: (i) for every point \(x \in \mL^0\) we have \(m_x \leq (3/2)k\); (ii) for every line \(L \in \mL^1\) we have \(m_L \leq k\).
	
	(i) Let \(x \in \mL^0\). Since \(\mH\) is essential, we can take a plane \(H_0\) that does not go through \(x\). Every pair of planes that contain \(x\) must meet \(H_0\) at different lines, hence
	\begin{equation*}
	m_x \leq |\mH^{H_0}| =  k + 1.
	\end{equation*}
	Since \(k \geq 2\), we have \(k + 1 \leq k + (k/2)\)
	and therefore \(m_x \leq (3/2)k\).
	
	Let \(L \in \mL^1\) and take a plane \(H \in \mH\) that contains \(L\). Since all the planes in \(\mH_L\) cut \(H\) along the same line \(L\) whereas the elements in \(\mH \setminus \mH_L\) cut \(H\) in at most \(|\mH| - m_L\) lines, we have
	\begin{equation*}
	k = |\mH^{H}| - 1  \leq |\mH| - m_L
	\end{equation*}
	from which we get \(m_L \leq k\).
\end{proof}

\begin{remark}
	In Lemma \ref{lem:hirdim3} we require that \(\mH\) is essential. If \(\mH \subset \CP^3\) is an arrangement of \(4 =2k\) planes that intersect at a common point and whose pairwise intersection are all distinct, then every member of \(\mH\) intersects the others along \(3 = k+1\) lines but \(\mH\) is not Hirzebruch.
\end{remark}

The next example shows that the class of Hirzebruch arrangements is not closed under restriction.

\begin{example}\label{ex:sweklines}
	Let \(L\) and \(L'\) be two skew lines in \(\CP^3\) and let \(\mH\) be the arrangement of \(2k\) planes with \(k\) intersecting along \(L\) and \(k\) intersecting along \(L'\). Then \(\mH\) is a product Hirzebruch arrangement, it is a particular case of Example \ref{ex:prodhir}.
	
	Let \(H\) be one of the planes containing \(L\) and let \(x\) be the intersection point of \(H\) and \(L'\). The induced arrangement \(\mH^H\) is the near-pencil consisting of \(k\) concurrent lines meeting at \(x\) together with the extra line \(L \subset H\). The arrangement \(\mH^H\) is \emph{not} Hirzebruch if \(k \geq 3\).
\end{example}

\begin{example}\label{ex:prodhir}
	Suppose that \(\mH_1 \subset \CP^{n_1}\) and \(\mH_2 \subset \CP^{n_2}\) are Hirzebruch. Then the product arrangement \(\mH_1 \times \mH_2 \subset \CP^{n_1 + n_2 + 1}\) is Hirzebruch if and only if
	\[
	\frac{|\mH_1|}{n_1 + 1} = \frac{|\mH_2|}{n_2+1} \,.
	\]
\end{example}

\subsection{Complex reflection arrangements}\label{sec:compref}

We recall the basic definitions on complex reflection arrangements, we follow \cite{orlikterao}. 
A linear map \(f \in U(n+1)\) is a \emph{complex reflection} if it has finite order and its fixed point set is a hyperplane \(H_f \subset \C^{n+1}\).
We call \(H_f\) the \emph{reflecting hyperplane} of \(f\).
A finite subgroup \(G \subset U(n+1)\) is a \emph{complex reflection group} if it is generated by complex reflections. 
The collection \(\mH \subset \CP^n\) of reflecting hyperplanes of \(G\)
is called the \emph{reflection arrangement} of \(G\). 
The group \(G\) is \emph{irreducible} if the only \(G\)-invariant subspaces are \(\{0\}\) and \(\C^{n+1}\). 

For the rest of this section we let \(\mH \subset \CP^n\) be a complex reflection arrangement of an irreducible complex reflection group \(G \subset U(n+1)\). 
Let \(N = |\mH|\). 

Let \(\langle v, w \rangle = \sum_i v_i \bar{w}_i\) be the usual Hermitian inner product on \(\C^{n+1}\).	
For each \(H \in \mH\) choose a unit vector \(e_H \in \C^{n+1}\) orthogonal  to the corresponding linear hyperplane. 
Write \(P_{H^{\perp}}\) for the orthogonal projection to \(H^{\perp}\), given by
\[
P_{H^{\perp}} (v) = \langle v, e_H \rangle \cdot e_H .
\]

\begin{lemma}[{\cite[Proposition 6.93]{orlikterao}}]\label{lem:schur}
	For every \(v \in \C^{n+1}\) we have
	\begin{equation}\label{eq:projH}
	\sum_{H \in \mH} P_{H^{\perp}}(v) = a_0 \cdot v \,,
	\end{equation}
	where \(a_0 = N / (n+1) \)\,.
\end{lemma}

\begin{proof}
	This is essentially a consequence of Schur's lemma.
	The left hand side of Equation \eqref{eq:projH} defines an invariant Hermitian form for the action of \(G\) on \(\C^{n+1}\). The assumption that the action is irreducible implies that \(\sum_{H} P_{H^{\perp}}\) is a scalar multiple of the identity.
	The value of \(a_0\) can be calculated by taking traces.
\end{proof}

Let \(\mL\) be the set of non-empty and proper subspaces obtained as intersection members of \(\mH\).
The multiplicity \(m_L\) of \(L \in \mL\) is the number of hyperplanes \(H \in \mH\) that contain \(L\), i.e., \(m_L = |\mH_L|\) where \(\mH_L\) is the localization of \(\mH\) at \(L\).
By slight abuse of notation, we also write \(L\) for the corresponding linear subspace of \(\C^{n+1}\), and \(L^{\perp} \subset \C^{n+1}\) for its orthogonal complement.
Recall that \(\mLi \subset \mL\) denotes the subset of irreducible subspaces. 
 
\begin{lemma}\label{lem:projL}
	Suppose that \(L \in \mLi\). Then, for all \(v \in L^{\perp}\), we have
	\begin{equation}\label{eq:projL}
	\sum_{H \in \mH_L} P_{H^{\perp}}(v) = a_L \cdot v \,,
	\end{equation}	
	where \(a_L = m_L / \codim L\).
\end{lemma}

\begin{proof}
	Consider the subgroup \(G_L \subset G\) made of elements that fix \(L\). By \cite[Theorem 6.25]{orlikterao}, the group \(G_L\) is generated by the reflections in \(G\) whose reflecting hyperplanes belong to the localized arrangement \(\mH_L\). 
	We identify \(L^{\perp} = \C^p\) with \(p = \codim L\) and \(G_L \subset U(p)\), as the group \(G_L\) acts faithfully on \(L^{\perp}\). Thus \(\mH_L\) is a complex reflection arrangement.
	Since, by assumption, the arrangement \(\mH_L\) is irreducible, the action of \(G_L\) on \(\C^p\) is also irreducible. 
	Equation \eqref{eq:projL} follows by Lemma \ref{lem:schur}\,. 
\end{proof}

\begin{remark}
	Lemma \ref{lem:projL} implies that
	\[
	\sum_{H \in \mH_L} P_{H^{\perp}} = a_L \cdot P_{L^\perp} \,,
	\]
	where \(P_{L^{\perp}}\) is the orthogonal projection to \(L^{\perp}\). This result is well known in the context of \emph{Dunkl connections}, see \cite[Lemma 2.13 and Example 2.5]{chl}\,.
\end{remark}

\begin{lemma}\label{lem:rarrh1}
	If \(L \in \mL\) then the following inequality holds:
	\begin{equation*}
		\frac{m_L}{\codim L} < \frac{N}{n+1} \,.
	\end{equation*}
\end{lemma}

\begin{proof}
	Write \(a_L = m_L / \codim L\) and \(a_0 = N / (n+1)\). We have to show that \(a_L < a_0\). 
	Suppose first that \(L \in \mLi\).
	
	Let \(S\) be the common intersection of all hyperplanes in \(\mH \setminus \mH_L\). Note that \(L^{\perp} \not\subset S\), otherwise \(L^{\perp}\) would be invariant by \(G\).
	Let \(v \in L^{\perp} \setminus S\) with \(|v|=1\). 
	Since \(v \notin S\), there is at least one \(H \in \mH \setminus \mH_L\) such that \(\langle v, e_H \rangle \neq 0\).
	Taking the inner product of Equations \eqref{eq:projL} and \eqref{eq:projH} with \(v\) we obtain
	\begin{equation*}
	\begin{aligned}
	a_L &= \sum_{H \in \mH_L} |\langle v, e_H \rangle|^2  \\
	&<  \sum_{H \in \mH} |\langle v, e_H \rangle|^2  = a_0 \,.
	\end{aligned}
	\end{equation*}
	
	To finish the proof, note that if \(a_L < a_0\) for \(L \in \mLi\) then \(a_L < a_0\) for all \(L \in \mL\), c.f. Lemma \ref{lem:kltcondition}\,.
\end{proof}

\begin{lemma}\label{lem:rarrh2}
	For every \(H \in \mH\) we have
	\[
	|\mH^H| = \left( 1 - \frac{2}{n+1} \right) |\mH| + 1 \,.
	\]
\end{lemma}

\begin{proof}
	This is \cite[Theorem 6.97]{orlikterao}.
\end{proof}

\begin{proposition}\label{prop:reflhir}
Let \(\mH \subset \CP^n\) be the complex reflection arrangement of an irreducible reflection group \(G \subset U(n+1)\).
Then \(\mH\) is Hirzebruch.
\end{proposition}

\begin{proof}
	Lemma \ref{lem:rarrh1} implies Item \ref{item:h1} and
	Lemma \ref{lem:rarrh2} implies Item \ref{item:h2}.
\end{proof}

\begin{remark}
	At the moment, the only examples we know in dimension \(n \geq 2\) of irreducible Hirzebruch arrangements stem from Proposition \ref{prop:reflhir}. For \(n=2\), it is an old question of Hirzebruch \cite[\S 3]{hirzebruch2} whether all Hirzebruch arrangements come from complex reflection groups.
	It is proved in \cite{dima} that all \emph{real} Hirzebruch line arrangements come from reflection groups.
\end{remark}

Using our results one can prove the following.

\begin{theorem}\label{thm:refarr}
	Let \(\mH \subset \CP^n\) be the complex reflection arrangement of an irreducible reflection group \(G \subset U(n+1)\).
	Then the quadratic form \(Q: \R^{\mH} \to \R\) is negative semidefinite.
\end{theorem}

\begin{proof}
	Lemma \ref{lem:rarrh1} implies that Equation \eqref{eq:sthir} is satisfied. By Lemma \ref{lem:rarrh2}\,, Equation \eqref{eq:hir} holds. By Theorem \ref{thm:hir}\,, the quadratic form \(Q\) is negative semidefinite.
\end{proof}

\begin{remark}
    It would be desirable to find a direct proof of Theorem \ref{thm:refarr}\,, for example, using the Shephard-Todd classification of irreducible complex reflection groups. 
\end{remark}

\subsection{The Hirzebruch quadratic form of a matroid}\label{sec:matroid}

We formulate our results in combinatorial terms, using the language of matroids. We begin with a review of basic definitions, a standard reference for this material is  
\cite{oxley}.

\subsubsection{Matroids: basic definitions}\label{sec:matroidbasics}
In a nutshell, matroids are combinatorial structures that mimic finite collections of vectors in a vector space, including information about their linear dependencies (see Example \ref{ex:repmat}). Formally, a \emph{matroid} is a pair \(M = (E, \mB)\), where \(E\) is a finite set and \(\mB\) is a non-empty collection of subsets \(B \subset E\) which satisfy the \emph{exchange property}: for any \(B_1, B_2 \in \mB\) and \(b_1 \in B_1 \setminus B_2\) there exists \(b_2 \in B_2 \setminus B_1\) such that \(\left(B_1 \setminus \{b_1\}\right) \cup \{b_2\}\) is in \(\mB\). The elements of \(\mB\) are called \emph{basis} of the matroid \(M\).
The set \(E\) is referred as the \emph{ground set} of \(M\). Two matroids \(M_1 = (E_1, \mB_1)\) and \(M_2 = (E_2, \mB_2)\) are \emph{isomorphic} if there is a bijection \(f: E_1 \to E_2\) such that \(B_1 \in \mB_1\) if and only if \(B_2 = f(B_1) \in \mB_2\)\,.

It follows from the exchange property that all bases \(B \in \mB\) have the same number of elements \(|B|\). The \emph{rank} \(r\) of the matroid \(M\) is \(r = |B|\) for any \(B \in \mB\). A subset of the ground set, \(I \subset E\), is \emph{independent} if it is contained in a basis, i.e., if \(I \subset B\) for some \(B \in \mB\). 
In particular, we agree that the empty set is independent.
We denote by \(\mI\) the collection of all independent subsets.
The \emph{rank of a subset} \(S \subset E\) is the largest size of an independent set contained in it,
\[
\rk (S) = \max_{I \in \mI} \{|I| \,\,|\,\, I \subset S \} \,.
\]

An element \(x\) of the ground set of \(M\) is a \emph{loop} if the singleton \(\{x\}\) is not an independent set, equivalently \(\rk(x) = 0\). A pair of elements \(x, y\) of \(M\) are \emph{parallel} if none is a loop and \(\rk(\{x, y\}) = 1\). 

\begin{example}\label{ex:repmat}
	Suppose that \(k\) is a field and let \(E = \{v_1, \ldots, v_{N}\}\) be vectors of \(V = k^{n+1}\) that span the whole space. Let \(\mB\) be the collection of all subsets \(B \subset E\) that form a basis of the vector space \(V\). Then \(M = (E, \mB)\) is a matroid of rank \(r = n+1\). 
	An element \(v_i \in E\) is a loop if \(v_i=0\). A pair of non-zero vectors \(v_i, v_j \in E\) are parallel if one is a scalar multiple of the other.
\end{example}

\begin{definition}
	If a matroid \(M\) is isomorphic to a matroid as in Example \ref{ex:repmat}\,, we say that \(M\) is \emph{representable} over the field \(k\). 
\end{definition}

A subset \(F\) of the ground set of \(M\) is a \(k\)-\emph{flat} if \(\rk(F) = k\) and 
\[
\rk \left(F \cup \{e\}\right) > \rk(F)
\] 
for any \(e \in E \setminus F\).
The set of all flats is denoted by \(\mF\).
The \emph{closure} \(\overline{S}\) of a subset \(S \subset E\) is the intersection of all flats containing it,
\[
\overline{S} = \bigcap_{F \in \mF \, |\, F \supset S} F \,.
\]
The intersection of flats is also a flat, so the closure \(\overline{S} \in \mF\). The set of all flats of the matroid \(M\) equipped with the partial order given by inclusion, \(F_1 \leq F_2\) if \(F_1 \subset F_2\), is a  poset \((\mF, \leq)\). This poset is actually a \emph{lattice}, meaning that any two elements \(F_1, F_2 \in \mF\) have a greatest lower bound (their \emph{meet} \(F_1 \wedge F_2\)) and a
least upper bound (their \emph{join} \(F_1 \vee F_2\)). These are necessarily unique and given by
\[
F_1 \wedge F_2 = F_1 \cap F_2, \hspace{2mm} F_1 \vee F_2 = \overline{F_1 \cup F_2} \,.
\]

The matroid \(M\) is \emph{simple} if it has no loops and no pairs of parallel points. If \(M\) is simple then the poset of flats \((\mF, \leq)\) is a \emph{geometric lattice}, meaning that it is \emph{semimodular} and \emph{atomic}. The \emph{semimodular} property is:
\[
\rk(F_1) + \rk(F_2) \geq \rk(F_1 \vee F_2) + \rk (F_1 \wedge F_2) \,.
\]
\emph{Atomic} means that every flat \(F \in \mF\) is a join of singletons \(\{x\}\). The elements \(x \in M\) of the ground set (or the singletons \(\{x\} \in \mF\))  are called \emph{atoms}. Conversely, every geometric lattice is the poset of flats of a simple matroid \cite[Theorem 3.8]{stanley}.

\begin{example}\label{ex:mathyp}
	Let \(\mH = \{H_1, \ldots, H_N\}\) be a finite collection of pairwise distinct hyperplanes in \(\P_k^n = \P(k^{n+1})\) with empty common intersection. For each \(H_i\) choose a defining linear equation \(h_i \in (k^{n+1})^*\). The collection of vectors \(E = \{h_1, \ldots, h_{N}\}\) defines, as in Example \ref{ex:repmat}\,, a matroid \(M = (E, \mB)\). The elements of \(\mB\) correspond to \((n+1)\)-tuples of hyperplanes with empty common intersection. The fact that each \(h_i \neq 0\) and that every pair of hyperplanes \(H_i\) and \(H_j\) with \(i \neq j\) satisfy \(H_i \neq H_j\), implies that \(M\) is simple. 
	
	The correspondence \(L \mapsto \mH_L\) (where \(L\) is a linear subspace and \(\mH_L\) is the set of hyperplanes in \(\mH\) that contain it) defines an isomorphism between the intersection poset \(\overline{\mL}\) of \(\mH\) equipped with the order by reverse inclusion and the poset of flats \(\mF\) equipped with the order by inclusion. The whole space \(\P_k^{n} \in \overline{\mL}\) (corresponding to the intersection of an empty collection of hyperplanes) is mapped to the minimal element \(\hat{0} = \overline{\emptyset}\) of \(\mF\), while \(\emptyset \in \overline{\mL}\) (the centre of the arrangement) corresponds to the unique maximal flat \(\hat{1}\) (the join of all \(h_i\)).
\end{example}

\begin{remark}
	If \(M\) is simple and representable over the field \(k\), then \(M\) must be isomorphic to a matroid as in Example \ref{ex:mathyp}\,. In this case, we say that \(M\) is the matroid \emph{associated} to the hyperplane arrangement \(\mH \subset \P^n_k\).
\end{remark}

\begin{remark}
	If the matroid \(M\) is associated to the arrangement \(\mH\), then \(\mH\) is assumed to be essential.
\end{remark}

\subsubsection{Quadratic form for matroids}

Let \(M\) be a simple matroid of rank \(n+1\) on the ground set \([N] = \{1, \ldots, N\}\).
Let \(F\) be a rank two flat. We say that \(F\) is \emph{irreducible} if \(|F| \geq 3\). Otherwise, if \(|F|=2\) then \(F\) is \emph{reducible}.

For \(i \in [N]\) let \(\sigma_i\) be the number of irreducible rank \(2\) flats that contain \(i\).

\begin{definition}
	The \emph{Hirzebruch quadratic form of} \(M\) is the function \(Q_M: \R^N \to \R\) given by
	\[
	Q_M (\bx) = \bx^t \cdot Q \cdot \bx \,,
	\]
	where \(Q\) is the real symmetric \(N \times N\) matrix with integer entries 
	\begin{equation}\label{eq:matrix}
		Q_{ij} = \begin{cases}
			-(n+1)\sigma_i + 2n &\text{ if } i = j \,, \\
			-2          &\text{ if } |i \vee j| = 2 \,, \\
			n - 1 &\text{ if } |i \vee j| \geq 3 \,.
		\end{cases}
	\end{equation}
 By slight abuse of notation, we shall also write \(Q\) for the quadratic form \(Q_M\) as well as for the symmetric matrix.
\end{definition}

\begin{remark}
	If \(M\) is the matroid associated to a hyperplane arrangement \(\mH \subset \CP^n\)
	then Equation \eqref{eq:matrix} agrees with Equation \eqref{eq:Qmatrix}.
\end{remark}

\begin{remark}[c.f. Lemma \ref{lem:sumqij}]
	The sum of the \(i\)-th column (or \(i\)-th row) of the matrix \(Q\) is given by
	\begin{equation}\label{eq:sumqijmat}
	\sum_{j=1}^{N} Q_{ij} = (n-1) \cdot N - (n+1) \cdot (t_i-1)
	\end{equation}
	where \(t_i\) is the number of rank \(2\) flats that contain \(i\).
\end{remark}

Let \(P \subset \R^N\) be the matroid polytope of \(M\), i.e., \(P\) is the convex hull of indicator functions of bases of \(M\).
Our results give us the following.

\begin{theorem}\label{thm:matroids}
	Suppose that the matroid \(M\) is representable over \(\C\). Then the Hirzebruch quadratic form of \(M\) is non-positive on the cone over the matroid polytope, i.e.,
	\begin{equation}\label{eq:thmmat}
	Q (\bx)\, \leq \, 0 	\,\, \textup{ for all } \,\, \bx \in \R_{\geq 0} \times P \,.
	\end{equation}
\end{theorem}

\begin{proof}
    This is Theorem \ref{thm:maingeneral} for
	the essential arrangement \(\mH \subset \CP^n\)  associated to \(M\). 
\end{proof}

\begin{comment}
	\begin{lemma}\label{lem:sumaij}
	The sum of the \(i\)-th column (or \(i\)-th row) of the matrix \(A\) is given by
	\begin{equation}\label{eq:sumaij}
	\sum_{j=1}^{N} A_{ij} = (n +1) \cdot t_i - \left(n+1 + (n-1) N \right) \,,
	\end{equation}
	where \(t_i\) is the number of rank \(2\) flats that contain \(i\).
	\end{lemma}
	
	\begin{proof}
	Let \(\mathcal{R}\) and \(\mI\) be the set of all reducible and irreducible rank \(2\) flats that contain \(i\) and let \(T = \mathcal{R} \cup \mI\). We write \(\tau_i = |\mathcal{R}|\) and
	\(\sigma_i = |\mI|\), so that \(t_i = \tau_i + \sigma_i\).
	It follows from Equation \eqref{eq:matrix} that
	\[
	\begin{aligned}
	\sum_{j=1}^{n} A_{ij} &= (n+1) \cdot \sigma_i - 2n + 2 \tau_i + (1-n) \cdot \sum_{F \in \mI} (|F| - 1) \\
	&= (2n) \cdot \sigma_i - 2n + 2 \tau_i + (1-n) \cdot \sum_{F \in \mI} |F| \\
	&=  (2n) \cdot t_i - 2n + (1-n) \cdot (2\tau_i) + (1-n) \cdot \sum_{F \in \mI} |F| \\
	&=  (2n) \cdot t_i - 2n  + (1-n) \cdot \sum_{F \in T} |F| \,.
	\end{aligned}
	\]
	Since for every \(j \neq i\) there is a unique rank \(2\) flat that contains \(\{i,j\}\) (namely \(i \vee j\)), 
	\[
	\sum_{F \in T} (|F| -1) = N-1 \,.
	\]
	Therefore,
	\[
	\sum_{j=1}^{n} A_{ij} = (n + 1) \cdot t_i - 2n  + (1-n) \cdot (N-1) 
	\]
	which implies \eqref{eq:sumaij}.
	\end{proof}
\end{comment}

\subsubsection{Examples}

The advantages of presenting our results within the broader context of matroids are twofold. On the one hand, Theorem \ref{thm:matroids} brings in a potential obstruction for representability of a matroid over the complex numbers. On the other hand, it rises the question for which classes of matroids the statement of Theorem \ref{thm:matroids} holds true. The next two examples illustrate this. 

\begin{example}
	Let \(p\) be a prime number and let \(\mathbb{F}_p\) be the finite field of order \(p\). For \(n \geq 2\), let
	\(M\) be the matroid associated to the hyperplane arrangement  consisting of all hyperplanes in \(\PG(n, p) = \P(\mathbb{F}_p^{n+1})\). As Example \ref{ex:finiteps} shows, \(\mathbf{1} \in \R_{\geq 0} \times P\) but \(Q (\mathbf{1}) > 0\).
	Therefore, Theorem \ref{thm:matroids} does not hold for \(M\). As a result, \(M\) is not realizable over the complex numbers.
\end{example}

\begin{comment}
\begin{example}[Fano matroid]
Let \(\P^2(F_2)\) be the finite projective plane over the field \(F_2\) with two elements, known as the Fano plane. Let \(\mH\) be the arrangement of all the \(7\) lines in \(\P^2(F_2)\). The projective space \(\P^2(F_2)\) has \(7\) points \(p_1, \ldots, p_7\). There are \(3\) lines going through each of the points \(p_i\), thus their multiplicities are \(m_{p} = 3\). 
Each line \(H \in \mH\) contains \(3\) points, thus \(B_H = 3 -1 =2\). It is easy that the vector \(\mathbf{1} \in \R^7\) satisfies the inequalities \eqref{eq:starr} that define the stable cone. However,
\[
\begin{aligned}
Q(\mathbf{1}) &= 7 \cdot \left(\frac{m_p}{2}\right)^2 - \frac{7}{2} \cdot B_H - \frac{s^2}{2(n+1)} \\
&=  7 \cdot \left( \frac{9}{4} - 1 - \frac{7}{6} \right) > 0 \,.
\end{aligned}
\]
\end{example}	
\end{comment}

\begin{example}
	The Non-Pappus matroid \(M\) is the matroid of rank \(3\) on the set of nine elements \([9] = \{1, \ldots, 9\}\). This matroid has \(8\) circuits (minimal dependent sets) shown in Figure \ref{fig:nonpapus} as lines connecting triplets of dependant points. The number of bases is equal to
	\[
	\binom{9}{3} - 8 = 76 \,,
	\] 
	as any \(3\) points not connected by a line form a basis.
	
\begin{figure}[h]
		\centering
		\scalebox{.7}{
		\begin{tikzpicture}
			\node (1) at (-4,4) [thick,circle,draw,fill=yellow!50] {\(1\)};
			\node (2) at (0,4) [thick,circle,draw,fill=yellow!50] {\(2\)};
			\node (3) at (4,4) [thick,circle,draw,fill=yellow!50] {\(3\)};
			\node (7) at (-2,0) [thick,circle,draw,fill=yellow!50] {\(7\)};
			\node (8) at (0,0) [thick,circle,draw,fill=yellow!50] {\(8\)};
			\node (9) at (2,0) [thick,circle,draw,fill=yellow!50] {\(9\)};
			\node (4) at (-4,-4) [thick,circle,draw,fill=yellow!50] {\(4\)};
			\node (5) at (0,-4) [thick,circle,draw,fill=yellow!50] {\(5\)};
			\node (6) at (4,-4) [thick,circle,draw,fill=yellow!50] {\(6\)};
			
			\draw (1) -- (2) -- (3);
			\draw (4) -- (5) -- (6);
			
			\draw (1) -- (7) -- (5);
			\draw (2) -- (7) -- (4);
			
			\draw (2) -- (9) -- (6);
			\draw (3) -- (9) -- (5);
			
			\draw (1) -- (8) -- (6);
			\draw (3) -- (8) -- (4);
		\end{tikzpicture}
		}
		\caption{The Non-Pappus matroid.}
		\label{fig:nonpapus}
\end{figure}
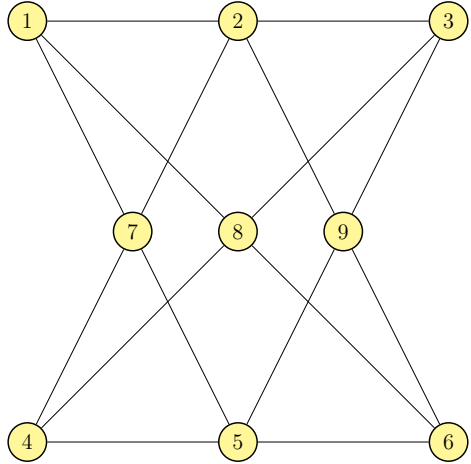

We see that \(| i \vee j | = 3\)  if \(i\) and \(j\) are joined by a line and \(| i \vee j | = 2\) otherwise. Also,
\[
\sigma_i = \begin{cases}
3 &\text{ if } 1 \leq i \leq 6 \,, \\
2 &\text{ if } i = 7, 8, 9 \,.
\end{cases}
\]
Therefore, the matrix \(Q\) of the Hirzebruch quadratic form of \(M\) is
\[
\begin{pmatrix}
	-5 & 1 & 1 & -2 & 1 & 1 & 1 & 1 & -2 \\
	1 & -5 & 1 & 1 & -2 & 1 & 1 & -2 & 1 \\
	1 & 1 & -5 & 1 & 1 & -2 & -2 & 1 & 1 \\
	-2 & 1 & 1 & -5 & 1 & 1 & 1 & 1 & -2 \\
	1 & -2 & 1 & 1 & -5 & 1 & 1 & -2 & 1 \\
	1 & 1 & -2 & 1 & 1 & -5 & -2 & 1 & 1 \\
	1 & 1 & -2 & 1 & 1 & -2 & -2 & -2 & -2 \\
	1 & -2 & 1 & 1 & -2 & 1 & -2 & -2 & -2 \\
	-2 & 1 & 1 & -2 & 1 & 1 & -2 & -2 & -2
\end{pmatrix}
\, .
\]

The matroid \(M\) is not representable over any field. However, the statement of Theorem \ref{thm:matroids} holds for \(M\). Indeed, it can be checked with computer\footnote{See \url{https://github.com/MdeBorbon/nonpappus}} that the matrix \(P = -Q\) is \emph{copositive}, meaning that \(\bx^t \cdot P \cdot \bx \geq 0\) for any vector \(\bx \in \R^9\) with components \(x_i \geq 0\).
\end{example}

%In light of the above example, we think it is an interesting question to ask whether Theorem \ref{thm:matroids} holds for \emph{oriented matroids} or \emph{pseudo arrangements}. 
%\vspace{2mm}

\subsubsection{Pseudoline arrangements and symplectic BMY.} 
It would be very interesting to extend Theorem \ref{thm:matroids} to some large classes of non-realizable matroids. In this section we speculate about one such possibility.

Recall that a \emph{pseudoline arrangement} is a collection of circles smoothly embedded into $\mathbb{RP}^2$ such that each circle is isotopic to a projective line, any two circles intersect in exactly one point, and their common intersection is empty (see \cite{ziegler}). Just as with line arrangements, one can associate a matroid to any pseudoline arrangement. Note that the non-Pappus matroid can be obtained in this manner: take a collection of $9$ straight lines in $\mathbb{RP}^2$ realizing the Pappus configuration and slightly perturb three lines in the neighbourhood of a triple point.

In the spirit of Arnold's \emph{topologocial economy principle}\footnote{This principle reads as follows: ``If you have a geometrical or topological phenomenon, which you can realize by algebraic objects, then \it{the simplest algebraic realizations are topologically as simple
as possible}".} \cite{arnold}, 
%[Topological Problems in Wave Propagation Theory and Topological Economy Principle in Algebraic Geometry], 
we would like to state the following provocative (and quite possibly over-optimistic) conjecture.

%\begin{question} Let $\mathcal L$ be an essential pseudoline arrangement and $M$ be its matroid. Is it true that the Hirzebruch quadratic form $Q$ of $M$ is non-positive on the semistable cone of $M$, and moreover, whenever $Q$ vanishes in its interior, is $\mathcal L$ realizable over $\R$?
%\end{question}

\begin{conjecture}\label{conj:provocation}
Let $\mathcal L\subset\mathbb{RP}^2$ be an essential pseudoline arrangement and $M$ be its matroid. Then the Hirzebruch quadratic form $Q$ of $M$ is non-positive on the semistable cone of $M$, and moreover, whenever $Q$ vanishes in its interior, $M$ is realizable over $\R$ and the pseudoline arrangement \(\mathcal{L}\) is \emph{stretchable}.
\end{conjecture}

Recall that there is a one-to-one correspondence between (equivalence classes of) pseudoline arrangements and (reorientation classes of) simple rank 3 oriented
matroids \cite[Section 6]{ziegler}.
%[Bj\"orner, A., Las Vergnas, M., Sturmfels, B., White, N., Ziegler, G.M.: Oriented Matroids. Encyclopedia of Mathematics and Its Applications, vol. 46. Cambridge University Press, Cambridge (1993) [Section 6]]. 
There exists a complete classification of pseudoline arrangements with up to $11$ pseudolines. Moreover, in case of simplicial pseudoline arrangements (that partition $\mathbb {RP}^2$ into triangles), the classification exists for up to $27$ pseudolines \cite{cuntz}.
%[Cuntz, M.: Simplicial arrangements with up to 27 lines. Discrete. Comput. Geom. 48(3), 682–701 (2012)]. 
This is a promising class of arrangements for finding counter-examples to Conjecture \ref{conj:provocation}, as simplicial arrangements seem to be \emph{extremal} for the Hirzebruch quadratic form\footnote{It follows from \cite{panov, dima} that any essential real arrangement for which the Hirzebruch quadratic form vanishes in the stable cone is simplicial.}.

We don't have much evidence for Conjecture \ref{conj:provocation} except that it holds for usual line arrangements and for the non-Pappus arrangement. However, we have a motivation coming from symplectic geometry. Recall, that to every compact $4$-dimensional manifold $(X,\omega)$ one can associate two Chern numbers $c_1^2(X)$ and $c_2(X)=e(X)$. Using Seiberg-Witten theory, one can also define the Kodaira dimension of $X$, which agrees with the standard definition in the case of K\"ahler surfaces. There exists a {\it folklore conjecture} (that we learned from Ivan Smith in 2010), that symplectic $4$-manifolds of general type satisfy Bogomolov-Myaoka-Yau inequality as well. In particular, no one has been able so far to construct a symplectic $4$-manifold of general type with $c_1(X)^2>3c_2(X)$.

\begin{comment}
The relation between Conjecture \ref{conj:provocation} and the symplectic BMY inequality is as follows. It turns out that any pseudoline arrangement $\mathcal L\subset \mathbb{RP}^2$ can be extended to an arrangement of symplectic spheres in $\mathbb{CP}^2$ disjoint in $\mathbb{CP}^2\setminus \mathbb{RP}^2$. The proof of this statement is contained in \cite{rubermanstarkston}[Theorem 1.4] 
%[Topological realizations of line arrangements Daniel Ruberman, Laura Starkston] 
(we learned this fact from Stepan Orevkov  in 2011). 
Following Hirzebruch, one can then consider a (symplectic) blow up of $\mathbb{CP}^2$ in points of $\mathcal L$ of multiplicity $\ge 3$ and then take an appropriate ramified cover in a hope of being lucky enough to get a counter-example to symplectic BMY\footnote{This is also a suggestion hinted upon in \cite{rubermanstarkston}.}. Indeed, in his seminal paper [], using Fermat covers of $\mathbb {CP}^2$ Hirzebruch was able to construct three surfaces with $c_1^2=3c_2$, each associated to one complex reflection group. 

%[Topological realizations of line arrangements Daniel Ruberman, Laura Starkston].
However, one can just treat Conjecture \ref{conj:provocation} as  a logarithmic version of symplectic BMY conjecture, and look for counter-examples directly to it.   
\end{comment}

A counterexample to Conjecture \ref{conj:provocation} would be interesting, as it could be interpreted as a failure of a logarithmic version of the symplectic BMY inequality.
Indeed, any pseudoline arrangement $\mathcal L\subset \mathbb{RP}^2$ can be extended to an arrangement of symplectic spheres in $\mathbb{CP}^2$ disjoint in $\mathbb{CP}^2\setminus \mathbb{RP}^2$. Such extensions are constructed in \cite{rubermanstarkston}[Theorem 1.4] 
%[Topological realizations of line arrangements Daniel Ruberman, Laura Starkston] 
(we learned this fact from Stepan Orevkov  in 2011). 

Going further, following Hirzebruch, one can then consider a (symplectic) blow up of $\mathbb{CP}^2$ in points of $\mathcal L$ of multiplicity $\ge 3$ and then take an appropriate ramified cover in a hope of being lucky enough to get a counter-example to symplectic BMY\footnote{This is also a suggestion hinted upon in \cite{rubermanstarkston}.}. In his seminal paper \cite{hirzebruch3}, using Kummer covers of $\mathbb {CP}^2$ Hirzebruch was able to construct three algebraic surfaces with $c_1^2=3c_2$, each associated to one complex reflection group.

\appendix

\newpage

\section{Auxiliary Results}\label{app:auxres}

Section \ref{ap:filt} provides self contained proofs on basic linear algebra related to filtrations of vector spaces, adapted basis, and nested sets.
These results are used in the proof of the locally abelian property of our parabolic bundle, Theorem \ref{thm:locab}. \newpar

\noindent Sections \ref{sec:extalg}\,, \ref{sec:satsubsheaf}, and \ref{sec:distributions} contain results on exterior algebra, saturated subsheaves, and distributions on \(\CP^n\) that are used in the proof of the stability Theorem \ref{thm:stability}.

\subsection{Filtrations, adapted basis, and nested sets}\label{ap:filt}

The key results proved in this section are: 
\begin{itemize}
	\item Lemma \ref{lem:comp}\,, which characterizes splittings of tuples of filtrations in terms of adapted bases;
	
	\item Corollary \ref{cor:adapbasis}\,, which shows that nested sets have adapted bases.
\end{itemize}

\subsubsection{Filtrations of vector spaces}

Let \(V\) be a finite dimensional vector space.

\begin{definition}
	 A filtration \(\mF = \{F_a \, | \, a \in \R\}\) of \(V\) is a family of vector subspaces  \(F_a \subset V\) parametrized by \(a \in \R\) satisfying the following conditions.
	\begin{enumerate}[label=\textup{(\roman*)}]
		\item Increasing: \(F_a \subset F_{a'}\) if \(a < a'\).
		\item Semi-continuity: for every \(a\) there is \(\epsilon >0\) such that \(F_{a+\epsilon} = F_a\).
		\item Normalization: \(F_a = \{0\}\) for \(a < 0\) and \(F_{a} = V\) for \(a \geq 1\).
	\end{enumerate} 
\end{definition}

The increasing and semi-continuity properties are equivalent to 
\[
	\forall a: \,\, F_a = \bigcap_{a'>a} F_{a'} \,.
\]
Given a filtration \(\mF = \{F_a\}\) of \(V\), we write
\[
F_{< a} = \bigcup_{a'<a} F_{a'} \,.
\]
The increasing property implies that \(F_{<a}\) is a vector subspace of \(F_a\). 

\begin{definition}
The graded components of \(\mF\) are the quotient vector spaces
\[
\Gr_a = F_a \, \big/ \, F_{< a} \,.
\]	
\end{definition}

\begin{definition}
We say that \(a \in \R\) is a \emph{weight} of \(\mF\)
if \(\Gr_a \neq 0\). The set of all weights of \(\mF\) is 
denoted by
\[
\wt(\mF) = \{a \, | \, \Gr_a \neq 0\} \,.
\]
\end{definition}

The normalization condition, together with the fact that \(V\) is finite dimensional,
imply that \(\wt(\mF)\) is a finite subset of \([0,1]\).
If we write
\[
\wt(\mF) = \{a^{\mF}_0, \ldots, a^{\mF}_p\} \,,
\]
then \(\mF\) determines a strictly increasing flag of vector subspaces
\begin{equation*}
0 \subsetneq F_{a^{\mF}_0} \subsetneq F_{a^{\mF}_1} \subsetneq \ldots \subsetneq F_{a^{\mF}_p} = V \,.	
\end{equation*}
Conversely, \(\mF\) is determined by this flag and the weights \(\wt(\mF)\) via
\begin{equation*}
F_a = \begin{cases}
\{0\} \,&\textup{ if }\, a < a^{\mF}_0 \,,\\
F_{a^{\mF}_i} \,&\textup{ if }\,  a^{\mF}_{i} \leq a < a^{\mF}_{i+1} \,, \\
V \,&\textup{ if }\, a \geq a^{\mF}_p \,. 
\end{cases}	
\end{equation*}

\subsubsection{Tuples of filtrations and splittings}\label{sec:tuplesfilt}

Suppose that \(\mF = (\mF^1, \ldots, \mF^k)\) is a tuple of \(k\) distinct filtrations \(\mF^i = \{F^i_{a}\}\) of \(V\).

\begin{notation}\label{not:Fa}
For \(\ba = (a_1, \ldots, a_k) \in \R^k\)\,, let \(F_{\ba}\) be the vector subspace
\begin{equation}\label{eq:fba}
	F_{\ba} = \bigcap_{i=1}^k F^i_{a_i} \,.
\end{equation}	

Given \(\ba = (a_1, \ldots, a_k)\) and \(\ba'=(a_1', \ldots, a_k')\) in \(\R^k\)\,, we say that \(\ba' \leq \ba\) if \(a'_i \leq a_i\) for all \(1 \leq i \leq k\).
This equips \(\R^k\) with a partial order such that \(F_{\ba'} \subset F_{\ba}\) if \(\ba' \leq \ba\).
We write \(\ba' \lneq \ba\) if \(\ba' \leq \ba\) and \(\ba' \neq \ba\). Note that
\begin{equation}\label{eq:sumint}
	\sum_{\ba' \lneq \ba} F_{\ba'} = F^1_{< a_1} \cap F^2_{a_2}\ldots \cap F^k_{a_k} + \ldots + F^1_{a_1} \cap F^2_{a_2}\ldots \cap F^k_{<a_k} 
\end{equation}
is a linear subspace of \(F_{\ba}\).	
\end{notation}

\begin{definition}\label{def:weight}
	We say that \(\ba \in \R^k\) is a weight of \(\mF\)\,, if
	\[
	\sum_{\ba' \lneq \ba} F_{\ba'} \, \subsetneq \, F_{\ba} \,.
	\]
	The set of all weights of \(\mF\) is denoted by \(\wt(\mF)\).
\end{definition}

\begin{lemma}
	The set \(\wt(\mF)\) is contained in \(\wt(\mF^1) \times \ldots \times \wt(\mF^k)\). In particular, \(\wt(\mF)\) is a finite subset of \([0,1]^k\).
\end{lemma}

\begin{proof}
	If \(a_i \notin \wt(\mF^i)\) then \(F^i_{< a_i} = F^i_{a_i}\)\,, so
	\[
	F^1_{a_1} \cap \ldots \cap F^i_{<a_i} \cap \ldots \cap F^k_{a_k} = F_{\ba} \,.
	\] 
	It follows from this, together with Equation \eqref{eq:sumint}, that if \(\ba \notin \wt(\mF^1) \times \ldots \times \wt(\mF^k)\)\,, then \(\sum_{\ba' \lneq \ba} F_{\ba'} = F_{\ba}\)\,; hence \(\ba \notin \wt(\mF)\).
\end{proof}

The main notion we want to introduce is that of a splitting of a tuple of filtrations. Before doing so, we recall the following standard definition.

\begin{definition}
	Let \(U_i\) be linear subspaces of \(V\) indexed by \(i \in I\) and let \(U = \sum_i U_i\) be their sum, i.e., \(U\) is the smallest linear subspace of \(V\) (with respect to the partial order given by inclusion) that contains all \(U_i\). We say that the sum \(U = \sum_i U_i\) is direct and write \(U = \oplus_i U_i\) if any of the following equivalent conditions holds.
	\begin{enumerate}[label=\textup{(\roman*)}]
		\item \(U_i \cap \sum_{j \neq i} U_j = \{0\}\) for all \(i \in I\).
		\item For every \(u \in U\) there are \emph{unique} vectors \(u_i \in U_i\) such that \(u = \sum_i u_i\).
		\item \(\dim U = \sum_i \dim U_i\)\,.
		\item There is a basis \(B\) of \(U\) such that \(B_i = B \cap U_i\) is a basis of \(U_i\)\,, the sets \(B_i\) are pairwise \emph{disjoint}, and \(B = \cup_i B_i\)\,.
	\end{enumerate}
\end{definition}

The main definition of this section is the next.

\begin{definition}\label{def:split}\label{def:compfilt}
	A splitting of \(\mF\) is a family of linear subspaces \(U_{\ba} \subset V\) indexed by \(\ba = (a_1, \ldots, a_k) \in \R^k\) satisfying the following properties.
	\begin{enumerate}[label=\textup{(\roman*)}]
		\item The subspaces \(U_{\ba}\) are zero except for a finite number of \(\ba \in \R^k\), and they form a direct sum decomposition
		\begin{equation}\label{eq:split1}
		V = \bigoplus_{\ba \in \R^k} U_{\ba} \,.
		\end{equation}
		\item For any \(\ba \in \R^k\), we have
		\begin{equation}\label{eq:split2}
		F_{\ba} = \bigoplus_{\ba' \leq \ba} U_{\ba'} \,.
		\end{equation}
	\end{enumerate}

	The filtrations \(\mF^1, \ldots, \mF^k\) are \emph{compatible} if the tuple \(\mathcal{F} = (\mathcal{F}^1, \ldots, \mathcal{F}^k)\) admits a splitting.
\end{definition}

\begin{remark}
	Item (i) follows from (ii) by taking \(\ba = (1, \ldots, 1)\). For the sake of clarity, we keep item (i) as part of Definition \ref{def:split}\,.	
\end{remark}

\begin{example}
	If \(k=1\) and \(\mF = \{F_a \,|\, a \in \R\}\), then we can construct a splitting of \(\mF\) as follows. Set \(U_a = \{0\}\) if \(a \notin \wt(\mF)\) and, for every \(a \in \wt(\mF)\) choose a linear subspace \(U_a \subset F_a\) such that \(F_{a} = F_{< a} \oplus U_a\). It is then easy to verify that \(\{U_a \,|\, a \in \R\}\) satisfies \eqref{eq:split1} and \eqref{eq:split2}\,. Note that the subspaces \(U_a\) are isomorphic to the graded components of \(\Gr_a = F_a \, / \, F_{<a}\) of \(\mF\).
\end{example}

\begin{lemma}
	Suppose that \(\{U_{\ba} \, | \, \ba \in \R^k\}\) is a  splitting of \(\mF\), then the following holds.
	\begin{enumerate}[label=\textup{(\roman*)}]
		\item \(U_{\ba}\) is non-zero if and only if \(\ba \in \wt(\mF)\).
		\item For any \(\ba \in \R^k\), we have
	\begin{equation}\label{eq:uba}
	U_{\ba} \cong \dfrac{F_{\ba}}{\sum_{\ba' \lneq \ba} F_{\ba'}} \,.
	\end{equation}
		\item If \(\{\widetilde{U}_{\ba}\}\) is another splitting of \(\mF\) then there is a linear isomorphism \(\Phi \in GL(V)\) with \(\Phi(\mF) = \mF\) such that \(\Phi(U_{\ba}) = \widetilde{U}_{\ba}\) for all \(\ba \in \R^k\).
	\end{enumerate}

\end{lemma}

\begin{proof}
	We begin by proving item (ii). For this, we notice that Equation \eqref{eq:split2} implies
	\[
	\sum_{\ba' \lneq \ba} F_{\ba'} = \bigoplus_{\ba' \lneq \ba} U_{\ba'} \,.
	\]
	Therefore,
	\begin{equation}\label{eq:fbaua}
		F_{\ba} = \left(\sum_{\ba' \lneq \ba} F_{\ba'}\right) \oplus U_{\ba}	
	\end{equation}
	and item (ii) follows. By Definition \ref{def:weight}\,, (ii) implies (i). To show (iii), we note that by (ii), for each \(\ba \in \wt(\mF)\), we can choose a linear isomorphism \(\Phi_{\ba}: U_{\ba} \to \widetilde{U}_{\ba}\). Since
	\[
	V = \bigoplus_{\ba \in \wt(\mF)} U_{\ba} = \bigoplus_{\ba \in \wt(\mF)} \widetilde{U}_{\ba} \,,
	\]
	the maps \(\Phi_{\ba}\) define a linear isomorphism \(\Phi \in GL(V)\) with \(\Phi(U_{\ba}) = \widetilde{U}_{\ba}\)\,. If follows from Equation \eqref{eq:split2}\,, that \(\Phi\) preserves each of the subspaces \(F^i_a\) of the filtration \(\mF^i\) for all \(i\), so \(\Phi(\mF) = \mF\).
\end{proof}

\begin{remark}
	For \(\ba \in \wt(\mF)\) the sum \(\sum_{\ba' \lneq \ba} F_{\ba'}\) is a proper linear subspace of \(F_{\ba}\). Therefore, for any \(\mF\) we can find non-zero subspaces \(U_{\ba} \subset F_{\ba}\) for \(\ba \in \wt(\mF)\) such that
	\eqref{eq:fbaua} holds. By construction, \(F_{\ba} = \sum_{\ba' \leq \ba} U_{\ba'}\). The subspaces \(U_{\ba}\) split \(\mF\) if and only if they form a direct sum. However, in general, the sum \(\sum U_{\ba}\) will not be direct, as the next example shows.
\end{remark}

\begin{example}\label{ex:linesc2}
	Let \(L_1, \ldots, L_k\) be \(k\) distinct lines in \(\C^2\) through the origin with \(k \geq 3\).
	Let \(\lambda_1, \ldots, \lambda_k \in (0, 1)\) and consider the filtrations \(\mF^1, \ldots, \mF^k\)  given by
	\[
	F^i_a = \begin{cases}
	\{0\} &\text{ if } a < 0 \,, \\
	L_i &\text{ if } 0 \leq a < \lambda_i \,, \\
	\C^2 &\text{ if } a \geq \lambda_i \,.
	\end{cases}
	\]
	For \(\ba = (a_1, \ldots, a_k) \in \R^k\), the intersection \(F_{\ba}\) is zero if \(a_i < \lambda_i\) and \(a_j < \lambda_j\) for a pair of distinct indices \(i, j\). Hence,
	the weights of \(\mF = (\mF^1, \ldots, \mF^k)\) are
	\[
	\wt(\mF) = \{(0, \lambda_2, \ldots, \lambda_k), \, (\lambda_1, 0, \ldots, \lambda_k), \ldots, \, (\lambda_1, \lambda_2, \ldots, 0)\} \,.
	\]
	If \(\ba \in \wt(\mF)\) then \(F_{\ba} = L_i\), where \(i\) is the unique component of \(\ba\) with \(a_i = 0\), and \(\sum_{\ba' \lneq \ba} F_{\ba'} = 0\)\,. Therefore, if Equation \eqref{eq:fbaua} is satisfied, we must take \(U_{\ba} = L_i\). The subspaces \(U_{\ba}\) do not make a direct sum. Therefore, the filtrations \(\mF^1, \ldots, \mF^k\) are \emph{not} compatible.
\end{example}

%In the next section, we reformulate the compatibility condition in terms of \emph{adapted bases}.

\subsubsection{Adapted bases}

Let \(B = \{e_1, \ldots, e_n\}\) be a basis of \(V\) and 
let \(L \subset V\) be a linear subspace.

\begin{definition}\label{def:adaptedbasis}
	We say that \(B\) is \emph{adapted} to \(L\) if \(B \cap L\) is a basis of \(L\). 
	If \(\mS\) is a collection of linear subspaces of \(V\), we say that \(B\) is adapted to \(\mS\) if \(B\) is adapted to \(L\) for every \(L \in \mS\). 
\end{definition}

\begin{comment}
	\begin{remark}\label{rmk:adapted}
		Suppose that \(B\) is adapted to \(\mS\). If \(L = \cap_i L_i\) with \(L_i \in \mS\), then \(B\) is adapted to \(L\). Similarly, if \(L = \sum_i L_i\) with \(L_i \in \mS\), then \(B\) is adapted to \(L\).
	\end{remark}
\end{comment}

Let \(\mF = (\mF^1, \ldots, \mF^k)\) be a tuple of filtrations \(\mF^i = \{F^i_{a}\}\) as in Section \ref{sec:tuplesfilt}.
A basis \(B\) of \(V\) is adapted to \(\mF\), if any of the following equivalent conditions holds:
\begin{itemize}
	\item  \(B\) is adapted to \(F^i_a\) for every \(1 \leq i \leq k\) and \(a \in \R\)\,;
	\item  \(B\) is adapted to \(F_{\ba} = \bigcap_{i=1}^k F^i_{a_i}\) for every \(\ba = (a_1, \ldots, a_k) \in \R^k\).
\end{itemize}
The main result that we are after is the next.

\begin{lemma}\label{lem:comp}
	The following two conditions are equivalent.
	\begin{enumerate}[label= \textup{(\roman*)}]
		\item The filtrations \(\mF^1, \ldots, \mF^k\) are compatible. 
		%i.e., there is a splitting \(\{U_{\ba} \,|\, \ba \in \R^k\}\)  of \(\mF\).
		\item There is a basis \(B\) of \(V\) that is adapted to \(\mF = (\mF^1, \ldots, \mF^k)\).
	\end{enumerate}
\end{lemma}

\begin{proof}
	(i)\(\implies\)(ii). Suppose that \(\{U_{\ba} \,|\, \ba \in \R^k\}\)
	is a splitting of \(\mF\). For each \(\ba \in \R^k\) such that \(U_{\ba}\) is non-zero, choose a basis \(B_{\ba}\) of \(U_{\ba}\). Then \(B =\bigcup_{\ba} B_{\ba}\) is adapted to \(\mF\). 
	
	(ii)\(\implies\)(i). Suppose that \(B\) is a basis of \(V\) adapted to \(\mF\). 	
	We want to define subspaces \(U_{\ba} \subset V\) that split \(\mF\).
	To do this, for each \(e \in B\) we consider the subspace \(F_{\ba}\) with smallest \(\ba\) that contains \(e\).
	In detail, note that if both \(F_{\ba}\) and \(F_{\ba'}\) contain \(e\), then \(F_{\min \{\ba, \ba'\}}\) also does, where \(\min \{\ba, \ba'\}\) is the vector with components \(\min \{a_i, a'_i\}\).
	Thus, we have a map \(\Phi: B \to \R^k\) given by
	\[
	\Phi(e) = \min \{ \ba \, | \, e \in F_{\ba} \} \,.
	\] 
	It is easy to see that the image of \(\Phi\) is \(\wt(\mF)\), the weights of \(\mF\). The preimages \(\Phi^{-1}(\ba)\) partition \(B\) into disjoint sets. We let \(U_{\ba}\) be the span of the vectors in \(\Phi^{-1}(\ba)\)\,,
	\[
	U_{\ba} = \spn \{ e \in B  \, | \, \Phi(e) = \ba \} \,.
	\]
	We want to show that \eqref{eq:split2} holds. Clearly, if
	\(\Phi(e) = \ba'\) and \(\ba' \leq \ba\), then \(e \in F_{\ba}\). Conversely, if \(e \in B \cap F_{\ba}\) and \(\Phi(e) = \ba'\), then \(\ba' \leq \ba\). Since \(B\) is adapted to \(\mF\), 
	\[
	\begin{aligned}
	F_{\ba} &= \spn \{ e \, | \, e \in B \cap F_{\ba} \} \\
	&= \spn \{ e \, | \, \Phi(e) \leq \ba \} = \bigoplus_{\ba' \leq \ba} U_{\ba'} \,.
	\end{aligned}
	\]
	This shows that \(\{U_{\ba} \, |\, \ba \in \R^k\}\) splits \(\mF\).
\end{proof}

\begin{example}
	If \(k=2\) then any two filtrations \(\mF^1\) and \(\mF^2\) on \(V\) are compatible,  see \cite[Lemma 3.5]{xu}. However, if \(k \geq 3\) then generic \(k\)-tuples of filtrations \(\mF^1, \ldots, \mF^k\) are not compatible. For example if \(k=3\), \(n=2\) and \(\mF^i\) define \(3\) distinct lines in \(\C^2\) then it is clear that there is no basis of \(\C^2\) adapted to \(\mF^1, \mF^2, \mF^3\); c.f. Example \ref{ex:linesc2}.
\end{example}

\subsubsection{Nested sets of linear subspaces}\label{sec:nestedlinear}

Let \(V\) be a finite dimensional vector space of dimension \(n\) and let \(V^*\) be its dual. We use the partial order on linear subspaces of \(V^*\) given by inclusion, i.e., \(U_1 \leq U_2\) if  \(U_1 \subset U_2\). 

%Two linear subspaces are comparable if one is contained in the other. Two subspaces are non comparable if none of them is contained in the other.

\begin{definition}\label{def:nesteddual}
	A \(*\)-\emph{nested} set  is a finite set \(\mN\) of non-zero linear subspaces \(N \subset V^*\) such that, for any collection of pairwise \emph{non comparable} elements \(N_1, \ldots, N_k \in \mN\) (i.e. \(N_i \not\subset N_j\) for \(i \neq j\)), the following holds:
	\begin{enumerate}[label=\textup{(\roman*)}]
		\item their sum \(S = \sum_i N_i\) is direct, i.e., \(S = \oplus_i N_i\)\,;
		\item \(S \notin \mN\).
	\end{enumerate}
	To emphasize the ambient space, we say that \(\mN\) is a \(*\)-nested set in \(V^*\).
\end{definition}

\begin{remark}
	One can also omit condition (ii) to get a sensible definition, but we will  need it for our purposes later on, see also Remark \ref{rmk:nouse}\,.
\end{remark}

The next lemma follows immediately from Definition \ref{def:nesteddual}\,, we omit its proof.

\begin{lemma}\label{lem:subnested}
	If \(\mN\) is a \(*\)-nested set and \(\mN_0\) is a subset of \(\mN\), then \(\mN_0\) is also a \(*\)-nested set. 
	Moreover, if all elements of \(\mN_0\) are contained in a subspace \(W \subset V^*\), then \(\mN_0\) is a \(*\)-nested set in \(W\).
\end{lemma}

Recall from Definition \ref{def:adaptedbasis} that, if \(\mN\) is a collection of linear subspaces of  \(V^*\), a basis of \(V^*\) is adapted to \(\mN\), if every  element of \(\mN\) is spanned by the basis vectors contained in it.

\begin{lemma}\label{lem:adapbasistar}
	Every \(*\)-nested set has an adapted basis.
\end{lemma}

\begin{proof}
	We proceed by induction on the dimension \(n\) of the ambient space \(V^*\). The statement is trivially true if \(n=1\).
	Let \(\mN\) be a \(*\)-nested set in \(V^*\) with
	\(n = \dim V^*\) and assume that the statement holds true for dimensions \(< n\).
	We want to show that there is a basis \(B^*\) of \(V^*\) adapted to \(\mN\). 
	
	We can assume that \(V^* \notin \mN\); otherwise take \(\mN_0 = \mN \setminus \{V^*\}\) and note that if a basis of \(V^*\) is adapted to \(\mN_0\) then it is also adapted to \(\mN\).
	Let \(N_1, \ldots, N_k\) be the maximal elements of \(\mN\) with respect to the partial order given by inclusion.
	Since \(V^* \notin \mN\), we have \(\dim N_i < n\) for all \(i\).
	The sets \(N_1, \ldots, N_k\) are pairwise non comparable, thus they form a direct sum 
	\[
	\bigoplus_{i=1}^k N_i \subset V^* \,.
	\]
	For \(1 \leq i \leq k\), let 
	\[
	\mN_i = \{N \in \mN \,|\, N \subset N_i\} \,.
	\]
	By construction, every \(N \in \mN\) is contained in a subspace \(N_i\) for some \(i\). On the other hand, if a subspace \(N\) is contained in both \(N_i\) and \(N_j\)  with \(i \neq j\), then \(N = \{0\}\). But, by Definition \ref{def:nesteddual}, \(\{0\} \notin \mN\).
	Thus, we have a disjoint union
	\[
	\mN = \bigcup_{i=1}^k \mN_i \,.
	\]
	By Lemma \ref{lem:subnested}\,,
	each \(\mN_i\) is a \(*\)-nested set in \(N_i\). 
	By induction hypothesis, since \(\dim N_i < n\),  we can chose a basis \(B^*_i\) of \(N_i\) adapted to \(\mN_i\).
	Their union \(\cup_i B^*_i\), extended to a basis of \(V^*\) if necessary, is a basis \(B^*\) of \(V^*\) adapted to \(\mN\).
\end{proof}

Next, we recall the notion of transversal intersection of linear subspaces.

\begin{definition}\label{def:transint}
	Let \(L_1, \ldots, L_k\) be linear subspaces of \(V\). We say that \(L_1, \ldots, L_k\) intersect transversely, or that the 
	their common intersection
	\[
	M = \bigcap_{i=1}^k L_i
	\]
	is transversal, if 
	\begin{equation}\label{eq:transversal}
	\codim M = \sum_{i=1}^k \codim L_i \,.		
	\end{equation}
\end{definition}

\begin{example}
	The most familiar case is when \(k=2\). In this case, the subspaces \(L_1\) and \(L_2\) are transversal if and only if \(L_1+L_2 = V\). This follows from the identity \(\dim(L_1 + L_2) = \dim L_1 + \dim L_2 - \dim L_1 \cap L_2\) together with Equation \eqref{eq:transversal}.
\end{example}

\begin{notation}\label{not:ann2}
	Recall that if \(L \subset V\) is a linear subspace, then its annihilator \(L^{\perp}\) is the linear subspace of \(V^*\) made of all linear functions on \(V\) that vanish on \(L\). The map \(L \mapsto L^{\perp}\) defines an inclusion reversing correspondence between linear subspaces of \(V\) and \(V^*\), with \(\dim L^{\perp} = \codim L\), and \((L^{\perp})^{\perp} = L\) under the natural identification \(V^{**} = V\).
\end{notation}

\begin{lemma}\label{lem:transvdirsum}
	Let \(L_1, \ldots, L_k\) be linear subspaces of \(V\). Then their common intersection \(M = \cap_i L_i\) is transversal if and only if \(M^{\perp} = \oplus_i L_i^{\perp}\).
\end{lemma}

\begin{proof}
	This follows from the identity
	\[
	\left( \bigcap_{i=1}^k L_i \right)^{\perp} = \sum_{i=1}^{k} L_i^{\perp}
	\]
	together with Equation \eqref{eq:transversal} and the fact that a sum of linear subspaces \(U = \sum_i U_i\) is direct if and only if \(\dim U = \sum_i \dim U_i\).
\end{proof}

\begin{remark}
Assume  \(L_1, \ldots, L_k\) intersect transversely. Then by  Lemma \ref{lem:transvdirsum} we can choose a basis in \(M^{\perp}\) as a union of bases in \(L_i^{\perp}\). It follows that  \(L_1, \ldots, L_k\) intersect transversely if and only if there are linear coordinates \(x_1, \ldots, x_n\) on \(V\) and pairwise disjoint subsets \(I_i \subset [n]\) for \(1 \leq i \leq k\) such that
\[
L_i = \{x_j = 0 \,|\, j \in I_i\} \,.
\]
In particular, 	if \(L_1, \ldots, L_k\) intersect transversely, then any subset of them also does.
%\(L_{i_1}, \ldots, L_{i_{\ell}}\) also intersect transversely for any \(1 \leq i_1 < \ldots < i_{\ell} \leq k\).	
\end{remark}

Having recalled the necessary background on transversal intersections,
we consider the corresponding dual notion to \(*\)-nested sets.

\begin{definition}\label{def:nestedV}
	A \emph{nested} set in \(V\) is a finite set \(\mS\) of proper linear subspaces \(L \subsetneq V\) such that, for any collection \(L_1, \ldots, L_k \in \mS\) of pairwise non comparable elements, the following holds: 
	\begin{enumerate}[label=\textup{(\roman*)}]
		\item their intersection \(M = \cap_i L_i\) is transversal;
		\item \(M \notin \mS\).
	\end{enumerate}
\end{definition}

Next, we relate Definitions \ref{def:nesteddual} and \ref{def:nestedV}\,.
If \(\mS\) is a collection of linear subspaces \(L \subset V\), then \(\mS^{\perp}\) denotes the corresponding collection of subspaces of \(V^*\) which are annihilators of elements in \(\mS\),
\[
\mS^{\perp} = \{L^{\perp} \,|\, L \in \mS\} \,.
\] 
With this notation,
the correspondence between
nested and \(*\)-nested sets can be stated as follows.

\begin{lemma}\label{lem:nestedstarnested}
	\(\mS\) is a nested set in \(V\) if and only if
	\(\mS^{\perp}\) is a \(*\)-nested set in \(V^*\).
\end{lemma}

\begin{proof}
	Suppose that \(\mS\) is a nested set in \(V\) and let \(\mN = \mS^{\perp}\). We want to show that \(\mN\) is a \(*\)-nested set in \(V^*\).
	First of all, note that the elements \(N \in \mN\) are non-zero subspaces of \(V^*\). Indeed, \(N = L^{\perp}\) with \(L \in \mS\), since the elements of \(\mS\) are proper subspaces \(L \subsetneq V\), their annihilators \(L^{\perp}\) are non-zero. In order to verify Definition \ref{def:nesteddual}, let \(N_1, \ldots, N_k \in \mN\) be pairwise non comparable. We want to show: (i) we have a direct sum \(\oplus_i N_i\) and (ii) \(\oplus_i N_i \notin \mN\).
	
	(i) For each \(1 \leq i \leq k\),
	we can write \(N_i = L_i^{\perp}\) with \(L_i \in \mS\).
	The corresponding elements \(L_1, \ldots, L_k \in \mS\) must also be pairwise non comparable. By item (i) of Definition \ref{def:nestedV}\,, the intersection \(M = \cap_i L_i\) is transversal. By Lemma \ref{lem:transvdirsum}\,, we have a direct sum \(M^{\perp} = \oplus_i N_i\); which proves (i).
	
	(ii) If \(\oplus_i N_i \in \mN\) then \(M = (\oplus_i N_i)^{\perp} \in \mS\). However, by item (ii) of Definition \ref{def:nestedV}\,, \(M \notin \mS\).
	Therefore, \(\oplus_i N_i \notin \mN\); which proves (ii).
	
	The proof that if \(\mS^{\perp}\) is \(*\)-nested then \(\mS\) is nested is similar and we omit it.
\end{proof}

\begin{notation}
	Let \(B = \{e_1, \ldots, e_n\}\) be a basis of \(V\). The dual basis \(B^* = \{e^*_1, \ldots, e^*_n\}\) is the basis of \(V^*\) defined by \(e^*_i(e_i) = 1\) and \(e^*_i(e_j) = 0\) for \(j \neq i\).
\end{notation}

\begin{lemma}\label{lem:dualbasis}
	Let \(\mS\) be a set of linear subspaces of \(V\). A basis \(B = \{e_1, \ldots, e_n\}\) of \(V\) is adapted to \(\mS\) if and only if the dual basis \(B^* = \{e^*_1, \ldots, e^*_n\}\) is adapted to \(\mS^{\perp}\).
\end{lemma}

\begin{proof}
	This follows from the fact that if \(L \subset V\) is the span of \(\{e_i \,|\, i \in I\}\) for a subset \(I \subset [n]\), then \(L^{\perp}\) is the span of \(\{e^*_j \,|\, j \in [n] \setminus I\}\).
\end{proof}

Putting all things together, we arrive at the main result of this section.

\begin{corollary}\label{cor:adapbasis}
	If \(\mS\) if a nested set in \(V\), then there is a basis of \(V\) adapted to \(\mS\).
\end{corollary}

\begin{proof}
	By Lemma \ref{lem:nestedstarnested}\,, \(\mS^{\perp}\) is \(*\)-nested. By Lemma \ref{lem:adapbasistar}\,, there is a basis \(B^*\) of \(V^*\) that is adapted to \(\mS^{\perp}\). By Lemma \ref{lem:dualbasis}\,, the dual basis \(B\) of \(V\) is adapted to \(\mS\).
\end{proof}

\begin{remark}\label{rmk:nouse}
	In this section we haven't made any use of item (ii) in the definition of nested and \(*\)-nested sets. In particular, Corollary \ref{cor:adapbasis} also holds for families of subspaces which don't satisfy this extra condition; but we won't need to use this. 
	
	On the other hand, using item (ii) in Definition \ref{def:nestedV}\,, it is not hard to show that if \(\mS\) is a nested set in \(V\), then \(|\mS| \leq \dim V\). If item (ii) is not satisfied then this is no longer true. For example, if \(L_1\) and \(L_2\) are two distinct lines in \(V =\C^2\) going through the origin, then the set \(\mS = \{\{0\}, L_1, L_2\}\) satisfies item (i) of Definition \ref{def:nestedV} but it does not satisfy item (ii); and \(3 = |\mS| > \dim V = 2\).
\end{remark}

\subsubsection{Nested sets of projective subspaces}\label{sec:nestedproj}

\begin{definition}\label{def:nestedProj}
	A nested set in \(\CP^n\) is
	a finite set \(\mS\) of non-empty and proper projective subspaces \(L \subsetneq \CP^n\) such that, for any collection \(L_1, \ldots, L_k \in \mS\) with \(k \geq 2\) of pairwise non comparable elements (i.e. \(L_i \not\subset L_j\) for \(i \neq j\)), the following holds:
	\begin{enumerate}[label=\textup{(\roman*)}]
		\item their common intersection \(M = \cap_i L_i\) is non-empty and transversal, i.e., 
		\[
		\codim M = \sum_i \codim L_i \,;
		\]
		\item \(M \notin \mS\).
	\end{enumerate}
\end{definition}

\begin{remark}
	If \(\mS\) is a nested set in \(\CP^n\), then 
	\[
	\bigcap_{L \in \mS} L
	\]
	is a non-empty subspace of \(\CP^n\).
\end{remark}

Let
\[
\pi: \C^{n+1} \setminus \{0\} \to \CP^n
\]
be the quotient projection. If \(L \subset \CP^n\) is a projective subspace, then we write \(L^{\bc}\) for the unique linear subspace of \(\C^{n+1}\) such that \(\pi(L^{\bc}) = L\).
If \(\mS\) is a set of projective subspaces, then
\[
\mS^{\bc} = \{L^{\bc} \,|\, L \in \mS\}
\]
is the corresponding set of linear subspaces of \(\C^{n+1}\).

\begin{lemma}\label{lem:sbc}
	If \(\mS\) is a nested set in \(\CP^n\), then \(\mS^{\bc}\) is a nested set in \(\C^{n+1}\). 
\end{lemma}

\begin{proof}
	Immediate consequence of Definitions \ref{def:nestedProj} and \ref{def:nestedV}\,.
\end{proof}

%If there is \(L \in \mS\) such that \(L' \subset L\) for every \(L' \in \mS\) then \(\mS\) is linearly ordered by inclusion.

The proof of the next lemma is straightforward and we omit it.

\begin{lemma}\label{lem:nestproj}
	Let \(V\) be a vector space and let \(\mS\) be a nested set in \(V\). Suppose that \(T\) is a non-zero linear subspace of \(V\) such that \(T \subset L\) for every \(L \in \mS\). Let \(\pr: V \to V / T\) be the quotient projection. Then the following holds.
	\begin{enumerate}[label=\textup{(\roman*)}]
		\item The set \(\mS / T = \{\pr(L) \,|\, L \in \mS\}\) is nested in \(V / T\).
		\item If \(B\) is a basis of \(V\) adapted to \(\mS\) such that \(T = \spn(B \cap T)\), then \(\pr(B) = \{\pr(e) \,|\, e \in B \setminus (T \cap B)\}\) is a basis of \(V / T\) adapted to \(\mS / T\).
	\end{enumerate}
\end{lemma}

Next, we consider frames of vectors of \(T\CP^n\) adapted to a nested set.

\begin{lemma}\label{lem:nestedframe}
	Let \(\mS\) be a nested set in \(\CP^n\). Let \(M\) be the common intersection of all the members of \(\mS\) and let \(p\) be a point in \(M\). Then the following holds. 
	\begin{enumerate}[label=\textup{(\roman*)}]
		\item The set 
		\[
		T_p \mS := \{ T_p L \,|\, L \in \mS
		\}
		\]  
		is a nested set in \(T_p \CP^n\).
		
		\item There is a neighbourhood \(U\) of \(p\) in \(M\) and a holomorphic frame \(\{e_1, \ldots, e_n\}\) of \(T\CP^n|_M\) defined on \(U\) such that, for every \(q \in U\) the basis \(\{e_1(q), \ldots, e_n(q)\}\) of \(T_q\CP^n\) is adapted to \(T_q \mS\).
	\end{enumerate}
\end{lemma}

\begin{proof}
	(i) Let \(\bar{p} \in \C^{n+1}\) such that \(\pi(\bar{p}) = p\).
	The differential \(d\pi_{\bar{p}}: T_{\bar{p}} \C^{n+1} \to T_p \CP^n\) gives us an identification of \(T_p \mS\) with \(\mS^{\bc} / T\) where \(T = \C \cdot \bar{p}\). The statement follows from Lemmas \ref{lem:sbc} and \ref{lem:nestproj}  (i).
	
	(ii) Let \(\{e_1(p), \ldots, e_n(p)\}\) be a basis of \(T_p\CP^n\) adapted to \(T_p \mS\). Let \(\bar{e}_1, \ldots, \bar{e}_n \in T_{\bar{p}}\C^{n+1}\)
	such that \(d\pi_{\bar{p}} (\bar{e}_i) = e_i(p)\) and extend \(\bar{e}_i\) as constant vectors on \(\C^{n+1}\).
	Let \(e_0\) be the Euler vector field in \(\C^{n+1}\) and let \(M^{\bc}\) be the unique linear subspace of \(\C^{n+1}\) such that \(\pi(M^{\bc}) = M\). 
	
	\emph{Claim}: there is a neighbourhood \(\bar{U}\) of \(\bar{p}\) in \(M^{\bc}\) such that for every \(\bar{q} \in \bar{U}\) the vectors \(\bar{B}(\bar{q}) = \{e_0(\bar{q}), \bar{e}_1, \ldots, \bar{e}_n\}\) form a basis of \(\C^{n+1}\) adapted to \(\mS^{\bc}\). 
	
	\emph{Proof of the claim}: the vectors \(\bar{B}(\bar{p})\) form a basis of \(\C^{n+1}\) adapted to \(\mS^{\bc}\). By continuity, the vectors \(\bar{B}(\bar{q})\) form a basis of \(\C^{n+1}\) for \(\bar{q}\) close to \(\bar{p}\). To check that \(\bar{B}(\bar{q})\) is adapted to \(\mS^{\bc}\), we need to show that each \(L^{\bc} \in \mS^{\bc}\) contains \(\dim L^{\bc}\) vectors from \(\bar{B}(q)\). However, the number of vectors of \(\bar{B}(q)\) contained in \(L^{\bc}\) is independent of \(q\) for \(q \in M^{\bc}\) because \(e_0(\bar{q}) \in M^{\bc} \subset L^{\bc}\) and the rest of the vectors \(\bar{e}_i\) with \(i \geq 1\) in \(\bar{B}(q)\) are constant. Since, \(\bar{B}(\bar{p})\) is adapted to \(\mS^{\bc}\), it follows that \(\bar{B}(\bar{q})\) is also adapted to \(\mS^{\bc}\) for all \(\bar{q}\) in a neighbourhood \(\bar{U}\) of \(\bar{p}\) in \(M^{\bc}\). This finishes the proof of the claim.
	
	Let \(U = \pi(\bar{U})\) where \(\bar{U}\) is as in the claim.
	Take an affine hyperplane \(K \subset M^{\bc}\) that goes through \(\bar{p}\) and is transversal to \(e_0(\bar{p})\). This way, for every \(q \in U\) there is a unique \(\bar{q} \in K \cap \bar{U}\) such that \(\pi(\bar{q}) = q\); define \(e_i(q)\) to be the projection of \(\bar{e}_i(\bar{q})\) by \(d\pi_{\bar{q}}\). It follows from Lemma \ref{lem:nestproj} (ii) that \(e_1(q), \ldots, e_n(q)\) is a basis of \(T_q\CP^n\) adapted to \(T_q \mS\). This finishes the proof of (ii).
\end{proof}

\subsection{Exterior algebra}\label{sec:extalg}

Let \(V\) be a finite dimensional complex vector space and let \(\Lambda^rV\) be the \(r\)-th exterior product of \(V\). We refer to elements \(v \in \Lambda^r V\) as \(r\)-vectors or multivectors.

\subsubsection{Multivectors and subspaces}

Let \(S \subset V\) be a linear subspace. Then \(\Lambda^r S\) is naturally embedded in \(\Lambda^r V\) as the subspace spanned by vectors of the form \(v_1 \wedge \ldots \wedge v_r\) with \(v_i \in S\). 

\begin{definition}\label{def:tansub}
	We say that the \(r\)-vector \(v \in \Lambda^r V\) is tangent to \(S\) if \(v\) belongs to the subspace \(\Lambda^r S\).	
\end{definition}

\begin{example}\label{ex:tanvec}
	Let \(v_1, \ldots, v_n\) be a basis of \(V\)  and let \(S\) be the subspace spanned by \(v_1, \ldots, v_s\). Let \(v  = \sum_I c_I v_I\) where \(I = \{i_1, \ldots, i_r\}\) runs over all \(r\)-subsets of \(\{1, \ldots, n\}\) and \(v_I = v_{i_1} \wedge \ldots \wedge v_{i_r}\) with \(i_1 < \ldots < i_r\) are the basis vectors. Then \(v\) is tangent to \(S\) if and only if \(c_I = 0\) for all \(I\) such that \(I \not\subset \{1, \ldots, s\}\).
\end{example}

\begin{remark}\label{rmk:tansub}
	(i) If \(v = 0\) then the tangency condition is trivially satisfied, i.e., \(v\) is tangent to any subspace. (ii) If \(v\) is tangent to \(S\) and \(S \subset H\) then \(v\) is also tangent to \(H\). (iii) If \(v \in \Lambda^rV\) is non-zero and tangent to a subspace \(S\) then \(\dim S \geq r\) because \(\Lambda^rS\) is zero if \(\dim S < r\).
\end{remark}

\begin{definition}
	We say that \(v \in \Lambda^rV\) is decomposable if \(v\) is non-zero and
	tangent to an \(r\)-dimensional subspace \(S \subset V\).
\end{definition}

\begin{example}\label{ex:dec}
	If \(S \subset V\) is an \(r\)-dimensional subspace and
	\(v_1, \ldots, v_r\) is basis of \(S\) then 
	\begin{equation}\label{eq:dec}
		v = v_1 \wedge \ldots \wedge v_r	
	\end{equation}
	is non-zero and tangent to \(S\), thus \(v\) is decomposable. 
	
	Taking different basis of \(S\), say \(v'_1, \ldots, v'_r\), gives a scalar multiple \(v' = \lambda v\) where \(v' = v_1' \wedge \ldots \wedge v_r'\) and \(\lambda\) is the determinant of the change of basis.
\end{example}

The following lemma shows that all decomposable multivectors are of the form given by Example \ref{ex:dec}\,. 

\begin{lemma}\label{lem:decvec}
	Suppose that \(v \in \Lambda^rV\) is decomposable and let \(S\) be an \(r\)-dimensional subspace such that \(v\) is tangent to \(S\).
	\begin{enumerate}[label=\textup{(\roman*)}]
		\item If \(v_1, \ldots, v_r\) is a basis of \(S\) then \(v = \lambda \cdot (v_1 \wedge \ldots \wedge v_r)\) with \(\lambda \in \C^*\).
		
		\item The subspace \(S\) is uniquely determined by \(v\) and it is given by
		\begin{equation}
			S = \ker(\wedge v) = \{u \in V \,\, | \,\, u \wedge v = 0 \} .
		\end{equation}
		
		\item If \(H \subset V\) is a linear subspace and \(v\) is tangent to \(H\) then \(S \subset H\).
	\end{enumerate}
\end{lemma}

\begin{proof}
	(i) Extend \(v_1, \ldots, v_r\) to a basis \(v_1, \ldots, v_n\) of \(V\) and write \(v  = \sum_I \lambda_I v_I\) where \(v_I = v_{i_1} \wedge \ldots \wedge v_{i_r}\) are the basis vectors of \(\Lambda^rV\). Since \(v\) is tangent to \(S\), we have \(\lambda_I = 0\) for all \(I \neq \{1, \ldots, r\}\) by Example \ref{ex:tanvec}\,.
	
	(ii) Take a basis of \(V\) as above. It follows from item (i) that \(S \subset \ker(\wedge v)\). Conversely, if \(u = \sum_{i=1}^{n}\lambda_i v_i\) is such that \(u \wedge v = 0\) then we must have \(\lambda_i=0\) for all \(i >r\), so \(u \in S\).
	
	(iii) If \(v \in \Lambda^rH\) and \(u \notin H\) then \(u \wedge v \neq 0\). Therefore, using (ii), if \(u \in S = \ker(\wedge v)\) we must have \(u \in H\).
\end{proof}

More generally, for any non-zero \(v \in \Lambda^r V\), the dimension of \(\ker(\wedge v)\) is \(\leq r\) and it is \(= r\) precisely when \(v\) is decomposable. The set of decomposable vectors in \(\Lambda^rV\) is the zero locus of a set of homogeneous quadratic equations, known as the Pl\"ucker relations.

If \(S \subset V\) is an \(r\)-dimensional subspace then \(\Lambda^r S\) is a complex line thorough the origin in \(\Lambda^r V\) whose non-zero elements are decomposable \(r\)-vectors. Conversely, a decomposable \(r\)-vector determines a unique \(r\)-dimensional subspace \(S \subset V\) and if two decomposable \(r\)-vectors \(v\) and \(v'\) determine the same subspace then \(v\) and \(v'\) are scalar multiples of each other. This correspondence between \(r\)-dimensional subspaces of \(V\) and the closed subvariety of \(\P(\Lambda^rV)\) of decomposable vectors is known as the Pl\"ucker embedding of the Grassmannian of \(r\)-planes in \(V\).

\subsubsection{Contraction}

\begin{definition}[{\cite[p. 42]{shafarevich}}]
	Let \(V\) be a vector space and let \(\omega \in V^*\). The contraction
	is a linear map \( \omega \iprod: \Lambda^r V \to \Lambda^{r-1}V \)
	defined by the following properties.
	\begin{itemize}
		\item If \(v \in \Lambda^1V = V\) then \(\omega \iprod v = \omega(v)\). 
		\item If \(v_1 \in \Lambda^r V\) and \(v_2 \in \Lambda^s V\) then
		\begin{equation}\label{eq:iprod}
			\omega \iprod (v_1 \wedge v_2) = (\omega \iprod v_1) \wedge v_2 + (-1)^r v_1 \wedge (\omega \iprod v_2) .
		\end{equation}
	\end{itemize}
\end{definition}

The contraction defines a bilinear a map from the direct product  \(V^* \times \Lambda^rV\) to \(\Lambda^{r-1}V\); or equivalently a linear map
\[
V^* \otimes \Lambda^r V \xrightarrow{\iprod} \Lambda^{r-1}V .
\]

\begin{lemma}\label{lem:convnondeg}
	The contraction is a non-degenerate bilinear pairing. More precisely,  if \(v \in \Lambda^rV\) is non-zero then there is \(\omega \in V^*\) such that \(\omega \iprod v \neq 0\). Conversely, if \(\omega \in V^*\) is non-zero then there is \(v \in \Lambda^rV\) such that \(\omega \iprod v \neq 0\).
\end{lemma}

\begin{proof}
	Let \(v_1, \ldots, v_n\) be a basis of \(V\) and write \(v =  \sum_I \lambda_I v_I\) where the sum is over multi-indices \(I=(i_1, \ldots, i_r)\) and \(v_I = v_{i_1} \wedge \ldots \wedge v_{i_r}\) are the basis elements of \(\Lambda^rV\). Let \(\eta_1, \ldots, \eta_n\) be the dual basis defined by \(\eta_i(v_j) = 1\) if \(i=j\) and \(0\) otherwise. Suppose that \(v\) is non-zero. Fix \(I\) such that \(\lambda_I \neq 0\) and let \(i \in I\). Take \(\omega = \eta_i\),  we claim that \(\omega \iprod v \neq 0\). Indeed, if \(I' = I \setminus \{i\}\) then
	\[
	\omega \iprod v = \pm \lambda_I v_{I'} + \sum_{J \neq I'} \tilde{\lambda}_J v_J \neq 0 
	\]
	where the sum runs over multi-indices \(J \neq I'\) with \(|J|=r-1\) which do not contain \(i\) and \(\tilde{\lambda}_J = \pm \lambda_{J\cup \{i\}}\).
	Conversely, if \(\omega \neq 0\) then we can assume that \(\omega = \eta_1\) and
	\[
	\omega \iprod (v_1 \wedge \ldots \wedge v_r) \neq 0 \,.
	\]
    This concludes the proof.
\end{proof}

If \(U \subset V\) is a linear subspace then \(\Lambda^r U\) is naturally embedded in \(\Lambda^r V\) as the subspace spanned by elements of the form 
\(u_1 \wedge \ldots \wedge u_r\) with \(u_i \in U\).

\begin{lemma}\label{lem:tangpreserve}
	Let \(U \subset V\) be a linear subspace and let \(\omega \in V^*\). Then 
	\begin{equation}
		\omega \iprod \Lambda^r U \subset \Lambda^{r-1} U .	
	\end{equation}
\end{lemma}

\begin{proof}
	By linearity, it suffices to consider the contraction of \(\omega\) with decomposable \(r\)-vectors \(u = u_1 \wedge \ldots \wedge u_r\) with \(u_i \in U\). The fact that \(\omega \iprod u \in \Lambda^{r-1} U\) follows from Equation \eqref{eq:iprod} and induction on \(r\). 
\end{proof}

Let \(e \in V\) be a non-zero vector. Let \(V_r^{\circ}\) be the subspace of all \(v \in \Lambda^r V\) such that \(e \wedge v = 0\) and
let \(W^{\circ} \subset V^*\) be the subspace of all \(1\)-forms \(\omega\) with \(\omega(e) = 0\).

\begin{lemma}\label{lem:extalg}
	The following holds.
	\begin{enumerate}[label=\textup{(\roman*)}]
		\item If \(v \in \Lambda^r V\) is non-zero and \(r \geq 2\) then there is \(\omega \in W^{\circ}\) such that \(\omega \iprod v \neq 0\).
		\item If \(v \in V^{\circ}_r\) and \(\omega \in W^{\circ}\) then \(\omega \iprod v \in V^{\circ}_{r-1}\).
		\item If \(v \in V^{\circ}_r\) and \(\omega(e) = 1\) then  \(v' = \omega \iprod v\) satisfies \(v = e \wedge v'\).
	\end{enumerate}
\end{lemma}

\begin{proof}
	(i) We complete the vector \(e\) to a basis of \(V\), say \(v_1=e, v_2, \ldots, v_n\). Let \(\eta_1, \ldots, \eta_n\) be the dual basis of \(V^*\) defined by \(\eta_i(v_j) = \delta_{ij}\). The subspace \(W^{\circ}\) is spanned by \(\eta_2, \ldots, \eta_n\). Write \(v =  \sum_I \lambda_I v_I\) where \(v_I = v_{i_1} \wedge \ldots \wedge v_{i_r}\) are the basis elements of \(\Lambda^r V\) and let \(I\) be such that \(\lambda_I \neq 0\). Since \(|I| = r \geq 2\), we can find an index \(i \in I\) with \(i\neq 1\) and we take \(\omega = \eta_i\). Then \(\omega \in W^{\circ}\) and the argument in the proof of Lemma \ref{lem:convnondeg} shows that \(\omega \iprod v \neq 0\).
	
	(ii) We need to show that \(e \wedge (\omega \iprod v)\) is zero. This follows from \(e \wedge v = 0\) and \(\omega (e) = 0\) together with the identity
	\[
	\omega \iprod (e \wedge v) = \omega(e) v - e \wedge (\omega \iprod v) .
	\]
	
	(iii) Since \(e \wedge v = 0\) and \(\omega (e) = 1\), 
	\begin{equation*}
		0 = \omega \iprod (e \wedge v) = \omega(e) v - e \wedge (\omega \iprod v) = v - e \wedge v' 
	\end{equation*}
    and the statement follows.
\end{proof}

\subsection{Saturated subsheaves}\label{sec:satsubsheaf}

\subsubsection{Subsheaves and subbundles}\label{sec:sheavebundles}

Let \(X\) be a complex manifold. For clarity, in this section \ref{sec:sheavebundles} only we distinguish between vector bundles and locally free sheaves by writing holomorphic vector bundles on \(X\) with straight letters \(E\) and use curly letters \(\mE\) to denote their corresponding sheaves of holomorphic sections. The correspondence \(E \mapsto \mE\)
defines an equivalence between the categories of holomorphic vector bundles on \(X\) and locally free sheaves of \(\mO_X\)-modules.
We recall the induced bijection between morphisms on these two categories. 

Let \(x \in X\). We denote by \(\mE_x\) the stalk of \(\mE\) at \(x\) and by \(E_x\) the fibre of \(E\) at \(x\). The two are related by
\begin{equation}
	E_x = \mE_x / \mathfrak{m}_x \mE_x
\end{equation}
where \(\mathfrak{m}_x \subset \mO_{X, x}\) is the ideal of germs of functions that vanish at \(x\).
If \(E\) and \(F\) are holomorphic vector bundles with sheaves of sections \(\mE\) and \(\mF\) then there is a natural correspondence between homomorphisms of \(\mO_X\)-modules
\[
\varphi \in \Hom_{\mO_X}(\mE, \mF)
\] 
and linear holomorphic maps of vector bundles 
\[
\phi \in H^0(\Hom(E, F))
\] 
where \(\Hom(E, F)\) is the vector bundle whose fibre over \(x\) is the set of linear maps \(E_x \to F_x\). More precisely, an element \(\varphi\) acts on stalks \(\varphi_x : \mE_x \to \mF_x\)
as a linear map of \(\mO_{X, x}\)-modules and, since \(\varphi_x(\mathfrak{m}_x \mE_x) \subset \mathfrak{m}_x \mF_x\), it induces a linear map of \(\C\)-vector spaces \(\phi_x : E_x \to F_x\) 
giving the action on the fibres of \(\phi\). Conversely, an element \(\phi\) acts pointwise on sections to give an element \(\varphi\). The constructions \(\varphi \mapsto \phi\) and \(\phi \mapsto \varphi\) are inverses of each other.

\begin{example}
	It is clear that if \(\varphi_x: \mE_x \to \mF_x\) is surjective then \(\phi_x: E_x \to F_x\) is also surjective. However,
	it can happen that \(\varphi_x\) is injective but \(\phi_x\) is not. For example, if \(\mE = \mO_X\) and \(\mF = \mO_X(D)\) is the locally free sheaf of meromorphic functions on \(X\) with simple poles along a divisor \(D \subset X\). Then the inclusion of \(\mO_X\)-modules \(\mE \subset \mF\) defines an element of \(H^0(\Hom(E, F))\), or equivalently a section of \(F\), that vanishes along \(D\).
\end{example}

We recall the following standard result.
\begin{lemma}\label{lem:constrank}
	Let \(f: E \to F\) be a surjective holomorphic map of vector bundles. Then \(\ker(f)\) is a vector subbundle of \(E\).
\end{lemma}

\begin{definition}
	Let \(\mE\) be a locally free sheaf and let \(\mV \subset \mE\) be a subsheaf. 
	We say that \(\mV\) is a vector subbundle of \(\mE\)
	if \(\mV\) is locally free and the natural map of vector bundles \(V \to E\) is injective.
\end{definition}

\begin{lemma}\label{lem:subbund}
	Let \(\mE\) be a locally free sheaf and let \(\mV \subset \mE\) be a subsheaf such that \(\mE / \mV\) is locally free. Then \(\mV\) is a vector subbundle of \(\mE\).
\end{lemma}

\begin{proof}
	The quotient projection \(\mE \to \mF := \mE / \mV\) corresponds to a surjective map of vector bundles \(f: E \to F\). 
	By Lemma \ref{lem:constrank} \(\ker(f)\) is a vector subbundle of \(E\) and by construction \(\mV\) is the sheaf of sections of \(\ker(f)\). 
\end{proof}

\subsubsection{Saturated subsheaves and subbundles}

\begin{definition}\label{def:saturated}
	Let \(\mE\) be a locally free sheaf and let \(\mV \subset \mE\) be a coherent subsheaf.
	The subsheaf \(\mV\) is saturated if the quotient sheaf \(\mE/\mV\) is torsion-free.
\end{definition}

\begin{proposition}[\cite{okonek}]\label{prop:torsionfree}
	If \(\mF\) is a torsion-free coherent sheaf on \(X\) then there is a closed analytic subset \(Z \subset X\) with \(\codim Z \geq 2\) such that \(\mF\) is locally free on \(X \setminus Z\).
\end{proposition}

\begin{corollary}\label{cor:subbund}
	Let \(\mE\) be a locally free sheaf on \(X\) and suppose that
	\(\mV \subset \mE\) is a saturated subsheaf. Then there exists a closed analytic subset \(Z \subset X\) with \(\codim Z \geq 2\) such that \(\mV\) is a vector subbundle of \(\mE\) on \(X \setminus Z\).
\end{corollary}

\begin{proof}
	By Proposition \ref{prop:torsionfree} there is a closed analytic subset \(Z \subset X\) with \(\codim Z \geq 2\) such that  \(\mE/\mV\) is locally free on \(X \setminus Z\). By Lemma \ref{lem:subbund}\,, \(\mV\) is a vector subbundle of \(\mE\) on \(X \setminus Z\).
\end{proof}

\subsubsection{Determinant line bundle}

Let \(\mE\) be a vector bundle on a complex manifold \(X\) and let \(\mV\) be a vector subbundle of \(\mE\) defined on an open set \(U = X \setminus Z\) where \(Z \subset X\) is a closed analytic subset with \(\codim Z \geq 2\). Let \(\imath: U \to X\) be the inclusion map and let \(r = \rk \mV\). 

\begin{definition}\label{def:det}
	The determinant line bundle of \(\mV\) is the sheaf \(\det \mV\) on \(X\) defined as the double dual of the push-forward of \(\Lambda^r\mV\) by the inclusion map
	\begin{equation}
		\det \mV = (\imath_{*}(\Lambda^r \mV))^{**} .
	\end{equation}
\end{definition}

In particular, on the open set \(U\), the sheaf \(\det(\mV)\) is canonically isomorphic to the line bundle \(\Lambda^r\mV\). The sheaf \(\det(\mV)\) provides a canonical extension of \(\Lambda^r \mV\) to the whole \(X\) with the following desirable property.

\begin{lemma}
	\(\det \mV\) is a line bundle on \(X\).
\end{lemma}

\begin{proof}
	By definition, \(\det\mV\) is a reflexive sheaf of rank \(1\). A reflexive sheaf of rank \(1\) is a line bundle, see \cite[Proposition 1.9]{hartshorne2}.
\end{proof}

\begin{definition}\label{def:c1saturated}
	Let \(\mV\) be a saturated subsheaf of a locally free sheaf. The first Chern class of \(\mV\) is defined as
	\[
	c_1(\mV) = c_1(\det \mV) \,,
	\]
	where \(\det \mV\) is the determinant line bundle of \(\mV\) as in Definition \ref{def:det} \,.
\end{definition}

\subsection{Distributions on \texorpdfstring{\(\CP^n\)}{CPn}}\label{sec:distributions}

A distribution on \(\CP^n\) is a holomorphic vector subbundle \(\mV \subset T\CP^n\) defined outside a closed analytic subset \(Z \subset \CP^n\) of complex codimension \( \geq 2\). 
The rank \(r\) of the subbundle is also called the rank of \(\mV\), we assume that \(1 \leq r \leq n-1\).
The \emph{singular set} of \(\mV\) is the smallest of all such \(Z\)'s, the \emph{regular set} is the complement \(U = \CP^n \setminus Z\).

\begin{definition}\label{def:index}
The index of \(\mV\) is the unique integer \(\imath\) such that there is an isomorphism of line bundles
\begin{equation}
	\det \mV \cong \, \imath \cdot \mO_{\P^n}(1) \,,
\end{equation}
where \(\det \mV\) is as in Definition \ref{def:det}\,.	
\end{definition}

\begin{lemma}\label{lem:indlessrank}
	The index is less or equal than the rank, \(\imath \leq r\).
\end{lemma}

\begin{proof}
	Take a generic line \(P\) contained in the regular set of \(\mV\) that is not tangent to the distribution, in the sense that \(TP \not\subset \mV|_P\). Then \(\mV|_P\) is a subbundle of \(\mO_{\P^1}(2) \oplus \mO_{\P^1}(1)^{\oplus(n-1)}\) that doesn't contain \(\mO_{\P^1}(2)\), so \(\imath = \deg(\mV|_{P}) \leq r\).
\end{proof}

We handle distributions using multivector fields, as detailed next.

\begin{lemma}\label{lem:multvf}
	Let \(\mV\) be a distribution on \(\CP^n\) of index \(\imath\) and rank \(r\). Then there is a multivector field 
	\[
	\bv \in H^0 \left( \Lambda^r T \CP^n \otimes \mO_{\P^n}(-\imath) \right) ,
	\]
	uniquely determined up to scalar multiplication, such that, on the regular set \(U\) of \(\mV\),
	\[
	\mV = \{ w \in T \CP^n \,\, | \,\, w \wedge \bv = 0 \} .
	\]
	In particular, the multivector field \(\bv\) is nowhere zero outside a codimension \(2\) analytic subset of \(\CP^n\).
\end{lemma}

\begin{proof}
	This is a global version of the Pl\"ucker embedding of the Grassmannian.
	Let \(x\) be a point in \(U\) and let
	\(v_1, \ldots, v_r\) be tangent vectors at \(x\) that make a basis of \(\mV_x\). Define
	\begin{equation}
		\bv_x' = (v_1 \wedge \ldots \wedge v_r) \otimes \ell 
	\end{equation}
	where \(\ell \in (\det \mV_x)^*\) is given by \(\ell (v_1 \wedge \ldots \wedge v_r) = 1\). The element \(\bv_x'\) is independent of the choice of basis and varying \(x\) we obtain a nowhere zero holomorphic section \(\bv'\) of \(\Lambda^r T\CP^n \otimes (\det \mV)^*\) defined over \(U\).
	By Hartogs, \(\bv'\) extends across the singular set of \(\mV\) as a holomorphic section on the whole \(\CP^n\).
	
	Fixing an isomorphism \(F: (\det \mV)^* \to \mO_{\P^n}(-\imath)\) we obtain a section \(\bv\) of \(\Lambda^r T\P^n \otimes \mO_{\P^n}(-\imath)\). Taking a different isomorphism \(\widetilde{F}\) produces another section \(\tilde{\bv}\). We can write \(\widetilde{F} = \Phi \circ F\) where \(\Phi\) is an automorphism of the line bundle \(\mO_{\P^n}(-\imath)\). Since \(\CP^n\) is compact, \(\Phi = \lambda \in \C^*\) and therefore \(\tilde{\bv} = \lambda \bv\). 
\end{proof}

\begin{remark}\label{rmk:dualdf}
	It is common to present distributions as kernels of  differential forms, see for example \cite[\S 1.3]{cukiermanpereira}. To get this description, fix a nowhere zero trivializing section \(\Omega\) of \(\Lambda^n T^*\CP^n \otimes \mO_{\P^n}(n+1)\). Let \(\bv\) be as in Lemma \ref{lem:multvf}\,. The contraction of \(\Omega\) with \(\bv\) defines a
	(twisted) differential form
	\[
	\omega = \bv \iprod \Omega \in H^0 \left( \Lambda^{n-r} T^*\CP^n  \otimes \mO_{\P^n}(n+1-\imath) \right) 
	\]
	such that \(\mV = \ker \omega\) on the regular set of \(\mV\).
\end{remark}

\begin{definition}
	The degree \(d\) of \(\mV\) is the difference
	\begin{equation}
		d = r -\imath .
	\end{equation}
	By Lemma \ref{lem:indlessrank}\,, \(d\) is a non-negative integer.
\end{definition}

\begin{remark}[Geometric interpretation of degree]
If \(Q\) is a generic linear subspace of complimentary dimension \(\dim Q = n-r\), then \(d\) is equal to the degree of the hypersurface \(Y \subset Q\) made of points \(x \in Q\) where the subspaces \(\mV_x\) and \(T_x Q\) have non-zero intersection.
To see this take \(\omega\) as in Remark \ref{rmk:dualdf} and note that the pullback of \(\omega\) to \(Q\) defines a non-zero section of \(\mO_{\P^{n-r}}(d)\).
\end{remark}

\subsubsection{Homogeneous multivector fields on \(\C^{n+1}\)}\label{sec:homogmultivec}

We work on \(\C^{n+1}\) with linear coordinates \(x_0, \ldots, x_n\).
For non-negative integers \(d\) and \(r\) we consider the finite dimensional vector space
\[
V_{d, r} = \C_{d}[x_0, \ldots, x_n] \otimes \Lambda^{r} \C^{n+1} 
\]
where \(\C_{d}[x_0, \ldots, x_n]\) is the space of homogeneous polynomials of degree \(d\).
We use the basis of \(\Lambda^{r}\C^{n+1}\) given by the multivectors
\[
\p x_I =  \frac{\p}{\p x_{i_1}} \wedge \ldots \wedge \frac{\p}{\p x_{i_{r}}} 
\] 
where \(I = (i_1, \ldots, i_r)\) is a multi-index with \(0 \leq i_1 < i_2 < \ldots < i_{r} \leq n\). An element \(\bv \in V_{d, r}\) takes the form
\[
\bv = \sum_I a_I \p x_I ,
\]
where the coefficients \(a_I\) are homogeneous polynomials of degree \(d\).

\begin{example}
	The Euler vector field \(e\) is the element in \(V_{1,1} \)  given by
	\begin{equation}
	e = \sum_{i=0}^{n} x_i \frac{\p}{\p x_i} .
	\end{equation}
	More generally, the elements of \(V_{1,1}\) are linear vector fields in \(\C^{n+1}\).
\end{example}

Taking wedge product with \(e\) defines a linear map \(V_{d, r} \xrightarrow{\wedge \, e} V_{d+1, r+1}\). We define \(V^{\circ}_{d, r}\) to be the kernel of \(\wedge \, e\), that is
\begin{equation}
V_{d, r}^{\circ} = \bigl\{ \bv \in  V_{d, r}  \, \text{ such that } \,  \bv \wedge e = 0 
\bigr\} .
\end{equation}
As we shall see, elements in \(V_{d, r}^{\circ}\) correspond to twisted multivector fields on \(\CP^n\). We begin with a preliminary lemma whose proof we omit.

\begin{lemma}[{\cite[Theorem 4.1.3]{hirzebruch}}]\label{lem:hirz}
	An exact sequence of vector bundles \(0 \to L \xrightarrow{\imath_0} V \xrightarrow{\pi} W \to 0\) with \(\rk L = 1\) determines an exact sequence 
	\[
	0 \to \Lambda^r W \otimes L \xrightarrow{\imath_0} \Lambda^{r+1} V \xrightarrow{\pi} \Lambda^{r+1} W \to 0 ,
	\]
	given by \(\pi (v_1 \wedge \ldots \wedge v_{r+1}) = \pi(v_1) \wedge \ldots \wedge \pi(v_{r+1})\) and \(\imath_0 ((w_1 \wedge \ldots \wedge w_r) \otimes \ell) = v_1 \wedge \ldots \wedge v_r \wedge \imath_0(\ell)\) where \(v_i \in V\) are such that \(\pi(v_i) = w_i\).
\end{lemma}

The desired correspondence between twisted multivector fields on \(\CP^n\) and homogeneous multivector fields on \(\C^{n+1}\) that annihilate \(e\) is given by the next.

\begin{proposition}\label{prop:homogmulti}
	Let \(r, d\) be non-negative integers and set \(\imath = r - d\). Then there is a natural linear isomorphism between \(H^0(\Lambda^rT\CP^n \otimes \mO_{\P^n}(-\imath))\) and \(V^{\circ}_{d+1, r+1}\).
\end{proposition}

\begin{proof}
	This is essentially a consequence of Euler's exact sequence. Let 
	\[E = \mO_{\P^n}(1) \otimes \C^{n+1}\] 
	be the vector bundle on \(\CP^n\) equal to the direct sum of \((n+1)\)-copies of \(\mO_{\P^n}(1)\). There is an obvious isomorphism between \(V_{1,1}\) and the space of global sections \(H^0(E)\) given by
	\[
	\sum_{i=0}^{n} \ell_i \frac{\p}{\p x_i} \mapsto (\ell_0, \ldots, \ell_n) \,,
	\]
	where \(\ell_i\) are linear functions on \(\C^{n+1}\).
	Under this isomorphism, the Euler vector field \(e \in V_{1,1}\) corresponds to a nowhere zero section \(e \in H^0(E)\). 
	
	Euler's exact sequence is given by
	\begin{equation}\label{eq:eulerseq}
	0 \to \C \xrightarrow{\imath_0} E \xrightarrow{\pi} T\CP^n \to 0 \,,
	\end{equation}
	where \(\imath_0(\lambda) = \lambda e\) is the inclusion of the trivial line bundle into \(E\) defined by the nowhere zero section \(e \in H^0(E)\) and
	\(\pi\) projects a linear vector field on \(\C^{n+1}\) down to \(\CP^n\) by the differential of the quotient map by scalar multiplication.
	By Lemma \ref{lem:hirz}\, taking exterior power of \eqref{eq:eulerseq} gives us an exact sequence 
	\begin{equation}\label{eq:euler2}
	0 \to \Lambda^r T\CP^n \xrightarrow{\imath_0} \Lambda^{r+1} E \xrightarrow{\pi} \Lambda^{r+1} T\CP^n \to 0 \,,
	\end{equation}
	where the maps \(\imath_0\) and \(\pi\) in \eqref{eq:euler2} act on decomposable vectors as follows:
	\begin{itemize}[leftmargin=*]
		\item if \(w = w_1 \wedge \ldots \wedge w_r \in \Lambda^rT\CP^n\) take \(v_i \in E\) such that \(\pi(v_i) = w_i\) then
		\begin{equation}\label{eq:i0}
		\imath_0(w) = v_1 \wedge \ldots \wedge v_r \wedge e \,;	
		\end{equation} 
		\item if \(v = v_1 \wedge \ldots \wedge v_{r+1} \in \Lambda^{r+1}E\) then
		\begin{equation}\label{eq:pi}
		\pi (v) = \pi(v_1) \wedge \ldots \wedge \pi(v_{r+1}) \,.	
		\end{equation}
	\end{itemize}
	
	Take the tensor product of the short exact sequence of vector bundles \eqref{eq:euler2} with \(\mO_{\P^n}(-\imath)\) to obtain
 
	\begin{equation}\label{eq:euler3}
	0 \to \Lambda^r T\P^n (-\imath) \xrightarrow{\imath_0} \Lambda^{r+1} E (-\imath) \xrightarrow{\pi} \Lambda^{r+1} T\P^n (-\imath) \to 0 \,.
	\end{equation}
 
	Composing \(\Lambda^{r+1} E (-\imath) \xrightarrow{\pi} \Lambda^{r+1} T\CP^n (-\imath)\) with the inclusion \(\Lambda^{r+1} T\CP^n (-\imath) \xrightarrow{\imath_0} \Lambda^{r+2} E (-\imath)\) and using that \(\imath_0 \circ \pi = \wedge \, e\) we obtain an exact sequence
	\[
	0 \to \Lambda^r T\P^n (-\imath) \xrightarrow{\imath_0} \Lambda^{r+1} E (-\imath) \xrightarrow{\wedge \, e} \Lambda^{r+2} E(-\imath) \,,
	\]
	where exactness in the middle term follows from \(\ker \pi = \ker (\imath_0 \circ \pi)\) since \(\imath_0\) is injective. 
	Taking global sections (which is a left exact functor) gives an exact sequence of vector spaces
	\begin{equation}\label{eq:hoimbed}
	0 \to H^0(\Lambda^r T\P^n (-\imath)) \xrightarrow{\imath_0} H^0(\Lambda^{r+1}E (-\imath)) \xrightarrow{\wedge \, e} H^0(\Lambda^{r+2}E (-\imath)) \,.	
	\end{equation}
	Finally, we note that, since \(\imath = r - d\), the space of global sections of
	\[
	\Lambda^{r+1} E \otimes \mO_{\P^n}(-\imath) = \mO_{\P^n}(d+1) \otimes \Lambda^{r+1} \C^{n+1}
	\]
	is naturally isomorphic to \(V_{d+1, r+1}\). Then it follows from Equation \eqref{eq:hoimbed} that 
	\[H^0(\Lambda^r T\CP^n \otimes \mO_{\P^n}(-\imath))\]
	embeds in \(V_{d+1, r+1}\) as the subspace of vectors \(\bv\) such that \(\bv \wedge e =0.\)
\end{proof}

\begin{remark}
	The case \(\imath = r -d = 0\) corresponds to genuine multivector fields on \(\CP^n\) - without any twisting. In this case the induced action \(\C^*\) action on \(V_{r+1, r+1}\) by scalar multiplication is trivial.
\end{remark}

Following next, we combine Lemma \ref{lem:multvf} and Proposition \ref{prop:homogmulti} together to associate a homogeneous multivector field \(\bv\) on \(\C^{n+1}\) to a distribution \(\mV \subset T\CP^n\). In order to state the result we first recall the notion of pullback distribution. 
Let \(X\) and \(Y\) be complex manifolds and suppose that \(f: X \to Y\) is a holomorphic submersion, i.e., the differential \(df_x\) is surjective for all \(x \in X\). In this situation, given a distribution \(\mV\) on \(Y\) we define its pullback \(f^*(\mV)\) as a distribution on \(X\) which is equal to \(df_x^{-1}(\mV_y)\) at all points \(y=f(x)\) that belong to the regular set of \(\mV\).
We consider the case where
\[
f : \C^{n+1} \setminus \{0\} \to \CP^n 
\]
is the quotient map by scalar multiplication.

\begin{corollary}\label{cor:distmult}
	Let \(\mV \subset T\CP^n\) be a distribution of rank \(r\) and degree \(d\).
	Then \(\mV\) defines a unique up to scalar multiplication homogeneous multivector field \(\bv \in V^{\circ}_{d+1, r+1}\) such that
	\(\ker (\wedge \, \bv)\) is equal to the \(\C^*\)-invariant distribution \(f^*(\mV)\) on \(f^{-1}(U)\), where \(U \subset \CP^n\) is the regular set of \(\mV\). In particular, \(\bv\) is nowhere zero and pointwise decomposable outside a codimension \(2\) analytic subset of \(\C^{n+1}\). 
	%We say that \(\bv\) is the multivector field of \(\mV\).  
\end{corollary}

\begin{proof}
	Lemma \ref{lem:multvf} and Proposition \ref{prop:homogmulti} combined together give us a vector \(\bv \in V^{\circ}_{d+1, r+1}\). If \(y \in \CP^n\) belongs to the regular set of \(\mV\) and \(x \in \C^{n+1} \setminus \{0\}\) projects down to \(y\) under the quotient map
	\[
	f : \C^{n+1} \setminus \{0\} \to \CP^n \,,
	\]
	then, unravelling definitions, the value of \(\bv\) at \(x\) is equal (up to a scalar factor) to
	\begin{equation}\label{eq:pullback}
	\bv_x =  \tv_1 \wedge \ldots \wedge \tv_r \wedge e_x \,,	
	\end{equation}
	where \(\tv_i \in T_x \C^{n+1}\) satisfy \(df_x(\tv_i) = v_i\) and \(v_1, \ldots, v_r\) make a basis of \(\mV_y\). It follows from Equation \eqref{eq:pullback} that \(\ker(\wedge \bv_x)\) is equal to \(df_x^{-1}(\mV_y)\)
\end{proof}

%In the context of Corollary \ref{cor:distmult},. Note however that a general element \(\bv \in V^{\circ}_{d+1, r+1}\) does not define a distribution because \(\bv\) must be pointwise decomposable, which impose certain quadratic equations on the coefficients known as the Pl\"ucker relations, see \(\S 4.1\) in \cite[Chapter I]{shafarevich}.

\begin{example}\label{ex:distansubspace}
	Let \(1 \leq r \leq n-1\) and consider the linear subspace \(M\) of \(\CP^n\) of dimension \(r-1\) given by the set of points \([x_0: \ldots :x_n]\) such that
	\[
	\{x_{r} = x_{r+1} = \ldots = x_n = 0 \} \,.
	\]
	The collection of \(r\)-planes that contain \(M\) defines a distribution \(\mV\) of rank \(r\) and degree \(d = 0\). The multivector field \(\bv \in V^{\circ}_{1, r+1}\) of \(\mV\) is given by
	\[
	\bv = \frac{\p}{\p x_0} \wedge \ldots \wedge \frac{\p}{\p x_{r-1}} \wedge e \,.
	\]
\end{example}

\begin{remark}\label{rmk:homog1form}
	Let \(p=n-r\) denote the codimension of a distribution. Then,
	an analogue of Proposition \ref{prop:homogmulti}, shows that  there is a canonical identification between \(\Omega^p_{\P^n}(n+1-\imath)\) (as in Remark \ref{rmk:dualdf}) and the vector space of \(p\)-forms on \(\C^{n+1}\) with homogeneous coefficients of degree \(d+1\) whose contraction with the Euler vector field is identically zero. In particular,
	codimension \(1\) distributions on \(\CP^n\) correspond to homogeneous \(1\)-forms
	\[
	\omega = \sum_{i=0}^{n} a_i dx_i \,,
	\]
	where \(a_i\) are homogeneous polynomials of degree \(d+1\) with \(\sum_i x_i a_i = 0\) and such that the common zero set of the polynomials \(a_i\) has codimension \(\geq 2\).
\end{remark}

\subsubsection{Multivector fields tangent to hyperplanes}\label{sec:tanhyp}

\begin{definition}\label{def:multvftanhyp}
	Let \(X\) be a complex manifold and let \(D \subset X\) be a smooth complex hypersurface. We say that an \(r\)-vector field \(\bv \in H^0(\Lambda^r TX)\) is tangent to \(D\) at \(x\) if  \(\bv_x\) belongs to the linear subspace \(\Lambda^rT_xD \subset \Lambda^rT_xX\)\,. We say that \(\bv\) is tangent to \(D\) if \(\bv\) is tangent to \(D\) at all points \(x \in D\).
\end{definition}

\begin{remark}\label{rmk:dense}
	By continuity, if \(D^{\circ}\) is an open dense subset of \(D\) and \(\bv\) is tangent to \(D^{\circ}\), then \(\bv\) is also tangent to \(D\).
\end{remark}

We are primarily interested in the case that \(X\) is \(\C^{n+1}\) and \(D\) is a complex hyperplane going through the origin. 

\begin{lemma}\label{lem:xjdivide}
	Let \(\bv = \sum_I a_I \p x_I\) be an \(r\)-vector field on \(\C^{n+1}\) with polynomial coefficients \(a_I \in \C[x_0, \ldots, x_n]\). Then \(\bv\) is tangent to the hyperplane \(\{x_j = 0\}\) if and only if \(x_j\) divides \(a_I\) for every \(I\) such that \(j \in I\).
\end{lemma}

\begin{proof}
	By Example \ref{ex:tanvec}, the tangency condition means that the coefficients \(a_I\) vanish along \(\{x_j = 0\}\) for every  \(I\) such that \(j \in I\). On the other hand, since \(a_I\) is a polynomial,  \(a_I\) vanishes along \(\{x_j = 0\}\) if and only if it is divisible by \(x_j\).
\end{proof}

\begin{lemma}\label{lem:posindex}
	Let \(\bv \in V_{d, r}\) be a non-zero homogeneous multivector field on \(\C^{n+1}\). Suppose that
	\(\bv\) is tangent to the \(k\) coordinate hyperplanes \(\{x_j = 0\}\) for \(j \in I_0\) where \(I_0 \subset \{0, \ldots, n\} = [n]\) with \(|I_0|=k\). Then
	\begin{equation}\label{eq:boundd}
		r - (n+1) + k \leq d .
	\end{equation}
\end{lemma}

\begin{proof}
	Write \(\bv = \sum_I a_I \p x_I\), where the sum runs over all \(I \subset [n]\) with \(|I| = r\) and the coefficients \(a_I\) are homogeneous polynomials of degree \(d\).
	
	Fix \(I\) such that \(a_I\) is non-zero. By Lemma \ref{lem:xjdivide}, if \(j \in I \cap I_0\) then \(x_j\) divides \(a_I\). In particular,
	\begin{equation}\label{eq:bound1lem}
		|I \cap I_0| \leq \deg a_I = d.
	\end{equation}
	On the other hand, since at worst \(I\) contains all the \(n+1-k\) elements from \([n] \setminus I_0\)
	\begin{equation}\label{eq:bound2lem}
		|I| - (n+1-k) \leq 	|I\cap I_0| .
	\end{equation}
	Combining inequalities \eqref{eq:bound1lem} and \eqref{eq:bound2lem} gives \eqref{eq:boundd}\,.
\end{proof}

\begin{definition}\label{def:tanv}
Let \(\bv\) be multivector field in \(\C^{n+1}\). We write \(\Tan(\bv)\) for the set of all linear hyperplanes \(H \subset \C^{n+1}\) such that \(\bv\) is tangent to \(H\). 	
\end{definition}

\begin{proposition}\label{prop:posindex}
	Let \(\bv \in V_{d, r}\) be a non-zero homogeneous multivector field on \(\C^{n+1}\). Suppose that \(d < r\) so \(\imath = r - d >0\). Then there is a linear subspace \(M \subset \C^{n+1}\) with \(\dim M \geq \imath\) such that \(M \subset H\) for all \(H \in \Tan(\bv)\).
\end{proposition}

\begin{proof}
	Let \(H_j\) be \(k\) linearly independent hyperplanes in \(\Tan (\bv)\). We prove the proposition by showing that 
	\begin{equation}\label{eq:bound0}
	\imath \leq n+1-k .
	\end{equation}
	
	Without loss of generality, we can assume that \(H_j = \{x_j = 0\}\) for \(j \in I_0\) where \(I_0  = \{0, \ldots, k-1\} \subset  [n]\). By Lemma \ref{lem:posindex}
	\[
	r - (n+1) + k \leq d = r - \imath
	\]
	which is equivalent to Equation \eqref{eq:bound0}\,.
\end{proof}

We say that a distribution \(\mV \subset T\CP^n\) is tangent to a hyperplane \(H \subset \CP^n\) if for every point \(x \in H \cap U\), where \(U\) is the regular set of \(\mV\), we have \(\mV_x \subset T_x H\),
where \(\mV_x \subset T_x \CP^n\) denotes the fibre at \(x\) of the vector subbundle \(\mV|_U\).
We denote by \(\Tan(\mV)\) the collection of all hyperplanes \(H\) such that \(\mV\) is tangent to \(H\).
We note the following.

\begin{lemma}\label{lem:tanhm}
	Let \(\mV\) be a distribution on \(\CP^n\) and let \(\bv \in V^{\circ}_{d+1, r+1}\) be its multivector as in Corollary \ref{cor:distmult}\,. Then a linear hyperplane \(H \subset \C^{n+1}\) belongs to \(\Tan(\bv)\) if and only if its projection \(\P(H) \subset \CP^n\) belongs to \(\Tan(\mV)\). 
\end{lemma}

\begin{proof}
	Let \(U \subset \CP^n\) be the regular set of \(\mV\) and let \(\widetilde{U} = f^{-1}(U)\), where
	\[
	f:  \C^{n+1} \setminus \{0\} \to \CP^n 
	\]
	is the quotient projection.
	Given a linear hyperplane \(H \subset \C^{n+1}\), we write \(H^{\circ} = H \cap \widetilde{U}\), and note that \(H^{\circ}\) is an open dense subset of \(H\). As pointed out in Remark \ref{rmk:dense}, \(\bv\) is tangent to \(H\) if and only if \(\bv_x\) is tangent to \(T_xH\) at all points \(x \in H^{\circ}\), . 
	
	On the other hand, given \(x \in H^{\circ}\), the decomposable multivector \(\bv_x\) is tangent to \(H\) at \(x\) if and only if \(\ker (\wedge \bv_x)\) is contained in \(T_xH\), see Remark \ref{rmk:tansub} and Lemma \ref{lem:decvec}\,. 
	If we let \(\pi = df_x\)\,, then \(\ker (\wedge \bv_x)\) is equal to \(\pi^{-1}(\mV_y)\) where \(y = f(x)\). We get that \(\bv_x\) is tangent to \(T_xH\) if and only if
	\begin{equation}\label{eq:piinverse}
		\pi^{-1}(\mV_y) \subset \pi^{-1} (T_y \P(H)) = T_x H .
	\end{equation}
	Since \(\pi\) is surjective, 
	Equation \eqref{eq:piinverse} is satisfied if and only if 
	\begin{equation}\label{eq:htan}
		\mV_y \subset T_y \P(H) .	
	\end{equation}
	We conclude that \(H \in \Tan(\bv)\) if an only if Equation \eqref{eq:htan} is satisfied at all points \(y \in \P(H) \cap U\), which by definition means that \(\P(H) \in \Tan(\mV)\).
\end{proof}

\begin{definition}\label{def:tr}
Let \(T_r\) be the linear subspace of \(V_{r,r}\) consisting of \(r\)-vector fields on \(\C^{n+1}\), with homogeneous polynomial coefficients of degree \(r\), that are tangent to the \(n+1\) coordinate hyperplanes \(\{x_j = 0\}\).	
Write \(T_r^{\circ} \subset T_r\) be the linear subspace of all \(\bv\) such that \(\bv \wedge e = 0\).
\end{definition}

\begin{lemma}\label{lem:zeroindx0}
	An element \(\bv \in V_{r, r}\) belongs to \(T_r\) if and only if \(\bv\) is of the form
	\begin{equation}\label{eq:tandeg0}
	\bv = \sum_{I = (i_1, \ldots, i_r)} c_I \cdot x_{i_1} x_{i_2} \ldots x_{i_r} \frac{\p}{\p x_{i_1}} \wedge \frac{\p}{\p x_{i_2}} \wedge \ldots \wedge \frac{\p}{\p x_{i_r}}
	\end{equation} 
	where \(c_I \in \C\).
\end{lemma}

\begin{proof}
	Write \(\bv = \sum_I a_I \p x_I\) with \(\deg a_I = r\) and suppose that \(\bv\) belongs to \(T_r\). Let us fix an index \(I = (i_1, \ldots, i_r)\). By Lemma \ref{lem:xjdivide} every \(x_j\) with \(j \in I\) divides \(a_I\). Since \(\deg a_I = r\) we must have that \(a_I = c_I x_{i_1} x_{i_2} \ldots x_{i_r}\) with \(c_I \in \C\). Conversely, if \(\bv \in V_{r, r}\) is of the form given by Equation \eqref{eq:tandeg0} then by Lemma \ref{lem:xjdivide} \(\bv\) is tangent to the coordinate hyperplanes and hence belongs to \(T_r\).
\end{proof}

\begin{example}
	The space \(T_1\) is the set of all linear vector fields \(\bv\) on \(\C^{n+1}\) that are tangent to the coordinate hyperplanes. These are precisely the vector fields of the form
	\[
	\bv = \sum_{i=0}^{n} c_i x_i \frac{\p}{\p x_i}
	\]
	with \(c_i \in \C\).
\end{example}

Let \(W\) be the \((n+1)\)-dimensional vector space made of meromorphic \(1\)-forms \(\omega\) of the form given by
\[
\omega = \sum_{i=0}^{n} \lambda_i \frac{dx_i}{x_i} 
\]
with \(\lambda_i \in \C\).

\begin{lemma}\label{lem:tr}
	If \(\omega \in W\) and \(\bv \in T_r\), then \(\omega \iprod \bv\) belongs to \(T_{r-1}\).
\end{lemma}

\begin{proof}
	Let \(j \in [n]\) and \(I = \{i_1, \ldots, i_r\} \subset [n]\).
	Write \(x_I = x_{i_1} x_{i_2} \ldots x_{i_r}\). Then
	\begin{equation}\label{eq:pairing}
	\frac{dx_j}{x_j} \iprod x_I \p x_I = \begin{cases}
	0 & \text{ if } j \notin I , \\
	\pm \, x_{I'} \p x_{I'} & \text{ if } j \in I ,
	\end{cases}
	\end{equation}
	where \(I' = I \setminus \{j\}\).
	The statement follows from Lemma \ref{lem:zeroindx0} and linearity.
\end{proof}

Consider the linear subspace  \(T_r^{\circ} \subset T_r\)  of all \(\bv \in T_r\) such that \(\bv \wedge e = 0\). Under the isomorphism of Proposition \ref{prop:homogmulti}, in the case \(\imath = 0\), the vector space \(T^{\circ}_{r+1}\) is identified with the subspace of \(H^0(\Lambda^rT\CP^n)\) of multivector fields  tangent to the coordinate hyperplanes. Let \(W^{\circ} \subset W\) be the linear subspace of all \(\omega\) such that \(\omega(e) = 0\). More explicitly, if \(\omega = \sum_{i=1}^{n} \lambda_i dx_i / x_i\) then \(\omega \in W^{\circ}\) if \(\sum_{i=0}^{n} \lambda_i = 0\). We can think of elements in \(W^{\circ}\) as meromorphic \(1\)-forms on \(\CP^n\) with simple poles at the coordinate hyperplanes.

\begin{lemma}\label{lem:zeroindxcontr}
	The following holds:
\begin{enumerate}[label=\textup{(\roman*)}]
	\item If \(\bv \in T_r\) is non-zero and \(r \geq 2\) then there is \(\omega \in W^{\circ}\) such that \(\omega \iprod \bv \neq 0\).
	\item If \(\bv \in T^{\circ}_r\) and \(\omega \in W^{\circ}\)  then \(\omega \iprod \bv \in T^{\circ}_{r-1}\).
	\item If \(\bv \in T^{\circ}_r\) and \(\omega (e) = 1\) then \(\bv' = \omega \iprod \bv\) satisfies \(e \wedge \bv' = \bv\).
\end{enumerate}
\end{lemma}

\begin{proof}
This Lemma follows from  \ref{lem:extalg} once we introduce a number of identifications.
We identify \(T_r \cong \Lambda^r \C^{n+1}\) by means of the	linear isomorphism 
\begin{equation}
\sum c_I \p x_I \mapsto \sum c_I x_I \p x_I.
\end{equation}
Similarly, we identify \(W \cong (\C^{n+1})^*\) by
\begin{equation}
\sum \lambda_i dx_i \mapsto \sum \lambda_i dx_i/x_i .
\end{equation}
Under these identifications, the pairing \(W \otimes T_r \xrightarrow{\iprod} T_{r-1}\) agrees with the usual contraction \((\C^{n+1})^* \otimes \Lambda^r\C^{n+1} \xrightarrow{\iprod} \Lambda^{r-1}\C^{n+1}\), as follows from Equation \eqref{eq:pairing}, and similarly for wedge products.
The subspace \(T^{\circ}_r\) is identified with the subspace of all \(v \in \Lambda^{r} \C^{n+1}\) such that \(v \wedge e = 0\) where \(e\) is the non-zero vector \(e = (1, \ldots, 1) \in \C^{n+1}\) and \(W^{\circ}\) is the subspace of all \( \omega \in (\C^{n+1})^*\) such that \(\omega(e) = 0\).	
\end{proof}

Recall that \(\Tan(\bv)\) denotes the set of all linear hyperplanes \(H \subset \C^{n+1}\) such that the multivector field \(\bv\) is tangent to \(H\). 
\begin{proposition}\label{prop:zeroindex}
	Let \(\bv\) be a non-zero element of \(T_r^{\circ}\) with \(r \geq 2\). Then there is a linear vector field \(\bv' \in T_1\) such that \(e \wedge \bv' \neq 0\) and \(\Tan(\bv) \subset \Tan(\bv')\).
\end{proposition}

\begin{proof}
	By repeated applications of items (i) and (ii) of Lemma \ref{lem:zeroindxcontr}, we can find \(1\)-forms \(\omega_i \in W^{\circ}\) for \(1 \leq i \leq r-2\)
	such that
	\[
	\tilde{\bv} = \omega_{r-2} \iprod \left( \omega_{r-3} \iprod \left( \ldots (\omega_1 \iprod X) \ldots \right) \right)
	\]
	is a non-zero element in \(T^{\circ}_2\). If \(r=2\) we simply take \(\tilde{\bv} = \bv\). 
	
	Let \(\omega \in W\) be such that \(\omega(e) = 1\), say \(\omega = dx_0 / x_0\), and define \(\bv' \in T_1\) as
	\[
	\bv' = \omega \iprod \tilde{\bv} .
	\]
	By item (iii) of Lemma \ref{lem:zeroindxcontr},
	\[
	e \wedge \bv' = \tilde{\bv} \neq 0 .
	\]
	
	We are left to show that \(\Tan(\bv) \subset \Tan(\bv')\). Notice that, since \(\bv' \in T_1\), the coordinate hyperplanes are contained in \(\Tan(\bv')\). Now, let \(H\) be a hyperplane in \(\Tan(\bv)\) different from the coordinate hyperplanes and take a point \(x \in H\) outside the origin and not lying in any of the coordinate hyperplanes. The multivector \(\bv'_x\) is obtained by successive contractions of \(\bv_x\) by elements in \((T_x \C^{n+1})^*\). Since \(\bv_x\) is tangent to \(H\), it follows from Lemma \ref{lem:tangpreserve} that \(\bv'_x\) is tangent to \(H\). Therefore, \(H \in \Tan(\bv')\) and the proposition follows.
\end{proof}

%\begin{proof}
%	Suppose that \(\mV\) is tangent to the hyperplanes \(\{x_j = 0\}\) for all \(j \in I_0\) where \(I_0\) is a subset of the index set \([n] = \{0, \ldots, n\}\). We prove the proposition by showing that 
%	\begin{equation}\label{eq:bound0}
%		|I_0| \leq n+1-\imath .
%	\end{equation}

%	Let \(X = \sum_I a_I \p x_I\) be the multivector of \(\mV\), where the sum runs over all subsets \(I \subset [n]\) with \(|I| = r+1\) and \(a_I\) are homogeneous polynomials of degree \(d+1\).  By Lemma \ref{lem:vanishing}, if \(j \in I \cap I_0\) then \(x_j\) divides \(a_I\). In particular,
%	\begin{equation}\label{eq:bound1}
%		\deg a_I = d + 1 \geq |I \cap I_0| .
%	\end{equation}

%	On the other hand, since \(|I| \leq |I\cap I_0| + |I_0^c|\) where \(I_0^c\) is the complement of \(I_0 \subset [n]\), we obtain
%	\begin{equation}\label{eq:bound2}
%		|I\cap I_0| \geq |I| - (n+1-|I_0|)
%	\end{equation}
%	Combining inequalities \eqref{eq:bound1} and \eqref{eq:bound2} and replacing \(d=r-\imath\) we obtain
%	\[
%	r - \imath + 1 \geq (r+1) - (n+1) + |I_0|
%	\]
%	which is equivalent to Equation \eqref{eq:bound0}.
%\end{proof}

\newpage
\section{Essential and irreducible configurations}\label{app:essirr}

%In the next Sections \ref{sec:git} and \ref{app:essirr} we work with  configurations of points. The correspondence with hyperplane arrangements is straightforward, by taking annihilators, and made explicit in Section \ref{sec:backarr}\,.

A configuration of points \(\mP = \{p_1, \ldots, p_m\}\) in \(\CP^n\) is \emph{essential} if
\[
\sum_{i=1}^m p_i = \CP^n \,.
\]
Equivalently, \(\mP\) is essential if it contains \(n+1\) linearly independent points.
We say that \(\mP\) is \emph{reducible} if there are two disjoint linear subspaces \(U\) and \(V\) such that
\[
\mP \subset U \cup V
\]
with both \(U \cap \mP\) and \(V \cap \mP\) non-empty.
We say that \(\mP\) is \emph{irreducible} if it is not reducible.
The main result of this section is the following.

\begin{proposition}\label{prop:stable}
	If the configuration of points \(\mP\) is essential and irreducible, then there is a weight vector \(\ba \in \R_{>0}^{m}\) such that \((\mP, \ba)\) is stable.
\end{proposition}

A hyperplane arrangement \(\mH \subset \CP^n\) corresponds to a configuration of points \(\mP \subset (\CP^n)^*\). 
It is straightforward to verify that: \(\mH\) is essential (irreducible) if and only if \(\mP\) is essential (irreducible).
Proposition \ref{prop:stable} together with Lemma \ref{lem:stabHP}\,, give us the following.

\begin{corollary}\label{cor:essirrst}
	If the hyperplane arrangement \(\mH \subset \CP^n\) is essential and irreducible, then there is a weight vector \(\ba \in \R^{\mH}_{>0}\) such that \((\mH, \ba)\) is stable.
\end{corollary}

\begin{remark}
Proposition \ref{prop:stable} follows from the fact that the \emph{matroid polytope} (see Section \ref{sec:matpol}) of \(\mP\) has dimension \(|\mP| - 1\) precisely when \(\mP\) is essential and irreducible. See \cite[Theorem 1.12.9]{borovikgelfand} and \cite[p. 698]{schrijver}. For completeness, we present a self-contained proof.
\end{remark}

We will use the following notion.
\begin{definition}
	We say that \(\mB \subset \mP\) is a \emph{basis} if:
	\begin{enumerate}[label=\textup{(\roman*)}]
		\item \(\sum\limits_{p \in \mB} p = \CP^n\)\,;
		\item \(|\mB| = n+1\)\,.
	\end{enumerate}
\end{definition}

\begin{remark}
	If \(\mP\) is essential then it contains at least one basis.
\end{remark}

We shall need the following elementary result, whose proof we omit.

\begin{lemma}\label{lem:sumdim}
	Let \(P\) and \(Q\) be linear subspaces of \(\CP^n\) such that \(P + Q = \CP^n\) and \(P \cap Q \neq \emptyset\). Then \(\dim P + \dim Q \geq n\)\,.
\end{lemma}

The key result we need to prove Proposition \ref{prop:stable} is the next. 

\begin{lemma}\label{lem:basis}
	Suppose that \(\mP\) is essential and irreducible. Then, for every \(U\) in the poset \(\mU\), there is a basis \(\mB \subset \mP\) such that
	\begin{equation}\label{eq:basisbound}
		| \mB \cap U | < \dim U + 1 \,.
	\end{equation}
\end{lemma}

\begin{proof}
	Let \(U \in \mU\). Consider the projective subspace \(Q \subset \CP^n\) given by
	\begin{equation*}
		Q = \spn \, \{\, p \in \mP \,|\, p \notin U  \} \,. 
	\end{equation*}
	Since \(\mP\) is essential, 
	\begin{equation}\label{eq:lperplusq}
		U + Q = \CP^n \,.
	\end{equation}
	In particular, since \(U \subsetneq \CP^n\), the subspace \(Q\) is non-empty.
	Choose a linearly independent subset 
	\begin{equation}\label{eq:mq}
		\mQ \subset \{\, p \in \mP \,|\, p \notin U  \} 	
	\end{equation}
	with \(|\mQ| = \dim Q + 1\) that spans the subspace \(Q\).
	Since \(\mP\) is essential,
	we can extend \(\mQ\) to a basis \(\mB \subset \mP\). It follows from \eqref{eq:mq} that \(\mB \cap U \subset \mB \setminus \mQ\)\,, hence
	\begin{equation}\label{eq:bcaplperp}
		\begin{aligned}
			|\mB \cap U| &\leq | \mB \setminus \mQ| \\ &=
			n + 1 - |\mQ| = n - \dim Q \,.
		\end{aligned}		
	\end{equation}
	
	On the other hand,
	\[
	\mP \subset U \cup Q
	\]
	and both \(U \cap \mP\), \(Q \cap \mP\) are non-empty.
	Since \(\mP\) is irreducible, we must have
	\begin{equation}\label{eq:lpercapq}
		U \cap Q \neq \emptyset \,.
	\end{equation}
	It follows from Equations \eqref{eq:lperplusq} and \eqref{eq:lpercapq} together with
	Lemma \ref{lem:sumdim}  that
	\begin{equation}\label{eq:dimQ}
		n - \dim Q \leq \dim U \,.
	\end{equation}
	It follows from Equations \eqref{eq:bcaplperp} and \eqref{eq:dimQ} that
	\[
	| \mB \cap U | \leq \dim U
	\]
	which is equivalent to Equation \eqref{eq:basisbound}.
\end{proof}

\emph{The semi-stable and stable cones.}
Let \(m = |\mP|\) and label the points, say 
\[
\mP = \{p_1, \ldots, p_m\} \,.
\]
Let \(s\) be the linear function on \(\R^m\) equal to the total sum of the components,
\[
s(\ba) = \sum_{i=1}^m a_i \,,
\]
where \(\ba = (a_1, \ldots, a_m) \in \R^m\).
Similarly, for an element \(U\) in the sum poset \(\mU\) of \(\mP\), let \(s_U\) be the linear function on \(\R^m\) given by
\[
s_U(\ba) =  \sum_{i \, | \, p_i \in U} a_i \,.
\]
For \(U \in \mU\), let \(f_U\) be the linear function on \(\R^m\) given by
\begin{equation}\label{eq:fU}
	f_U(\ba) = \frac{\dim U + 1}{n+1} \cdot s(\ba) - s_U(\ba) \,.
\end{equation}

\begin{definition}
	The semi-stable cone \(C\) is the closed convex polyhedral cone in \(\R^m\) obtained as intersection of the half-spaces \(\{f_U \geq 0\}\) for \(U \in \mU\) together with \(\{a_i \geq 0\}\),
	\[
	C = \R^m_{\geq 0} \cap
	\left( \bigcap_{U \in \mU} \{f_U \geq 0\} \right)  \,.
	\]
	The stable cone is the interior of the semi-stable cone:
	\[
	C^{\circ} = \R^m_{>0} \cap
	\left( \bigcap_{U \in \mU} \{f_U > 0\} \right)  \,.
	\]
\end{definition}

\begin{lemma}\label{lem:stcone}
	The weighted configuration of points \((\mP, \ba)\) is stable if and only if \(\ba\) belongs to the stable cone.
\end{lemma}

\begin{proof}
	For \(U \in \mU\), Equation \eqref{eq:stU} is equivalent to \(f_U > 0\). The statement then follows from Lemma \ref{lem:stU}\,.
\end{proof}

Since \(C\) is a cone defined by linear inequalities, we have the following.

\begin{lemma}\label{lem:sumss}
	If \(\ba, \bb \in C\) then \(\ba + \bb \in C\).
\end{lemma}

\begin{comment}
	\begin{proof}
		By assumption, \(\ba, \bb \in \R^{m}_{\geq 0}\) with \(f_U(\ba) \geq 0\) and \(f_U(\bb) \geq 0\) for all \(U \in \mU\). Clearly \(\ba + \bb \in \R^m_{\geq 0}\) and
		\[
		f_U(\ba + \bb) = f_U(\ba) + f_U(\bb) \geq 0 \,;
		\]
		showing that \(\ba + \bb \in C\).
	\end{proof}
\end{comment}

Let \(\mB \subset \mP\) be a basis. The indicator function of \(\mB\) is the vector \(\be_{\mB} \in \R^m\)  with components given by
\[
(\be_{\mB})_i =
\begin{cases}
	1 &\text{ if } p_i \in \mB \,, \\
	0 &\text{ if } p_i \notin \mB \,.
\end{cases}
\]

\begin{lemma}\label{lem:delta}
	Let \(\mB \subset \mP\) be a basis and let \(U \in \mU\). Then
	\begin{equation}\label{eq:fLdelta}
		f_U(\be_{\mB}) = \dim U + 1 - | \mB \cap U | \,.
	\end{equation}
	In particular, the following holds:
	\begin{enumerate}[label = \textup{(\roman*)}]
		\item \(f_U(\be_{\mB}) >0\) if \( |\mB \cap U| < \dim U + 1\)\,;
		\item 	\(f_U(\be_{\mB}) \geq 0\) \,.
		\item  \(\be_{\mB} \in C\) \,;
	\end{enumerate}
\end{lemma}

\begin{proof}
	Note that
	\begin{equation}\label{eq:sdelta}
		s(\be_{\mB}) = |\mB| = n+1
	\end{equation}
	and
	\begin{equation}\label{eq:sLdelta}
		s_U (\be_{\mB}) = |\mB \cap U| \,. 
	\end{equation}
	Equation \eqref{eq:fLdelta} follows from \eqref{eq:fU} together with Equations \eqref{eq:sdelta} and \eqref{eq:sLdelta}.
	
	Items (i) and (ii) follow from Equation \eqref{eq:fLdelta} together with the fact that, since the elements of \(\mB\) are linearly independent,
	\[
	| \mB \cap U | \leq \dim U + 1 \,.
	\]
	Item (iii) holds since (ii) holds for all \(U\).
\end{proof}

\begin{lemma}\label{lem:posweightss}
	If \(\mP\) is essential then \(C \cap \R^m_{>0}\) is non-empty.
\end{lemma}

\begin{proof}
	Since \(\mP\) is essential, for any \(p_i \in \mP\) we can find a basis \(\mB_i \subset \mP\) with \(p_i \in \mB_i\). The vector
	\[
	\ba_0 = \sum_{i=1}^m \be_{\mB_i}
	\]
	belongs to \(\R^m_{>0}\). It follows from Lemma \ref{lem:delta} (iii) and Lemma \ref{lem:sumss}\,, that \(\ba_0 \in C\).
\end{proof}

\begin{proof}[Proof of Proposition \ref{prop:stable}]
	Start with \(\ba_0 \in C \cap \R_{> 0}^{m}\) as given by Lemma \ref{lem:posweightss}. For each \(U \in \mU\), let \(\mB_U \subset \mP\) be a basis with \(|\mB_U \cap U| < \dim U + 1\) as provided by Lemma \ref{lem:basis}\,. Set
	\[
	\ba = \ba_0 + \sum_{U \in \mU} \be_{\mB_U} \,.
	\]
	Clearly, \(\ba \in \R_{>0}^m\).
	It follows from Lemma \ref{lem:delta} (i) and (ii) that \(f_U(\ba) > 0\) for all \(U \in \mU\). Therefore, \(\ba \in C^{\circ}\). By Lemma \ref{lem:stcone}\,, \((\mP, \ba)\) is stable.
\end{proof}

%\subsubsection{Back to arrangements.}\label{sec:backarr}

\newpage
\section{Miyaoka-Yau and K\"ahler-Einstein metrics}\label{app:ke}

For compact K\"ahler manifolds,
the Miyaoka-Yau inequality follows from the existence of K\"ahler-Einstein metrics together with the Chern-Weil formula
\begin{equation}\label{eq:myidentity}
	\left( 2(n+1) c_2(X) - n c_1(X)^2 \right) \cdot [\omega_{\KE}]^{n-2} = \lambda_n \int_X |\mathring{\textup{Riem}}|^2 \cdot \omega_{\KE}^n \,\, ,	
\end{equation}
where \((X^n, \omega_{\KE})\) is a compact K\"ahler-Einstein manifold,
\(\mathring{\textup{Riem}}\) is the trace free part of the Riemann curvature tensor of \(\omega_{\KE}\), and \(\lambda_n\) is a positive constant depedning on dimension. As a consequence, if equality holds in the Miyaoka-Yau inequality, then \(\omega_{\KE}\) must have constant holomorphic sectional curvature. 

The Miyaoka-Yau inequality has been extended to the setting of log pairs \((X, \Delta)\); see for instance \cite{langer}, \cite{li}, and  \cite{guenancia}\,. However, a general version of the Miyaoka-Yau inequality for ktl pairs that characterizes the equality case through the existence of constant holomorphic sectional curvature metrics with conical singularities is still missing.
In this appendix we show that, under the hypothesis of Theorem \ref{thm:main}\,, the pairs \((\CP^n, \Delta)\) with \(\Delta = \sum a_H \cdot H\)\,, admit \emph{weak} Ricci-flat K\"ahler metrics with prescribed singularities at the hyperplanes. To state this precisely, we introduce some notation.
%as detailed in Proposition \ref{prop:weakrf} later in this section. Conjecturally, an analogue of the Chern-Weil formula \eqref{eq:myidentity} should hold for these singular metrics. 

Let \((\mH, \ba)\) be a weighted arrangement of hyperplanes \(H \subset \CP^n\) and let \((\CP^n)^{\circ}\) be the arrangement complement.
Let \(\omega_{\FS}\) be the Fubini-Study metric on \(\CP^n\) with \(\omega_{\FS} \in 2\pi \cdot c_1(\mO_{\P^n}(1))\) and let \(|\cdot|\) be the usual Hermitian metric on \(\mO_{\P^n}(1)\) with curvature \(-i \cdot \omega_{\FS}\). 
Consider the function 
\begin{equation}
	f =  \prod_{H \in \mH} |\ell_H|^{-2a_H} \,
\end{equation}
where \(\ell_H\) are sections of \(\mO_{\P^n}(1)\)  with \(H = \{\ell_H=0\}\). The main result of this appendix is the next.

\begin{proposition}\label{prop:weakrf}
	Suppose that the weighted arrangement \((\mH, \ba)\) is klt and CY. Then there is a Ricci-flat K\"ahler metric \(\omega_{\RF}\) on the arrangement complement \((\CP^n)^{\circ}\) whose volume form is proportional to \(f \cdot \omega_{\FS}^n\). Moreover, we can write 
    \[
    \omega_{\RF} = \omega_{\FS} + i \p \bar{\p} \varphi \,,
    \] 
    where \(\varphi\) extends continuously to \(\CP^n\). 
\end{proposition}

To prove Proposition \ref{prop:weakrf}\,, we need to establish some preliminary lemmas.
Write \(L^p = L^p(\CP^n)\) for the Lebesgue space of measurable functions \(u\) on \(\CP^n\) such that 
\[
\int_{\CP^n} |u|^p \cdot \omega_{\FS}^n < \infty \,.
\] 

\begin{lemma}
	If the klt condition \eqref{eq:klt} is satisfied then \(f \in L^p\) for some \(p > 1\). 
\end{lemma}

\begin{proof}
	Let \(\mL\) be the collection of non-empty and proper subspaces \(\emptyset \neq L \subsetneq \CP^n\) obtained as intersection of members of \(\mH\). Let
	\[
	p^* = \min_{L \in \mL} \frac{ \codim L}{\sum\limits_{H | L \subset H} a_H} \,.
	\]
	The klt condition implies that \(p^* > 1\). Fix any \(1 < p < p^*\). We will prove that \(f \in L^p\).
	
	Consider the stratification of \(\CP^n\) given by the elements
	of \(\mL\). Specifically, let
	\[
	K_i = \bigcup_{L \in \mL^{n-i}} L
	\]
	where \(\mL^{n-i}\) is the subset of \(L \in \mL\) with \(\codim L = i\). E.g. \(K_1\) is union of all hyperplanes, \(K_2\) is the union of all codimension \(2\) subspaces and so on.
	Since every \(L \in \mL\) is contained in another \(L' \in \mL\) with \(\dim L' = \dim L + 1\), we have a decreasing filtration
	\[
	K_1 \supset K_2 \supset \ldots \supset K_n \supset K_{n+1} = \emptyset .
	\]
	The complements \(\Omega_i = \CP^n \setminus K_{i+1}\) form an increasing sequence of open sets
	\[
	\Omega_0 \subset \Omega_1 \subset \ldots \subset \Omega_n = \CP^n
	\]
	where \(\Omega_0 = (\CP^n)^{\circ}\) is the arrangement complement and 
	\[
	\Omega_i \setminus \Omega_{i-1} = K_i \setminus K_{i+1} = \bigcup_{\codim L = i} L^{\circ} 
	\]
	where the union is taken over all \(L \in \mL\) of codimension \(i\) and 
	\[
	L^{\circ} = L \setminus \bigcup_{H | L \not\subset H} (L \cap H) .
	\]
	
	\emph{Claim:}  \(f \in L^p_{\loc}(\Omega_i)\) for all \(0 \leq i \leq n\), where \(L^p_{\loc}(\Omega_i)\) denotes the space of locally \(L^p\) functions on \(\Omega_i\). \emph{Proof of the claim:} by induction on \(i\). For \(i=0\) the function \(f\) is smooth on \(\Omega_0\) and the statement is obvious. Let \(1 \leq i \leq n\) and assume that \(f \in L^p_{\loc}(\Omega_i)\). Let \(p \in \Omega_i \setminus \Omega_{i-1}\). We want to show that there is an open set \(U\) that contains \(p\) such that \(f \in L^p(U)\). By construction of the stratification, \(p \in L^{\circ}\) for some \(L \in \mL\) with \(\codim L = i\). We can assume that if \(H \in \mH\) intersects \(U\) then \(H \supset L\). Take linear coordinates \(z_1, \ldots, z_n\) centred at \(p\), so that \(L = \{z_1 = \ldots = z_i = 0\}\).
	Up to multiplication by a smooth positive factor, we can assume that
	\[
	f = \prod_{L \in \mH_L} |\ell_H|^{-2a_H}
	\]
	where \(\ell_H\) are linear functions on \(z_1, \ldots, z_i\) and \(|\cdot|\) is the usual absolute value on \(\C\). By Fubini's theorem it is enough to show that \(f\) is locally in \(L^p\) in a neighbourhood of the origin of the \(\C^i\) factor transversal to \(L\). At the same time, \(U \setminus L \subset \Omega_{i-1}\) and by induction hypothesis \(f \in L^p_{\loc}(U \setminus L)\).
	In particular, the integral of \(|f|^p\) on a sphere about \(0 \in \C^i\) is finite. Taking spherical polar coordinates on \(\C^i\) and by homogeneity of \(f\), we see that \(|f|^p\) is locally integrable around \(0 \in \C^i\) if and only if
	\begin{equation}\label{eq:lp}
		\int_{0}^{1} r^{-2 \cdot p \cdot a_L} r^{2i-1} dr < \infty
	\end{equation}
	where \(a_L = \sum\limits_{H | L \subset H} a_H\). Clearly, Equation \eqref{eq:lp} holds if and only if
	\[
	p \cdot a_L < i
	\]
	which is guaranteed for our choice of \(p\). This finishes the proof of the claim.
	
	The claim for \(i=n\) implies that \(f \in L^p_{\loc}(\CP^n)\). On the other hand, since \(\CP^n\) is compact, we have \(L^p_{\loc}(\CP^n) = L^p(\CP^n)\).
\end{proof}

We say that two volume forms \(dV_1\) and \(dV_2\) on a manifold are proportional, if they are equal up to a constant positive factor. We write this as \(dV_1 \propto dV_2\).

\begin{lemma}\label{lem:rf}
	Suppose that \(\omega\) is a K\"ahler metric on \((\CP^n)^{\circ}\) with
	\[
	\omega^n \propto f \cdot \omega_{\FS}^n \,.
	\]
	If the CY condition \eqref{eq:cy} is satisfied then \(\Ric(\omega) = 0\).
\end{lemma}

\begin{proof}
	The identity \(i \bar{\p} \p \log (\omega^n / \omega_{\FS}^n) = \Ric(\omega) - \Ric(\omega_{\FS})\) together with 
	\[
	\Ric(\omega_{\FS}) = (n+1) \cdot \omega_{\FS}
	\] 
	imply that
	\begin{equation}\label{eq:ric1}
		\Ric(\omega) = (n+1) \cdot \omega_{\FS} + i \bar{\p}\p \log f \,.
	\end{equation}
	On the other hand, since \(i \bar{\p} \p \log |\ell_H|^2 = \omega_{\FS}\) on \(\{\ell_H\neq 0\}\), we have
	\begin{equation}\label{eq:ric2}
		i \bar{\p}\p \log f = -  \left(\sum_{H \in \mH} a_H \right) \cdot \omega_{\FS} \,.	
	\end{equation}
	Equations \eqref{eq:ric1} and \eqref{eq:ric2} give us
	\[
	\Ric(\omega) = \left( n+ 1 - \sum_{H \in \mH} a_H \right) \cdot \omega_{\FS} = 0 ,
	\]
	where the second equality holds because of the CY condition.
\end{proof}

Next, we finish the proof of Proposition \ref{prop:weakrf}.

\begin{proof}[Proof of Proposition \ref{prop:weakrf}]
	Since \(f \in L^p\) for some \(p>1\) and \(L^p \subset L^1\), the integral of \(f\) on \(\CP^n\) is finite. Let \(F = C \cdot f\) where \(C\) is a positive constant such that
	\[
	\int_{\CP^n} F \cdot \omega_{\FS}^n = \int_{\CP^n} \omega_{\FS}^n \,.
	\]
	Consider the complex Monge-Amp\`ere equation
	\begin{equation}\label{eq:ma}
		\left( \omega_{\FS} + i \p\bar{\p} \varphi \right)^n = F \cdot \omega_{\FS}^n
	\end{equation}
	for an unknown real valued function \(\varphi\).
	By \cite[Theorem 2.4.2 and Example 2 in p.91]{kolodziej}  there is a continuous solution \(\varphi\) to \eqref{eq:ma}\,. On the arrangement complement, \(\varphi\) is smooth and \(\omega_{\RF} = \omega_{\FS} + i \p\bar{\p} \varphi\) is a Ricci-flat K\"ahler metric satisfying the required conditions.
\end{proof}

\begin{remark}
By uniqueness of solutions to \eqref{eq:ma}, our Conjecture \ref{conjecture} implies that if \(Q(\ba) = 0\) then the Ricci-flat metric \(\omega_{\RF}\) is actually flat. 
Heuristically, an extension of the Chern-Weil formula \eqref{eq:myidentity} to metrics with conical singularities, should express
the `energy' or \(L^2\)-norm of the Riemann curvature tensor of \(\omega_{\RF}\) as a function of the weights \(E(\ba)\) so that \(\omega_{\RF}\) is flat precisely when \(E(\ba) = 0\). However, little is known on the behaviour of \(\omega_{\RF}\) near the support of \(\mH\), see \cite{dbs} for the case \(n=2\). In particular, it is not known whether \(|\textup{Riem}(\omega_{\RF})|^2\) is locally integrable.
%If the arrangements \(\mH_L / L\) are stable for all \(L \in \mLi^{n-2}\), then \cite[\S 6]{dbs} suggests that \(E(\ba)\) is equal to \(-Q(\ba)\) up to a constant dimensional factor.    
\end{remark}

\newpage
\section*{List of Symbols}
\addcontentsline{toc}{section}{List of Symbols}

\small{
\begin{longtable}{l l}
	\textbf{Symbol} & \textbf{Description} \\
	\hline 
	% a
	\(a_L\) & weight at the irreducible subspace \(L\), Definition \ref{def:weightaL} \\
	
	% b
	\(B_H\) &   Definition \ref{def:bh} \\
	\(b_L\) &  polarization coefficients, Notation \ref{not:bL} \\
	
	% c
    \(C\) & semistable cone, Definition \ref{def:semistabcone} \\
	\(C^{\circ}\) & stable cone, Definition \ref{def:stabcone}  \\
	\(\gamma_L\) &  Poincar\'e dual of \(D_L\), Notation \ref{not:gammaL} \\
	\((\CP^n)^{\circ}\) & arrangement complement, Section \ref{sec:habasicdef} \\
	
	% d
	\(D_L\) & irreducible component of \(D = \pi^{-1}(\mH)\) with \(\pi(D_L) = L\), Theorem \ref{thm:conproc} \\
	
	% e
	\(\mE\) & pullback tangent bundle, Section \ref{sec:pardef} \\
	\(\mE_{*}\) & parabolic bundle on \((X, D)\), Definition \ref{def:parstr} \\
	
	% F
	% filtrations
	\(F^i_a\) & filtration of \(\mE|_{D_i}\) by vector subbundles, Definition \ref{def:parbun} \\
	\(F_{\ba}\) & intersection of a tuple of filtrations, Notation \ref{not:Fa} and Notation \ref{not:Fa2} \\
	
	% G
	\(\Gr^i_a\) & graded component, Definition \ref{def:gr} \\
	\(\Gr^{i, j}_{a, b}\) & Definition \ref{def:grij} \\
	
	% H
	\(h\) &  generator of \(H^2(\CP^n, \Z)\) equal to  \(c_1(\mO_{\P^n}(1))\), Notation \ref{not:h} \\
	
	\(\mH\) &  an arrangement of hyperplanes \(H \subset \CP^n\) \\
	
	\((\mH, \ba)\) & weighted arrangement, Definition \ref{def:weightedarr} \\

	\(\mH^L\) & induced arrangement, Section \ref{sec:habasicdef} \\
	\(\mH_L\) & localization of \(\mH\) at \(L\), Section \ref{sec:habasicdef} \\
	
	\(H \pitchfork \mV\) & hyperplane transverse to a distribution, Definition \ref{def:HtransV} \\
	
	% i
	\(\Irr(L)\) & irreducible components of \(L \in \mL\), Notation \ref{not:IrrL} \\

	% L
	\(L^{\bc}\) &  linear subspace in \(\C^{n+1}\) that projects to \(L \subset \CP^n\), Section \ref{sec:habasicdef} \\
	\(L^{\circ}\) & complement of the induced arrangement \(\mH^L\)\,, Section \ref{sec:habasicdef}  \\
	\(L^{\perp}\) & annihilator of \(L\), Notation \ref{not:ann1} and Notation \ref{not:ann2} \\

	\(\mL\) &  poset of non-empty and proper intersections of \(\mH\), Section \ref{sec:habasicdef} \\
	\(\mLi\) & non-empty and proper irreducible subspaces of \(\mH\), Notation \ref{not:Lirr}  \\
	\(\mLi^{\circ}\) & non-empty irreducible subspaces of codimension \(\geq 2\), Notation \ref{not:Lirrcirc} \\
	\(L_1 \pitchfork L_2\) & reducible intersection of two irreducible subspaces, Notation \ref{not:transint} \\

	% m
	\(m_L\) & multiplicity of \(L\), Section \ref{sec:habasicdef} \\
	
	%n
	\(N\) & number of hyperplanes, \(N = |\mH|\) \\
	% O
	\(O(k^{n-2})\) & Notation \ref{not:On} \\
	
	%p
	\(P\) & matroid polytope, Definition \ref{def:matpol} \\
	\(P_k\) & polarization on \(X\), Lemma \ref{lem:polarization} \\
	
	% q
	\(Q\) & quadratic form of \(\mH\), Definition \ref{def:quadraticform} \\

	%s
	\(s\) & sum of all weights \(a_H\), Equation \eqref{eq:s} \\
	\(s_H\) & sum of weights on the induced arrangement, Equation \eqref{eq:sh} \\

	% t
	\(\Tan(\mV)\) &  set of all hyperplanes tangent to a distribution \(\mV \subset \CP^n\), Definition \ref{def:disttanhyp} \\

	% general
	%\([n]\) &  set of integers \(\{1, \ldots, n\}\) \\

	% resolution
	\(X\) & minimal De Concini-Procesi model -or resolution- of \(\mH\), Definition \ref{def:resolution} \\
	\hline &
\end{longtable}
}
\clearpage

\addcontentsline{toc}{section}{References}
\bibliographystyle{alpha}
\bibliography{refs}

\Address

\end{document}